%% file: main.tex
\documentclass[10pt]{article}

\usepackage[plainpages=false,pdfpagelabels,colorlinks=true,linkcolor=mypurple,urlcolor=mypurple,citecolor=mypurple]{hyperref}

\usepackage{tptstyle}

\usepackage{fullpage}
\usepackage[font={footnotesize,it}]{caption}
\numberwithin{equation}{section}

\newtoggle{vaos}
\togglefalse{vaos}
\newcommand{\aosversion}[2]{\iftoggle{vaos}{#1}{#2}}

\newcommand{\cG}{{\mathcal{G}}}
\newcommand{\cF}{{\mathcal{F}}}
\newcommand{\cE}{{\mathcal{E}}}
\newcommand{\sig}[1]{{\mathrm{sig}(#1)}}

\definecolor{revisecolor}{HTML}{ae1908}

\begin{document}

\title{Causality Pursuit from Heterogeneous Environments via Neural Adversarial Invariance Learning}

\author{Yihong Gu$^1$
~~~~ Cong Fang$^2$
~~~~ Peter B{\"u}hlmann$^3$
~~~~ Jianqing Fan$^1$\thanks{
J. Fan was supported by ONR grant N00014-25-1-2317 and NSF grants DMS-2053832, DMS-2210833, and DMS-2412029. Y. Gu was supported by the Charlotte Elizabeth Procter Honorific Fellowship from Princeton University. C. Fang was supported by the NSF China No.s 92470117.}\\
 $^1$Department of Operations Research and Financial Engineering, Princeton University \\
$^2$School of Intelligence Science and Technology, Peking University \\
$^3$Seminar for Statistics, ETH Z{\"u}rich}

\date{This version: November 2025}
\maketitle

\begin{abstract}
\input{abstract.tex}
\end{abstract} 

\newpage

\input{text_main}

\section*{Acknowledges}
The authors would like to thank the anonymous referees, an Associate Editor, and the Editor for their constructive comments that improved the quality and accessibility of this paper. We thank Yiran Jia for helpful discussions on presenting a generic identification result on SCM using the unified graph including $E$, Yimu Zhang for the help with the numerical implementation for robust prediction application, and Xinwei Shen for drawing our attention to Gumbel approximation in the implementation.

\newpage
\renewcommand{\contentsname}{Table of Contents}
\tableofcontents

\appendix
\newpage
\begin{center}
    \Large Supplement to ``Causality Pursuit from Heterogeneous Environments via Neural Adversarial Invariance Learning''
\end{center}

\section*{Organization of the Supplemental Material}
\input{text_supp}


\bibliographystyle{apalike2}
\bibliography{main.bib}

\end{document}

%% file: abstract.tex
Pursuing causality from data is a fundamental problem in scientific discovery, treatment intervention, and transfer learning. This paper introduces a novel algorithmic method for addressing nonparametric invariance and causality learning in regression models across multiple environments, where the joint distribution of response variables and covariates varies, but the conditional expectations of outcome given an unknown set of quasi-causal variables are invariant. The challenge of finding such an unknown set of quasi-causal or invariant variables is compounded by the presence of endogenous variables that have heterogeneous effects across different environments. The proposed Focused Adversarial Invariant Regularization (FAIR) framework utilizes an innovative minimax optimization approach that drives regression models toward prediction-invariant solutions through adversarial testing. Leveraging the representation power of neural networks, FAIR neural networks (FAIR-NN) are introduced for causality pursuit.  It is shown that FAIR-NN can find the invariant variables and quasi-causal variables under a minimal identification condition and that the resulting procedure is adaptive to low-dimensional composition structures in a non-asymptotic analysis.  Under a structural causal model, variables identified by FAIR-NN represent pragmatic causality and provably align with exact causal mechanisms under conditions of sufficient heterogeneity. Computationally, FAIR-NN employs a novel Gumbel approximation with decreased temperature and a stochastic gradient descent ascent algorithm. The procedures are demonstrated using simulated and real-data examples.

%% file: text_main.tex
\section{Introduction}

A fundamental problem in statistics and machine learning is to use collected data to predict the response variable $Y$ based on explanatory covariates $X\in \mathbb{R}^d$.  The objective often centers on estimating the regression function $m_0(x) = \mathbb{E}[Y|X=x]$, which minimizes the population $L_2$ risk $\mathsf{R}(m) = \int |y - m(x)|^2 \mu_0(dx, dy)$, starting from the pioneering work of least squares by
\cite{legendre1805leastsquares} and \cite{gauss1809leastsquares}. 
The problem of achieving sample-efficient estimation of $m_0$ has been extensively studied, and there are many methods that attempt to exploit a low-dimensional structure such as sparsity, low-rankness, or additivity, and develop corresponding optimal methods tailored to this assumed structure \citep{hastie2009elements,wainwright2019high, fan2020statistical}. However, these methods may suffer from model misspecification due to their reliance on imposed structures. As an alternative, algorithmic methods \citep{breiman2001statistical} like neural networks can be adaptive to the low-dimensional structure efficiently \citep{schmidt2020nonparametric, fan2024factor} with no supervision of function structure. This nature endows them with universal applicability across various tasks and data. 

Despite many celebrated efforts for the efficient estimation of $m_0$ or its variants like quantile functions, the ultimate goal of statistical learning is to predict on unseen data, elucidate the causal relationships among variables, and guide decision-making in real-world scenarios. We instinctively regard $m_0$ as such a target function for achieving decent prediction and causal attribution. However, this can be flawed: $m_0$ can produce unstable predictions on unseen data, and we risk false scientific conclusions in numerous cases. Consider a simple thought experiment where we aim to classify an object in a picture as either a cow $(Y=1)$ or a camel $(Y=0)$ using two provided features $X_1$ (body shape) and $X_2$ (background color). In the data we collected from $\mu_0$, the cows usually appear on green grass, while camels often stay on yellow sand. Consequently, the conditional expectation $m_0(x_1, x_2) = \mathbb{E}_{\mu_0}[Y|X_1=x_1, X_2=x_2]$ would be heavily dependent on $x_2$. Such a model is problematic both for attribution and prediction in an unseen environment. Its application in a setting with a different background, such as zoos, would lead to unreliable predictions. Furthermore, attributing the determination of an object to the background surrounding it also contradicts our understanding of causality. In the above case, we may prefer $m_\star(x)=\mathbb{E}[Y|X_1=x_1]$ for prediction and attribution as we know the causal mechanisms.

We refer to the above problem as the ``{\it curse of endogeneity}'', namely, the conditional expectation of the residual for the ``potential'' interested (causal) $m_\star$ is not zero given all the explanatory variables, i.e., $\mathbb{E}[Y-m_\star(X)  |X]\neq 0$. Such a problem will lead to a misalignment between $m_0$ and $m_\star$, i.e., $m_0(X)-m_\star(X)\neq 0$. Hence, traditional regression techniques for estimating $m_0$ will result in an unsatisfactory solution.    

Causal inference methods offer remarkable remedies to the curse of endogeneity. Based on the potential outcome \citep{rubin1974estimating} or structural causal model (SCM) \citep{pearl2009causality}, efficient estimation via various regression techniques \citep{chernozhukov2018double, athey2019generalized} is possible. However, all these methods rely on relatively strong assumptions that are often untestable from data. This, in turn, leads to a high risk of severe misspecification of models and assumptions.

This paper proposes an algorithmic remedy for the ``curse of endogeneity'' taking advantage of data from multiple sources and a high-level invariance principle. Motivated by causal discovery under the SCM framework \citep{peters2016causal}, the invariance principle argues that causal relations remain constant across different environments from multiple sources. Leveraging this invariance principle, we propose an algorithmic framework that estimates the most predictive association, which we refer to as \emph{data-driven causality}, that is invariant across diverse environments. Methodologically and in contrast to previous work, our framework is nonparametric and assumption-lean, making it scalable and robust to model misspecification. From a statistical viewpoint, our estimator requires a minimal number of environments and achieves optimal sample complexity. Furthermore, our approach identifies the causal structure in the setting of an SCM under minimal assumptions of heterogeneity across different environments.

\subsection{The Canonical Model under Study}
\label{sec:intro-problem}

Consider the following multi-environment regression problem. Let $\mathcal{E}$ be the set of environments. For each environment $e\in \mathcal{E}$,  $n$ i.i.d. data $\{(X_i^{(e)},Y_i^{(e)})\}_{i=1}^n$ are drawn from $\mu^{(e)}$ -- the joint distribution of $(X^{(e)},Y^{(e)})$ satisfying
\begin{align}
\label{eq:intro-model}
    Y^{(e)} = m^\star(X_{S^\star}^{(e)}) + \varepsilon^{(e)} \qquad \text{with} \qquad \mathbb{E}[\varepsilon^{(e)}|X_{S^\star}^{(e)}] \equiv 0.
\end{align} 
Here $S^\star$, the unknown true important variable set, and $m^\star: \mathbb{R}^{|S^\star|} \to \mathbb{R}$, the target regression function, are both {\it invariant} across different environments; but the joint distributions $\mu^{(e)}$ can vary. We aim to learn the set of important variables $S^\star$ and estimate the {\it invariant regression function} $m^\star$ using data $\{\{(X_i^{(e)}, Y_i^{(e)})\}_{i=1}^n\}_{e\in \mathcal{E}}$ from $|\mathcal{E}|$ heterogeneous environments. The same $n$ in the formulation is just for expository simplicity; the extension to varying $n^{(e)}$ is straightforward. We refer to the above problem as \emph{nonparametric invariance pursuit}.  

Now we temporarily refrain from causal discussions and frame it as a pure statistical estimation problem. We will use a running example in \cref{sec:1.3} right after introducing our method to provide a causal interpretation of $S^\star$ and then offer in \cref{sec:ident-fairnn} a rigorous and comprehensive interpretation of what $S^\star$ is in the SCM with interventions on $X$. It is also notable to mention that model \eqref{eq:intro-model} only requires invariance in the first moment instead of full distributional invariance, i.e., $\varepsilon^{(e)} \sim F_\varepsilon$ and independent of $X_{S^\star}^{(e)}$, as typically required for causal discovery \citep{peters2016causal}. It is more realistic and allows for between-environment heteroscedastic errors. 

It is important to note that the standard nonparametric regression generally diverges from our target $m^\star$, i.e., $\mathbb{E}[Y^{(e)}|X^{(e)}=x] \neq m^\star(x_{S^\star})$. This mismatch is due to $\mathbb{E}[\varepsilon^{(e)}|X^{(e)}] \neq 0$. Such a ``curse of endogeneity'' problem is the main challenge we need to address. Including even one of the endogenously spurious variables, e.g., $X_2$ background color in the above example, in the regression function will create an inconsistent estimation of $m^\star$.  Thus, it is essential to design an algorithm to eliminate all endogenously spurious variables.

\subsection{Our Algorithmic Remedy: FAIR Estimation}\label{sec1.2}

This paper proposes a unified estimation framework -- the {\it Focused Adversarial Invariance Regularized (FAIR)} estimator. It regularizes the user-specified risk loss $\ell(y, v)$ by a novel regularizer. Specifically, the FAIR estimator is the solution of the following minimax optimization program
\begin{align}
\label{eq:intro-obj}
    \min_{g\in \mathcal{G}} \max_{ \substack{f^{(e)} \in \mathcal{F}_{S_g} \\ \forall e\in \mathcal{E}}} \underbrace{\sum_{e\in \mathcal{E}} \mathbb{E}_{\mu^{(e)}}\left[\ell(Y, g(X))\right]}_{\mathsf{R}(g)} + \gamma \underbrace{\sum_{e\in \mathcal{E}}\mathbb{E}_{\mu^{(e)}}\left[\{Y-g(X)\} f^{(e)}(X) - \{f^{(e)}(X)\}^{2}/2\right]}_{\mathsf{J}(g, \{f^{(e)}\}_{e\in \mathcal{E}})}.
\end{align} 
Here $\ell(\cdot,\cdot)$ is a loss whose population solution leads to the conditional expectation, $\gamma>0$ is the regularization hyper-parameter to be determined, $(\mathcal{G}, \mathcal{F})$ are the function classes to be specified by the user satisfying $\mathcal{G} \subseteq \mathcal{F}$. The first part is the risk minimization, and the second component is the test of exogeneity of the variables $S_g = \mathrm{supp}(g)$ used by the regression function $g$, where $\mathcal{F}_{S_g} = \{f \in \mathcal{F}: f(x) = h(x_{S_g})\text{ for some }h: \mathbb{R}^{|S_g|}\to \mathbb{R}\}$ is the testing function class for the prediction functions in $\mathcal{G}$ that only ``focuses'' on the variables $S_g$ that $g$ used. Two useful classes of functions are linear and square-integrable classes for $(\mathcal{G}, \mathcal{F})$, which correspond respectively to linear models and nonparametric regression models; see \cref{sec:model} for additional details.  Note that the second component is nonnegative after maximization by comparing with $f^{(e)} = 0$ so that the penalty is nonnegative. For the empirical counterpart, we solve a similar minimax optimization program that substitutes $\mathbb{E}_{\mu^{(e)}}[\cdot]$ with the corresponding sample means.

To see why such a FAIR penalty works, let us consider the nonparametric regression setting in which $\mathcal{F} = \{f: \mathbb{E}_{\mu^{(e)}}[f^2(X_{S_g})] < \infty\}$. By conditioning on $X_{S_g}$, for $f^{(e)} \in \mathcal{F}_{S_g}$,
$$
\mathbb{E}_{\mu^{(e)}}\left[\{Y-g(X)\} f^{(e)}(X) \right ] =
\mathbb{E}_{\mu^{(e)}} \left [ \left \{ \mathbb{E}_{\mu^{(e)}}[ Y|X_{S_g}] - g(X) \right \} f^{(e)}(X) \right ].
$$
Then, the supremum in \eqref{eq:intro-obj} can be explicitly found and the objective now becomes
\begin{align}
    \min_{g\in \mathcal{G}} \mathsf{R}(g) + \gamma \cdot \mathsf{J}^\star (g) ~~~  \text{with} ~~~ \mathsf{J}^\star (g) = \frac{1}{2} \sum_{e\in \mathcal{E}}\mathbb{E}_{\mu^{(e)}}\left[\left|g(X) - \mathbb{E}_{\mu^{(e)}}[Y|X_{S_g}]\right|^2\right]. \label{fan1}
\end{align}
Therefore, $g(X) = m^\star(X_{S^\star})$ is a minimax solution.

To motivate \eqref{eq:intro-obj}, let us first consider the additional constraint $\mathbb{E}_{\mu^{(e)}}[f^{(e)}(X_{S_g})^2] = 1$ so that the first part of the second component in \eqref{eq:intro-obj} is basically the maximal correlation between the residual $\{Y-g(X_{S_g})\}$ and testing functions $f^{(e)}(X_{S_g})$. Hence, the criterion \eqref{eq:intro-obj} is to find a set of variables $X_{S_g}$ as exogenous (weakly correlated) with the residuals as possible for all testing functions in $\mathcal{F}_{S_g}$.  By the Lagrange multiplier method, the constrained maximization problem can be written as
$$
    \max_{ f^{(e)} \in \mathcal{F}_{S_g} } \mathbb{E}_{\mu^{(e)}}\left[\{Y-g(X)\} f^{(e)}(X) - \lambda \{f^{(e)}(X)\}^{2}\right].
$$
Choosing the multiplier $\lambda = 1/2$ gives rise to the objective function \eqref{eq:intro-obj}.

FAIR penalty screens out all the endogenously spurious variables when $\gamma$ is sufficiently large.  This is easily seen when the penalty in \eqref{eq:intro-obj} is not zero, such a $g$ is dominated by $g=m^\star$ for large $\gamma$.  After deleting endogenously spurious variables, we can apply the commonly-used statistical variable selection methods \citep{hastie2009elements,wainwright2019high, fan2020statistical} to further eliminate exogenously spurious or weak causal variables such as the time and temperature at which the photos were taken.

\begin{figure}
\centering
\includegraphics[width=0.9\textwidth]{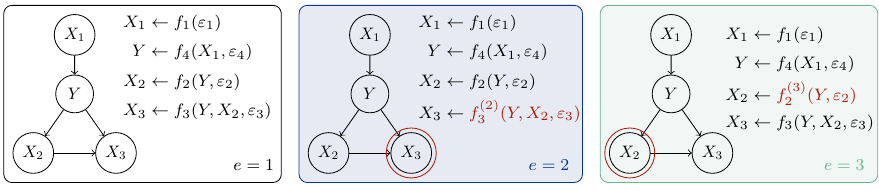}
\caption{The running example when $d=3$ and $|\mathcal{E}|=3$. The arrow from node $x$ to $y$ indicates that $x$ affects $y$ directly. The data-generating process of $(X_1,\ldots, X_3, Y)$ in each environment is described by the set of assignments in each panel, and $\varepsilon_1,\ldots,\varepsilon_4$ are independent noises. Compared with the first environment $e=1$, the assignment for $X_3$ perturbs in $e=2$ and the assignment for $X_2$ perturbs in $e=3$, which are marked by \myred{red}. }
\label{fig:intro}
\end{figure}

\subsection{A Running Example and the Roadmap}
\label{sec:1.3}

Here we use an example to demonstrate the philosophy of our method, namely, to describe the target regression function of our method when data from different environments are observed. Leveraging this example, we also illustrate the key idea of the main theoretical results and offer an overview of this paper.

Let us use the running example with $d=3$ in \cref{fig:intro} to illustrate the causal interpretation of $S^\star$ which our FAIR estimation pursues. The data-generating processes of $(X, Y)$ in the environments are described by the SCMs shown in \cref{fig:intro}: for example, the data-generating process of the first environment is described by the four assignments in the left panel, where $\{f_j\}_{j=1}^4$ are arbitrary nonparametric functions and $\varepsilon_1,\ldots, \varepsilon_4$ are some independent noises. Here, the presentation of the assignments and the cause-effect relationship is for illustration, our algorithm is blind to this knowledge. 

When only data from the first environment is observed, the standard least squares will use all the variables to predict $Y$ if $\{f_j, \varepsilon_j\}_{j=1}^4$ are ``\emph{nondegenerate}''. This is because besides $X_1$ (direct cause of $Y$), both $X_2$ and $X_3$ can help predict the noise $\varepsilon_4$, excluding some ``degenerate'' cases that rarely happen; for example, when $f_2(y,\varepsilon_2) = h(y) + \varepsilon_2$ and $f_3(y,x_2,\varepsilon_3)=h(y) - x_2 + \varepsilon_3 = -\varepsilon_2 +  \varepsilon_3$, only $X_1$ and $X_2$ contribute to prediction; see a formal definition of the ``\emph{nondegenerate}'' cases in \cref{sec:ident-fairnn}. The FAIR estimation will pursue the same $S^\star=\{1,2,3\}$ as the invariance constraint trivially holds when $|\mathcal{E}|=1$. 

Things will be different when data from the second environment $e=2$ becomes accessible. Here, the change of assignment of $X_3$ (\cref{fig:intro} middle panel), or the intervention on $X_3$, will make the conditional moment invariance no longer hold for any variable set containing $X_3$ under nondegenerate cases if such an intervention is also ``\emph{nondegenerate}''. Therefore, the FAIR method will pursue the \emph{maximum invariant set} $S^\star=\{1,2\}$. We use the name ``maximum invariant set'' because it is the most predictive set that preserves the conditional moment invariance constraint. The set $S^\star=\{1,2\}$ includes both the direct cause of $Y$ ($X_1$) and the effect of $Y$ ($X_2$), and this is the best we can get from the currently available data. In this case, all the sets $\{1,2\}, \{1\}, \{2\}, \emptyset$ preserve the invariant structure. The rule we follow is to pick one from the four candidates based on the following question: given available data from $\mathcal{E}=\{1,2\}$, what is the best prediction model from a pragmatic perspective? A model including $X_3$ is not robust -- because we have observed the perturbations of association (i.e., the ``non-invariance'') when $X_3$ is included in the prediction model, the adversarial effect of $X_3$ can make the prediction very bad in an unknown future environment. Here, making predictions on $\{X_1,X_2\}$ may be the best choice. This is because if we hold the belief that in the future, the interventions are made within $X_3$, then the association between $X_{S^\star}=X_{\{1,2\}}$ and $Y$ will be maintained and is the most predictive one among all the maintaining associations. Therefore, the maximum invariant set FAIR estimation pursues can be interpreted as either contemporary direct causes (the candidate direct causes that haven't been falsified) or pragmatic direct causes (for pragmatic considerations in future predictions).

Finally, when we observe additional data from environment $e=3$ (\cref{fig:intro} right panel), the maximum invariant set $S^\star$ will match the exact direct causes $X_1$ under this model. As a comparison, when $\mathcal{E}=\{1,2,3\}$ are observed, the standard least squares, group distributional robust optimization procedure \citep{meinshausen2015maximin, duchi2021learning, sagawa2019distributionally,agarwal2022minimax}, and IRM \citep{arjovsky2019invariant} will produce prediction models using all the variables; and the previous hypothesis-test based procedure from nonlinear ICP \citep{heinze2018invariant} will result in the null prediction because $\emptyset$ is also a set maintaining invariant structure.

We also remark that it is possible to recover the direct causes $X_1$ non-trivially when only one environment is observed. But it is at the cost of imposing additional structural assumptions, for example, assuming $f_4$ is linear \citep{fan2014endogeneity}. This can be implemented in our framework by restricting $\mathcal{G}$ within a linear class and choosing a nonparametric $\mathcal{F}$.

\medskip

\noindent \textbf{Roadmap.} The theoretical claims in the paper will extend the above intuitions to arbitrary multivariate cases in a rigorous manner. \cref{sec:fairnn} and \cref{sec:ident-fairnn} focus on the method and theoretical results for the nonparametric invariance pursuit \eqref{eq:intro-model}. \cref{sec:fairnn} considers the pure estimation problem \emph{nonparametric invariance pursuit} itself. \cref{thm:fairnn} shows that under the existence of the maximum invariant set (which is testable), realizing the prediction and testing function class $\mathcal{G}$ and $\mathcal{F}$ by neural networks can allow us to estimate the regression function $m^\star$ induced by the maximum invariant set efficiently in several aspects. \cref{sec:ident-fairnn} offers a causal interpretation of the maximum invariant set under extra structural assumptions in SCMs: \cref{prop:ident-transfer-learning} shows that there always exists a maximum invariant set under nondegenerate cases, it can be represented as pragmatic direct causes in general (\cref{prop:rtl}) and will match the direct causes under sufficient interventions (\cref{prop:ident-causal-discovery}). 

The above nonparametric invariance pursuit, as a special instance, helps to illustrate the main idea and philosophy of our general invariance pursuit problem and FAIR estimation framework, which will be formally presented in \cref{sec:ip}. In the main text, we provide a sketch of the abstract unified result, from which all the non-asymptotic results are derived as corollaries, along with other applications in \cref{sec:sketch}. This includes the case that is identifiable using only one environment. We provide a computationally efficient implementation using variants of gradient descent and Gumbel approximation, followed by its application to the simulation and real data analysis in \cref{sec:exp}. A robust prediction of water birds and land birds, similar to the thought experiment, is deferred to \aosversion{Appendix C.3}{\cref{sec:app-bird}}.

\subsection{New Contributions}

We propose a unified, algorithmic, and sample-efficient methodological framework that can discover the invariant regression function, i.e., to solve a generalized version of the problem in \cref{sec:intro-problem}. The method is simple, universal, fully algorithmic, and sample-efficient:  It is just one optimization objective \eqref{eq:intro-obj} complemented by one extra hyper-parameter $\gamma$; it accommodates many losses and can be seamlessly integrated by various machine learning algorithms; it does not require any prior structural knowledge, and it is almost as statistically efficient as standard regression under various cases. 

As a special instance in our framework, the FAIR neural network (FAIR-NN) estimator is proposed for which $\mathcal{G}$ and $\mathcal{F}$ are neural networks to unveil $m^\star$ in \eqref{eq:intro-model}. It is the \emph{first} theoretically guaranteed estimator that can \emph{efficiently} recover $m^\star$ under a single \emph{general} and \emph{minimal} identification condition associated with the heterogeneity of the environments. Its sample efficiency can be understood in several notable aspects: it requires the minimal identification condition, leading to fewer required environments; it exhibits the same $L_2$ error rate as if directly regressing $Y$ on known $X_{S^\star}$, regardless of the complexity of spurious associations; and it adapts to the unknown low-dimension structure of the invariant association $m^\star$ in a same manner as \cite{kohler2021rate}. 

We also establish the \emph{first} general causal interpretation of the $S^\star$ in the canonical model \eqref{eq:intro-model} under SCM with interventions on $X$ setting. Specifically, we demonstrate that under \emph{arbitrary nondegenerate interventions}, there always exists a set -- explicitly characterized by the cause-effect relationships and the intervened variables and referred to as ``pragmatic direct causes'' -- that satisfies the aforementioned general identification condition. As a corollary, we establish the general sufficient and necessary condition under which the direct causes of $Y$ can be recovered, while previous recovery results \citep{peters2016causal} only consider special cases. Moreover, even when the interventions are insufficient to identify direct causes, such an $S^\star$ also has implications in robust transfer learning. These results are, to the best of our knowledge, the \emph{first} in the literature and are of independent interest.

While the complicated combinatorial constraint and minimax optimization are introduced in \eqref{eq:intro-obj}, we show that a variant of gradient descent -- gradient descent-ascent with Gumbel approximation \citep{jang2017categorical,maddison2017concrete} to handle the combinatorial-nature ``focused'' constraint $f\in \mathcal{F}_{S_g}$ -- continues to apply to our specifically designed algorithm and neural network estimators with no curse-of-dimension in implementation. Numerical results in \cref{sec:exp} support this. 

Though our framework is designed for algorithmic learning, it is versatile in that the user can also incorporate their strong prior structural knowledge, such as linearity or additivity of $m^\star$, into the FAIR estimation. This can be realized by restricting the function class $\mathcal{G}$ within this known structure and designating $\mathcal{F}$ as a more expansive class. We demonstrate that harnessing such strong structural knowledge can relax the condition for identification. It is worth pointing out that identification is viable even when $|\mathcal{E}|=1$ corresponding to observational data; see examples in \aosversion{Appendix B.6}{\cref{sec:theory-linearx}}. At the methodology level, our method bridges the invariance principle \citep{peters2016causal} and asymmetry principle \citep{janzing2016algorithmic} for observational data into a unified framework.

\subsection{Related Works and Comparisons}

\label{subsec:related}

Starting from the pioneering work of \cite{peters2016causal}, there is considerable literature proposing methods to estimate $m^\star$ in \eqref{eq:intro-model}, predominantly when $m^\star$ is linear. These methods broadly fall into two categories: hypothesis test-based methods and optimization-based methods. For the hypothesis test-based methods \citep{peters2016causal, heinze2018invariant, pfister2019invariant}, the Type-I error is controlled for an estimator $\hat{S}$ with $\mathbb{P}(\hat{S}\subseteq S^\star) \ge 1-\alpha$. Nonetheless, these procedures may result in missing important variables or conservative solutions like $\hat{S} = \emptyset$ due to the inherent worst-case construction in the algorithm. Additionally, the introduction of hypothesis tests also hinders their seamless integration by machine learning algorithms, limiting their scalability. On the other hand, some optimization-based methods \citep{ghassami2017learning, rothenhausler2019causal, rothenhausler2021anchor} focus on linear $m^\star$ and tackle the problem under additional structures such as linear SCMs with additive interventions \citep{rothenhausler2019causal}. This limitation curtails its applicability to a broader nonparametric setting. Some optimization-based methods \citep{pfister2021stabilizing, yin2021optimization} designed for linear models are heuristic and lack finite sample guarantees. In summary, there is still a crucial gap towards efficiently estimating $m^\star$ without additional assumptions on the underlying model. Although \cite{fan2024environment} recently bridged this gap for linear $m^\star$ through an optimization-based method, it is still unclear under the general nonparametric setting. This paper is the first to attain sample-efficient estimation for the general model with non-asymptotic guarantees in terms of both $|\mathcal{E}|$ and $n$. Additionally, it is the first to provide a general sufficient and near-necessary conditions for interventions that enable the exact recovery of direct causes within the SCM framework.

\cite{arjovsky2019invariant} considers a general task, which aims to search for a data representation such that the optimal solution given that representation is optimal across diverse environments. They propose an optimization-based approach called invariant risk minimization (IRM), with many subsequent variants proposed later. However, their method comes with no statistical guarantees and requires at least $d$ environments even for the linear model, and the improvement over standard empirical risk minimization is not clear \citep{rosenfeld2020risks, kamath2021does}. Our paper is the first to offer a comprehensive theoretical analysis of general invariance learning when the representation class is $\{(x_1,\ldots, x_d) \to (a_1 x_1, \ldots, a_d x_d): a_1,\ldots, a_d \in \{0,1\}\}$ and to show that sample efficient estimation is in general viable even when $|\mathcal{E}|=2$. The main reason why this is attainable is due to the \emph{exact} invariance pursued by our FAIR penalty and its ``focused'' nature, see the discussion in \aosversion{Appendix A.2}{\cref{appendix:dis2}}.

Under the SCM framework, there is considerable literature on causal discovery using observational data \citep{spirtes2000causation, richardson1996feedback, chickering2002optimal, hyttinen2013discovering, hyttinen2014constraint}.
However, most of them only attain identification up to Markov equivalent class \citep{geiger1990logic}. To overcome the issue,  existing methods can be roughly divided into two categories -- one based on the invariance principle and the other based on the asymmetry principle. The invariance-based approaches \citep{peters2016causal} use samples from multiple experiments where some unknown intervention may apply to the variables other than $Y$. It leverages the idea that the cause-effect mechanism will remain constant while the reverse effect-cause association may vary. On the other hand, the asymmetry-based approaches \citep{shimizu2006linear, hoyer2008nonlinear, zhang2009identifiability, janzing2012information, peters2014causal} only observe one sample of observational data and use the idea that the cause-effect mechanism admits a simple prior known structure, whereas its inverse does not, example includes the additive noise structure \citep{hoyer2008nonlinear}. These two principles for causal discovery seem to have been orthogonal before. Our estimation framework is the first to offer a unified methodological perspective on these two principles with theoretical guarantees. It demonstrates the ability to simultaneously leverage both principles for identification and estimation.

Adversarial estimation is introduced in \cite{goodfellow2014generative} for generative modeling. Its application in the statistics spans distribution estimation \citep{liang2021well}, instrumental variable regression \citep{dikkala2020minimax}, estimating the (implicit) influence function \citep{chernozhukov2020adversarial, hirshberg2021augmented}, and so on. We adopted adversarial estimation from two novel aspects. Firstly, it allows us to use a simple objective function that homogenizes different tasks and prediction models for estimation. Moreover, such a minimax optimization objective and the Gumbel trick in the implementation jointly relax the combinatorial nature in \eqref{fan1} and make a variant of gradient descent continue to work numerically.

\subsection{Notations} 

We use upper case $(X, Y, Z)$ to represent random variables/vectors and denote their instances as $(x,y,z)$. Define $[n]=\{1,\ldots, n\}$. For a vector $x = (x_1,\ldots, x_d)^\top\in \mathbb{R}^d$, we let $\|x\|_2=(\sum_{j=1}^d x_j^2)^{1/2}$. For given index set $S=\{j_1,\ldots, j_{|S|}\}\subseteq [d]$ with $j_1<\cdots<j_{|S|}$, we denote $[x]_S=(x_{j_1},\ldots, x_{j_{|S|}})^\top \in \mathbb{R}^{|S|}$ and abbreviate it as $x_S$ if there is no ambiguity. We let $a\lor b = \max \{a, b\}$ and $a\land b = \min\{a, b\}$. We use $a(n) \lesssim b(n)$, $b(n) \gtrsim a(n)$, or $a(n) = O(b(n))$ if there exists some constant $C>0$ such that $a(n) \le Cb(n)$ for any $n \ge 3$. Denote $a(n) \asymp b(n)$ if $a(n)\lesssim b(n)$ and $a(n) \gtrsim b(n)$.  In the theorem statement and proof, we will use $C$ to represent the universal constants that may vary from line to line and will use $\tilde{C}, \tilde{C}_1,\ldots$ to represent the constants that may depend on the other defined constants. 

In the context of the multi-environment setup, for each $e\in \mathcal{E}$, let $\Theta^{(e)}=L_2(\mu^{(e)}_x) := \{f: \int f^2(x) \mu_x^{(e)}(dx) < \infty\}$, and denote $\|f\|_{2,e} = \{\int f^2(x) \mu^{(e)}_x(dx)\}^{1/2}$. Given $n$ observations $\{(X_i^{(e)}, Y_i^{(e)})\}_{i=1}^n \subseteq \mathbb{R}^d \times \mathbb{R}$ drawn i.i.d. from $\mu^{(e)}$, we define $\mathbb{E}[f(X^{(e)},Y^{(e)})] = \int f(x, y) \mu^{(e)}(dx, dy)$ and $\hat{\mathbb{E}}[f(X^{(e)},Y^{(e)})] = \frac{1}{n} \sum_{i=1}^n f(X^{(e)}_i, Y^{(e)}_i)$ for any $f \in \Theta^{(e)}$. We assume $\mathbb{E}[|Y^{(e)}|^2] < \infty$. Let $\bar{\mu} = \frac{1}{|\mathcal{E}|} \sum_{e\in \mathcal{E}} \mu^{(e)}$, and $\Theta = L_2(\bar{\mu}_x)$ equipped with the norm $\|\cdot\|_2 = \{\int f^2(x) \bar{\mu}_x(dx)\}^{1/2}$. It is easy to verify that $\Theta = \bigcap_{e\in \mathcal{E}} \Theta^{(e)}$. 

Let $S\subseteq [d]$ be any index set. Given a function class $\mathcal{H} \subseteq \{h: \mathbb{R}^d \to \mathbb{R}\}$, we define $\mathcal{H}_S$ be the class of functions in $\mathcal{H}$ that only depend on variables $x_S$, i.e., $\mathcal{H}_S=\{h\in \mathcal{H}, h(x) \equiv u(x_S) \text{ for some } u: \mathbb{R}^{|S|}\to \mathbb{R} ~~\mu^{(e)}\text{-}a.s. \forall e\in \mathcal{E}\}$. We sometimes also write $h(x_S)$ instead of $h(x)$ for $h \in \mathcal{H}_S$ since $h$ only depends on $x_S$. For any $h\in \mathcal{H}$, we use $S_h \subseteq [d]$ to represent the index set of the variables $h$ depends on. We let $\{\mathcal{H}\}^k = \{(h_1,\ldots, h_k): h_i \in \mathcal{H} ~\forall i\in [k]\}$. For any $(X, Y)$'s joint distribution $\nu$, we use $\nu_x$ to denote the marginal distribution of $X$, and $\nu_{x, S}$ to denote the marginal distribution of $X_S$.

\medskip
\noindent \textbf{Neural Networks.}  We use neural networks as a scalable nonparametric technique: we adopt the fully connected deep neural network with ReLU activation $\sigma(\cdot) = \max\{0, \cdot\}$, and call it \emph{deep ReLU network} for short. Let $L,N$ be any positive integer, a \emph{deep ReLU network with depth $L$ width $N$} admits the form of
\begin{align}
\label{eq:nn-architecture}
    g(x) = T_{L+1} \circ \bar{\sigma}_L \circ T_L \circ \bar{\sigma}_{L-1} \circ \cdots \circ T_2 \circ \bar{\sigma}_1 \circ T_1(x).
\end{align} Here $T_{l}(z) = W_l z + b_l: \mathbb{R}^{d_l} \to \mathbb{R}^{d_{l+1}}$ is a linear map with weight matrix $W_l \in \mathbb{R}^{d_{l}\times d_{l-1}}$ and bias vector $b_{l} \in \mathbb{R}^{d_{l}}$, where $(d_0,d_1\ldots, d_L, d_{L+1}) = (d, N, \ldots, N, 1)$, and $\bar{\sigma}_l: \mathbb{R}^{d_l} \to \mathbb{R}^{d_l}$ applies the ReLU activation $\sigma(\cdot)$ to each entry of a $d_l$-dimensional vector. Here, the equal width is for presentation simplicity.

\begin{definition}[Deep ReLU network class]
    Define the family of deep ReLU networks taking $d$-dimensional vector as input with depth $L$, width $N$, truncated by $B$ as 
        $\mathcal{H}_{\mathtt{nn}}(d, L, N, B) = \{\tilde{g}(x) = \mathrm{Tc}_B(g(x)): g(x) \text{ in } \eqref{eq:nn-architecture}\}$, 
        where $\mathrm{Tc}_B: \mathbb{R} \to \mathbb{R}$ is the truncation operator defined as $\mathrm{Tc}_B(z) = {\min}\{|z|, B\} \cdot \mathrm{sign}(z)$.
\end{definition}

\section{FAIR Least Squares Estimator Using Neural Networks}
\label{sec:fairnn}

In this section, we show that one can use the FAIR-NN least squares estimator, a realization of the FAIR estimator by setting $\ell(y,v)=\frac{1}{2}(y-v)^2$ and specifying both $(\mathcal{G}, \mathcal{F})$ as neural networks, to attain sample-efficient estimation in nonparametric invariance pursuit. 

The main messages of this section are twofold. From a theoretical perspective, it shows that sample-efficient estimation (in both $n$ and $|\mathcal{E}|$) in the general nonparametric invariance pursuit problem is viable under a minimal identification condition related to the heterogeneity of the environments. From a methodological perspective, it demonstrates one key feature of our proposed framework: one can seamlessly integrate black-box machine learning models (e.g., neural networks) into it and fully exploit these models' sample efficiency and capability in being adaptive to low-dimensional structures.

\subsection{Setup}
\label{sec:nip-setup}

Recall that $\mu^{(e)}$ is the joint distribution of $(X,Y)$ in environment $e$. Let $m^{(e,S)}(x) := \mathbb{E}[Y^{(e)}|X_S^{(e)}=x_S]$ be the conditional expectation of $Y$ given $X_S$ in environment $e$. Recall that $\nu_{x,S}$ is the marginal distribution of $X_S$ for $(X,Y)\sim \nu$. It is easy to see that $\mu^{(e)}_{x,S}$ is absolutely continuous with respect to $\bar{\mu}_{x,S}=[\frac{1}{|\mathcal{E}|} \sum_{e\in \mathcal{E}} \mu^{(e)}]_{x,S}$ for any $S\subseteq[d]$ hence $\rho^{(e)}_S$, the Radon–Nikodym derivative of $\mu_{x,S}^{(e)}$ with respect to $\bar{\mu}_{x,S}$, is well defined. We define $\bar{m}^{(S)}(x) = \sum_{e\in \mathcal{E}} \rho^{(e)}_S(x_S) m^{(e,S)}(x)$, which can be interpreted as the population-level least squares that regress $Y$ on $X_S$ using all the data in $\mathcal{E}$. 

\begin{condition}[Model and Regularity Conditions] There exists some positive constants $(C_0, s_{\min})$ such that the following conditions hold. 
\label{cond:regularity-fairnn}

\begin{itemize}
\item[(a)] \underline{Data Generating Process:} We collect data from $|\mathcal{E}| \in \mathbb{N}^+$ environments with $|\mathcal{E}| \le n^{C_0}$. For each environment $e\in \mathcal{E}$, we observe $\{(X_i^{(e)}, Y_i^{(e)})\}_{i=1}^n \overset{i.i.d.}{\sim} \mu^{(e)}$.

\item[(b)] \underline{Invariance Structure:} There exists some set $S^\star$ and $m^\star: \mathbb{R}^{|S^\star|} \to \mathbb{R}$ such that $m^{(e,S^\star)}(x) \allowbreak\equiv m^\star(x_{S^\star})$ for any $e\in \mathcal{E}$.

\item[(c)] \underline{Sub-Gaussian Response:} For any $e\in \mathcal{E}$ and $t\ge 0$, $\mathbb{P}\left[|Y^{(e)}| \ge t\right]\le C_0 e^{-t^2/(2C_0)}$.

\item[(d)] \underline{Boundedness:} $X \in [-C_0, C_0]^d$ $\bar{\mu}$-a.s. and $\|m^{(e,S)}\|_\infty \le C_0$ for any $S\subseteq [d]$ and $e\in \mathcal{E}$.

\item[(e)] \underline{Nondegenerate Covariate:} $\forall S\subseteq [d]$ with $S^\star \setminus S \neq \emptyset$, $\inf_{m\in \Theta_S} \|m - m^\star\|_{2}^2 \ge s_{\min}>0$. 

\end{itemize}

\end{condition}

\cref{cond:regularity-fairnn} (a)--(b) is just a restatement of \eqref{eq:intro-model} together with i.i.d. data within each environment; data across different environments may be dependent. (c)--(d) are standard in nonparametric regression. (e) rules out some degenerate cases, for example, $m^\star(x_1)=x_1^2$ with $S^\star=\{1\}$ and $X_2=X_1^4$, or $m^\star(x_1,x_2)=f(x_1)$ with $S^\star=\{1,2\}$, and is imposed for technical convenience. This condition is not necessary for deriving the $L_2$ error rate, but it is necessary for the variable selection. The target (invariant) regression function in nonparametric invariance pursuit is $m^\star$.  

\subsection{Proposed FAIR-NN Least Squares Estimator}
\label{sec:fairlse}

Given the data $\{\{(X_i^{(e)}, Y_i^{(e)})\}_{i=1}^n\}_{e\in \mathcal{E}}$ from heterogeneous environments, we consider using the following FAIR-NN least squares estimator to learn $m^\star$ in \eqref{eq:intro-model}. Specifically, the FAIR-NN least squares estimator is the solution to the subsequent minimax optimization objective
\begin{align}
\label{eq:def-fair-lse}
    \hat{g} \in \argmin_{g\in \mathcal{G}} \sup_{f^{\mathcal{E}} \in \{\mathcal{F}_{S_g}\}^{|\mathcal{E}|}} \frac{1}{|\mathcal{E}| \cdot n} \sum_{e\in \mathcal{E}, i\in [n]} \left\{Y^{(e)}_i - g(X^{(e)}_i) \right\}^2 + \gamma \hat{\mathsf{J}}(g, f^{\mathcal{E}}).
\end{align} where the first part of the objective $\hat{\mathsf{Q}}_\gamma(g, f^{\mathcal{E}})$ is the pooled least squares loss preventing the estimator from collapsing to conservative solutions, $\gamma$ is the hyper-parameter to be determined, and $\hat{\mathsf{J}}(g, f^{\mathcal{E}})$ is the empirical counterpart of the focused adversarial invariance regularizer defined as
\begin{align}
\label{eq:method-empirical-j}
    \hat{\mathsf{J}}(g, f^{\mathcal{E}}) = \frac{1}{|\mathcal{E}|\cdot n}\sum_{e\in \mathcal{E}, i\in [n]} \left[\big\{Y^{(e)}_i - g(X_i^{(e)})\big\} f^{(e)}(X^{(e)}_i) - \frac{1}{2} \big\{f^{(e)}(X^{(e)}_i)\big\}^2\right].
\end{align} 
The minimax program \eqref{eq:def-fair-lse} is the empirical version of \eqref{eq:intro-obj} via setting $\ell(y, v) = \frac{1}{2}(y-v)^2$. Here we specify the predictor function class $\mathcal{G}$ and testing function class $\mathcal{F}$ as
\begin{align}
\label{eq:fairnn-function-class}
    \mathcal{G}=\mathcal{H}_{\mathtt{nn}}(d, L, N, B) \qquad \text{and} \qquad \mathcal{F} = \mathcal{H}_{\mathtt{nn}}(d, L+2, 2N, 2B)
\end{align} for neural network architecture hyper-parameters $N, L$ and truncation parameter $B=C_0$. Here $B$ can be larger than $C_0$ but should satisfy $B=O(1)$. One can also adopt a larger width, depth, and truncation parameter for $\mathcal{F}$. Our choice of $(N, L, B)$ for $\mathcal{F}$ here is for technical purposes, that is, any $m^{(e,S)}-g$  for $g\in \mathcal{G}$ can be well approximated by some $f\in \mathcal{F}$.

\subsection{Non-Asymptotic Result for FAIR-NN}

\begin{condition}[Identification for Nonparametric Invariance Pursuit]
    \label{cond-fairnn-ident}
    For any $S\subseteq [d]$ such that $\bar{\mu}(\{m^\star \neq \bar{m}^{(S\cup S^\star)}\})>0$, there exists some $e, e'\in \mathcal{E}$ such that $\min\{\mu^{(e)}, \mu^{(e')}\}\allowbreak(\{m^{(e,S)} \neq m^{(e',S)}\})\} > 0$.
\end{condition}

\begin{remark}[Minimal Heterogeneity Condition for Identification]
The above identification condition necessitates that whenever a bias emerges when regressing $Y$ on $X_{S\cup S^\star}$ using least squares, there should be noticeable shifts in the conditional expectation $m^{(e, S)}$ across environments. In other words, $S^\star$ is the maximum set preserving the invariant associations.
This condition is minimal. If it is violated, it would imply that $\exists \tilde{S}\subseteq [d] ~\text{with}~ \tilde{S}\setminus S^\star \neq \emptyset$ such that
\begin{align*}
     \forall e\in \mathcal{E}~~\mathbb{E}[Y^{(e)}|X_{\tilde{S}}^{(e)}] \equiv g(X_{\tilde{S}}^{(e)}) ~~\mu^{(e)}\text{-}a.s.~~\text{for some}~ g: \mathbb{R}^{|S|}\to \mathbb{R},
\end{align*} in which both set $S^\star$ and $\tilde{S}$ embody the invariant conditional expectation structure, thus more environments are needed in this case to pinpoint $S^\star$. Such a minimal identification condition underscores that our proposed FAIR-NN estimator is ``sample efficient'' regarding the number of environments $|\mathcal{E}|$ required; see the discussions in \cref{sec:ident-fairnn}. Notably, such an identification condition relaxes those employed in approaches using intersections like ICP \citep{peters2016causal}. These approaches require the shifts of conditional distributions for all the $S$ with $\bar{m}^{(S)}\neq m^\star$ for identifying $S^\star$. 
\end{remark}

The following theorem provides an oracle-type inequality for the FAIR-NN least squares estimator in a structure-agnostic manner. It shows that under \cref{cond-fairnn-ident}, one can expect consistent estimation and further establish non-asymptotic upper bounds on the $L_2$ error between the estimator \eqref{eq:def-fair-lse} and the invariant regression function $m^\star$. In addition, the theorem quantifies the amount of penalty needed, namely $\gamma^\star_{\mathtt{NN}}$, which is of constant order and is related to the signal-to-noise ratio of the problem.

\begin{theorem}[Oracle-type Inequality for FAIR-NN Least Squares Estimator]
\label{thm:fairnn}
    Assume \cref{cond:regularity-fairnn} and \ref{cond-fairnn-ident} hold. Then $\gamma^\star_{\mathtt{NN}}=\sup_{S\subseteq [d]: \mathsf{b}_{\mathtt{NN}}(S)>0} (\mathsf{b}_{\mathtt{NN}}(S)/\bar{\mathsf{d}}_{\mathtt{NN}}(S)) < \infty$, where \begin{align}
    \label{eq:bnndnn}
         \mathsf{b}_{\mathtt{NN}}(S) = \|m^\star - \bar{m}^{(S\cup S^\star)}\|_2^2 \qquad \text{and}\qquad
         \bar{\mathsf{d}}_{\mathtt{NN}}(S) = \frac{1}{|\mathcal{E}|} \sum_{e\in \mathcal{E}} \|m^{(e,S)} - \bar{m}^{(S)}\|_{2,e}^2.
    \end{align} Consider the estimator that solves \eqref{eq:def-fair-lse} using $\gamma \ge 8\gamma^\star_{\mathtt{NN}}$ and function classes \eqref{eq:fairnn-function-class} with $L, N$ satisfying $NL\le n$ and $N\ge 4$. Then, there exists some constant $\tilde{C}$ depending on $(d, C_0)$ such that for any $n\ge 3$, 
    \begin{align}
    \label{eq:fairnn-error}
        \frac{\|\hat{g} - m^\star\|_{2}}{\tilde{C}} \le \max_{e\in \mathcal{E}} \inf_{h\in \mathcal{G}_{S^\star}} \|m^\star - h\|_{2,e} + \frac{NL\log^{3/2} n}{\sqrt{n}} + 1_{\{{\delta}_{\mathtt{NN},1} > \mathsf{s}\}} \cdot \left(\gamma {\delta}_{\mathtt{NN},1}\right)
    \end{align} occurs with probability at least $1-\tilde{C}n^{-100}$. Here $\delta_{\mathtt{NN}, 1}=\max_{e\in \mathcal{E}, S\subseteq[d]} \inf_{h\in \mathcal{G}_S} \|m^{(e,S)} - h\|_{2,e} + \frac{NL\log^{3/2} n}{\sqrt{n}}$ and $\mathsf{s} = \tilde{C}^{-1} [1\land s_{\min} \land \{\gamma \inf_{S:\bar{\mathsf{d}}_{\mathtt{NN}}(S)>0} \bar{\mathsf{d}}_{\mathtt{NN}}(S)\}]/(1+\gamma)$, where $s_{\min}$ is defined in \cref{cond:regularity-fairnn}(e). Moreover, under the above event, if $\delta_{\mathtt{NN},1} \le \mathsf{s}$, then the variable selection property holds, for $\hat{S} = S_{\hat{g}}$,
    \begin{align}
    \label{eq:var-sel-fairnn}
        S^\star \subseteq \hat{S}  \qquad \text{and}  \qquad \forall e\in \mathcal{E}, ~~ m^{(e,\hat{S})} = m^\star .
    \end{align}
\end{theorem}

\begin{remark}[Interpretation of $\mathsf{b}_{\mathtt{NN}}(S)$ and $\bar{\mathsf{d}}_{\mathtt{NN}}(S)$]
\label{remark:bias-mean-variance}
We refer to $\mathsf{b}_{\mathtt{NN}}(S)$ as bias mean since it exactly characterizes the bias of the least squares estimator in the presence of endogenously spurious variables like the background color in the thought experiment. In particular, letting $\hat{g}_{\mathtt{LSE}(S)}$ be the least squares estimator that regresses $Y$ on $X_{S}$ using all the data, namely, the FAIR-NN estimator with $\gamma = 0$, \aosversion{Proposition B.1}{\cref{prop:bias}} implies
\begin{align*}
    \left|\frac{\|\hat{g}_{\mathtt{LSE}(S)} - m^\star\|_2^2}{\mathsf{b}_{\mathtt{NN}}(S)} - 1\right| = o_{\mathbb{P}}(1) \qquad \text{if}~~ S^\star \subseteq S ~\text{and}~ \mathsf{b}_{\mathtt{NN}}(S)>0.
\end{align*}
We refer to $\bar{\mathsf{d}}_{\mathtt{NN}}(S)$ as the bias variance because it measures the variations of bias across environments. Specifically, when $S^\star\subseteq S$, the bias in environment $e$ is $(m^{(e, S)}-m^\star)$, and $\bar{\mathsf{d}}_{\mathtt{NN}}(S)$ can be viewed as the variance of the bias concerning the uniform distribution on $\mathcal{E}$ since $\bar{\mathsf{d}}_{\mathtt{NN}}(S)  = \frac{1}{|\mathcal{E}|} \sum_{e\in \mathcal{E}} \|(m^{(e, S)} - m^\star) - (\bar{m}^{(S)} - m^\star)\|_{2,e}^2$. We have $\bar{\mathsf{d}}_{\mathtt{NN}}(S^\star) = 0$ by the invariance structure in \cref{cond:regularity-fairnn}(b).
\end{remark}

\begin{remark}[Identification]
\cref{thm:fairnn} combines the identification result, which characterizes when it is possible to consistently estimate $m^\star$, and the finite-sample estimation error result, which characterizes how accurately we can estimate $m^\star$. The main identification message disentangled from the above theorem is that if the minimal heterogeneity condition \cref{cond-fairnn-ident} holds, then one can consistently estimate $m^\star$ provided $\gamma$ is larger than some threshold $8\gamma^\star_{\mathtt{NN}}$ that is independent of $n$.
\end{remark}

Here $\delta_{\mathtt{NN},1}$ can be interpreted as the sum of the worst-case approximation error of neural networks to all the conditional moments $\{m^{(e,S)}\}_{e\in \mathcal{E}, S\subseteq [d]}$ and the stochastic error. One can expect $\delta_{\mathtt{NN},1}=o(1)$ if $NL\log^{1.5}n=o(\sqrt{n})$ and all the conditional moments are Lipschitz functions. Moreover, $\mathsf{s}$ can be explained as the minimum of the signal of true important variables in $S^\star$ and the signal of heterogeneity. Given $\mathsf{s}$ is of constant order and $\delta_{\mathtt{NN},1} = o(1)$, the error bound \eqref{eq:fairnn-error} shows that as $\delta_{\mathtt{NN},1} \le \mathsf{s}$, i.e., if $n$ is large enough, the $L_2$ error is composed of the approximation error of neural networks to $m^\star$ and the stochastic error. In this case, all endogenously spurious variables can be surely screened \citep{fan2008sure}, i.e., \eqref{eq:var-sel-fairnn}, and $m^\star$ can be estimated as well as if the invariant set of variables $S^\star$ is known. At the same time, given our results are non-asymptotic, for a given (not large enough) $n$, we may not be able to eliminate all endogenously spurious variables as in \eqref{eq:var-sel-fairnn}. The error rate in this case will be $(\gamma+1) \delta_{\mathtt{NN}, 1}$ as in \eqref{eq:fairnn-error} given that our method may select wrong variables. The error bounds and $\delta_{\mathtt{NN},1}$ will be presented explicitly as \eqref{eq:rate-fairnn} in \cref{coro:fair-nn-fast} when we impose assumptions on the function class.

\subsection{Adapting to the Low-dimensional Structures Algorithmically}

In this section, we present the convergence rate of the FAIR-NN when $m^\star$ lies within the hierarchical composition model \citep{bauer2019deep}. This is the function class that neural networks can efficiently estimate \citep{schmidt2020nonparametric, kohler2021rate,fan2024factor} with little guidance regarding the forms of functions. We show that FAIR-NN can obtain the same result as standard regression blind to both the knowledge of $S^\star$ and function structure. This example demonstrates our framework's ability to fully leverage the sample efficiency of the adopted machine learning model while also providing a concrete instance that realizes several quantities defined in the structure-agnostic setting of \cref{thm:fairnn}.

\begin{definition}[$(\beta,C)$-smooth Function]
    Let $\beta = r+s$ for some nonnegative integer $r\ge 0$ and $0<s\le 1$, and $C>0$. A $d$-variate function $f$ is $(\beta, C)$-smooth if for every non-negative sequence $\alpha \in \mathbb{N}^d$ such that $\sum_{j=1}^d \alpha_j = r$, the partial derivative $\partial^{\alpha} f=(\partial f)/(\partial x_1^{\alpha_1}\cdots x_d^{\alpha_d})$ exists and satisfies $|\partial^{\alpha} f(x) - \partial^{\alpha} f(z)| \le C \|x-z\|_2^s$.
    We use $\mathcal{H}_{\mathtt{HS}}(d, \beta, C)$ to denote the set of all the $d$-variate $(\beta, C)$-smooth functions.
\end{definition}

\begin{definition}[Hierarchical Composition Model $\mathcal{H}_{\mathtt{HCM}}(d, l, \mathcal{O}, C)$]
\label{hcm}
    We define function class of hierarchical composition model $\mathcal{H}_{\mathtt{HCM}}(d, l, \mathcal{O}, C)$ \citep{kohler2021rate} with $l, d \in \mathbb{N}^+$, $C\in \mathbb{R}^+$, and $\mathcal{O}$, a subset of $[1,\infty) \times \mathbb{N}^+$, in a recursive way as follows. Let $\mathcal{H}_{\mathtt{HCM}}(d, 0,\mathcal{O}, C)=\{h(x)=x_j, j\in [d]\}$, and for each $l\ge 1$, 
    \begin{align*}
    \mathcal{H}_{\mathtt{HCM}}(d, l,\mathcal{O}, C) = \big\{&h: \mathbb{R}^d \to \mathbb{R}: h(x) = g(f_1(x),...,f_t(x))\text{, where} \\
    &~~~~~ g\in \mathcal{H}_{\mathtt{HS}}(t, \beta, C) \text{ with } (\beta, t)\in \mathcal{O} \text{ and } f_i \in \mathcal{H}_{\mathtt{HCM}}(d, l-1,\mathcal{O}, C)\big\}.
\end{align*} 
\end{definition}

Following \cite{kohler2021rate}, we assume all the compositions are at least Lipschitz functions to simplify the presentation. The minimax optimal $L_2$ estimation risk over $\mathcal{H}(d, l, \mathcal{O}, C_h)$ is $n^{-\alpha^\star/(2\alpha^\star + 1)}$, where $\alpha^\star = \min_{(\beta, t)\in \mathcal{O}} (\beta/t)$ is the smallest dimensionality-adjusted degree of smoothness \citep{fan2024factor} that represents the hardest component in the composition.  For example, if $m^\star(x) = f_1(x_1) + f_2(f_3(x_2, x_3), f_4(x_4, x_5)) + f_5(x_1, x_3, x_5)$ and all functions have a bounded second derivative, then the hardest component is the last one, and the dimensionality-adjusted degree of smoothness is $\alpha^* = 2/3$.

\begin{condition}[Function Complexity]
\label{cond:fair-nn-fast}
The following holds: 

\noindent (a) $m^{(e,S)} \in \mathcal{H}_{\mathtt{HCM}}(|S|, l, \mathcal{O}, C_h)$ for any $e\in \mathcal{E}$ and $S\subseteq [d]$ with $\alpha_0 = \inf_{(\beta, t)\in \mathcal{O}} (\beta/t)$.

\noindent (b) $m^\star \in \mathcal{H}_{\mathtt{HCM}}(|S^\star|, l, \mathcal{O}^\star, C_h)$ with $\alpha^\star = \inf_{(\beta, t)\in \mathcal{O}^\star} (\beta/t)$.

\noindent (c) $\max\{C_0, d, l, C_h, \sup_{(\beta, t) \in \mathcal{O}} (\beta \lor t), \sup_{(\beta, t) \in \mathcal{O}^\star} (\beta \lor t)\} \le C_1$ for some constant $C_1>1$.

\noindent (d) The neural network architecture hyper-parameters diverge: $(\log n)/(N\land L)=o(1)$.
\end{condition}

\begin{corollary}[Convergence Rate for FAIR-NN]
\label{coro:fair-nn-fast}
    Under the setting of \cref{thm:fairnn}, assume further that \cref{cond:fair-nn-fast} holds. Then, for any $n \ge 3$, with probability at least $1-\tilde{C}n^{-100}$, the following holds
    \begin{align}
    \label{eq:rate-fairnn}
        \frac{\|\hat{g} - m^\star\|_2}{\tilde{C} \log^{1.5\lor 4\alpha^\star}(n)} \le (NL)^{-2\alpha^\star} + \frac{NL}{\sqrt{n}} + 1_{\{n< n_0\}} \gamma \underbrace{\left[(NL)^{-2\alpha_0} + \frac{NL}{\sqrt{n}}\right]}_{{\delta_{\mathtt{NN,1}}}},
    \end{align} where $n_0$ depends on $(C_1, \gamma, s_{\min}, \inf_{S:\bar{\mathsf{d}}_{\mathtt{NN}}(S)>0} \bar{\mathsf{d}}_{\mathtt{NN}}(S))$, and $\tilde{C}$ is a constant dependent only on $C_1$. Under the optimal choice of network architecture hyper-parameters $N,L$ satisfying $LN \asymp n^{\frac{1}{2(2\alpha^\star+1)}}$, the R.H.S. of \eqref{eq:rate-fairnn} is $n^{-\alpha^\star/(2\alpha^\star+1)}+1_{\{n<n_0\}} \gamma n^{-\alpha_0/(2\alpha^\star+1)}$
\end{corollary}

From \cref{coro:fair-nn-fast}, we can get (up to logarithmic factors) minimax convergence rate $n^{-\alpha^\star/(2\alpha^\star + 1)}$, which is independent of both $\alpha_0$ and $\gamma$, when $n$ is larger than some constant $n_0$. Utilizing neural networks in predictor and discriminator function classes allows the estimator to adapt to the invariant regression function $m^\star$ efficiently from two crucial perspectives. Firstly, similar to using neural networks in nonparametric regression \citep{schmidt2020nonparametric}, adopting neural networks in $\mathcal{G}$ endows the estimator with the capability of being adaptive to the low-dimensional hierarchical structure algorithmically. Secondly, the choice of model parameter $(N, L)$ and the convergence rate depend only on $m^\star$. The (spurious) conditional expectations $m^{(e, S)}$ can be much more complex than $m^\star$. Notably, this complexity will not affect the convergence rate. This can be credited to the scalability of neural networks used as discriminators, i.e., their adaptivity capability in the regularization part of FAIR. 

\begin{remark}[{Error Guarantees} for All $n$] The error bound \eqref{eq:rate-fairnn} is applicable for any $n\ge 3$, even when it selects the wrong variables. {This is the benefit brought by our proposed regularized least squares and cannot be easily attained by alternative two-state procedures, for example, first running some variable selection procedure similar to ICP and then refitting the model.} Furthermore, the error bound will not inflate if the invariant signal $s_{\min}$ and the heterogeneity signal $\inf_{S\subseteq [d]: \bar{\mathsf{d}}_{\mathtt{NN}}(S)>0} \bar{\mathsf{d}}_{\mathtt{NN}}(S)$ is small. Though the error bound scales linearly with $\gamma$, the estimator we propose is not vulnerable to ``weak spurious'' variables, e.g., $x_j$ with $\sup_{e\in \mathcal{E}} \|m^{(e,S^\star \cup \{j\})} - m^\star\|_{2,e} \le \epsilon$, provided all the ratio of the bias $\mathsf{b}_{\mathtt{NN}}(S)$ to heterogeneity $\bar{\mathsf{d}}_{\mathtt{NN}}(S)$ gets controlled. 
\end{remark}

\begin{remark}[Choice of the Hyper-parameter $\gamma$] Though we have to choose a hyper-parameter $\gamma$ larger than a certain threshold to attain such a rate, the convergence rate is independent of $\gamma$. This implies that when the sample size $n$ is large, we do not need to tune the hyper-parameter $\gamma$ for optimal performance. Instead, we can choose some conservative (large) $\gamma$ such that the lower bound $\gamma \ge 8\gamma^\star_{\mathtt{NN}}$ is guaranteed.
\end{remark}

\section{Nonparametric Invariance Pursuit under SCMs}
\label{sec:ident-fairnn}

The results in \cref{sec:fairnn} are for the problem \emph{nonparametric invariance pursuit} itself. In a population-level view, if there exists a ``maximum invariant set'' $S^\star$ satisfying
\begin{align}
\label{eq:maximum-invariant-set}
\begin{split}
    &m^{(e,S^\star)} \equiv \bar{m}^{(S^\star)} ~(\text{invariant}) ~~~ \text{and} ~~~ \\
    &\qquad \forall S\subseteq [d], ~ m^{(e,S)} \equiv \bar{m}^{(S)} \Longrightarrow \bar{m}^{(S\cup S^\star)} = \bar{m}^{(S^\star)}~(\text{maximum})
\end{split}
\end{align} simultaneously, then both $S^\star$ and the induced $m^\star$ can be estimated well as if standard regression by the FAIR-NN estimator. It is natural to ask
\begin{center}
\emph{
Does such a maximum invariant set $S^\star$ exist? What's the semantic meaning of it?}
\end{center}

We offer a general answer to the question under the SCM with arbitrary interventions (on $X$) setting. The short answer is: Yes, it can be interpreted as the ``pragmatic direct causes''. 

\subsection{Structural Causal Model with Interventions on Covariates}

We first introduce the concept of the structural causal model \citep{glymour2016causal}. See \cref{fig:scm-ident} for examples of SCM. It says that each variable in the directed graph is a function of its parents (if any) and an independent innovation or noise.

\begin{definition}[Structural Causal Model] A structural causal model $M=(\mathcal{S}, \nu)$ on $p$ variables $Z_1,\ldots, Z_p$ can be described using $p$ assignment functions $\{f_1,\ldots, f_p\} = \mathcal{S}$: 
\begin{align*}
    Z_j \leftarrow f_j(Z_{\mathtt{pa}(j)}, U_j) \qquad j=1,\ldots, p,
\end{align*} 
where $\mathtt{pa}(j) \subseteq \{1,\ldots, p\}$ is the set of parents, or the direct causes, of the variable $Z_j$, and the joint distribution $\nu(du) = \prod_{j=1}^p \nu_j(du_j)$ over $p$ independent exogenous variables $(U_1,\ldots, U_p)$. For a given model $M$, there is an associated directed graph $G(M)=(V, E)$ that describes the causal relationships among variables, where $V=[p]$ is the set of nodes, $E$ is the edge set such that $(i,j)\in E$ if and only if $i\in \mathtt{pa}(j)$. $G(M)$ is acyclic if there is no sequence $(v_1,\ldots, v_k)$ with $k\ge 2$ such that $v_1=v_k$ and $(v_i, v_{i+1}) \in E$ for any $i\in [k-1]$.
\end{definition}

As in \cite{peters2016causal}, we consider the following data-generating process in $|\mathcal{E}|$ environments. For each $e\in \mathcal{E}$, the process governing $p=d+1$ random variables $Z^{(e)} = (Z_1^{(e)},\ldots, Z_{d+1}^{(e)})=(X_1^{(e)}, \ldots, X_d^{(e)},Y^{(e)})$ is derived from an SCM $M^{(e)}(\mathcal{S}^{(e)}, \nu)$, whose induced graph $G(M^{(e)})$ is acyclic, and assignments as
\begin{align}
\label{eq:scm-model} 
\begin{split}
    X_j^{(e)} &\leftarrow f_j^{(e)}(Z_{\mathtt{pa}(j)}^{(e)}, U_j), \qquad \qquad  j=1,\ldots, d \\
    Y^{(e)} &\leftarrow f_{d+1}(X_{\mathtt{pa}(d+1)}^{(e)}, U_{d+1}).
\end{split}
\end{align}
Here the distribution of exogenous variables $(U_1,\ldots, U_{d+1})$, the cause-effect relationship graph $G$, and the structural assignment $f_{d+1}$ are \emph{invariant} across $e\in \mathcal{E}$, while the structural assignments for $X$ may vary among $e\in \mathcal{E}$. We use superscript $(e)$ to highlight this heterogeneity. This heterogeneity may arise from performing arbitrary interventions on the variables $X$. We use $Z_{\mathtt{pa}(j)}$ to emphasize that $Y$ can be the direct cause of some variables in the covariate vector. See an example in \cref{fig:scm-ident} (a). Here we restrict to the case without hidden confounders; see the statement under the presence of hidden confounders in  \aosversion{Appendix A.6}{\cref{appendix:confounder}}.

To present the result, we consider an augmented SCM that incorporates the environment label $e$ as a variable $E$. We consider the case where $\mathcal{E}=\{0,\ldots, |\mathcal{E}|-1\}$. We let $0$ be the observational environment, and the rest are the interventional environments where some {\it unknown}, {\it arbitrary} interventions are applied to the variables in some given set $I \subseteq [d]$ defined as $I:=\{j: \exists e\in \mathcal{E} ~\text{s.t.}~ f_j^{(e)} \neq f_j^{(0)}\}$. The interventions can be arbitrary: it can be a ``hard'' do-intervention via setting $X_j$ to be $v_j$, or a soft intervention that slightly perturbs the association, e.g., replacing $X_j \gets 2 X_k + U_j$ by $X_j \gets 1.5 X_k + U_j$. The shared cause-effect relationships in all the environments are encoded by $G$, or $\{\mathtt{pa}(j)\}_{j=1}^{d+1}$.

The following SCM $\tilde{M}=(\tilde{\mathcal{S}},\tilde{\nu})$ on $d+2$ variables $Z=(Z_1,\ldots, Z_d, Z_{d+1}, Z_{d+2}) = (X_1,\ldots, X_d, Y, E)$ encodes all the information of $|\mathcal{E}|$ models $\{M^{(e)}(\mathcal{S}^{(e)}, \nu)\}_{e\in \mathcal{E}}$ in \eqref{eq:scm-model}. Denote $\nu_b \sim \mathrm{Uniform}(\mathcal{E})$. Here $\tilde{\nu}(du_1,\ldots, du_{d+2}) = \nu(du_1,\ldots, du_{d+1}) \nu_b(du_{d+2})$, and the assignments $\tilde{\mathcal{S}}=\{\tilde{f}_1,\ldots, \tilde{f}_{d+2}\}$ are defined as
\begin{align}
\begin{split}
    E &\gets \tilde{f}_{d+2}(U_{d+2}) := U_{d+2} \\
    X_j &\gets \begin{cases} \tilde{f}_j(Z_{\mathtt{pa}(j)}, U_{j}) := f^{(0)}_j(Z_{\mathtt{pa}(j)}, U_{j}) &\qquad \forall j\in [d] \setminus I \\
    \tilde{f}_j(Z_{\mathtt{pa}(j)}, E, U_{j}) := f^{(E)}_j(Z_{\mathtt{pa}(j)}, U_j) &\qquad \forall j \in I 
    \end{cases} \\
    Y &\gets \tilde{f}_{d+1}(X_{\mathtt{pa}}(d+1), U_{d+1}) := f_{d+1}(X_{\mathtt{pa(d+1)}}, U_{d+1}),
\end{split}
\label{eq:model-2scm}
\end{align} 
where $I$ is the set of all intervention variables in $\cE$.
It should be noted that throughout this section, the direct cause map $\mathtt{pa}: [d+1]\to [d+1]$ matches the causal relationship $G$ instead of $\tilde{G}=G(\tilde{M})$. See a graphical illustration of the construction in \cref{fig:scm-ident} (b). 

We summarize the above construction as a condition.

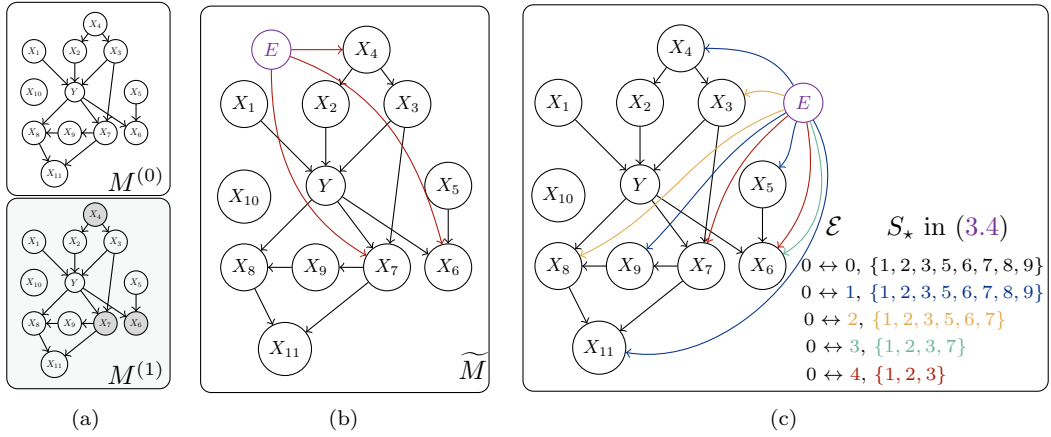
\begin{figure}
\centering

\subfigure[]{
\begin{tikzpicture}[scale=0.36, state/.style={circle, draw, minimum size=0.72cm, scale=0.36}]
\draw[black, rounded corners] (-1, -5.3) rectangle (5, 1.8);

\node[state] at (0, 0) (x1) {$X_1$};
\node[state] at (2*0.75, -1.5) (y) {$Y$};
\node[state] at (2*0.75, 0) (x2) {$X_2$};
\node[state] at (4*0.75, 0) (x3) {$X_3$};
\node[state] at (3*0.75, 1) (x4) {$X_4$}; 
\node[state] at (5*0.75, -1.5) (x5) {$X_5$};
\node[state] at (5*0.75, -3) (x6) {$X_6$};
\node[state] at (3.5*0.75, -3) (x7) {$X_7$};
\node[state] at (1.75*0.75, -3) (x9) {$X_{9}$};
\node[state] at (0, -3) (x8) {$X_8$};
\node[state] at (0, -1.5) (x10) {$X_{10}$};
\node[state] at (1*0.75, -4.5) (x11) {$X_{11}$};

\draw[->] (x1) -- (y);
\draw[->] (x2) -- (y);
\draw[->] (x3) -- (y);
\draw[->] (x4) -- (x2);
\draw[->] (x5) -- (x6);
\draw[->] (x4) -- (x3);
\draw[->] (y) -- (x6);
\draw[->] (y) -- (x7);
\draw[->] (y) -- (x8);
\draw[->] (x7) -- (x9);
\draw[->] (x7) -- (x11);
\draw[->] (x8) -- (x11);
\draw[->] (x3) -- (x7);
\draw[->] (x9) -- (x8);

\draw node at (5*0.75, -4.7) {$M^{(0)}$};

\draw[black, rounded corners, fill=mylightblue!10] (-1, -12.4) rectangle (5, -5.4);

\node[state] at (0, -7) (ix1) {$X_1$};
\node[state] at (2*0.75, -8.5) (iy) {$Y$};
\node[state] at (2*0.75, -7) (ix2) {$X_2$};
\node[state] at (4*0.75, -7) (ix3) {$X_3$};
\node[state, fill=gray!30] at (3*0.75, -6) (ix4) {$X_4$}; 
\node[state] at (5*0.75, -8.5) (ix5) {$X_5$};
\node[state, fill=gray!30] at (5*0.75, -10) (ix6) {$X_6$};
\node[state, fill=gray!30] at (3.5*0.75, -10) (ix7) {$X_7$};
\node[state] at (1.75*0.75, -10) (ix9) {$X_{9}$};
\node[state] at (0, -10) (ix8) {$X_8$};
\node[state] at (0, -8.5) (ix10) {$X_{10}$};
\node[state] at (1*0.75, -11.5) (ix11) {$X_{11}$};

\draw[->] (ix1) -- (iy);
\draw[->] (ix2) -- (iy);
\draw[->] (ix3) -- (iy);
\draw[->] (ix4) -- (ix2);
\draw[->] (ix5) -- (ix6);
\draw[->] (ix4) -- (ix3);
\draw[->] (iy) -- (ix6);
\draw[->] (iy) -- (ix7);
\draw[->] (iy) -- (ix8);
\draw[->] (ix7) -- (ix9);
\draw[->] (ix7) -- (ix11);
\draw[->] (ix8) -- (ix11);
\draw[->] (ix3) -- (ix7);
\draw[->] (ix9) -- (ix8);

\draw node at (5*0.75, -11.8) {$M^{(1)}$};

\end{tikzpicture}
}
\subfigure[]{
\begin{tikzpicture}[scale=0.72, state/.style={circle, draw, minimum size=0.72cm, scale=0.72}]

\draw[black, rounded corners] (-0.8, -5.3) rectangle (4.5, 1.8);

\node[state] at (0, 0) (x1) {$X_1$};
\node[state] at (2*0.75, -1.5) (y) {$Y$};
\node[state] at (2*0.75, 0) (x2) {$X_2$};
\node[state] at (4*0.75, 0) (x3) {$X_3$};
\node[state] at (3*0.75, 1) (x4) {$X_4$}; 
\node[state] at (5*0.75, -1.5) (x5) {$X_5$};
\node[state] at (5*0.75, -3) (x6) {$X_6$};
\node[state] at (3.5*0.75, -3) (x7) {$X_7$};
\node[state] at (1.75*0.75, -3) (x9) {$X_{9}$};
\node[state] at (0, -3) (x8) {$X_8$};
\node[state] at (0, -1.7) (x10) {$X_{10}$};
\node[state] at (1*0.75, -4.5) (x11) {$X_{11}$};
\node[state, color=mypurple] at (0.5, 1) (e) {$E$};

\draw[->] (x1) -- (y);
\draw[->] (x2) -- (y);
\draw[->] (x3) -- (y);
\draw[->] (x4) -- (x2);
\draw[->] (x5) -- (x6);
\draw[->] (x4) -- (x3);
\draw[->] (y) -- (x6);
\draw[->] (y) -- (x7);
\draw[->] (y) -- (x8);
\draw[->] (x7) -- (x9);
\draw[->] (x7) -- (x11);
\draw[->] (x8) -- (x11);
\draw[->] (x3) -- (x7);
\draw[->] (x9) -- (x8);

\draw[->, color=myred] (e) to[out=-90,in=150] (x7);
\draw[->, color=myred] (e) to[out=0, in=180] (x4);
\draw[->, color=myred] (e) to[out=-20,in=105] (x6);
\draw node at (4.2, -4.8) {$\tilde{M}$};

\end{tikzpicture}
} 
\subfigure[]{
\begin{tikzpicture}[scale=0.72, state/.style={circle, draw, minimum size=0.72cm, scale=0.72}]

\draw[black, rounded corners] (-0.65, -5.3) rectangle (9, 1.8);

\node[state] at (0, 0) (x1) {$X_1$};
\node[state] at (2*0.75, -1.5) (y) {$Y$};
\node[state] at (2*0.75, 0) (x2) {$X_2$};
\node[state] at (4*0.75, 0) (x3) {$X_3$};
\node[state] at (3*0.75, 1) (x4) {$X_4$}; 
\node[state] at (5*0.75, -1.5) (x5) {$X_5$};
\node[state] at (5*0.75, -3) (x6) {$X_6$};
\node[state] at (3.5*0.75, -3) (x7) {$X_7$};
\node[state] at (1.75*0.75, -3) (x9) {$X_{9}$};
\node[state] at (0, -3) (x8) {$X_8$};
\node[state] at (0, -1.7) (x10) {$X_{10}$};
\node[state] at (1*0.75, -4.5) (x11) {$X_{11}$};
\node[state, color=mypurple] at (6*0.75, 0) (e) {$E$};

\draw[->] (x1) -- (y);
\draw[->] (x2) -- (y);
\draw[->] (x3) -- (y);
\draw[->] (x4) -- (x2);
\draw[->] (x5) -- (x6);
\draw[->] (x4) -- (x3);
\draw[->] (y) -- (x6);
\draw[->] (y) -- (x7);
\draw[->] (y) -- (x8);
\draw[->] (x7) -- (x9);
\draw[->] (x7) -- (x11);
\draw[->] (x8) -- (x11);
\draw[->] (x3) -- (x7);
\draw[->] (x9) -- (x8);

\draw[->, color=myblue] (e) to[out=-150,in=45] (x9);
\draw[->, color=myblue] (e) to[out=120, in=0] (x4);
\draw[->, color=myblue] (e) to[out=-110,in=45] (x5);
\draw[->, color=myblue] (e) to[out=-60,in=-10] (x11);

\draw[->, color=myyellow] (e) to[out=-160,in=25] (x8);
\draw[->, color=myyellow] (e) to[out=160,in=20] (x3);

\draw[->, color=mygreen] (e) to[out=-70,in=30] (x6);

\draw[->, color=myred] (e) to[out=-140,in=75] (x7);
\draw[->, color=myred] (e) to[out=-80,in=50] (x6);

\node at (6.9-0.3, -2.3) {$\mathcal{E}$ ~~~~$S_\star$ in \eqref{eq:invariant-blanket}};

\node at (7-0.3, -3) {\scriptsize $0\leftrightarrow 0$, $\{1,2,3,5,6,7,8,9\}$};
\node at (7-0.3, -3.5) {\scriptsize $0 \leftrightarrow \myblue{1}$, $\myblue{\{1,2,3,5,6,7,8,9\}}$};
\node at (6.67-0.3, -4) {\scriptsize $0 \leftrightarrow \myyellow{2}$, $\myyellow{\{1,2,3,5,6,7\}}$};
\node at (6.33-0.3, -4.5) {\scriptsize $0 \leftrightarrow \mygreen{3}$, $\mygreen{\{1,2,3,7\}}$};
\node at (6.15-0.3, -5) {\scriptsize $0 \leftrightarrow \myred{4}$, $\myred{\{1,2,3\}}$};
\end{tikzpicture}
}
\caption{(a) is an illustration of the two-environment model, the SCMs in the two environments share the same associated graph: $M^{(0)}$ is an observational environment, and $M^{(1)}$ is an intervention environment where some unknown intervention is applied to $(X_4, X_6, X_7)$, where $M^{(0)}$ and $M^{(1)}$ are defined as \eqref{eq:scm-model}. (b) visualizes $\tilde{G}$, the associated graph of $\tilde{M}$ constructed based on $(M^{(0)}, M^{(1)})$ and \eqref{eq:model-2scm}, which is another plot of the environments in (a). (c) An illustration of \cref{prop:ident-transfer-learning} by showing how $S_\star$ therein will change as we see more and more environments: the arrow from $E$ to $X_j$ with color $e$ means $X_j$ is intervened in $e\in \{\myblue{1}, \myyellow{2}, \mygreen{3}, \myred{4}\}$. For example, $0 \leftrightarrow \mygreen{3}$ means with interventions in environments \myblue{1}, \myyellow{2}, and \mygreen{3}, the invariant variable set is $\mygreen{\{1,2,3,7\}}$.  Although $X_7$ and is reverse causal and hence related to $Y$, we do not know this based only on the given environments.  }
\label{fig:scm-ident}
\end{figure}

\begin{condition}[SCM with Interventions on $X$]
\label{cond:scm-model}
Suppose $M^{(0)},\ldots, M^{(|\mathcal{E}|-1)}$ are defined by \eqref{eq:scm-model}, and $G$ is acyclic. Let $\tilde{M}$ be the model constructed as \eqref{eq:model-2scm} by $\{M^{(e)}\}_{e\in \mathcal{E}}$ with $I$ being given set of variables intervened.
\end{condition}

\subsection{Maximum Invariant Set as the Pragmatic Direct Causes}

We characterize what $S^\star$ would satisfy \eqref{eq:maximum-invariant-set} given a fixed intervention set $I$, and how large $I$ should be to recover the $Y$'s direct causes under arbitrary types of interventions. Define $\mathtt{ch}(k):=\{j: k\in \mathtt{pa}(j)\}$ as the set of children of variable $k$ and $\mathtt{at}(k)$ as the set of all the ancestors of the variable $Z_k$, defined recursively as $\mathtt{at}(k) = \mathtt{pa}(k) \cup \cup_{j\in \mathtt{pa}(k)} \mathtt{at}(j)$ in the topological order of $G$.
The following condition rules out some degenerate cases.

\begin{condition}[Nondegenerate Interventions]
\label{cond:expected-faithfulness}
    The following holds for $\tilde{M}$: (a) $\forall S\subseteq [d]$ containing $Y$'s descendants, 
    if $E \notindep_{\tilde{M}} Y | X_S$, then there exists some $e, e'\in \mathcal{E}$ such that $(\mu^{(e)} \land \mu^{(e')})(\{m^{(e,S)} \neq m^{(e',S)}\}) > 0$; (b) $\widetilde{M}$ is faithful, i.e.,$\forall ~\text{Disjoint}~ A, B, C\subseteq [d+2]$, if $Z_{A} \perp \!\!\! \perp Z_{B} | Z_{C}$, then$Z_{A} \perp \!\!\! \perp _{\widetilde{G}} Z_{B}|Z_{C}$. Here $Z_{A} \perp \!\!\! \perp _{\widetilde{G}} Z_{B} | Z_{C}$ means the node set $A$ and $B$ are d-separated by $C$ in the graph $\widetilde{G}$; see Definition 2.4.1 in \cite{glymour2016causal} for a formal definition of $d$-separation.
\end{condition}

The condition (b), faithfulness on the graph $\tilde{G}$ constraining that the graph $\tilde{G}$ truly depicts all the conditional independence relationships, is widely used in the causal discovery literature. Condition (a) is further imposed since we only leverage the information of conditional expectations instead of conditional distributions. We impose \cref{cond:expected-faithfulness} such that the dependence on $E$ in the conditional expectation of $Y$ given $X_S$ with any $S\subseteq[d]$ can be represented by the graph $\tilde{G}$ itself. The imposed \cref{cond:expected-faithfulness} rules out the possibility of some degenerate cases; see the justifications for \cref{cond:expected-faithfulness} and some degenerate examples in \aosversion{Appendix A.4}{\cref{sec:discussion:faithfulness}}. It should be noted that our general results in \cref{prop:ident-transfer-learning} and \cref{prop:ident-causal-discovery} apply to arbitrary forms of interventions under \cref{cond:expected-faithfulness}, which is a mild condition
as the violation of faithfulness in \cref{cond:expected-faithfulness} occurs with probability zero under some suitable measure on the model \citep{spirtes2000causation}.

\begin{theorem}[Existence of Maximum Invariant Set] 
\label{prop:ident-transfer-learning}
Under \cref{cond:scm-model}, for
\begin{align}
\label{eq:invariant-blanket}
    S_\star = \mathtt{pa}(d + 1) \cup A(I) \cup \bigcup_{j\in A(I)} \left(\mathtt{pa}(j)  \setminus \{d+1\}\right)
\end{align} with $A(I) = \{j: j\in \mathtt{ch}(d+1), j\notin I, \mathtt{at}(j) \cap \mathtt{ch}(d+1) \cap I = \emptyset \}$, we have the invariance $m^{(e,S_\star)} \equiv \bar{m}^{(S_\star)}:= m_\star$. Suppose further
\cref{cond:expected-faithfulness} holds, then \cref{cond-fairnn-ident} holds with $(S^{\star }, m^{\star})= (S_{\star},  m_{\star})$.
\end{theorem}

\cref{prop:ident-transfer-learning} exactly characterizes what $S^\star$ is in our nonparametric invariance pursuit under the SCM with interventions on $X$ -- it doesn't require intervention to be ``sufficient''. Firstly, such a $S^\star$ is well-defined in that there exists one maximum set $S_\star$ satisfying the invariant condition \eqref{eq:intro-model} and heterogeneity condition \cref{cond-fairnn-ident} simultaneously. Secondly, in the SCM setting, such a $S^\star = S_\star$ can be represented in a simple way in \eqref{eq:invariant-blanket}, which lies in between the Markov blanket of the variable $Y$ and the set of $Y$'s direct causes. Note that $A(I)$ can be interpreted as the ``unaffected'' children of $Y$ from the interventions $I$.  As shown in the definition of $A(I)$, the ``unaffected'' children include the children of $Y$ unaffected by both direct interventions in $I$ (itself is not included in $I$) and indirect interventions (it does not have an ancestor that is both $Y$'s child and suffer from intervention). \cref{prop:ident-transfer-learning} states explicitly that the pursued set of invariant variables $S^\star$ is the union of parents of $Y$, unaffected children of $Y$, and parents of these unaffected children. The size of that set $S^\star$ will keep decreasing when $I$ enlarges. It will finally match the direct causes of $Y$ when $I$ includes ``root children set'' $I^\star$ as stated in \cref{prop:ident-causal-discovery} below; see an illustration in \cref{fig:scm-ident} (c).

\begin{proposition}[Direct Cause Recovery] 
\label{prop:ident-causal-discovery}

\noindent {\it (Sufficiency)} Under \cref{cond:scm-model}, define $I^\star = \{j: j\in \mathtt{ch}(d+1), \mathtt{at}(j) \cap \mathtt{ch}(d+1)=\emptyset \}$. If \cref{cond:expected-faithfulness} holds and $I\supseteq I^\star$, then \cref{cond-fairnn-ident} holds with $S^\star = \mathtt{pa}(d+1)$. 

\noindent {\it (Necessity)} Moreover, if $\bar{m}^{(S^\star \cup S)} \neq m^\star$ for any $S$ with $\mathtt{ch}(d+1) \cap S \neq \emptyset$, i.e., $Y$ does not have degenerate children, then \cref{cond-fairnn-ident} holds only if $I\supseteq I^\star$.
\end{proposition}

We refer to $I^\star$ as the \emph{minimal intervention set} because it is the exact minimal set of variables that should be intervened on for exact direct cause recovery in general, nondegenerate cases. The set $I^\star$ is determined by the cause-effect relationship graph $G$. In particular, $I^\star$ is $\{6, 7\}$ for the example in \cref{fig:scm-ident}. Notably, $X_8$ does not require intervention, as $X_7$, one of its ancestors, is included in $I^\star$.

Unfortunately, $S_\star \supsetneq \mathtt{pa}(d+1)$ when $I^\star\not\subseteq I$ in general. This is due to a lack of evidence in environments to falsify that some variables in $S^\star$ are not direct causes. Nevertheless, $S^\star=S_\star$ in this setup can still be interpreted as the ``contemporary direct causes'' or ``pragmatic direct causes'' of $Y$ based on the observed environments. We refer to it as ``pragmatic direct causes'' from the perspective of future prediction. The direct causes of $Y$ have implications in robust transfer learning because the conditional moment of $Y$ given direct causes is the most predictive one among all the transferable associations under the worst case where all the covariates are arbitrarily strongly intervened. The ``pragmatic direct causes'' can be understood similarly if future interventions are made within the intervened variables $X_I$. Particularly, if the future interventions are made within the set $I$, then $S^\star$ can be regarded as the direct causes from a pragmatic perspective since the conditional expectation of $Y$ given $X_{S^\star}$ will remain invariant in a new environment $t$. Moreover, it depicts the most predictive one among all the associations in the observational environment $e=0$ that remains in the environment $t$. 

\begin{proposition}[Robust Transfer Learning]
\label{prop:rtl}
 Under \cref{cond:scm-model}, for a new environment $t$ with SCM $M^{(t)}=\{\mathcal{S}^{(t)}, \nu\}$ satisfying $f_j^{(t)} \equiv f_j^{(0)}$ for any $j\in [d+1]\setminus I$, i.e., only $X_I$ is intervened, we have $\mathbb{E}[Y^{(t)}|X_{S_\star}^{(t)}] \equiv \mathbb{E}[Y^{(0)}|X_{S_\star}^{(0)}]$ with $S_\star$ in \eqref{eq:invariant-blanket}. If \cref{cond:expected-faithfulness} holds and $M^{(t)}$ satisfies a condition akin to \cref{cond:expected-faithfulness} (see \aosversion{Appendix A.5}{\cref{sec:rtl}}), then $S_\star$ is the maximum set whose conditional expectation is transferable in that for any $S\subseteq [d]$ such that $\mathbb{E}[Y^{(t)}|X_{S_\star \cup S}^{(t)}] \neq \mathbb{E}[Y^{(t)}|X_{S_\star}^{(t)}]$, one has $\mathbb{E}[Y^{(t)}|X_{S}^{(t)}] \neq \mathbb{E}[Y^{(0)}|X_{S}^{(0)}]$.
\end{proposition}

\section{A Unified Framework}
\label{sec:ip}

The proposed FAIR-NN least squares is a special instance of our generic FAIR estimation framework, which homogenizes different risk losses and prediction models. Moreover, our framework also allows the user to incorporate additional structural knowledge into estimation such that identification is sometimes viable when $|\mathcal{E}|=1$. The invariance pursuit problem, the estimation method, and the non-asymptotic results will be presented in a unified manner in this section. 

\subsection{General Invariance Pursuit from Heterogeneous Environments}
\label{sec:model}

In this section, we formalize the problem of \emph{invariance pursuit} using data from multiple environments, which admits the canonical \emph{nonparametric invariance pursuit} in \cref{sec:intro-problem} as a special case.

Let $Y \in \mathbb{R}$ be the response variable and $X\in \mathbb{R}^d$ be the explanatory variable.
We consider the general setting in which we have collected data from multiple environments $\mathcal{E} = \{e_1, \ldots, e_{|\mathcal{E}|}\}$, where $\mathcal{E}$ is the set of a finite number of environments. In each environment $e\in \mathcal{E}$, we observe $n$ i.i.d. observations $\{(X^{(e)}_i, Y^{(e)}_i)\}_{i=1}^n$ that follow from some distribution $\mu^{(e)}$. Let $\Theta_g, \Theta_f \subseteq \Theta$ be the class of prediction and testing functions, respectively. Our goal is to estimate the underlying \emph{invariant regression function} $g^\star \in \Theta_g$ satisfying the invariance structure
\begin{align}
\label{eq:invariant-structure}
\forall e\in \mathcal{E} \qquad \mathbb{E} \left[\left(Y^{(e)} - g^\star(X_{S^\star}^{(e)}) \right) f(X^{(e)}_{S^\star})\right] = 0 \qquad \forall f\in [\Theta_f]_{S^\star},
\end{align} where $S^\star$ is the \emph{unknown} set of true important variables.
We refer to the above problem as \emph{invariance pursuit} or \emph{causal pursuit} exchangeably, as no evidence against causality with the available experiments.

The problem of estimating $g^\star$ in \eqref{eq:invariant-structure} is a generalized version of the canonical 
\emph{nonparametric invariance pursuit} with $g^\star = m^\star$ in \eqref{eq:intro-model} and $\Theta_f = \Theta_g = \Theta$. It depicts a general form and unifies several problems of interest in predecessors. For example, when $\Theta_g$ and $\Theta_f$ are all linear function classes, it reduces to the \emph{linear invariance pursuit} problem, i.e., estimating $g^\star(x) = (\beta^\star)^\top x = (\beta^\star_{S^\star})^\top x_{S^\star}$ with $\beta^\star \in \mathbb{R}^d$ satisfying $\supp(\beta^\star) = S^\star$ in the multi-environment linear regression \citep{fan2024environment} with linear invariance structure
\begin{align} \label{fan2}
    \mathbb{E} \left[ \left(Y^{(e)} - (\beta^\star_{S^\star})^\top X_{S^\star}^{(e)} \right) X_{j}^{(e)} \right] = 0 \qquad \forall e\in \mathcal{E}, j\in S^\star.
\end{align} 

Another example is the \emph{augmented linear invariance pursuit} where $\Theta_g$ is linear and $\Theta_f=\{f(x)=\sum_{j=1}^d \beta_{0,j} x_j + \beta_{1,j} \phi(x_j) \}$ with some transform function $\phi: \mathbb{R}\to \mathbb{R}$.  This can further generalize this to multiple transformed testing functions such as $\phi_1(x_j)=x_j^2$ and $\phi_2(x_j)=|x_j|$ but we keep one here for simplicity. The augmented linear invariance structure that realizes \eqref{eq:invariant-structure} in this case is, for all $e\in \mathcal{E}, j\in S^\star$,
\begin{align}  \label{fan3}
    \mathbb{E} \left[ \left(Y^{(e)} - (\beta^\star_{S^\star})^\top X_{S^\star}^{(e)} \right) X_{j}^{(e)} \right] = \mathbb{E} \left[ \left(Y^{(e)} - (\beta^\star_{S^\star})^\top X_{S^\star}^{(e)} \right) \phi(X_{j}^{(e)}) \right] = 0.
\end{align} 
It coincides with the problem considered by \cite{fan2014endogeneity} when $|\mathcal{E}|=1$ and our method reduces to the FGMM method therein. The \emph{augmented linear invariance pursuit} leverages further a part of the structural knowledge that $\mathbb{E}[Y^{(e)}|X_{S^\star}^{(e)}] = (\beta^\star_{S^\star})^\top X_{S^\star}^{(e)}$, which is much weaker than the assumption $\mathbb{E}[Y^{(e)}|X^{(e)}] = (\beta^\star_{S^\star})^\top X_{S^\star}^{(e)}$ in the sparse linear regression. Identification is possible in this case even when $|\mathcal{E}|=1$.  This is important for most biological medical studies, where data are usually collected in similar settings.  In this case, the FAIR penalty eliminates endogenously spurious variables, making traditional variable selection methods applicable.

\begin{remark}
\label{remark:variable-sel}  
We point out here that there are two kinds of spurious variables. One is endogenously spurious variables such as $X_2=$ background color, and the other is exogenously spurious variables such as $X_3=$ the time the photo was taken or the types of camera used.  The former is harmful, and the latter is nearly harmless in statistical prediction, transfer learning, and even statistical attribution or causality, thinking of $X_3$ as a weak causal variable. The introduction of our FAIR method is to surely screen \citep{fan2008sure} the endogenously spurious variables while keeping all the important variables as in \eqref{eq:var-sel-fairnn}. Exogenously spurious variables can be reduced by using commonly used variable selection methods such as Lasso, SCAD, and best subsets. See \aosversion{Appendix A.7}{\cref{sec:varsel}} for how to attain variable selection consistency.
\end{remark}

Similar to the discussion in \cref{sec:intro-problem}, the main challenge here is the curse of endogeneity. To address this issue, we will harness the insight that the distributions of $(X, Y)$ across diverse environments capture the invariance structure \eqref{eq:invariant-structure}. The key idea is to exploit both the heterogeneity among different environments and the above invariance structure \eqref{eq:invariant-structure} to pinpoint the invariant regression function $g^\star$. 

It should be noted that both $g^\star$ and $S^\star$ are determined by $(\Theta_g, \Theta_f)$ and $\mathcal{E}$ through the structure \eqref{eq:invariant-structure}. It is required that $\partial \Theta_g = \{g - g': g, g'\in \Theta_g\} \subseteq \Theta_f$. In the case of $\Theta_f=\partial \Theta_g$, one uses only heterogeneity among different environments, or the ``invariance principle'', to identify the invariant regression function $g^\star$, as in \eqref{fan2}. Heterogeneous environments are essential in this case. By choosing substantially large $\Theta_f \supsetneq \partial \Theta_g$, one further injects the strong structural assumption that the invariant regression function lies in the class $\Theta_g$ rather than $\Theta_f \setminus \Theta_g$ as in \eqref{fan3}. In this case, one leverages both heterogeneity among environments, i.e., the ``invariance principle'', and the mentioned prior structure knowledge, i.e., the ``asymmetry principle'', to jointly identify $g^\star$. Only one environment may be enough for identifying $g^\star$ when the intersection of both principles gives sufficient conditions.

\subsection{General FAIR Estimation Framework}

Let $\ell: \mathbb{R} \times \mathbb{R} \to \mathbb{R}$ be a user-determined risk loss such that 
\begin{align}
\label{eq:loss-constrain}
    \frac{\partial \ell(y, v)}{\partial v} = (v - y) \psi(v) \qquad \text{and} \qquad \frac{\partial^2 \ell(y, v)}{\partial v^2} > 0,
\end{align} 
which is slightly more general than the quasi-likelihood in the generalized linear model
\citep{nelder1972generalized}.  The constraints in \eqref{eq:loss-constrain} ensure that the conditional expectation aligns with the unique global minima and can be satisfied by various risk losses. Two leading examples are the least square loss $\ell(y, v) = \frac{1}{2}(y-v)^2$ with $\psi(v)=1$ for regression, and the cross-entropy loss $\ell(y, v) = -\log(1-v) - y\log\{v/(1-v)\}$ with $\psi(v)=1/\{v(1-v)\}$ for classification. 

Given all the data $\{\{(X_i^{(e)}, Y_i^{(e)})\}_{i=1}^n\}_{e\in \mathcal{E}}$ from heterogeneous environments together with $(\Theta_g, \Theta_f)$ that may encode part of the prior information when $\Theta_g \neq \Theta$, our proposed focused adversarial invariance regularized estimator (FAIR estimator) is the solution to the subsequent minimax optimization objective
\begin{align}
\label{eq:def-fair-estimator}
    \hat{g} \in \argmin_{g\in \mathcal{G}} \sup_{f^{\mathcal{E}} \in \{\mathcal{F}_{S_g}\}^{|\mathcal{E}|}} \underbrace{\hat{\mathsf{R}}(g) + \gamma \hat{\mathsf{J}}(g, f^{\mathcal{E}})}_{=:\hat{\mathsf{Q}}_\gamma(g, {f}^{\mathcal{E}})}. 
\end{align} 
where $\mathcal{G} \subseteq \Theta_g$ and $\mathcal{F} \subseteq \Theta_f$ are function classes that approximates $\Theta_g$ and $\Theta_f$, respectively. Here $\hat{\mathsf{R}}(g)$ is the pooled sample mean of the user-specified loss across all the environments $\mathcal{E}$:
\begin{align}
\label{eq:method-empirical-r}
    \hat{\mathsf{R}}(g) = \frac{1}{|\mathcal{E}|} \sum_{e\in \mathcal{E}} \hat{\mathbb{E}} \left[\ell(Y^{(e)}, g(X^{(e)}))\right] = \frac{1}{|\mathcal{E}| \cdot n} \sum_{e\in \mathcal{E}, i\in [n]} \ell(Y^{(e)}_i, g(X^{(e)}_i)),
\end{align} 
$\gamma$ is the hyper-parameter to be determined, and $\hat{\mathsf{J}}(g, f^{\mathcal{E}})$ is defined the same as \eqref{eq:method-empirical-j}. We summarize the framework proposed in \cref{fair-algo}.

\begin{algorithm}
\caption{FAIR Estimation}
\begin{algorithmic}[1]
\State \textbf{Input:} Data $\{\mathcal{D}^{(e)}\}_{e\in \mathcal{E}}$ with $\mathcal{D}^{(e)} = \{(X_i^{(e)}, Y_i^{(e)})\}_{i=1}^n$ from $|\mathcal{E}|$ environments. Determine risk loss $\ell(\cdot, \cdot)$.
\State Choose predictor function class $\mathcal{G}$. 
\State \myred{Choose testing function class $\mathcal{F} = \partial \mathcal{G}$, unless with prior knowledge that the target function $\notin \Theta_f \setminus \partial \Theta_g$.}
\State \myred{Choose invariance hyper-parameter $\gamma$}.
\State Solve the \myred{minimax program in \eqref{eq:def-fair-estimator}.}
\end{algorithmic}
\label{fair-algo}
\end{algorithm}

The difference compared with the standard empirical risk minimizer is outlined in \myred{red}: the choice of testing function can be the default $\mathcal{F} = \partial \mathcal{G}$ in the absence of additional priors. Though one additional hyper-parameter $\gamma$ is introduced, our theorem and empirical studies show it has no effect when $n$ is large. So we recommend picking a large enough $\gamma$ like $\gamma=36$ for the causal discovery task and can use either one additional validation set or leave-one-out cross-validation to optimize the prediction error; see the idea of data-driven determination of $\gamma$ in Appendix D.7 of \cite{fan2024environment}. Our \cref{sec:implementation} proposes an efficient implementation of Step 5 if running least squares on $\mathcal{G}$ can be solved by gradient descent, which is quite mild.

From a high-level perspective, our proposed FAIR estimator searches for the most predictive variable set $S$ that preserves some invariance structure imposed by the specification of $(\Theta_g, \Theta_f)$. The framework presented has several limitations: (1) the loss $\ell$ has restrictions in that the conditional expectation must uniquely minimize it; (2) the environment label is discrete; and (3) the discussion still lies within the variable selection level invariance rather than general representation level invariance. We will discuss in \aosversion{Appendix A.3}{\cref{sec:extension}} that our entire framework can be easily extended to the cases where (1) and (2) fail to hold. We add some discussions on the rationale, comparison with IRM, and extension on (3) in \aosversion{Appendix A.2}{\cref{appendix:dis2}}.

\subsection{Sketch of the Generic Result and Its Applications}
\label{sec:sketch}

The non-asymptotic results in \cref{sec:fairnn} can be extended to the general FAIR estimation framework, formally stated in \aosversion{Theorem B.1}{\cref{thm:oracle}}, which unifies the identification condition and $L_2$ estimation errors for specific $(\Theta_g, \Theta_f)$ or $(\mathcal{G}, \mathcal{F})$  under the least squares loss $\ell(y,v)=\frac{1}{2}(y-v)^2$. We sketch the main idea here and defer the complete result and applications to \aosversion{Appendix B}{\cref{sec:theory-appendix}}.

Suppose $[\Theta_g]_S$ and $[\Theta_f]_S$ are closed subspaces of $\Theta_S$ for any $S\subseteq[d]$ so that one can define $\bar{g}^{(S)}(x) = \argmin_{g\in [\Theta_g]_S} \|g - \bar{m}^{(S)}\|_2$ and $f^{(e,S)}(x) = \argmin_{f\in [\Theta_f]_S} \|f - m^{(e,S)}\|_{2,e}$.
Then, the invariant structure and the invariant regression function in \eqref{eq:invariant-structure} can be simplified as
\begin{align}
\label{eq:invariance-thm}
    f^{(e,S^\star)}(x) \equiv \bar{g}^{(S^\star)}(x) := g^\star(x).
\end{align}
Similar to the nonparametric bias mean and bias variance in \cref{remark:bias-mean-variance}, we can define the generalized bias mean and bias variance with respect to $(\Theta_g, \Theta_f)$ as $\mathsf{b}(S) = \|\bar{g}^{(S\cup S^\star)} - g^\star\|_2^2$ and $\bar{\mathsf{d}}(S) = \frac{1}{|\mathcal{E}|} \sum_{e\in \mathcal{E}} \|\bar{g}^{(S)} - f^{(e,S)}\|_{2,e}^2$. The general identification condition akin to \cref{cond-fairnn-ident} is 
\begin{align}
\label{eq:ident-thm}
\forall ~ S\subseteq [d],  \qquad \mathsf{b}(S)>0 ~~ \Longrightarrow ~~ \bar{\mathsf{d}}(S)>0.
\end{align} 
It requires that whenever incorporating more variables in $S$ leads to better prediction performance, the set $S$ will not satisfy the invariance structure \eqref{eq:invariant-structure}. \cref{cond-fairnn-ident} instantiates \eqref{eq:ident-thm} by letting $\bar{\mathsf{d}}(S) =\bar{\mathsf{d}}_{\mathtt{NN}}(S)$ and ${\mathsf{b}}(S) = \mathsf{b}_{\mathtt{NN}}(S)$ with $(\mathsf{b}_{\mathtt{NN}}(S), \bar{\mathsf{d}}_{\mathtt{NN}}(S))$ defined in \eqref{eq:bnndnn}.

\begin{theorem}[Main Result for FAIR Least Squares Estimator, Informal] \label{thm:informal}
Under \eqref{eq:invariance-thm}, \eqref{eq:ident-thm}, and some regularity conditions in regression, one can consistently estimate $g^\star$ by choosing $\gamma \ge 8\sup_{S:\mathsf{b}(S)>0} \{\mathsf{b}(S)/\bar{\mathsf{d}}(S)\}$. In this case, the FAIR estimator $\hat{g}$ in \eqref{eq:def-fair-estimator} with $\ell(y,v)=\frac{1}{2}(y-v)^2$ satisfies, for any $n\ge 3$, w.h.p.,
\begin{align}
\label{eq:error-bound}
    \frac{\|\hat{g} - g^\star\|_2}{C_1} \le \delta_{\mathtt{stoc}} + \delta_{\mathtt{approx}}^\star + \gamma ( \delta_{\mathtt{stoc}} + \delta_{\mathtt{approx}}) 1_{\{\delta_{\mathtt{stoc}} + \delta_{\mathtt{approx}} \ge \frac{s}{1+\gamma}\}}
\end{align} Here $\delta_{\mathtt{stoc}}$ is the stochastic error characterized by the local Rademacher complexity of $\mathcal{F}, \partial \mathcal{G}$ and $n$, $\delta_{\mathtt{approx}}^\star$ measures certain approximation error of $(\mathcal{G}, \mathcal{F})$ w.r.t. $g^\star$, and $\delta_{\mathtt{approx}}$ measures the worst case approximation error of $(\mathcal{G}, \mathcal{F})$ w.r.t. all the $\{f^{(e,S)}\}$. The constant $s>0$ is the signal strength related to $\min_{S: \bar{\mathsf{d}}(S)>0} \bar{\mathsf{d}}(S)$ and $\min_{S: S^\star \setminus S\neq \emptyset} \inf_{g\in [\Theta_g]_S} \|g - g^\star\|_2$, and $C_1$ is a universal constant independent of the two quantities. 
\end{theorem}

\begin{table}[!t]
    \centering
    \small
    \begin{tabular}{lllllll}
      \hline
     $\Theta_g$ & $\Theta_f$ & $\mathcal{G}$ & $\mathcal{F}$ & Priors & $|\mathcal{E}|=1$ Ident & Result \\
      \hline
      \hline
      Linear & Linear & Linear & Linear & None & Impossible & Thm \aosversion{B.5}{\ref{thm:lglf}}\\
      Linear & Linear w/ $\phi$ & Linear & Linear w/ $\phi$ & Nearly Linear & Possible & Thm \aosversion{B.6}{\ref{thm:lgbf}} \\
      Linear & $\Theta$ & Linear & NN & Linear & Possible & Thm \aosversion{B.7}{\ref{thm:lgnf}} \\
      Additive & $\Theta$ & Additive NN & NN & Additive & Impossible & Thm \aosversion{B.4}{\ref{thm:fairann}}\\
      $\Theta$ & $\Theta$ & NN & NN & None & Impossible & Thm \ref{thm:fairnn}\\
      \hline
    \end{tabular}
    \caption{Applications of \cref{thm:informal}. Recall that $\Theta$ is the set of all $L_2(\bar{\mu}_x)$ functions. For the function classes in columns $\Theta_g, \Theta_f, \mathcal{G}$ and $\mathcal{F}$, ``Linear'' is $\{f(x)=\sum_{j=1}^d \beta_{j} x_j\}$, ``Linear w/ $\phi$'' is $\{f(x)=\sum_{j=1}^d \beta_{j} x_j + \alpha_j \phi(x_j)\}$, ``NN'' is deep ReLU network class, ``Additive'' is the additive functions $\{f(x)=\sum_{j=1}^d f_j(x_j)\}$ and ``Additive NN'' is a structured neural network approximating additive functions. The column ``Priors'' indicates what prior structure knowledge is injected by the choice of $(\Theta_g, \Theta_f)$. For the second row, it is ``nearly linear'' given it only requires that the residual is uncorrelated with all the $\phi(x_j)$ with $j\in S^\star$; the prior for the third row is exactly linear provided $\Theta_f = \Theta$. The column ``$|\mathcal{E}|=1$ Ident'' indicates whether identification for $S^\star$ in \eqref{eq:intro-model} is possible with only one environment. 
    }
    \label{table:estimators-main-text}
\end{table}

The complete and rigorous statement is deferred to \aosversion{Theorem B.1}{\cref{thm:oracle}} in \aosversion{Appendix B.1}{\cref{sec:main-result}}, with more loss function $\ell$ in \aosversion{Theorem B.2}{\cref{thm:oracle-2}}. These generic results can characterize several advantages in our FAIR framework's sample efficiency. Firstly, the error \eqref{eq:error-bound} is structure-agnostic in that it is represented by the sum of approximation error and stochastic error, indicating that (1) our framework can fully exploit the capability of $(\mathcal{G}, \mathcal{F})$ in learning low-dimensional structures, and (2) it has almost no additional cost in sample efficiency compared with standard regression. Moreover, the error rate applies to any $n$, implying the estimation error is guaranteed even when it selects the wrong variable, especially when the signal $s$ is weak. Finally, though a large enough regularization hyper-parameter $\gamma$ is needed to guarantee consistent estimation, the error will be free of $\gamma$ when $n$ is large enough. We also apply our unified result to various specifications of $(\mathcal{G}, \mathcal{F})$, including the non-asymptotic results in identification and convergence rate; see a summary in \cref{table:estimators-main-text}.

\section{Experiments}
\label{sec:exp}

\subsection{An End-to-End Implementation}
\label{sec:implementation}

We realize the minimax optimization using gradient descent ascent, a similar approach adopted in GAN \citep{goodfellow2014generative} training. The main challenge here is how to do ``focused regularization'' that enforces $f^{(e)} \in \mathcal{F}_{S_g}$. Here we consider a re-parameterization trick that disentangles the function $g$ and the variable $S_g$ it selects. To start with, we can write $g(x) = g(a \odot x) = g(x_1 a_1,\ldots, x_d a_d)$ with $a\in \{0,1\}^d$ indicating presence/absence of variables.
Then we can write the objective \eqref{eq:def-fair-estimator} as
\begin{align} \label{fan5}
    (\hat{g}, \hat{a}) \in \argmin_{g\in \mathcal{G}, a\in \{0,1\}^d} \sup_{f^{\mathcal{E}} \in \{\mathcal{F}\}^{|\mathcal{E}|}} \hat{\mathsf{R}}(g(a \odot \cdot)) + \gamma \hat{\mathsf{J}}(g(a \odot \cdot), f^{\mathcal{E}}(a \odot \cdot))
\end{align}
A naive implementation is to first enumerate all the possible $a\in \{0,1\}^d$ and then do gradient descent ascent for given $a$, which is computationally inefficient.  To avoid this, we first rewrite the optimization as a ``continuous'' optimization:
\begin{align*}
    (\hat{g}, \hat{w}) \in \argmin_{g\in \mathcal{G}, w \in R^d} \sup_{f^{\mathcal{E}} \in \{\mathcal{F}\}^{|\mathcal{E}|}} \mathbb{E}_{B(w)} \left [ \hat{\mathsf{R}}(g(B(w) \odot \cdot)) + \gamma \hat{\mathsf{J}}(g(B(w) \odot \cdot), f^{\mathcal{E}}(B(w) \odot \cdot)) \right ],
\end{align*}
where the $j^{th}$ component of $B(w) \in \{0, 1\}^d$ follows an independent Bernoulli with probability of success $\sig {w_j} = \exp(w_j)/(1+\exp(w_j))$.  This is easily seen by taking $\hat{w} = \mbox{logit}(\hat a)= \log (\frac{\hat a}{1 - \hat a})$.  Note that $B_j(w_j) = \indicator\{\mbox{logit} (U_j) \leq w_j\}$ is discontinuous in $w_j$ where $U_j \sim$ uniform[0,1], but can be approximated as
\begin{equation} \label{fan6}
    B_j(w_j) \approx \frac{1}{1 + e^{ (\mbox{\scriptsize logit}(U_j) - w_j))/\tau}} \equiv V_\tau(U_j, w_j) ~~~~ \text{as} ~~~~\tau \to 0^+,
\end{equation}
for which its gradient can be taken.  Let $A_\tau(U, w) = (V_\tau(U_{1}, w_1),\ldots, V_\tau(U_{d}, w_d))^\top \in \mathbb{R}^d$ with $\{U_j\}_{j=1}^d$ being i.i.d. uniform random variables.
One can thus approximate \eqref{fan5} by
\begin{align}
    (\hat{\theta}, \hat{w}) &\in \argmin_{\theta \in \mathbb{R}^{N_g}, w\in \mathbb{R}^d} \sup_{\forall e\in \mathcal{E}, \phi^{(e)} \in \mathbb{R}^{N_f}} \mathbb{E}_{U}[\hat{\mathsf{L}}(A_\tau(U,w), \theta, \{\phi^{(e)}\}_{e\in \mathcal{E}})],
\end{align} with $\hat{\mathsf{L}}(A, \theta, \{\phi^{(e)}\}_{e\in \mathcal{E}})]= \hat{\mathsf{R}}(g(A \odot \cdot; \theta)) + \gamma \hat{\mathsf{J}}(g(A \odot \cdot), f^{\mathcal{E}}(A \odot \cdot; \{\phi^{(e)}\}_{e\in \mathcal{E}}))$, where parametrizations of $g \in \cG$ and $f^{e} \in \cF$ are used.  Since $\mbox{logit}(U_j) \stackrel{d}{=} U_{j,1}- U_{j, 2}$ with $\{U_{j,1}, U_{j,2}\}_{j=1}^d$ being i.i.d. Gumbel(0,1) random variables, the approximation \eqref{fan6} is also referred to as the Gumbel approximation.

One can use similar implementation tricks widely used in stochastic gradient descent with Gumbel approximation that gradually anneals the Gumbel approximation hyperparameter $\tau$; see the pseudo-code in \aosversion{Appendix C.1}{\cref{sec:code}}. We include the simulation for linear models and applications of causal discovery in the main text and defer the simulation for FAIR-NN estimator to  \aosversion{Appendix C.2}{\cref{sec:fairnn-simulation}} and robust prediction of water/land birds to \aosversion{Appendix C.3}{\cref{sec:app-bird}}.

\subsection{Simulations for FAIR-Linear Estimator}
\label{sec:simu-linear}

In this section, we present the simulation result for the FAIR-Linear estimator implemented by the Gumbel approximation trick and gradient descent ascent algorithm. 

\noindent \textbf{Data Generating Process.} We consider the case where $|\mathcal{E}|=2$ and the data $(X^{(e)}, Y^{(e)})$ in each environment $e\in \{0,1\}$ are generated from two SCMs sharing the same causal relationship between variables. For each trial, we first generate the parent-children relationship among the variables. We enumerate all the $i\in [d+1]$. For each $i\in [d+1]$, we randomly pick at most $4$ parents for the variable $Z_i$ from $\{Z_1,\ldots, Z_{i-1}\}$, this step ensures that the induced graph is a  DAG. We use fixed $d=70$, and let the variable $Z_{36}$ be $Y$ and the rest variables constitute the covariate $X$, that is, we let $(Z_1,\ldots, Z_{35}, Z_{36}, Z_{37}, \ldots, Z_{71})=(X_1,\ldots, X_{35}, Y, X_{36}, \ldots, X_{70})$. We also enforce that $Y$ has at least $5$ parents and at least $5$ children by adding parents and children when needed.  The structural assignment for each variable $Z_j$ is defined as
\begin{align*}
    Z_j^{(e)} \gets \sum_{k\in \mathtt{pa}(j)} C_{j,k}^{(e)} f_{j,k}^{(e)}(Z_k^{(e)}) + C_{j, j}^{(e)} \varepsilon_j
\end{align*} where $(\varepsilon_1,\ldots, \varepsilon_{71})$ are independent standard normal distributed. For $j\neq 36$, $f_{j,k}^{(e)}$ are sampled randomly from the candidate functions $\{\cos(x), \sin(x), \sin(\pi x), x, 1/(1+e^{-x})\}$, $C_{j,k}^{(e)}$ are sampled from $\mathrm{Uniform}[-1.5, 1.5]$ with $|C_{j,j}^{(e)}| \ge 0.5$. For $j=36$ and $k<j$, we have $f_{36,k}^{(e)}(x)=x$ and $C_{36,k}^{(0)}\equiv C_{36,k}^{(1)}$ for linearity and invariance. The above data-generating process can be regarded as one observation environment $e=0$ and an interventional environment $e=1$ where the random and simultaneous interventions are applied to all the variables other than the variable $Y$, while the assignment from $Y$'s parent to $Y$ remains and furnishes the target regression function $m^\star(x) = \sum_{k\in \mathtt{pa}(36)} C_{36, k}^{(e)} x_k$ in pursuit. In this case, we let $S^\star = \mathtt{pa}(36)$ and $\beta^\star$ with support set $S^\star$ be such that $\beta_j^\star=C^{(0)}_{36, k}=C^{(1)}_{36, k}$ for any $k\in S^\star$. We also let the noise variance be different for the two environments, i.e., $C_{36,36}^{(0)} \neq C_{36,36}^{(1)}$. Now, the model only has conditional expectation invariance rather than the full conditional distribution invariance. \cref{fig:visualize-linear} (a) visualizes the induced graph in one trial. The complex cause-effect relationships in high-dimensional variables make it very challenging to estimate $\beta^\star$.

\begin{figure}[!t]
\centering
\begin{tabular}{cc}
\subfigure[]{
\includegraphics[width=0.35\linewidth]{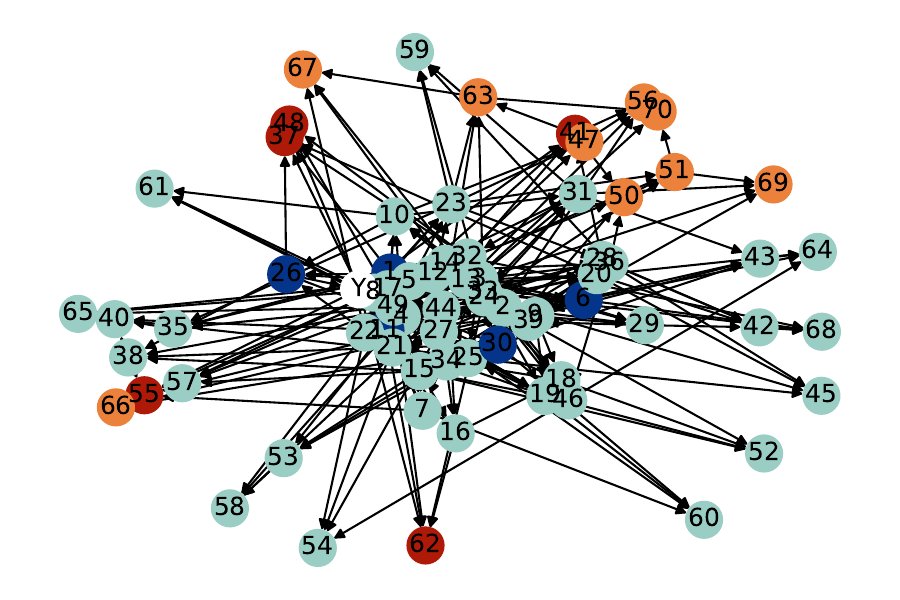}
}&
\subfigure[]{
\includegraphics[width=0.35\linewidth]{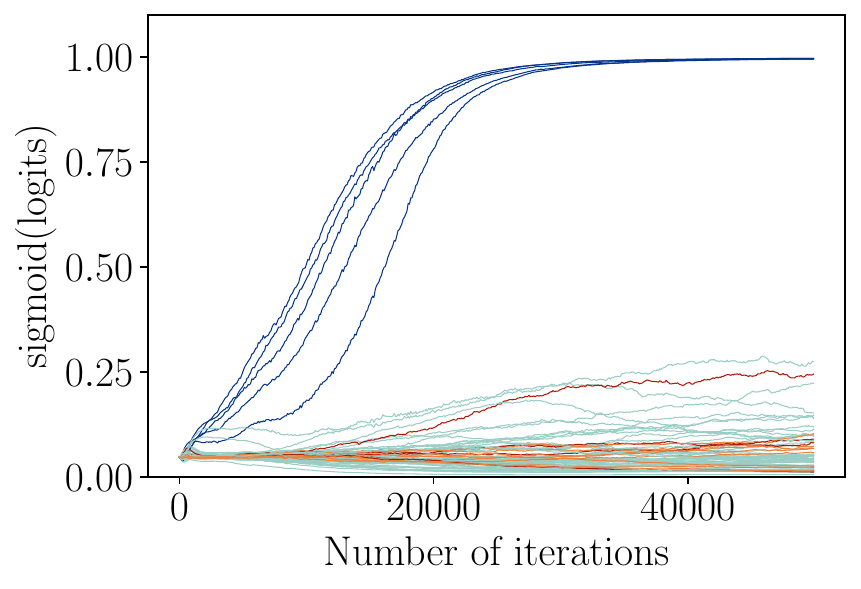}
}
\end{tabular}
\caption{The visualization of (a) the SCM and (b) the $\sig{w}$ during training in one trial for the FAIR-Linear estimator. We use different colors to represent the different relationships with $Y$: \myblue{blue} = parent, \myred{red} = child, \myorange{orange} = offspring, \mylightblue{lightblue} = other.}
\label{fig:visualize-linear}
\end{figure}

\noindent \textbf{Implementation.} For the FAIR-Linear estimator, we realize $\mathcal{G}$ and $\mathcal{F}$ by linear function classes, i.e., $\mathcal{G}=\{g(x)=\beta_g^\top x: \beta_g \in \mathbb{R}^d\}$ and $\mathcal{F}=\{f(x)=\beta_f^\top x: \beta_f \in \mathbb{R}^d\}$, and run gradient descent ascent using Adam optimizer with a learning rate of 1e-3, batch size $64$ for $50k$ iterations. In each iteration, one gradient descent update of the parameters of the predictor $\beta_g$ and Gumbel logits parameters $w$ is followed by the three gradient ascent updates of the discriminators' parameters $(\beta_f^{(1)}, \beta_f^{(2)})$. We adopt a fixed hyper-parameter $\gamma=36$ and report the performance of the following estimators using the median of the estimation error $\|\hat{\beta} - \beta^\star\|_2^2$ over $50$ replications and varying $n\in \{200, 500, 1000, 2000, 5000\}$.
\begin{itemize}[itemsep=0pt]
    \item[(1)] Pool-LS: it simply runs least squares on the full covariate $X$ using all the data. 
    \item[(2)] FAIR-GB: Our FAIR-Linear estimator with Gumbel trick that outputs $\beta_g \odot \sig{w}$.
    \item[(3)] FAIR-RF: it selects the variables $x_j$ with $\sig{w_j}>0.9$ of the fitted model in (2), i.e., $\hat{S}=\{j: \sig{w_j}>0.9\}$, and refits least squares again on $X_{\hat{S}}$ using all the data.
    \item[(4)] Oracle: it runs least squares on $X_{S^\star}$ using all the data.
    \item[(5)] Semi-Oracle: it runs least squares on $X_{G^c}$ using all the data, where $G$ is the set of all the descendants of $Y$. It is unbiased yet has a larger variance compared with the Oracle one. 
\end{itemize}

\cref{fig:visualize-linear} (b) visualizes how the Gumbel gate values for different covariates $\sig{w}$ evolve during training in one trial. We can see that $\sig{w_j}$ for $j\in S^\star$ quickly increases and dominates the values for other variables like children/offspring of $Y$. 

\noindent \textbf{Results. } The results are shown in \cref{fig:simulation-result} (a). We can see that the square of the $\ell_2$ estimation error $\|\hat{\beta} - \beta^*\|_2^2$ for the pooled least squares estimator (\myyellow{$\times$}) does not decrease and remains to be very large ($\approx 1.5$) as $n$ increases, indicating that it converges to a biased solution. At the same time, the estimation error for FAIR-GB (\mylightblue{$\blacklozenge$}) decays as $n$ grows ($\approx 0.01$ when $n=1k$) and lies in between that for least squares on $X_{G^c}$ (Semi-Oracle \myorange{$\blacktriangledown$}) and least squares on $X_{S^*}$ (Oracle \myred{$\blacktriangle$}). This is expected to happen since the FAIR-Linear estimator is not designed to screen out all the exogenously spurious variables: They can be further regularized using the commonly used variable selection techniques; see \cref{remark:variable-sel}. We also observe that the training dynamics of adversarial estimation are highly non-stable: though it can converge to an estimate around $\beta^\star$ when $n$ is very large, it fails to converge to $\beta^\star$ at a comparable rate compared to the standard least squares. The FAIR-RF (\myblue{$+$}) estimator then completes the last step towards attaining better accuracy in this regard: we can see that its performances are very close to that of the Oracle estimator when $n$ is very large ($n=5000$).

\begin{figure}[!t]
\centering
\begin{tabular}{cc}
\subfigure[]{
\includegraphics[width=0.35\linewidth]{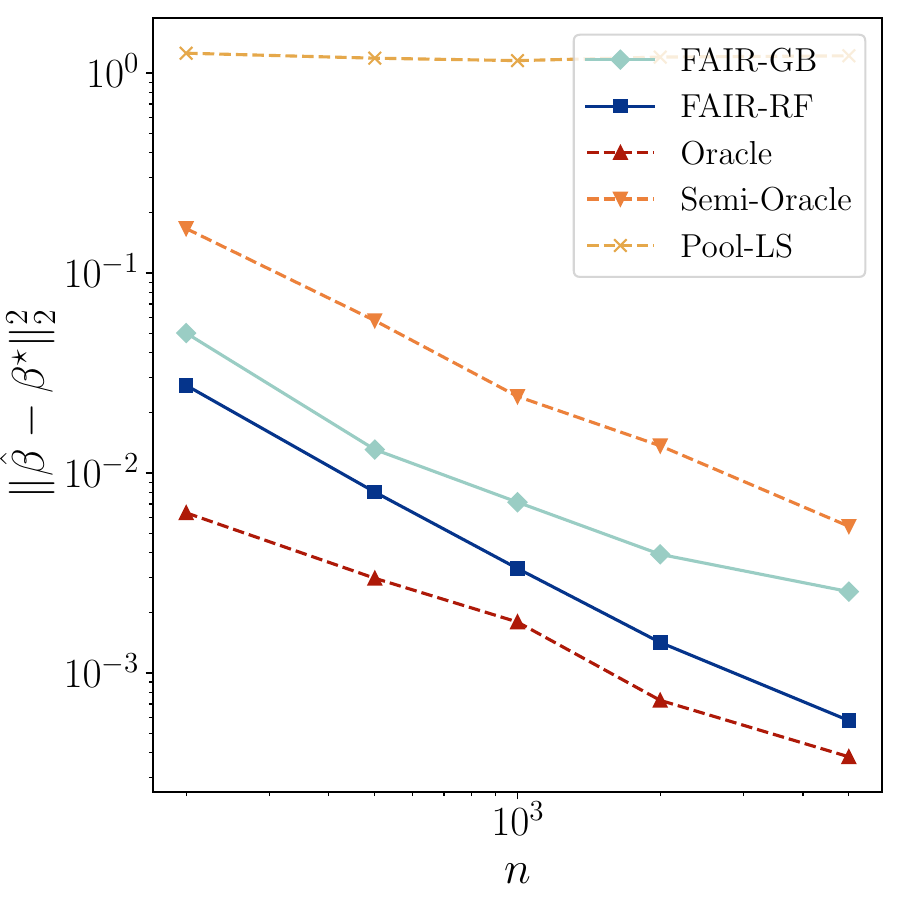}
}&
\subfigure[]{
\includegraphics[width=0.35\linewidth]{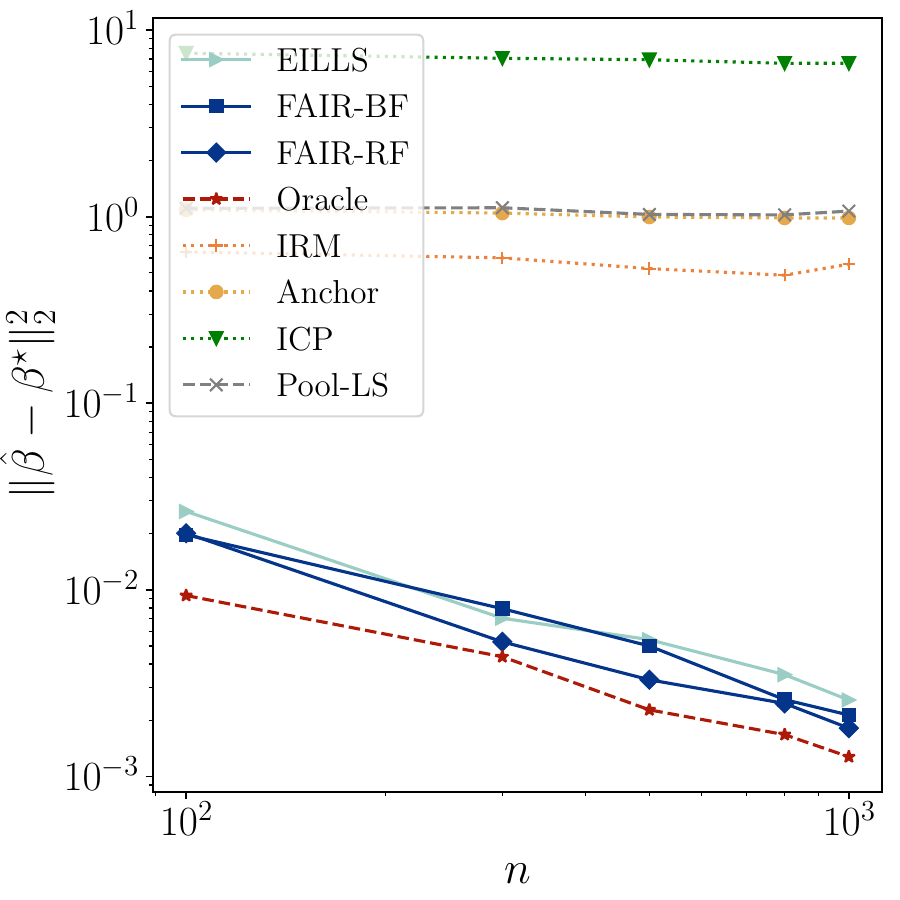}
}
\end{tabular}
\caption{The simulation results for linear models with (a) $d=70$ and (b) $d=15$. Both figures depict how the median estimation errors (based on $50$ replications, shown in log scale) for different estimators (marked with different shapes and colors) change when $n$ varies in (a) $\{200, 500, 1000, 2000, 5000\}$ and (b) $\{100, 200, 500, 800, 1000\}$, respectively. 
}
\label{fig:simulation-result}
\end{figure}

\definecolor{ao}{rgb}{0.0, 0.5, 0.0}

We also compare our FAIR-Linear estimator with the cousin estimator EILLS  (\mylightblue{$\blacktriangleright$}) in \cite{fan2024environment} and other invariance learning estimators (dotted lines), including invariant causal prediction \citep{peters2016causal} (ICP {\color{ao} $\blacktriangledown$}), invariant risk minimization \citep{arjovsky2019invariant} (IRM\myorange{$+$}), anchor regression \citep{rothenhausler2021anchor} (Anchor\myyellow{$\bullet$}) in a similar but smaller dimension setting with $d=15$, under which ICP and EILLS can be computed within affordable time. For the FAIR-Linear estimator, we report the performance of the FAIR-RF (\myblue{$\blacklozenge$}) and the one with brute force search (FAIR-BF\myblue{\textbf{$\blacksquare$}}). The results are shown in \cref{fig:simulation-result} (b): we can see that the FAIR family estimators (\mylightblue{$\blacktriangleright$}\myblue{\textbf{$\blacksquare$}}\myblue{$\blacklozenge$} with solid lines) are the only ones attaining consistent estimation among all the invariant learning methods; see a detailed discussion of the data generating process and results in \aosversion{Appendix C.4.1}{\cref{sec:linear-simu-append}}. 

\subsection{Application I: Discovery in Real Physical Systems}

We apply our method to perform causal discovery in the light tunnel datasets from \cite{gamella2024causal}. The data are collected from a real physical device under different manipulation settings. The tunnel device contains a controllable light source at one end and two linear polarizers mounted on rotating frames. Several sensors are deployed in various positions to measure the light intensity. The causal relationships between the variables of interest are known such that we can get access to the ground-truth cause-effect relationship; see Fig. 2(d) and Fig. 3(a) therein for the device diagram and the cause-effect graphs, respectively. It is worth noticing that the data are collected from a real-world device where the associations between the measurements follow from real-world physical laws. This realistic nature and the knowledge of ground-truth cause-effect knowledge make it an excellent testbed for causal discovery algorithms. 

Following the notations, we use the variables $(R, G, B, \theta_1, \theta_2, \tilde{V}_3, \tilde{V}_2, \tilde{V}_1, \tilde{I}_3, \tilde{I}_2, \tilde{I}_1, \tilde{C})$. Here $(R, G, B)$ is the intensity of the light source at three different wavelengths, $\tilde{C}$ is the drawn electric current, $(\theta_1, \theta_2)$ represent the angles of the polarizer frame, and $(\tilde{V}_3, \tilde{V}_2, \tilde{V}_1, \tilde{I}_3, \tilde{I}_2, \tilde{I}_1)$ are the measurement of light-intensity sensors in various positions. 

We plan to learn algorithmically the direct cause for $Y=\tilde{I}_3$, the infrared measurement of the light-intensity sensor after the polarizers, among a subset of manipulable variables and measurement variables $(X_1,\ldots, X_{11}) = (R, G, B, \theta_1, \theta_2, \tilde{V}_3, \tilde{V}_2, \tilde{V}_1, \tilde{I}_2, \tilde{I}_1, \tilde{C})$ under the following two-environment experimental setting: $e=0$ is the observational environment, $e=1$ is the interventional environment where the variables $\{\tilde{V}_j\}_{j=1}^3$ and $\{\tilde{I}_j\}_{j=1}^2$ are \emph{weakly} intervened on. This leads to the following ``equivalent'' ground-truth cause-effect relationship among those variables and the effect of ``environment intervention'' in \cref{fig:app1} (a). In this case, the variables $(R,G,B,\theta_1,\theta_2)$ are the direct causes, i.e., $S^\star = \{1, 2, 3, 4, 5\}$, $\tilde{V}_3$ are the spurious variables that will lead to biased estimation. The remaining variables are exogenous but have marginal predictive power, i.e., $\mathrm{Var}[Y|X_j] > 0$ for $j \ge 7$.

We will use the following dataset in the experiment: the environment dataset $\mathcal{D}_0$ with size $|\mathcal{D}_0|=10^4$, the weakly interventional environment dataset $\mathcal{D}_1$ with $|\mathcal{D}_1|=3000$, and five strongly interventional environment dataset $\mathcal{D}_{2,Z}$ with $Z\in \{\tilde{V}_1,\tilde{V}_1,\tilde{V}_3,\tilde{I}_1,\tilde{I}_2\}$ and $|\mathcal{D}_{2,Z}|=1000$. In each trial, different methods use the same random subsample $\breve{\mathcal{D}} = \{\breve{\mathcal{D}}_{0}, \breve{\mathcal{D}}_{1}\}$ with $\breve{\mathcal{D}}_{k} \subseteq \mathcal{D}_{k}$ and $|\breve{\mathcal{D}}_{k}|=n=1000$ to fit the model. How the fitted model $\hat{f}$ quantitatively depends on exogenously/endogenously spurious variable $Z$ is evaluated using the OOS $R^2$ in corresponding $\mathcal{D}_{2,Z}$ defined as
\begin{align*}
R^2_{\mathrm{OOS},Z} := \frac{\sum_{(X,Y)\in \mathcal{D}_{2,Z}} \{\hat{f}(X)-Y\}^2}{\sum_{(X,Y)\in \mathcal{D}_{2,Z}} \{Y - \bar{Y}\}^2} \qquad \text{with} \qquad \bar{Y} = \frac{\sum_{(X,Y)\in \breve{\mathcal{D}}_{0}\cup \breve{\mathcal{D}}_{1}} Y}{2n}.
\end{align*} See the detailed data collection and experimental configuration in \aosversion{Appendix C.5}{\cref{appendix:app1}}.

\begin{figure}[!t]
\centering
\begin{tabular}{ccc}
\subfigure[]{
\begin{tikzpicture}[state/.style={circle, draw, minimum size=0.8cm, scale=0.8}]
    \node at (0,0) (x1) {$(R, G, B)$};
    \node at (1.5, -0.5) (x2) {$\theta_1$};
    \node at (1.5, 0.5) (x3) {$\theta_2$};
    \node at (-1.5, -0.5) (x4) {$\tilde{C}$};
    \node at (-1.5, 1) (x5) {$\tilde{I}_2$};
    \node at (0, 1) (x6) {$\tilde{I}_1$};
    \node at (-0.7, 1.5) (x7) {$\tilde{V}_2$};
    \node at (0.8, 1.5) (x8) {$\tilde{V}_1$};
    \node at (-0.7, -1.5) (y) {$\tilde{I}_3$};
    \node at (0.7, -1.5) (x9) {$\tilde{V}_3$};

    \node at (-2, 2) (e) {\mypurple{$E$}};
    \draw[->] (x1) -- (x4);
    \draw[->] (x1) -- (x5);
    \draw[->] (x1) -- (x6);
    \draw[->] (x1) -- (x7);
    \draw[->] (x1) -- (x8);
    \draw[->] (x1) -- (y);
    \draw[->] (x1) -- (x9);
    \draw[->] (y) -- (x9);
    \draw[->] (x2) -- (y);
    \draw[->] (x3) -- (y);
    \draw[->] (x2) -- (x9);
    \draw[->] (x3) -- (x9);
    
    \draw[->, mypurple] (e) to[out=-10,in=150] (x7);
    \draw[->, mypurple] (e) to[out=0,in=160] (x8);
    \draw[->, mypurple] (e) to[out=-100,in=150] (x9);
    \draw[->, mypurple] (e) to[out=-50,in=95] (x5);
    \draw[->, mypurple] (e) to[out=-5,in=120] (x6);
\end{tikzpicture}
} & \subfigure[]{
\includegraphics[width=0.3\linewidth]{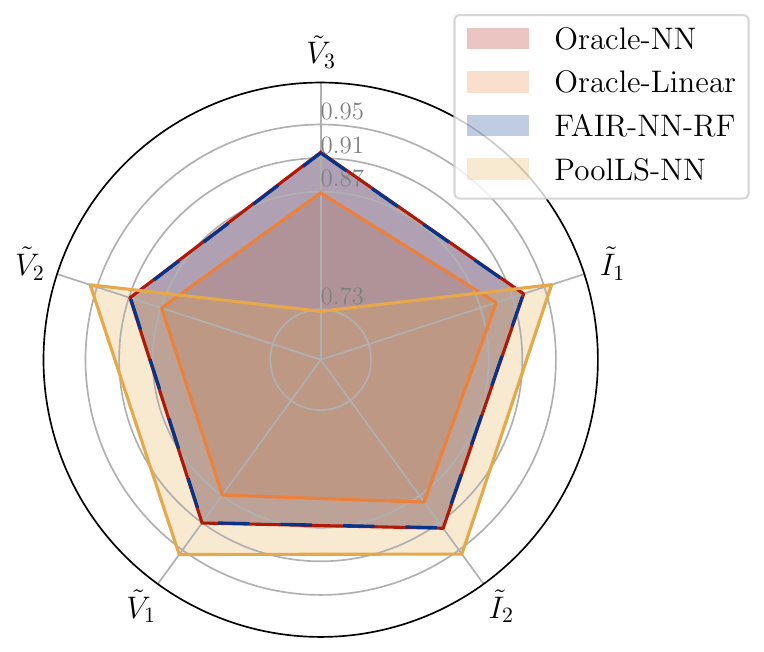}
} & \subfigure[]{
\includegraphics[width=0.28\linewidth]{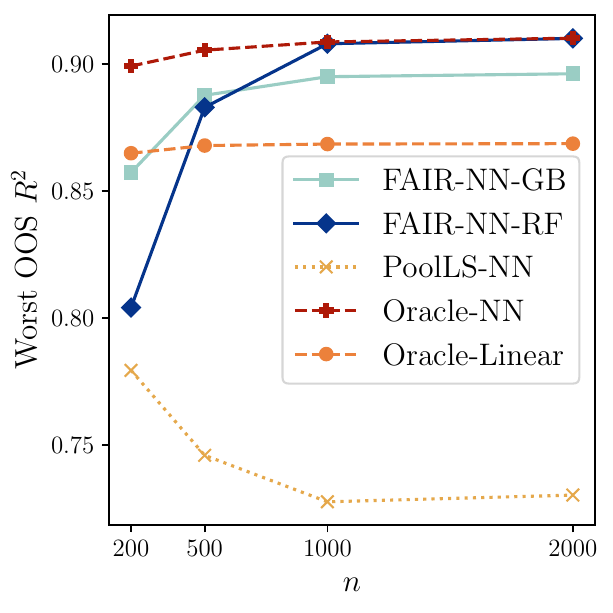}
}\\
\multicolumn{2}{c}{
\subfigure[]{
\includegraphics[width=0.65\linewidth]{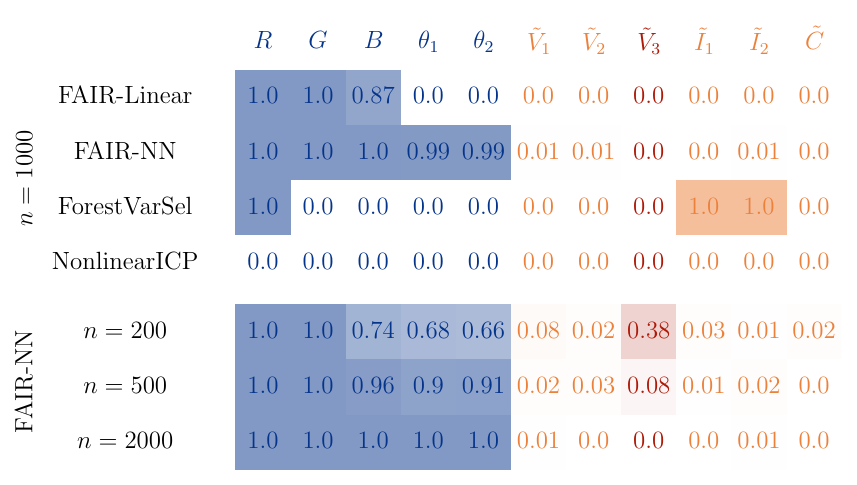}}
} & 
\subfigure[]{
\includegraphics[width=0.3\linewidth]{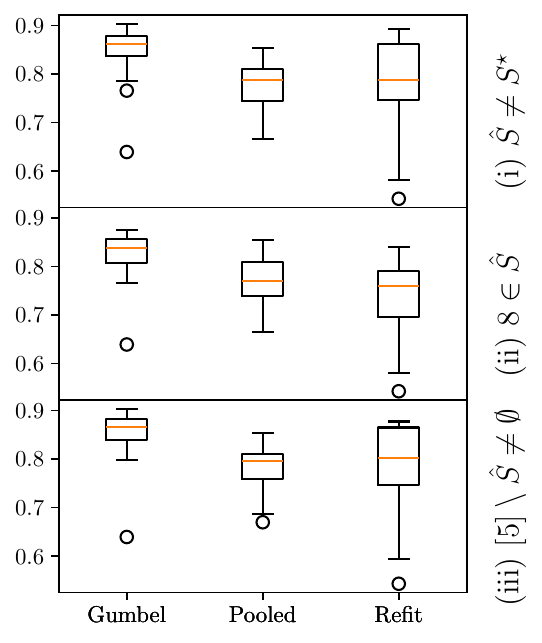}
}
\end{tabular}
\caption{Discovery in Real Physical Systems: (a) the unified cause-effect relationship and interventions similar to \cref{fig:scm-ident} (b). (b) the average out-of-sample $R^2$ for different estimators using the spider chart: the axis annotated by placeholder variable $Z$ corresponds to the test environment where $Z$ is strongly intervened on. We can see the performance of \myred{Oracle-NN} and \myblue{FAIR-NN-RF} is almost identical. 
(c) the average (based on 100 replications) of the worst-case (across 5 environments) of OOS $R^2$ for different methods as a function of $n$. (d) the variable selection rate over $100$ trials for different methods (top panel) and the variable selection rate for FAIR-NN for various $n$ (bottom panel). We use different colors to represent different relationships with $Y$: \myblue{blue}=parent, \myred{red}=child, \myorange{orange}=neither ancestor nor descendant. (e) the distribution of worst-case OOS $R^2$ (y-axis) for Gumbel-trick optimized FAIR-NN (Gumbel), the follow-up refitted estimator (Refit), and Pooled LS (Pooled) when FAIR-NN selects the wrong variables: the subplots from top to bottom consider the cases of (i) failure in selection consistency (ii) false positive that it falsely selects the child $X_8=\tilde{V}_3$ 
(iii) false negative that it does not select the entire ground-truth $(X_1,\ldots, X_5) = (R, G, B, \theta_1, \theta_2)$. 
}
\label{fig:app1}
\end{figure}
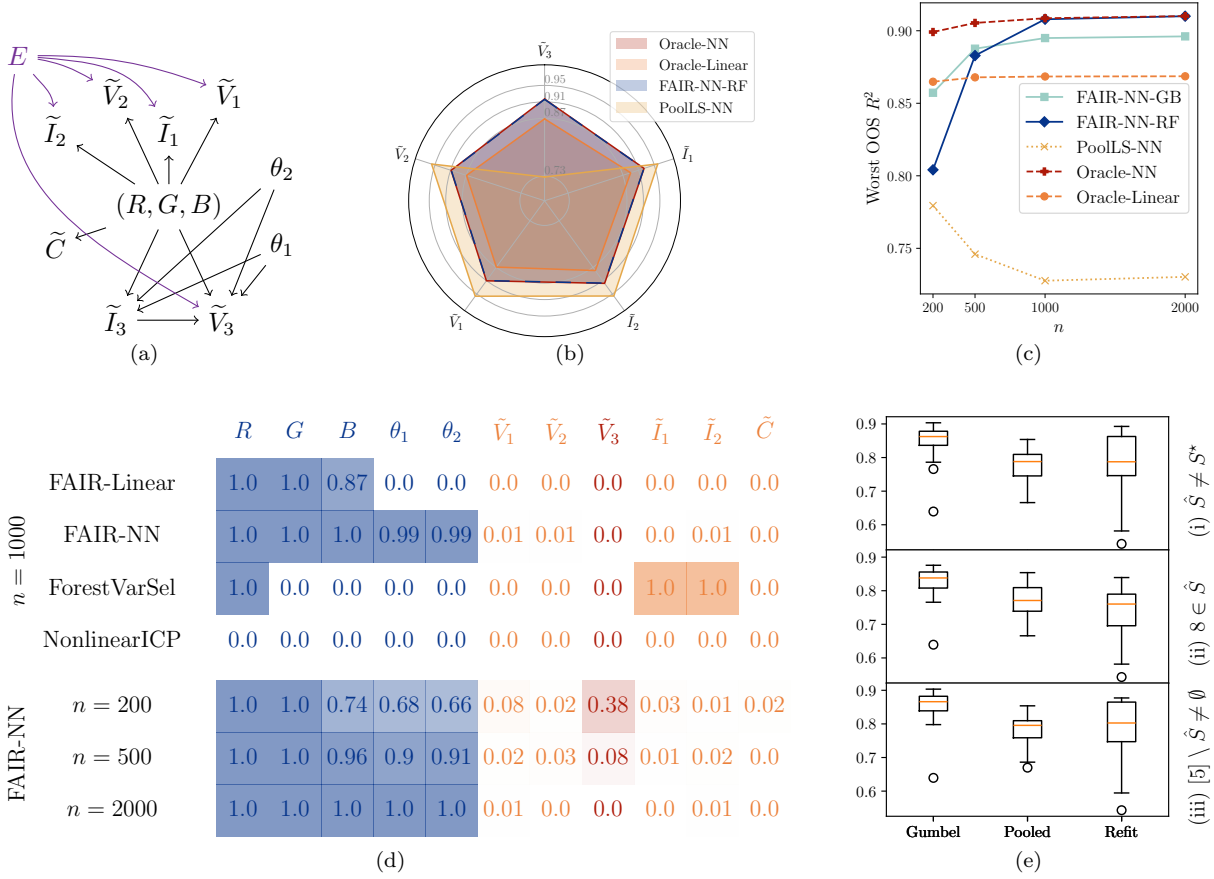

The first four rows in \cref{fig:app1} (d) report the variable selection result for several methods over $100$ trials. The nonlinear ICP \citep{heinze2018invariant} method does not select any variables because of its conservative nature and stronger heterogeneity condition to recover the direct cause. We can see that FAIR-NN can successfully recover the direct cause $(R, G, B, \theta_1, \theta_2)$ in this case. It exploits neural networks' capability in efficiently detecting the nonlinear associations (the Malus's law, $\tilde{I}_3 \propto \cos^2(\theta_1 - \theta_2)$ for fixed $(R, G, B)$), while the linear counterpart FAIR-Linear fails to select the variables $(\theta_1, \theta_2)$. It is worth pointing out that such a causality recovery cannot be attained by the traditional predictive power and simplicity tradeoff: the variable selection method based on random forest variable importance measures (ForestVarSel) in \cite{heinze2018invariant} cannot detect $(G, B, \theta_1, \theta_2)$ and falsely select $(\tilde{I}_1, \tilde{I}_2)$. The last three rows in \cref{fig:app1} (d) illustrate how the variable selection rate for the FAIR-NN estimator changes when $n$ grows.

\cref{fig:app1} (b) offers a quantitative illustration by showing the out-of-sample (OOS) $R^2$ of different estimators under environments with \emph{strong} interventions on $(\tilde{I}_1, \tilde{I}_2, \tilde{V}_1, \tilde{V}_2, \tilde{V}_3)$, respectively. The estimator denoted as Oracle-$M$ with $M \in \{\mathrm{Linear}, \mathrm{NN}\}$ referred to the method that runs regress $Y$ on $X_{S^\star}$ using model $M$. In the spider chart, the red shade represents the out-of-sample $R^2$ under different interventions for the {Oracle-NN} estimator that regresses $Y$ on its direct causes. We can see that its performances behave uniformly under various interventions: all the OOS $R^2$ are approximately equal to $0.91$. This is slightly better than that for the linear model ({Oracle-Linear}) by 0.04. This illustrates the capability of neural networks introduced to detect \emph{weak}, \emph{nonlinear} causal signals from heterogeneous environments. The {PoolLS-NN} estimator regressing $Y$ on $X$ using neural network and all the data fully exploits the strong spurious association between $\tilde{V}_3$ and $Y=\tilde{I}_3$, its heavy reliance on $\tilde{V}_3$ let it predict better (than the causal model {Oracle-NN}) when $\tilde{V}_3$ is not intervened. However, its OOS $R^2$ significantly decreases by $0.2$ when $\tilde{V}_3$ is strongly intervened hence the spurious association changes. On the contrary, the OOS $R^2$ for FAIR-NN after refitting ({FAIR-NN-RF}) behaves almost identically to that for {Oracle-NN}. This quantitative result illustrates its capability to correct non-trivial and strong biases without supervision and its efficiency in detecting nonlinear and weak signals. 

\cref{fig:app1} (c) shows how the worst-case OOS $R^2$ among the five, strong intervention environments changes for different estimators when $n$ grows. The performance of the Gumbel-trick optimized FAIR-NN estimator without refitting ({FAIR-NN-GB}) lies between {Oracle-NN} and {Oracle-Linear} and significantly outperforms that of the {PoolLS-NN} estimator. This suggests that the gradient descent optimized algorithm has already found predictions nearly independent of the spurious variable, and the success of variable selection in \cref{fig:app1} (d) is not because of truncating weak but non-negligible spurious signals. Moreover, as shown in \cref{fig:app1} (e), its performance significantly outperforms the least squares estimator using either the full covariate or the selected covariates when $n=200$ and it selects the wrong variables. This further supports the theoretical claims and the advantages of adopting penalized least squares.

%% file: text_supp.tex

The supplemental materials are organized as follows:

\begin{itemize}[itemsep=0pt, leftmargin=20mm]
    \item[\cref{appendix:discussion}] contains the omitted discussions in the main text, including the detailed related works, applicable scenarios for the nonparametric invariance pursuit, some discussions and extensions on the method, and some discussions on the conditions in \cref{sec:fairnn} and \cref{sec:ident-fairnn}.  
    \item[\cref{sec:theory-appendix}] contains the complete result that is sketched in \cref{sec:sketch}.
    \item[\cref{sec:exper}] contains omitted discussions and results in experiments section.
    \item[\cref{sec:proof-main-result}] contains the proofs for our main abstract results in \cref{sec:main-result} including \cref{prop:bias} and \cref{thm:oracle}, together with the proof for the result with general risk loss \cref{thm:oracle-2}.
    \item[\cref{sec:proof-application}] contains the proofs that applying \cref{thm:oracle} to different $(\mathcal{G}, \mathcal{F})$ specifications.
    \item[\cref{sec:proof-population}] contains the proofs for population-level results, including \cref{prop:ident-transfer-learning}, \cref{prop:ident-causal-discovery}, \cref{prop:rtl} in the main text, and \cref{prop:scm-invariant}, \cref{prop:transfer-learning}, \cref{thm:causal-identification-hc} in \cref{appendix:discussion}.
\end{itemize}

\section{Further Discussions}
\label{appendix:discussion}

\subsection{Applicable Scenarios for Nonparametric Invariance Pursuit}
\label{sec:scenario}
This section is devoted to providing a self-contained introduction to the motivation behind the nonparametric invariance pursuit using statements akin to previous literature \citep{peters2016causal, rojas2018invariant, fan2024environment}. 

\medskip
\noindent \textbf{Causal Discovery.} If we can expect $\mathcal{E}$ to be heterogeneous enough, recovering  $S^\star$ in nonparametric invariance pursuit coincides with discovering the direct cause of $Y$ when the multi-environment data come from SCM with intervention on $X$ setting. 

\begin{proposition}
\label{prop:scm-invariant}
    Under the model \eqref{eq:scm-model}, if we further assume that $\mathbb{E}[|Y^{(e)}|^2] < \infty$ for any $e\in \mathcal{E}$, then \eqref{eq:intro-model} holds with $S^\star = \mathtt{pa}(d+1)$.
\end{proposition}

The SCM \eqref{eq:scm-model} and \cref{prop:scm-invariant} extend the framework described in \cite{peters2016causal} (specifically Section 4.1 and Proposition 1). This model accommodates nonlinear structural assignments. Critically, the residuals $\varepsilon^{(e)} = Y^{(e)} - \mathbb{E}[Y^{(e)}|X_{S^\star}^{(e)}]$, do not need to be independent of $X_{S^\star}^{(e)}$ or remain invariant across various environments as represented by $\varepsilon^{(e)} \sim \mu_\varepsilon$. Such flexibility broadens the scope for various applications, including binary classification. According to \cref{prop:scm-invariant}, when restricted to model \eqref{eq:scm-model}, a specific instantiation of our generic statistical model \eqref{eq:intro-model}, identifying the true important variable set $S^\star$ is tantamount to pinpointing the direct cause of the target variable $Y$. Concurrently, unveiling the invariant association $m^\star$ aligns with uncovering the causal mechanism between $Y$ and its direct causes.

\medskip

\noindent \textbf{Transfer Learning.} Consider we collect data $\{(X_i^{(e)}, Y_i^{(e)})\}_{e\in \mathcal{E}, i\in [n]}$ from $|\mathcal{E}|$ distinct sources and aim to develop a model that produces decent predictions on the data $\{X^{(t)}_{i}\}_{i\in [n_t]}$ in an unseen environment $t$. A significant portion of transfer learning algorithms fundamentally relies on the covariate shift assumption, represented as
\begin{align*}
\mathbb{E}[Y^{(t)}|X^{(t)}] \equiv \mathbb{E}[Y^{(e)}|X^{(e)}] \qquad \forall e\in \mathcal{E}.
\end{align*} However, as illustrated in \cite{fan2024environment, rojas2018invariant}, it is hard for this to be true given collecting so many variables. Therefore, a more realistic assumption is that information from true important variables is transferable, articulated as $\mathbb{E}[Y^{(t)}|X^{(t)}_{S^\star}] = \mathbb{E}[Y^{(e)}|X^{(e)}_{S^\star}]$. The subsequent proposition suggests that though $m^\star$ might not be the optimal predictor in the unseen environment $t$, it does minimize the worst-case $L_2$ risk, and the associated excess risk can be decomposed as follows.

We suppose both the distribution $\mu^{(e)}$ we observed in $\mathcal{E}$ and the future distributions $\nu$ come from the following distribution family. 
\begin{align*}
    \mathcal{U}_{S^\star, m^\star, \sigma^2} = \Big\{\mu: \mathbb{E}_{\mu}[Y^2]<\infty, &\mathbb{E}_{\mu}[Y|X_{S^\star}] = m^\star(X_{S^\star}), \\
    &~~~~~~\mathbb{E}_{\mu}[\mathrm{Var}_{\mu}(Y|X_{S^\star})] \lor \max_{1\le j\le d}\mathbb{E}_{\mu}[X_j^2] \le \sigma^2\Big\},
\end{align*}

\begin{proposition}
\label{prop:transfer-learning}
Let $\nu\in \mathcal{U}_{S^\star, m^\star, \sigma^2}$ be arbitrary. Define 
\begin{align*}
\mathsf{R}_{\mathtt{oos}}(m; \nu_x) = \sup_{\mu\in \mathcal{U}_{S^\star, m^\star, \sigma^2}, \mu_x \sim \nu_x} \mathbb{E}_{(X,Y)\sim \mu} [|Y - m(X)|^2]
\end{align*}
and $\Theta^{(t)} = L_2(\nu_x)$. We have
\begin{align*}
    \forall m\in \Theta^{(t)} \qquad \mathsf{R}_{\mathtt{oos}}(m; \nu_x) - \mathsf{R}_{\mathtt{oos}}(m^\star; \nu_x) 
    = \|m - m^\star\|_{L_2(\nu_x)}^2 + 2\sigma \|m - \tilde{m}\|_{L_2(\nu_x)},
\end{align*} where $\tilde{m}(x) = \mathbb{E}_{X\sim \nu_x}[m(X)|X_{S^\star}=x_{S^\star}]$. The term $2\sigma \|m - \tilde{m}\|_{L_2(\nu_x)}$ is zero when $m\in \Theta_{S^\star}^{(t)}$.
\end{proposition}
Given the framework described above, our proposed method solving problem in \cref{sec:intro-problem} can be integrated with the re-weighting technique \citep{gretton2009covariate}, a strategy addressing discrepancies within the marginal distribution of $X$, to yield reliable predictions in the previously unobserved environment $t$.

\subsection{Discussion on the Methods}
\label{appendix:dis2}

We provide a discussion in a question-and-response manner. 

\bigskip

\noindent \emph{{[Q] You are doing ``focused regularizer'' that are of combinatorial nature in computation, can it be removed?}}

{\noindent \emph{Answer:} The short answer is No. The regularizer will be the same as running least squares if we do not enforce the discriminator using the same variables that the predictor uses. This is also the main computational difficulty in our framework and why we use randomness relaxation and Gumbel approximation in implementation. Indeed, even for linear invariance pursuit, there are certain fundamental computational limits in this such that no polynomial-time algorithm can attain consistent estimation in pursuing invariance without relying on additional structures other than invariance.}

\bigskip

\noindent \emph{{[Q] The method has a similar form to IRM, what's the major difference? }}

{\noindent \emph{Answer:} The main difference is we should at least let $\Theta_f \supseteq \Theta_g$, such a constraint leverage the idea of over-identification and make identification possible even when $|\mathcal{E}|=2$ provided enough heterogeneity. Suppose our regularizer, which can be seen as a ``correct'' method to pursue condition expectation invariance, is to make $u^{(1)} = u^{(2)}$ for two $s$-dimensional parameter vectors $u^{(1)}, u^{(2)} \in \mathbb{R}^s$, what IRM does is to let $\sum_{i=1}^s u^{(1)}_i = \sum_{i=1}^s u^{(2)}_i$. It is hard to say the latter constraint will make sense and can obtain a similar effect as the former. }

\bigskip

\noindent \emph{{[Q] Could your proposed framework be extended to the representation-level invariance like IRM?}}

{\noindent \emph{Answer:} The short answer is Yes given its algorithmic nature. But identification with two or constant-level environments is impossible now: a linear-in-dimension number of environments is required even for linear representation learning. For example, one can find some linear representation $\Phi: \mathbb{R}^d\to \mathbb{R}^r$ such that
\begin{align*}
    \mathbb{E}[Y^{(e)}|\Phi X^{(e)}] \equiv m^\star(\Phi X^{(e)})
\end{align*}
However, $|\mathcal{E}| \ge r$ is the necessary condition for identification even when the heterogeneity is enough and $r$ is pre-known to us. We conjecture that any finite number of environments $|\mathcal{E}|<\infty$ may be impossible for identification if $\Phi$ lies in some nonparametric function class. }

\subsection{Extensions to General Environment Variable and Loss Function}
\label{sec:extension}

In the main text, we propose an estimation framework leveraging conditional expectation invariance with respect to discrete environment variables. It is worth noticing that our adversarial estimation framework is indeed more versatile than this: one can easily extend it to other conditional point prediction invariance with respect to more general environment covariates. We briefly discuss the direct extension here and leave a rigorous treatment as future work. In the following discussions, suppose we observe data $\{(X_i, Y_i, E_i)\}_{i=1}^n$ drawn i.i.d. from some distribution $\mu_0$, where $X \in \mathbb{R}^d$ is the covariate we used for prediction, $Y \in \mathbb{R}$ is the target response, $E \in \mathbb{R}^{q}$ is the environment covariate we wish our prediction should be invariant with respect to. 

Let $\ell(u, y): \mathbb{R} \times \mathbb{R} \to \mathbb{R}$ be the user-defined risk whose population-level minimizer may not necessarily be the conditional expectation but satisfying certain regularity conditions. Let $\ell_u(u, y) = \partial \ell(u, y) / \partial u$ be the partial sub-gradient with respect to the prediction. Suppose the following general invariance structure with respect to $\ell$ and environment covariate holds, that there exists $S^\star \subseteq [d]$ and a function $g^\star$ that only depends $x_{S^\star}$ such that
\begin{align}
\label{eq:invariance-general}
    \mathbb{E}\left[\ell_u(g^\star(X_{S^\star}), Y)|X_{S^\star}, E\right] \equiv 0.
\end{align} 
It coincides with the main problem of study when $E$ is discrete and $\ell$ satisfies \eqref{eq:loss-constrain}, but also allows for other loss and continuous environment label. Other losses include, but are not limited to Huber loss for robust regression, or $L_1$ loss for median regression. 

We consider the following optimization minimax objective containing a min-max game between a predictor $g: \mathbb{R}^d \to \mathbb{R}$ and a discriminator $f: \mathbb{R}^d \times \mathbb{R}^q \to \mathbb{R}$:
\begin{align}
    \min_{g\in \mathcal{G}} \max_{f\in \mathcal{F}_{S_g}} \underbrace{\frac{1}{n} \sum_{i=1}^n \ell(g(X_i), Y_i)}_{\hat{\mathsf{R}}(g)} + \gamma \underbrace{\frac{1}{n} \sum_{i=1}^n \left[ \ell_u(g(X_i), Y_i) f(X_i, E_i) - 0.5 \{f(X_i, E_i)\}^2\right]}_{\hat{\mathsf{J}}(g, f)},
\end{align} where $\gamma$ is the hyper-parameter to be determined, and $\mathcal{F}_{S_g} = \{f(x, e) \in \mathcal{F}, f(x, e)=w(x_{S_g}, e)\text{ for some }w\}$.
Similar to the calculation in \cref{sec1.2}, one can expect that minimizing the population counterpart of the focused adversarial invariance regularizer $\max_{f\in \mathcal{F}_{S_g}} \hat{\mathsf{J}}(g, f)$ shares a similar nature of imposing \eqref{eq:invariance-general}. One can derive non-asymptotic identification and estimation error results akin to \cref{thm:oracle} and \cref{thm:oracle-2} provided strong convexity and certain Lipschitz property of the loss $\ell(u, y)$. We leave this for future studies.

\subsection{Discussion on the Nondegenerate Intervention Condition}
\label{sec:discussion:faithfulness}

The conditions (a) and (b) in \cref{cond:expected-faithfulness} are imposed to eliminate some degenerate cases. To illustrate the intuitions why such two conditions are needed, and how such a condition will hold in general. We consider the following two examples.

\paragraph{Introduction of condition (a)} From a high-level viewpoint, the introduction of condition (a) is to eliminate the cases where though there are shifts in condition distributions among different environments, it happens that there are no shifts in conditional expectations. This can be illustrated in the following example.

\begin{example}
\label{ex1}
Consider the following canonical model also presented in Example 4.1 in \cite{fan2024environment}.
\begin{align*}
    X_1^{(e)} &\gets \sqrt{0.5} U_1 \\
    Y^{(e)} &\gets X_1^{(e)} + \sqrt{0.5} U_3 \\
    X_2^{(e)} &\gets s^{(e)} Y^{(e)} + U_2 
\end{align*} where $U_1, U_2, U_3$ are independent standard normal variables, and $\mathcal{E}=\{1, 2\}$. We let $e=1$ be the observational environment and $e=2$ be the interventional environment where the linear effect of $Y$ on $X_2$ are intervened ($s^{(1)} \neq s^{(2)}$). We also focus on the regime where $s^{(1)} + s^{(2)} \neq 0$ such that running least squares will lead to a biased solution.
\end{example}

In the above model, we can see that
\begin{align*}
    Y^{(e)}|X_2^{(e)} \sim \mathcal{N}\left(\frac{s^{(e)}}{(s^{(e)})^2+1} X_2^{(e)}, \frac{1}{(s^{(e)})^2+1}\right)
\end{align*}
It is easy to check under the case of nondegenerate child ($s^{(1)} + s^{(2)} \neq 0$) and faithfulness on $\tilde{M}$ ($s^{(1)} \neq s^{(2)}$). We have
\begin{align*}
    Y^{(1)}|X_1^{(1)} \overset{d}{\neq} Y^{(2)}|X_2^{(2)},
\end{align*} or in other words, $Y \indep E | X_2$. However, when $s^{(1)} = 1/s^{(2)} = s$, the following holds
\begin{align*}
    \mathbb{E}[Y^{(1)}|X_2^{(1)}=x] = \frac{s^{(1)}}{(s^{(1)})^2+1} x = \frac{s}{s^2+1} x = \frac{s^{(2)}}{(s^{(2)})^2+1} x = \mathbb{E}[Y^{(2)}|X_2^{(2)}=x] 
\end{align*} The introduction of \cref{cond:expected-faithfulness} (a) is to rule out the cases where $s^{(1)} = 1/s^{(2)} = s$. And it is easy to see when $s^{(1)}$ and $s^{(2)}$ are independently generated from some prior distribution that is absolute continuous with respect to Lebesgue measure on $\mathbb{R}$, i.e., $S^{(1)}, S^{(2)} \sim p_s$, then 
\begin{align*}
    \mathbb{P}\left[S^{(1)} S^{(2)} = 1\right] = 0.
\end{align*}

\paragraph{Introduction of condition (b).} The condition (b), that the faithfulness condition on $\tilde{M}$, is to eliminate the cases where though the interventions are applied, it happens that such interventions do not make an impact on the variables intervened. The following example presents such an example.
\begin{example}
Consider the case where $\mathcal{E}=\{1,2\}$, and the data generating process is as follows
\begin{align*}
    Y^{(e)} &\gets U_3 \\
    X^{(e)}_1 &\gets Y^{(e)} + e + U_1 \\
    X_2^{(e)} &\gets 0.5 Y^{(e)} - s X^{(e)}_1 + e + U_2.
\end{align*} where $U_1, U_2, U_3$ are independent standard normal variables, $s \neq 0.5$ is a fixed parameter. We let $e=1$ be the observational environment and $e=2$ be the interventional environment where shifts in mean are applied to the variables $X_1$ and $X_2$.
\end{example}

In the above case, we have $S^\star = \mathtt{pa}(3) = \emptyset$, and there exists a effective simultaneous intervention on $(X_1, X_2)$. However, such an intervention will not affect $X_2$ if and only if $s=1$ because its direct effect on $X_2$ and the indirect effect passing through $X_1$ get canceled provided $s=1$. To be specific, $X_2^{(e)}$ can be written as 
\begin{align*}
    X_2^{(e)} = 0.5 Y^{(e)} - s (Y^{(e)} + e + U_1) + e  + U_2 = (0.5-s) Y^{(e)} - sU_1 + U_2 + e(1-s).
\end{align*} This implies that
\begin{align*}
    Y \indep E |X_2
\end{align*} provided $s=1$, under which the faithfulness on $\tilde{M}$ fails to hold because we have $Y \notindep_{\tilde{G}} E | X_2$ since the path $Y \to X_2 \gets E$ is not blocked by $X_2$. However, if the parameter $s$ is also generated from some prior distribution that is absolute continuous with respect to Lebesgue measure on $\mathbb{R}$, i.e., $S \sim p_s$, then 
\begin{align*}
    \mathbb{P}\left[S = 1\right] = 0.
\end{align*}

\subsection{The Complete Statement of Proposition \ref{prop:rtl}}
\label{sec:rtl}

Specifically, we construct a unified SCM $(X,Y,E) \sim \bar{M}(\bar{\mathcal{S}}, \nu)$ based on $M^{(0)}$ and new environment $M^{(t)}$ as follows:
\begin{align*}
    E &\gets \text{Uniform}(\{0,t\}) \\
    X_j &\gets \begin{cases} \bar{f}_j(X_{\mathtt{pa}(j)}, U_{j}) := f^{(0)}_j(X_{\mathtt{pa}(j)}, U_{j}) &\qquad \forall j\in [d] \setminus I \\
    \bar{f}_j(X_{\mathtt{pa}(j)}, E, U_{j}) := f^{(t)}_j(X_{\mathtt{pa}(j)}, U_j) &\qquad \forall j \in I 
    \end{cases} \\
    Y &\gets \bar{f}_{d+1}(X_{\mathtt{pa}}(d+1), U_{d+1}) := f_{d+1}(X_{\mathtt{pa(d+1)}}, U_{d+1}).
\end{align*}

We suppose the following condition similar to \cref{cond:expected-faithfulness} holds in the constructed graph.

\begin{condition}
\label{cond:expected-faithfulness-2}
    The following holds for $\bar{M}$: (1) $\forall S\subseteq [d]$ containing $Y$'s descendants, i.e., $d+1\in \cup_{j\in S} \mathtt{at}(j)$, if $E \notindep_{\bar{M}} Y | X_S$, then $(\mu^{(0)} \land \mu^{(t)})(\{m^{(0,S)} \neq m^{(t,S)}\}) > 0$; (2) $\bar{M}$ is faithful, that is,
    \begin{align*}
        \forall ~\text{Disjoint}~ A, B, C\subseteq [d+2], \qquad Z_A \indep Z_B | Z_C ~~ \overset{(a)}{\Longrightarrow} ~~ Z_A \indep_{\bar{G}} Z_B | Z_C,
    \end{align*} where $Z_A \indep_{\tilde{G}} Z_B | Z_C$ means the node set $A$ and $B$ and d-separated conditioned on $C$ in the graph $\bar{G}=G(\bar{M})$.
\end{condition}

We are ready to give a complete statement of \cref{prop:rtl}.

\begin{proposition}[Formal Statement of \cref{prop:rtl}]
\label{prop:rtl-full}
 Under the setting of \cref{prop:ident-transfer-learning}, for a new environment $t$ with SCM $M^{(t)}=\{\mathcal{S}^{(t)}, \nu\}$ satisfying $f_j^{(t)} \equiv f_j^{(0)}$ for any $j\in [d+1]\setminus I$, i.e., only $X_I$ is intervened, we also have $\mathbb{E}[Y^{(t)}|X_{S_\star}^{(t)}] \equiv \mathbb{E}[Y^{(0)}|X_{S_\star}^{(0)}]$. Suppose further that \cref{cond:expected-faithfulness-2} holds for the constructed SCM $\bar{M}$. Then $S_\star$ is the unique largest set whose conditional expectation is transferable, i.e., for any $S\subseteq [d]$ such that $\mathbb{E}[Y^{(t)}|X_{S_\star \cup S}^{(t)}] \neq \mathbb{E}[Y^{(t)}|X_{S_\star}^{(t)}]$, one has $\mathbb{E}[Y^{(t)}|X_{S}^{(t)}] \neq \mathbb{E}[Y^{(0)}|X_{S}^{(0)}]$.
\end{proposition}

\subsection{Discussion on SCM with Hidden Confounders}
\label{appendix:confounder}

One can derive rigorous causal interpretation results akin to \cref{prop:ident-transfer-learning} under hidden confounders that do not affect other direct causes directly. Here for simplicity, we present the result in the presence of \emph{one} hidden confounder $H$, and the causal graph still maintains acyclic: $H$ is the direct cause of $Y$ but is unobserved. We leave rigorous statements for multiple unobserved confounders to future studies.

Similar to the setup of \cref{prop:ident-transfer-learning}, we consider the case where the data generating process in each environment is governed by SCM on $(Z^{(e)}_1,\ldots, Z^{(e)}_{d+2})=(X^{(e)}_1,\ldots, X^{(e)}_d, Y^{(e)}, H^{(e)})$ as follows:
\begin{align}
\label{eq:scm-model-hc} 
\begin{split}
    X_j^{(e)} &\leftarrow f_j^{(e)}(Z_{\mathtt{pa}(j)}^{(e)}, U_j), \qquad \qquad  j=1,\ldots, d \\
    H^{(e)} &\leftarrow f_{d+2}(Z_{\mathtt{pa}(d+2)}^{(e)}, U_{d+2}) \\
    Y^{(e)} &\leftarrow f_{d+1}(X_{\mathtt{pa}(d+1)}^{(e)}, H^{(e)}, U_{d+1}).
\end{split}
\end{align}
and we cannot observe $Z_{d+2}=H$.  
Similar to the setup for \cref{prop:ident-transfer-learning}, we assume $f_{d+1}$, distribution of noise $U_1,\ldots, U_{d+2}$, and cause-effect relationship $\mathtt{pa}: [d+2]\to 2^{[d+2]}$ are the same across different environments. Here we assume that the assignment for $H^{(e)}$ is also invariant, otherwise, the intervention on $H$ is equivalent to direct intervention on $Y$ when $H$ is unobserved. Moreover, we consider the case where $H$ is cannot directly affect other direct causes of $Y$, that is
\begin{align*}
    d+2 \notin \cup_{j\in \mathtt{pa}(d + 1) \setminus \{d+2\}} \mathtt{at}(j).
\end{align*}

Now we define the augmented SCM $\tilde{M}=(\tilde{\mathcal{S}},\tilde{\nu})$ on $d+3$ variables 
\begin{align*}
Z=(Z_1,\ldots, Z_d, Z_{d+1}, Z_{d+2}, Z_{d+3}) = (X_1,\ldots, X_d, Y, H, E)
\end{align*}
encoding all the information of $|\mathcal{E}|$ models $\{M^{(e)}(\mathcal{S}^{(e)}, \nu)\}_{e\in \mathcal{E}}$ in \eqref{eq:scm-model-hc}. Denote $\nu_b \sim \mathrm{Uniform}(\mathcal{E})$. Here $\tilde{\nu}(du_1,\ldots, du_{d+3}) = \nu(du_1,\ldots, du_{d+2}) \nu_b(du_{d+3})$, and the assignments $\tilde{\mathcal{S}}=\{\tilde{f}_1,\ldots, \tilde{f}_{d+3}\}$ are defined as
\begin{align}
\begin{split}
    E &\gets \tilde{f}_{d+3}(U_{d+3}) := U_{d+3} \\
    X_j &\gets \begin{cases} \tilde{f}_j(Z_{\mathtt{pa}(j)}, U_{j}) := f^{(0)}_j(Z_{\mathtt{pa}(j)}, U_{j}) &\qquad \forall j\in [d] \setminus I \\
    \tilde{f}_j(Z_{\mathtt{pa}(j)}, E, U_{j}) := f^{(E)}_j(Z_{\mathtt{pa}(j)}, U_j) &\qquad \forall j \in I 
    \end{cases} \\
    H &\gets \tilde{f}_{d+2}(X_{\mathtt{pa}}(d+2), U_{d+2}) := f_{d+2}(X_{\mathtt{pa(d+1)}}, U_{d+2}), \\
    Y &\gets \tilde{f}_{d+1}(X_{\mathtt{pa}}(d+1), H, U_{d+1}) := f_{d+1}(X_{\mathtt{pa(d+1)}}, H, U_{d+1}),
\end{split}
\label{eq:model-2scm-hc}
\end{align} 

We summarize the data-generating process as a condition.

\begin{condition}[SCM with One Hidden Confounder and Interventions on $X$]
\label{cond:scm-model-hc}
Suppose $M^{(0)},\ldots, M^{(|\mathcal{E}|-1)}$ are defined by \eqref{eq:scm-model-hc}, and $G$ is acyclic. Let $\tilde{M}$ be the model constructed as \eqref{eq:model-2scm-hc} by $\{M^{(e)}\}_{e\in \mathcal{E}}$ with $I$ be given set of variables intervened. We assume that $d+2 \notin \cup_{j\in \mathtt{pa}(d + 1) \setminus \{d+2\}} \mathtt{at}(j)$. 
\end{condition}

Recall that $\mathtt{pa}(\cdot)$ matches the relationship in the graph $G=G(M^{(0)})$ not $\tilde{G}=G(\tilde{M})$, and $I \subseteq [d]$. We impose a condition akin to \cref{cond:expected-faithfulness}.

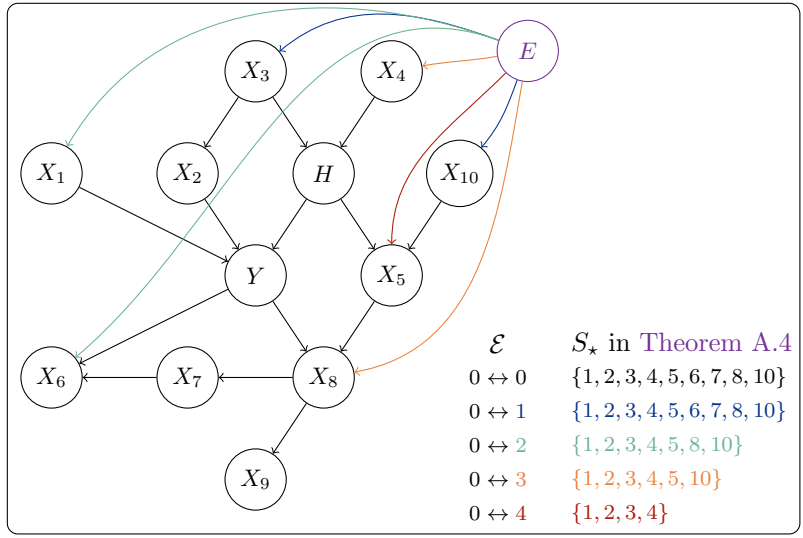
\begin{figure}
\centering
\begin{tikzpicture}[scale=0.9, state/.style={circle, draw, minimum size=0.9cm, scale=0.9}]
\draw[black, rounded corners] (-0.65, -5.3) rectangle (11, 2.5);

\node[state] at (0, 0) (x1) {$X_1$};
\node[state] at (2, 0) (x2) {$X_2$};
\node[state] at (3, 1.5) (x3) {$X_3$};
\node[state] at (5, 1.5) (x4) {$X_4$};
\node[state] at (4, 0) (h) {$H$};
\node[state] at (3, -1.5) (y) {$Y$};
\node[state] at (5, -1.5) (x5) {$X_5$};
\node[state] at (0, -3) (x6) {$X_6$};
\node[state] at (2, -3) (x7) {$X_7$};
\node[state] at (4, -3) (x8) {$X_8$};
\node[state] at (3, -4.5) (x9) {$X_9$};
\node[state] at (6, 0) (x10) {$X_{10}$};

\draw[->] (x1) -- (y);
\draw[->] (x2) -- (y);
\draw[->] (x3) -- (x2);
\draw[->] (x4) -- (h);
\draw[->] (h) -- (y);
\draw[->] (h) -- (x5);
\draw[->] (x5) -- (x8);
\draw[->] (x8) -- (x9);
\draw[->] (y) -- (x6);
\draw[->] (x7) -- (x6);
\draw[->] (x10) -- (x5);
\draw[->] (x3) -- (h);
\draw[->] (y) -- (x8);
\draw[->] (x8) -- (x7);

\node[state, color=mypurple] at (7, 1.8) (e) {$E$};

\draw[->, color=myblue] (e) to[out=160, in=40] (x3);
\draw[->, color=myblue] (e) to[out=-110, in=50] (x10);

\draw[->, color=mygreen] (e) to[out=160, in=60] (x1);
\draw[->, color=mygreen] (e) to[out=160, in=40] (x6);

\draw[->, color=myorange] (e) to[out=-170, in=10] (x4);
\draw[->, color=myorange] (e) to[out=-100, in=10] (x8);

\draw[->, color=myred] (e) to[out=-135, in=90] (x5);

\node[anchor=west] at (6.3, -2.5) {$\mathcal{E}$};
\node[anchor=west] at (7.5, -2.5) {$S_\star$ in \cref{thm:causal-identification-hc}};
\node[anchor=west] at (6, -3) {\footnotesize $0\leftrightarrow 0$};
\node[anchor=west] at (7.5, -3) {\footnotesize $\{1,2,3,4,5,6,7,8,10\}$};
\node[anchor=west] at (6, -3.5) {\footnotesize $0\leftrightarrow \myblue{1}$};
\node[anchor=west] at (7.5, -3.5) {\footnotesize $\myblue{\{1,2,3,4,5,6,7,8,10\}}$};
\node[anchor=west] at (6, -4) {\footnotesize $0\leftrightarrow \mygreen{2}$};
\node[anchor=west] at (7.5, -4) {\footnotesize $\mygreen{\{1,2,3,4,5,8,10\}}$};
\node[anchor=west] at (6, -4.5) {\footnotesize $0\leftrightarrow \myorange{3}$};
\node[anchor=west] at (7.5, -4.5) {\footnotesize $\myorange{\{1,2,3,4,5,10\}}$};
\node[anchor=west] at (6, -5) {\footnotesize $0\leftrightarrow \myred{4}$};
\node[anchor=west] at (7.5, -5) {\footnotesize $\myred{\{1,2,3,4\}}$};
\end{tikzpicture}
\caption{An illustrative example of \cref{thm:causal-identification-hc}. The arrow from $E$ to $X_j$ with color $e$ means $X_j$ is intervened in $e\in \{\myblue{1}, \mygreen{2}, \myorange{3}, \myred{4}\}$. For example, $0 \leftrightarrow \mygreen{3}$ means with interventions in environments \myblue{1}, \mygreen{2}, and \myorange{3}, the invariant variable set is $\myorange{\{1,2,3,4,5,10\}}$. Although $X_5$ is the effect of the hidden confounder $H$ and hence related to $Y$, we do not know this based only on the given environments. }
\label{fig:hc}
\end{figure}

\begin{condition}
\label{cond:expected-faithfulness-3}
    The following holds for $\tilde{M}$: (1) $\forall S\subseteq [d]$ containing $Y$ or $H$'s descendants, i.e., $d+1\in \cup_{j\in S} \mathtt{at}(j)$ or $d+2\in \cup_{j\in S} \mathtt{at}(j)$, if $E \notindep_{\tilde{M}} Y | X_S$, then $(\mu^{(e)} \land \mu^{(e')})(\{m^{(e,S)} \neq m^{(e',S)}\}) > 0$ for some $e,e'\in \mathcal{E}$; (2) $\tilde{M}$ is faithful, that is,
    \begin{align*}
        \forall ~\text{Disjoint}~ A, B, C\subseteq [d+3], \qquad Z_A \indep Z_B | Z_C ~~ \overset{(a)}{\Longrightarrow} ~~ Z_A \indep_{\tilde{G}} Z_B | Z_C,
    \end{align*} where $Z_A \indep_{\tilde{G}} Z_B | Z_C$ means the node set $A$ and $B$ and d-separated conditioned on $C$ in the graph $\tilde{G}=G(\tilde{M})$.
\end{condition}

Now we can state the main result.

\begin{theorem}[General Identification under SCM with One Hidden Confounder and Interventions on $X$] 
\label{thm:causal-identification-hc}
Under \cref{cond:scm-model-hc}, for
\begin{align}
\label{eq:invariant-blanket-hc}
    S_\star = (\mathtt{pa}(d + 1) \setminus \{d+2\}) \cup \mathtt{pa}(d+2) \cup A(I) \cup \bigcup_{j\in A(I)} \left(\mathtt{pa}(j) \setminus \{d+1, d+2\}\right)
\end{align} with $A(I) = \{j\in [d]: j\in \mathtt{ch}(d+1) \cup \mathtt{ch}(d+2), \mathtt{at}(j) \cap [\mathtt{ch}(d+1) \cup \mathtt{ch}(d+2)]\cap I = \emptyset \}$, we have the invariance $m^{(e,S_\star)} \equiv \bar{m}^{(S_\star)}:= m_\star$. Suppose further \cref{cond:expected-faithfulness-3} holds, then \cref{cond-fairnn-ident} holds with $S^\star = S_\star$ and $m^\star = m_\star$. 
\end{theorem}

We use the following example in \cref{fig:hc} to illustrate how $S^\star$ will vary when we observe more and more environments. Here are a few worth remarking on
\begin{itemize}
    \item We can see that in the presence of one hidden confounder, though identifying the direct causes of $Y$ is impossible (because $H$ is the direct cause and cannot be observed), identifying the direct causes plus a set of surrogate direct causes (the direct causes of the hidden confounder) is possible by our algorithm. To see this, as more and more interventions are applied such that $A(I)$ is $\emptyset$, the maximum invariant set $S^\star$ will collapse to
    \begin{align*}
        S^\star = (\mathtt{pa}(d + 1) \setminus \{d+2\}) \cup \mathtt{pa}(d+2),
    \end{align*} this generalizes \cref{prop:ident-causal-discovery}.
    \item Similar to \cref{prop:ident-transfer-learning}, such a set $S^\star$ will include some of $Y$'s child, or $H$'s child that survived from intervention $I$. In this case, the robust transfer learning property \cref{prop:rtl} still holds, we can claim that, for a new environment $t$, if the interventions are made within the set $I$, then $S^\star$ represents the most predictive associations which are transferable under the worst cases, namely
    \begin{align*}
        &m^{(t, S^\star)} = m^{(0, S^\star)} \qquad \text{and} \qquad \\
        &\qquad \forall S\in [d], m^{(0, S\cup S^\star)} \neq m^{(0, S^\star)} ~\Longrightarrow~ m^{(t, S)} \neq m^{(0, S)}.
    \end{align*}
\end{itemize}

\subsection{Attaining Variable Selection Consistency and Extension to High-dimension Regime}
\label{sec:varsel}

        It is worth noticing that with the help of the FAIR penalty, one can only guarantee the following screening \citep{fan2008sure} property rather than variable selection consistency. To be specific, in the setting of Section 2, the selected variable set $\hat{S}$ satisfies
		\begin{align}
			\label{eq:ss}
			\hat{S} \supseteq S^\star \qquad \text{and} \qquad \forall e\in \mathcal{E},~ \mathbb{E}[Y^{(e)}|X^{(e)}_{\hat{S}}] = m^\star(X_{S^\star}^{(e)}),
		\end{align} and it does not necessarily imply $\hat{S} = S^\star$. 
		
		There are two strategies to further attain variable selection consistency. In the low-dimensional regime, one can first run our FAIR estimator and then perform standard variable selection methods on the variables $\hat{S}$ that FAIR selects. This is because the screening property \eqref{eq:ss} guarantees the full covariate exogeneity on $X_{\hat{S}}$, that is,
		\begin{align*}
			\mathbb{E}[Y - m^\star(X_{S^\star}) | X_{\hat{S}}] = 0.
		\end{align*} 
        So it reduces to the standard setting of nonparametric variable selection. 
		
		In the high-dimensional regime, one can further add sparsity or variable selection penalty to attain variable selection consistency. Examples include $L_q$ with $q\in [0,1]$ \citep{tibshirani1997lasso, zhang2010analysis}, SCAD \citep{fan2001variable} for linear models, $(\mathcal{H},1)$-norm \citep{raskutti2012minimax} and group Lasso \citep{huang2010variable} for high-dimensional additive and structured nonparametric models, clipped-L1 weights \citep{fan2024factor} for structured neural networks. Now, the full objective function to be optimized admits the following form
		\begin{align*}
			\mathsf{R}^{\mathcal{E}}(g) + \sup_{\{f_e\}_{e\in \mathcal{E}}\in \mathcal{F}} \gamma \mathsf{J}(g, \{f_e\}_{e\in \mathcal{E}}) + \lambda \mathsf{P}(g)
		\end{align*} where $\mathsf{P}(g)$ is the standard sparsity or variable selection penalty. One can derive corresponding variable selection consistency results similar to Theorem 4.5 in \cite{fan2024environment} that augments least squares loss with an invariance regularizer similar to the FAIR penalty and $L_0$ penalty. We leave it for future studies.

\section{Generic Results and Its Applications}
\label{sec:theory-appendix}

\subsection{Main Result for the General FAIR Least Squares Estimator}
\label{sec:main-result}

This section is designed to offer a unified main result characterizing when the FAIR least squares estimator can identify the target regression function together with a non-asymptotic $L_2$ error bound for general $(\mathcal{G}, \mathcal{F})$. We first introduce some standard regularity conditions.

\begin{condition}[Data Generating Process]
\label{cond:general-dgp}
We collect data from $|\mathcal{E}| \in \mathbb{N}^+$ environments. For each environment $e\in \mathcal{E}$, we observe $(X_1^{(e)}, Y_1^{(e)}), \ldots, (X_n^{(e)}, Y_n^{(e)}) \overset{i.i.d.}{\sim} \mu^{(e)}$.
\end{condition}

\begin{condition}[Sub-Gaussian Response]
\label{cond:general-response}
    For any $e\in \mathcal{E}$ and $t\ge 0$, $\mathbb{P}\left[|Y^{(e)}| \ge t\right]\le C_y e^{-t^2/(2\sigma_y^2)}$, where $\sigma_y>0$ and $C_y>0$ are some constants independent of $e$ and $t$.
\end{condition}

To impose statistical complexity on the function classes we used, we introduce the definition of \emph{localized population Rademacher complexity}, described as follows.

\begin{definition}[Localized Population Rademacher Complexity]
For a given radius $\delta>0$, function class $\mathcal{H}$, and distribution $\nu$, define
\begin{align*}
    R_{n,\nu}(\delta;\mathcal{H}) = \mathbb{E}_{X,\varepsilon}\left[\sup_{h\in \mathcal{H}, \|h\|_{L_2(\nu)} \le \delta} \left|\frac{1}{n} \sum_{i=1}^n \varepsilon_i h(X_i)\right|\right],
\end{align*} where $X_1,\ldots, X_n$ are i.i.d. samples from distribution $\nu$, and $\varepsilon_1,\ldots, \varepsilon_n$ are i.i.d. Rademacher variables taking values in $\{-1,+1\}$ with equal probability which are also independent of $(X_1,\ldots, X_n)$.
\end{definition} 

\begin{condition}[Function Class]
\label{cond:general-function-class}
Suppose the following holds for the function class $\mathcal{G}$ and $\mathcal{F}$ we use:
\begin{itemize}[noitemsep]
\item[(1).] It is uniformly bounded by $B \ge 1$, i.e., $\sup_{h\in \mathcal{G} \cup \mathcal{F}} \|h\|_\infty \le B$.
\item[(2).] $0\in \mathcal{F}$ and the statistical complexity of the function classes $\mathcal{G}+\mathcal{F}:= \{g+f: g\in \mathcal{G}, f\in \mathcal{F}_{S_g}\}$ is upper-bounded by $\delta_n$. In particular, there exists some quantity $1/n \le \delta_n < 1$ such that
\begin{align*}
    R_{n,\mu^{(e)}}(\delta; \partial \mathcal{G}) \le B \delta_n \delta \qquad \text{and} \qquad R_{n,\mu^{(e)}}(\delta; \partial (\mathcal{G}+\mathcal{F})) \le 2B\delta_n \delta 
\end{align*} for any $e\in \mathcal{E}$ and $\delta \in [\delta_n, 2B]$, where $\partial \mathcal{H} = \{h - h': h, h'\in \mathcal{H}\}$.
\end{itemize}
\end{condition}

Note that when $-\cG = \cG$, $R_{n,\mu^{(e)}}(\delta; \partial \mathcal{G}) =R_{n,\mu^{(e)}}(\delta; \mathcal{G})$.
The above three assumptions \cref{cond:general-dgp}, \ref{cond:general-response}, \ref{cond:general-function-class} are standard in the theoretical analysis of regression. Recall the definition of $m^{(e,S)}$ and $\bar{m}^{(S)}$ in \cref{sec:nip-setup}, now we introduce the specific assumption in our multi-environment regression setting. 

\begin{condition}[Invariance and Identification]
\label{cond:general-gf}
For any $S$, let $\overline{\mathcal{G}_S} \supseteq \mathcal{G}_S$, $\overline{\mathcal{F}_S} \supseteq \mathcal{F}_S$ be closed subspaces of $\Theta_S$ satisfying $\overline{\mathcal{G}_S} \subseteq \overline{\mathcal{F}_S}$. In this case, we can define $\Pi_{\mathcal{A}}(h) = \argmin_{a\in \mathcal{A}} \|a - h\|_2$ and $\Pi^{(e)}_{\mathcal{A}}(h) = \argmin_{a\in \mathcal{A}} \|a - h\|_{2,e}$ when $\mathcal{A} \in \{\overline{\mathcal{F}_S}, \overline{\mathcal{G}_S}\}$ and $h\in \Theta_S$. Suppose the following holds:
\begin{itemize}
    \item[1.] (Invariance) There exists some index set $S^\star\subseteq [d]$ such that
    \begin{align*}
        \forall e\in \mathcal{E} \qquad \Pi^{(e)}_{\overline{\mathcal{F}_{S^\star}}}(m^{(e,S^\star)}) = \Pi_{\overline{\mathcal{G}_{S^\star}}}(\bar{m}^{(S^\star)}) := g^\star
    \end{align*}
    \item[2.] (Heterogeneity) For each $S\subseteq [d]$, if $\mathsf{b}_{\mathcal{G}}(S) > 0$, then $\bar{\mathsf{d}}_{\mathcal{G}, \mathcal{F}}(S) > 0$, where
    \begin{align}
        \label{eq:theory-bias-and-variance-general}
        \mathsf{b}_{\mathcal{G}}(S) = \|\Pi_{\overline{\mathcal{G}_{S\cup S^\star}}}(\bar{m}^{(S\cup S^\star)}) - g^\star\|_{2}^2 ~~ \text{and} ~~
        \bar{\mathsf{d}}_{\mathcal{G}, \mathcal{F}}(S) = \frac{1}{|\mathcal{E}|} \sum_{e\in \mathcal{E}} \|\Pi_{\overline{\mathcal{F}_{S}}}^{(e)} (m^{(e,S)}) - \Pi_{\overline{\mathcal{G}_S}}(\bar{m}^{(S)}) \|_{2,e}^2.
    \end{align}
    \item[3.] (Nondegenerate Covariate) For any $S\subseteq [d]$ such that $S^\star \setminus S \neq \emptyset$, we have $\inf_{g\in \overline{\mathcal{G}_S}} \|g - g^\star\|_{2}^2 \ge s_{\min}$ for some constant $s_{\min}>0$.
\end{itemize}
\end{condition}

The first condition ``invariance'' specifies the target regression function $g^\star$ of interests and states the invariance structure imposed for our theoretical analysis. It relaxes the general conditional expectation invariance \eqref{eq:intro-model} when $\overline{\mathcal{F}_S} \subsetneq \Theta_S$. Two leading examples are (1) the fully nonparametric class $\overline{\mathcal{G}_S} = \overline{\mathcal{F}_S} = \Theta_S$, and (2) linear class $\overline{\mathcal{G}_S} = \overline{\mathcal{F}_S} = \{f(x)=\beta_S^\top x_S: \beta_S\in \mathbb{R}^{|S|}\}$. In the first example, we are interested in estimating the invariant conditional expectation $g^\star = m^\star$, and the invariance condition requires the conditional expectation invariance \eqref{eq:intro-model}, that 
\begin{align*}
    \forall e\in \mathcal{E} \qquad m^{(e,S^\star)}(x) = m^\star(x_{S^\star}).
\end{align*} In the second example, when the covariance matrices $\mathbb{E}[X^{(e)} (X^{(e)})^\top]$ across all the environments are all positive definite, we are interested in estimating the invariant linear predictor $g^\star(x) = x^\top \beta^\star$, and such the ``invariance'' condition only requires that
\begin{align*}
    \forall e\in \mathcal{E} \qquad \beta^{(e,S^\star)} \equiv \beta^\star \qquad \text{where} ~~~~ \beta^{(e,S^\star)} = \argmin_{\beta \in \mathbb{R}^d, \beta_{(S^\star)^c}=0} \mathbb{E}[|Y^{(e)} - \beta^\top X^{(e)}|^2],
\end{align*} that is, the best linear predictors constrained on $S^\star$ among all the environment are the same. In this case, the conditional expectations $m^{(e, S^\star)}(x)$ can be nonlinear or different.

The second condition ``heterogeneity'' is for identification and is fundamental to derive the population-level strong convexity with respect to $g^\star$. The two quantities in \eqref{eq:theory-bias-and-variance-general} are general forms of the bias mean and the bias variance, respectively. We refer to $\mathsf{b}_{\mathcal{G}}(S)$ as the bias mean because $\mathsf{b}_{\mathcal{G}}(S)$ is the precise bias of the estimator that regress $Y$ on $X_S$ when $S^\star\subseteq S$ using all the data. This can be formally presented in the following proposition, which asserts that in the absence of our proposed regularizer, a vanilla least squares estimator will not consistently estimate $g^\star$, and the discrepancy $\|\hat{g} - g^\star\|_2^2$ is approximately equal to $\mathsf{b}(S)$ when $n$ is large. 

\begin{proposition}[Inconsistency of Least Squares Estimator]
\label{prop:bias}
Let $S$ be an index set such that $S^\star \subseteq S\subseteq [d]$.
Assume \cref{cond:general-dgp}, \ref{cond:general-response}, \ref{cond:general-function-class}--\ref{cond:general-gf} hold, and $\mathsf{b}_{\mathcal{G}}(S) > 0$. Suppose further that $ U \delta_{n,\log n}+\inf_{g\in \mathcal{G}_S} \|g - \Pi_{\overline{\mathcal{G}_{S}}}(\bar{m}^{(S)})\|_2 = o(1)$, where $U$ and $\delta_{n,t}$ are two constants defined in \cref{thm:oracle} below. Then the estimator $\hat{g}_{\mathtt{R}}$ that minimizes \eqref{eq:method-empirical-r} in $\mathcal{G}_{S}$ satisfies, for large enough $n$,
\begin{align*}
    0.99 \le \frac{\|\hat{g}_{\mathtt{R}} - g^\star\|_2^2}{\mathsf{b}_{\mathcal{G}}(S)} \le 1.01 
\end{align*} with probability at least $1-\{C_y(\sigma_y+1)+1\} n^{-100}$.
\end{proposition}

On the other hand, our proposed FAIR estimator will not converge to the biased solution under the condition ``heterogeneity''. The condition ``heterogeneity'' is an abstraction of the ``identification'' condition in previous subsections, for example, \cref{cond-fairnn-ident} for FAIR-NN.

The last condition ``nondegenerate covariate'' ensures that the target regression function $g^\star$ cannot be exactly fitted by any function $g$ whose dependent variable set $S_g$ does not cover $S^\star$. It reduces to be ``non-collinearity'' when $\mathcal{G}$ is linear. 

In practice, we may only get access to the approximate solution. In our theoretical analysis, we focus on the performance of the approximate solution $(\hat{g}, \hat{f}^{\mathcal{E}})$ satisfying
\begin{align}
\label{eq:approx-solution}
    \sup_{f^{\mathcal{E}} \in \{\mathcal{F}_{S_{\hat{g}}}\}^{|\mathcal{E}|}} \hat{\mathsf{Q}}_\gamma(\hat{g}, f^{\mathcal{E}}) -  (\gamma+1) \delta_{\mathtt{opt}}^2 \le \hat{\mathsf{Q}}_\gamma(\hat{g}, \hat{f}^{\mathcal{E}}) \le \inf_{g\in \mathcal{G}} \sup_{f^\mathcal{E} \in \{\mathcal{F}_{S_g}\}^{|\mathcal{E}|}} \hat{\mathsf{Q}}_\gamma(g, f^{\mathcal{E}}) + (1+\gamma) \delta_{\mathtt{opt}}^2
\end{align} with some optimization error $\delta_{\mathtt{opt}}^2>0$, here $\gamma$ in $(1 + \gamma)$ is the same as that in $\hat{\mathsf{Q}}_{\gamma}$. Now we are ready to state the main result regarding the statistical rate of convergence of our estimator $\hat{g}$ to $g^\star$, that is,
\begin{align*}
    \|\hat{g} - g^\star\|_2 = \left\{ \int (\hat{g} - g^\star)^2 \bar{\mu}_x(dx) \right\}^{1/2}.
\end{align*} 

\begin{theorem}[Main Result for the FAIR Estimator with $\ell_2$ Loss]
\label{thm:oracle}
Assume Conditions \ref{cond:general-dgp}--\ref{cond:general-gf} hold. Define the critical threshold
\begin{align*}
    \gamma^\star := \sup_{S\subseteq [d]: \mathsf{b}_{\mathcal{G}}(S) > 0} \frac{\mathsf{b}_{\mathcal{G}}(S)}{\bar{\mathsf{d}}_{\mathcal{G}, \mathcal{F}}(S)}.
\end{align*} There exists some universal constant $C$ such that, for any $\gamma \ge 8\gamma^\star$, the following holds:

\noindent (1) \underline{General $L_2$ error rate.} Let $t>0$ be arbitrary. Define general approximation errors with respect to the function class $\mathcal{G}$ and $\mathcal{F}$ as
\begin{align*}
&\delta_{\mathtt{a},\mathcal{G}} = \inf_{g\in \mathcal{G}_{S^\star}} \|g-g^\star\|_2 ~~~~\text{and} \\
&~~~~~~~~~~\delta_{\mathtt{a}, \mathcal{F},\mathcal{G}}(S) = \sqrt{\frac{1}{|\mathcal{E}|} \sum_{e\in \mathcal{E}} \sup_{g\in \mathcal{G}: S_g=S} \inf_{f\in {\mathcal{F}_{S_g}}} \|\Pi_{\overline{\mathcal{F}_S}}^{(e)}(m^{(e,S)}) - g - f\|_{2,e}^2},
\end{align*} and the stochastic error as $\delta_{n,t} = \delta_n + \{(\log(nB|\mathcal{E}|) + t + 1)/n\}^{1/2}$, where $\delta_n$ is the quantity in \cref{cond:general-function-class}. Let $U=B(B + \sigma_y \sqrt{\log(n|\mathcal{E}|)})$, then
\begin{align}
\label{eq:main-result-general-rate}
    \|\hat{g} - g^\star\|_2 \le C (1+\gamma) \left(U\delta_{n,t} + \delta_{\mathtt{a}, \mathcal{G}} + \delta_{\mathtt{a}, \mathcal{F}, \mathcal{G}}(S_{\hat{g}}) + \delta_{\mathtt{a}, \mathcal{F}, \mathcal{G}}(S^\star) + \delta_{\mathtt{opt}} \right).
\end{align} with probability at least $\mathfrak{p}=1-6e^{-t}-2C_y(\sigma_y+1)n^{-100}$.

\noindent (2) \underline{Faster $L_2$ error rate and variable selection property.} Moreover, if
\begin{align} 
\label{eq:main-result-faster-cond}
\begin{split}
    \delta_{\mathtt{opt}}^2 + \sup_{S\subseteq [d]}\delta^2_{\mathtt{a}, \mathcal{F}, \mathcal{G}}(S) &+ \delta^2_{\mathtt{a}, \mathcal{G}} + UB\delta_{n,t} \\
    &\le \left\{ 1\land \frac{s_{\min}}{\gamma+1} \land \left(\frac{\gamma}{\gamma+1} \inf_{S: \bar{\mathsf{d}}_{\mathcal{G}, \mathcal{F}}(S) > 0} \bar{\mathsf{d}}_{\mathcal{G}, \mathcal{F}}(S)\right) \right\}/C
\end{split}
\end{align} then the following holds, with probability at least $\mathfrak{p}$, the following holds
\begin{align}
\label{eq:main-result-rate}
    \|\hat{g} - g^\star\|_2 \le C \left(U\delta_{n,t} + \delta_{\mathtt{a}, \mathcal{G}} + \delta_{\mathtt{a}, \mathcal{F}, \mathcal{G}}^\star + \delta_{\mathtt{opt}} \right), \qquad S^\star \subseteq S_{\hat{g}} ~~\text{and}~~ \bar{\mathsf{d}}_{\mathcal{G}, \mathcal{F}}(S_{\hat{g}})=0,
\end{align} where $\delta_{\mathtt{a}, \mathcal{F}, \mathcal{G}}^\star = \{\frac{1}{|\mathcal{E}|} \sum_{e\in \mathcal{E}} \sup_{g\in \mathcal{G}} \inf_{f\in \mathcal{F}_{S_g}} \|g^\star - g - f\|_{2,e}^2\}^{1/2}$.
\end{theorem}

\cref{thm:oracle} generalizes Theorem 4.4 in \cite{fan2024environment} to a broad spectrum of $(\mathcal{G}, \mathcal{F})$ configurations. After specifying the function class $(\mathcal{G},\mathcal{F})$, one can further derive the corresponding identification condition by calculating $(\mathsf{b}_{\mathcal{G}}(S), \bar{\mathsf{d}}_{\mathcal{G}, \mathcal{F}}(S))$ and establish a high probability bound on the $L_2$ error by substituting approximation errors $(\delta_{\mathtt{a},\mathcal{G}}, \delta_{\mathtt{a}, \mathcal{F},\mathcal{G}}(S), \delta_{\mathtt{a}, \mathcal{F}, \mathcal{G}}^\star)$ and stochastic error $\delta_{n}$ for the function class $(\mathcal{G}, \mathcal{F})$. In particular, when $\mathcal{G}$ and $\mathcal{F}$ are restricted to the linear function class, they not only match but also significantly improve the result in \cite{fan2024environment}; see \cref{sec:theory-linearx}. All the results in \cref{table:estimators} are direct corollaries of our abstract result \cref{thm:oracle}.

It is required that $\gamma$ should be greater than a constant-level critical threshold $8\gamma^\star$ for consistent estimation of $g^\star$. \cref{thm:oracle} further establishes a crude instant-dependent and oracle-type error bound \eqref{eq:main-result-general-rate} that holds for arbitrary $n \ge 2$ and scales linearly with $\gamma$. Furthermore, when the stochastic error and approximation errors all go to $0$ as $n$ increases and $n$ is large enough such that \eqref{eq:main-result-faster-cond} holds, we have \eqref{eq:main-result-rate}, which improves the $L_2$ error bound \eqref{eq:main-result-general-rate} in two aspects -- the error bound is no longer dependent on either $\gamma$ or other $m^{(e, S)}$ with $S\neq S^\star$. The quantities in the RHS of \eqref{eq:main-result-faster-cond} can be interpreted as the smaller of (1) the signal of true important variables and (2) the signal of heterogeneity. When one of these signals is weak, one can expect to demand more data to differentiate whether it is signal or noise.

One important ingredient in the FAIR estimator is the choice of regularization hyper-parameter $\gamma$ that promotes the invariance. \cref{thm:oracle} offers some insights on choosing $\gamma$. Firstly, $\gamma \ge C\gamma^\star$ is required such that it will correctly identify $g^\star$ from a population-level perspective. Second, it will influence the $L_2$ error rate when $n$ is not large enough such that \eqref{eq:main-result-faster-cond} does not hold. 
Furthermore, the final $L_2$ error rate \eqref{eq:main-result-rate} when $n$ is large enough is independent of $\gamma$. This indicates that the estimator's performance is somewhat not very sensitive to the choice of hyper-parameter $\gamma$. In this case, one can adopt a slightly conservative large $\gamma$ to meet the population condition $\gamma \ge C\gamma^\star$.

\subsection{Extension to the General Risk Loss under the Nonparametric Setting}
\label{sec:general-loss}

\begin{condition}[Risk Loss]
\label{cond:general-loss}
    Let $v_l=\inf_{g\in \mathcal{G}\cup \{g^\star\}}\sup_{L}\{g(X)\ge L, \bar{\mu}_x\text{-}a.s.\}$ and $v_r=\sup_{g\in \mathcal{G}\cup \{g^\star\}}\inf_{U}\{g(X)\le U, \bar{\mu}_x\text{-}a.s.\}$
    Define $\mathcal{V} = [v_l, v_r]$ be the value that $g(X)$ takes, and $\mathcal{Y}= [\sup_{l}\{Y\ge l, \bar{\mu}_x\text{-}a.s.\}, \inf_{u}\{Y\le u, \bar{\mu}_x\text{-}a.s.\}]$ be the value that $Y$ takes. The loss $\ell(\cdot, \cdot)$ satisfies
    \begin{itemize}[noitemsep]
        \item[(1)]  $\ell(y, v) < \infty$ for any $y\in \mathcal{Y}$ and $v\in \mathcal{V}$ and twice continuously differentiable in $\mathcal{Y} \times \mathcal{V}$. $\frac{\partial \ell(y, v)}{\partial v} = (v-y)\psi(v)$ for some continuously differentiable $\psi(v): \mathbb{R} \to \mathbb{R}$.
        \item[(2)] There exists some universal constant $\zeta\ge 1$ such that
        \begin{align*}
            |\psi(v)| \le \zeta \qquad \text{and} \qquad {\zeta^{-1}} \le \frac{\partial^2 \ell}{\partial v^2}(Y, v) \le {\zeta} \qquad \forall v \in \mathcal{V} \text{ and } \bar{\mu}\text{-}a.s. ~.
        \end{align*}
    \end{itemize}
\end{condition}

The assumptions on risk loss in \cref{cond:general-loss} is standard: (1) ensures that $\ell$ is well-defined on optimal solutions and linear combination of them, (2) requires that the population-level global minima is conditional mean, (3) guarantees that the loss function is strongly convex and smooth in the domain, and satisfies $|\ell(y,v)-\ell(y,v')|\le \zeta|y-\tilde{v}||v-v'|$ for some universal constant $\zeta$, which slightly relaxes the Lipschitz condition in \cite{farrell2021deep} and \cite{foster2019orthogonal}. 

We now state the invariance and identification condition when the general risk loss is adopted.

\begin{condition}[Invariance and Identification for General Risk Loss]
\label{cond:general-gf-2}
Suppose the following holds
\begin{itemize}
    \item[1.] (Invariance) There exists some index set $S^\star\subseteq [d]$ such that
    \begin{align*}
        \forall e\in \mathcal{E} \qquad m^{(e,S^\star)} = \bar{m}^{(S^\star)} =: m^\star
    \end{align*}
    \item[2.] (Heterogeneity) For each $S\subseteq [d]$, if $\mathsf{b}(S) > 0$, then $\bar{\mathsf{d}}(S) > 0$, where
    \begin{align}
        \label{eq:theory-bias-and-variance-general-2}
        \mathsf{b}(S) := \|\bar{m}^{(S\cup S^\star)} - m^\star\|_{2}^2, \qquad 
        \bar{\mathsf{d}}(S) := \frac{1}{|\mathcal{E}|} \sum_{e=1}^m \|m^{(e,S)} - \bar{m}^{(S)} \|_{2,e}^2.
    \end{align}
    \item[3.] (Nondegenerate Covariate) For any $S\subseteq [d]$ such that $S^\star \setminus S \neq \emptyset$, we have $\inf_{g\in \Theta_S} \|g - m^\star\|_{2}^2 \ge s_{\min}$ for some constant $s_{\min}>0$.
\end{itemize}
\end{condition}

We are now ready to state the main result in this case.

\begin{theorem}[Main Result for the FAIR Estimator with General Risk Loss]
\label{thm:oracle-2}
Assume \cref{cond:general-dgp},\ref{cond:general-response},\ref{cond:general-function-class}, and \cref{cond:general-loss}--\ref{cond:general-gf-2} hold. Define the critical threshold
\begin{align*}
    \gamma^\star := \sup_{S\subseteq [d]: \mathsf{b}(S) > 0} \frac{\mathsf{b}(S)}{\bar{\mathsf{d}}(S)}.
\end{align*} There exists some universal constant $C$ such that, for any $\gamma \ge 8\zeta^2\gamma^\star$, the following holds:

\noindent (1) \underline{General $L_2$ error rate.} Let $t>0$ be arbitrary. Define general approximation errors with respect to the function class $\mathcal{G}$ and $\mathcal{F}$ as
\begin{align*}
\delta_{\mathtt{a},\mathcal{G}} = \inf_{g\in \mathcal{G}_{S^\star}} \|g-m^\star\|_2 ~~~~\text{and} ~~~~\delta_{\mathtt{a}, \mathcal{F},\mathcal{G}}(S) = \sqrt{\frac{1}{|\mathcal{E}|} \sum_{e\in \mathcal{E}} \sup_{g\in \mathcal{G}: S_g=S} \inf_{f\in {\mathcal{F}_{S_g}}} \|m^{(e,S)} - g - f\|_{2,e}^2},
\end{align*} and the stochastic error as $\delta_{n,t} = \delta_n + \{(\log(nB|\mathcal{E}|) + t + 1)/n\}^{1/2}$, where $\delta_n$ is the quantity in \cref{cond:general-function-class}. Let $U=B(B + \sigma_y \sqrt{\log(n|\mathcal{E}|)})$, then
\begin{align}
\label{eq:main-result-general-rate-2}
    \|\hat{g} - m^\star\|_2 \lor \|\hat{g} - m^\star\|_n \le C (\zeta+\gamma)\zeta \left(U\delta_{n,t} + \delta_{\mathtt{a}, \mathcal{G}} + \delta_{\mathtt{a}, \mathcal{F}, \mathcal{G}}(S_{\hat{g}}) + \delta_{\mathtt{a}, \mathcal{F}, \mathcal{G}}(S^\star) + \delta_{\mathtt{opt}} \right).
\end{align} with probability at least $\mathfrak{p}=1-6e^{-t}-2C_y(\sigma_y+1)n^{-100}$.

\noindent (2) \underline{Faster $L_2$ error rate.} Moreover, if
\begin{align} 
\label{eq:main-result-faster-cond-2}
\begin{split}
    &\delta_{\mathtt{opt}}^2 + \sup_{S\subseteq [d]}\delta^2_{\mathtt{a}, \mathcal{F}, \mathcal{G}}(S) + \delta^2_{\mathtt{a}, \mathcal{G}} + UB\delta_{n,t} \\
    &~~~~~~~~~~~~~~~~ \le \left\{ 1\land \frac{s_{\min}}{(\gamma+\zeta)\zeta} \land \left(\frac{\gamma}{\gamma+\zeta} \inf_{S: \bar{\mathsf{d}}_{\mathcal{G}, \mathcal{F}}(S) > 0} \bar{\mathsf{d}}_{\mathcal{G}, \mathcal{F}}(S)\right) \right\}/C
\end{split}
\end{align} then the following holds, with probability at least $\mathfrak{p}$,
\begin{align}
\label{eq:main-result-rate-2}
    \|\hat{g} - m^\star\|_2 \lor \|\hat{g} - m^\star\|_n \le C \zeta^2\left(U\delta_{n,t} + \delta_{\mathtt{a}, \mathcal{G}} + \delta_{\mathtt{a}, \mathcal{F}, \mathcal{G}}^\star + \delta_{\mathtt{opt}} \right),
\end{align} where $\delta_{\mathtt{a}, \mathcal{F}, \mathcal{G}}^\star = \{\frac{1}{|\mathcal{E}|} \sum_{e\in \mathcal{E}} \sup_{g\in \mathcal{G}} \inf_{f\in \mathcal{F}_{S_g}} \|m^\star - g - f\|_{2,e}^2\}^{1/2}$.
\end{theorem}

\subsection{Key Ideas and Proof Sketch of Theorem \ref{thm:oracle}}
\label{subsec:proof-sketch}
We first introduce some additional notations. Let
\begin{align*}
    \mathsf{A}^{(e)}(g, f^{(e)}) &= \mathbb{E} \left[ \{Y^{(e)} - g(X^{(e)})\} f^{(e)}(X^{(e)}) - \frac{1}{2} \{f^{(e)}(X^{(e)})\}^2\right] \\
    \hat{\mathsf{A}}^{(e)}(g, f^{(e)}) &= \frac{1}{n} \sum_{i=1}^n  \{Y_i^{(e)} - g(X_i^{(e)})\} f^{(e)}(X_i^{(e)}) - \frac{1}{2} \{f^{(e)}(X_i^{(e)})\}^2.
\end{align*} Define the population-level pooled risk and FAIR estimator loss as
\begin{align*}
    \mathsf{R}(g) = \frac{1}{|\mathcal{E}|} \sum_{e\in \mathcal{E}} \mathbb{E}\left[ \frac{1}{2} |Y^{(e)} - g(X^{(e)})|^2 \right] \qquad \text{and}\qquad \mathsf{Q}_\gamma(g, f^{\mathcal{E}}) = \mathsf{R}(g) + \gamma \mathsf{J}(g, f^{\mathcal{E}})
\end{align*} 

We will use the following theorem establishing approximate strong convexity with respect to $g^\star$.

\begin{theorem}
\label{thm:population}
    Assume \cref{cond:general-gf} hold, $\ell(y, v)=\frac{1}{2}(y-v)^2$. Let $\delta \in (0,1)$ be arbitrary. Then the following holds, for any $\gamma\ge 4\delta^{-1} \gamma^\star$,
    \begin{align*}
        &\mathsf{Q}_\gamma(g, f^{\mathcal{E}}) - \mathsf{Q}_\gamma(\tilde{g}, \tilde{f}^{\mathcal{E}}) \\
        &~~~\ge \frac{1-\delta}{2} \|g - \tilde{g}\|_2^2 +  \frac{\gamma}{4} \bar{\mathsf{d}}_{\mathcal{G}, \mathcal{F}}(S) + \frac{\gamma}{2} \|g - \Pi_{\overline{\mathcal{G}_S}}(\bar{m}^{(S)})\|_2^2\\
        &~~~~~~~~~ - \frac{\gamma}{2|\mathcal{E}|} \sum_{e\in \mathcal{E}}  \|f^{(e)} - \{\Pi_{\overline{\mathcal{F}_S}}^{(e)} (m^{(e,S)}) - g\}\|_{2,e}^2 - (\delta^{-1} + \gamma/2) \|\tilde{g} - g^\star\|_2^2  
    \end{align*} for any $g\in \mathcal{G}$, $\tilde{g} \in \mathcal{G}_{S^\star}$ and $S_{\tilde{g}} = S^\star$, $f^{\mathcal{E}} \in \{\overline{\mathcal{F}_{S_g}}\}^{|\mathcal{E}|}$, and $\tilde{f}^{\mathcal{E}} \in \{\overline{\mathcal{F}_{S^\star}}\}^{|\mathcal{E}|}$.
\end{theorem}

Recall our definition of
\begin{align*}
    \delta_{n,t} = \delta_n + \sqrt{\frac{t + \log (nB|\mathcal{E}|) + 1}{n}} \qquad \text{and} \qquad U = B(B + \sigma\sqrt{\log(n|\mathcal{E}|)})
\end{align*}
The first proposition establishes instance-dependent error bounds on 
\begin{align*}
    \Delta_\mathsf{R}(g, \tilde{g}) := \{\hat{\mathsf{R}}(g) - \hat{\mathsf{R}}(\tilde{g})\} - \{{\mathsf{R}}(g) - \mathsf{R}(\tilde{g})\},
\end{align*} and is standard in nonparametric regression literature.

\begin{proposition}[Instance-dependent error bounds for pooled risk]
\label{prop:nonasymptotic-pooled}
Suppose \cref{cond:general-dgp},\ref{cond:general-response}, \ref{cond:general-function-class} hold. There exists some universal constant $C$ such that for any $\eta>0$ and $t>0$, the following event
\begin{align*}
    \forall g, \tilde{g} \in \mathcal{G}, ~~ |\Delta_\mathsf{R}(g, \tilde{g})| \le C U \left\{ \delta_{n,t}^2 + \delta_{n,t} \frac{1}{|\mathcal{E}|} \sum_{e\in \mathcal{E}} \|g - \tilde{g}\|_{2,e}\right\}
\end{align*} occurs with probability at least $1-3e^{-t}-C_y(\sigma_y+1) n^{-100}$.
\end{proposition}

The analysis of the focused adversarial invariance regularizer is more involved. The next proposition establishes the instance-dependent error bound for the regularizer. We define
\begin{align*}
    \Delta_{\mathsf{A}}^{(e)}(g, \tilde{g}, f^{(e)}, \tilde{f}^{(e)}) = \mathsf{A}^{(e)}(f, g^{(e)}) - \mathsf{A}^{(e)}(\tilde{f}, \tilde{g}^{(e)}) - \left\{\hat{\mathsf{A}}^{(e)}(f, g^{(e)}) - \hat{\mathsf{A}}^{(e)}(\tilde{f}, \tilde{g}^{(e)})\right\}
\end{align*} and
\begin{align*}
    \mathcal{M}(\mathcal{G},\mathcal{F}) = \left\{(g, \tilde{g}, f, \tilde{f}): g, \tilde{g} \in \mathcal{G} ~\text{and}~ f \in \mathcal{F}_{S_g}, \tilde{f} \in \mathcal{F}_{S_{\tilde{g}}} \right\}.
\end{align*}

\begin{proposition}[Instance-dependent error bounds for regularizer]
\label{prop:nonasymptotic-a}
Suppose \cref{cond:general-dgp}, \ref{cond:general-response}, \ref{cond:general-function-class} hold. There exists some universal constant $C$ such that for any $t>0$, the following event 
\begin{align*}
    \forall e\in \mathcal{E}, &~\forall (g, \tilde{g}, f^{(e)}, \tilde{f}^{(e)}) \in \mathcal{M}(\mathcal{G}, \mathcal{F}),\\
    &~~~~~~~~~~|\Delta_{\mathsf{A}}^{(e)}(g, \tilde{g}, f^{(e)}, \tilde{f}^{(e)})| \le C U \left(\delta_{n,t} \left(\|\tilde{g} - g\|_{2,e} + \|\tilde{g} + \tilde{f}^{(e)} - g - f^{(e)}\|_{2,e}\right) + \delta_{n,t}^2\right)
\end{align*} occurs with probability at least $1-3e^{-t}-C_y(\sigma_y+1) n^{-100}$.
\end{proposition}

We first utilize \cref{prop:nonasymptotic-a} in a way that $g$ and $\tilde{g}$ are the same. In this case, the optimization problem of $\max$-$\mathcal{F}$ in one single environment $e\in \mathcal{E}$ for fixed $g\in \mathcal{G}$ is similar to least squares regression that fits the target regression function 
\begin{align*}
    \Pi_{\overline{\mathcal{F}_S}}^{(e)} (m^{(e,S)}) - g.
\end{align*} Thus one can establish high probability error bounds on the $\|\cdot\|_{2,e}$ norm between the empirical loss maximizer $\hat{f}_g^{(e)}$ and the above target function in terms of statistical error $\delta_{n,t}$ and approximation error rate $\delta_{\mathtt{a},\mathcal{F},\mathcal{G}}(e,S_g)$, defined as
\begin{align*}
\delta_{\mathtt{a},\mathcal{F},\mathcal{G}}(e,S) := \sup_{g\in \mathcal{G}: S_g = S} \inf_{f\in \mathcal{F}_S} \|\Pi_{\overline{\mathcal{F}_S}}^{(e)}(m^{(e,S)}) - g - f \|_{2,e}
\end{align*} We formally present the above intuition in the following instance-dependent error bound in \cref{prop:characterize-f} in a way that the optimization gap term is maintained in the error bound.

\begin{proposition}[Instance-dependent characterization of approximately optimal discriminator]
\label{prop:characterize-f}
Let $0<\eta<1/2$ be arbitrary, under the event defined in \cref{prop:nonasymptotic-a}, the following holds,
\begin{align*}
    &\forall e\in \mathcal{E}, \forall g \in \mathcal{G}, \forall f^{(e)} \in \mathcal{F}_{S_g}, \\
    &~~~~~~ \|\Pi_{\overline{\mathcal{F}_S}}(m^{(e,S)}) - g - f^{(e)} \|_{2,e}^2 \le \frac{2\eta^{-1}+2-4\eta}{1-2\eta}\delta^2_{\mathtt{a}, \mathcal{F}, \mathcal{G}}(e, S_g) + \frac{2\eta^{-1}+4}{1-2\eta} C^2 U^2\delta_{n,t}^2 \\
    &~~~~~~~~~~~~~~~~~~~~~~~~~~~~~~~~~~~~~~~~~~~~~~~~ + \frac{4}{1-2\eta} \left\{\sup_{\breve{f} \in \mathcal{F}_{S_g}} \hat{\mathsf{A}}^{(e)}(g, \breve{f}) - \hat{\mathsf{A}}^{(e)}(g, f^{(e)})\right\}
\end{align*} where $C$ is the universal constant defined in \cref{prop:nonasymptotic-a}. Averaging over all the $e\in \mathcal{E}$, we obtain
\begin{align*}
    &\forall g\in \mathcal{G}, ~\forall f^{\mathcal{E}} \in \{\mathcal{F}_{S_g}\}^{|\mathcal{E}|}, \\
    &~~~~~~ \frac{1}{|\mathcal{E}|} \sum_{e\in \mathcal{E}} \|\Pi_{\overline{\mathcal{F}_S}}(m^{(e,S)}) - g - f^{(e)} \|_{2,e}^2 \le \frac{2\eta^{-1}+2-4\eta}{1-2\eta}\delta^2_{\mathtt{a}, \mathcal{F}, \mathcal{G}}(S_g) + \frac{2\eta^{-1}+4}{1-2\eta} C^2 U^2\delta_{n,t}^2\\
    &~~~~~~~~~~~~~~~~~~~~~~~~~~~~~~~~~~~~~~~~~~~~~~~~~~~~~~ + \gamma^{-1} \frac{4}{1-2\eta} \left\{\sup_{\breve{f}^{\mathcal{E}} \in \{\mathcal{F}_{S_g}\}^{|\mathcal{E}|}} \hat{\mathsf{Q}}_\gamma(g, \breve{f}) - \hat{\mathsf{Q}}_\gamma(g, f^{\mathcal{E}})\right\}
\end{align*}
\end{proposition}

Now we are ready to prove \cref{thm:oracle}.

For the proof of (2) faster $L_2$ rate, we will divide the proof into two main steps as follows. 

\begin{enumerate}
\item In the first step, we establish a variable selection property claim that when the \cref{eq:main-result-faster-cond} holds, and the events defined in \cref{prop:nonasymptotic-pooled} and \ref{prop:nonasymptotic-a} occurs, then $\hat{S}$ satisfies
\begin{align*}
    \forall e\in \mathcal{E} \qquad \Pi_{\overline{\mathcal{F}_{\hat{S}}}}^{(e)}(m^{(e,\hat{S})}) = g^\star
\end{align*} using proof by contradiction that any $g$ such that such that the above constrain is violated in $S_g$, will not be the approximate solution of the minimax optimization $\inf_g\sup_{f^{\mathcal{E}}}\hat{\mathsf{Q}}_\gamma(g,f^{\mathcal{E}})$. This can be summarized as the following \cref{prop:variable-selection}. The variable selection property in \eqref{eq:main-result-rate} also follows from this.
\item In the second step, we proceed conditioned on the above claim and derive a sharp $L_2$ error bound. To derive a sharp error bound, we combine (1) the approximate strong convexity with respect to $g^\star$, i.e., \cref{thm:population}, (2) the instance-dependent error bound for $\mathsf{J}$ and $\mathsf{R}$, i.e., \cref{prop:nonasymptotic-pooled} and \ref{prop:nonasymptotic-a}, and (3) the key fact that, if the claim in step 1 holds, then 
\begin{align*}
\|\tilde{g} + \tilde{f}_{\tilde{g}}^{(e)} - g - f_g^{(e)}\|_{2,e} &\le \|\tilde{g} + \tilde{f}_{\tilde{g}}^{(e)} - g^\star + \Pi_{\overline{\mathcal{F}_{S_g}}}^{(e)} (m^{(e,S_g)}) - g - f_g^{(e)}\|_{2,e} \\
&\lesssim \|\tilde{g} + \tilde{f}_{\tilde{g}}^{(e)} - g^\star\|_{2,e} + \|g^\star - g - f_g^{(e)}\|_{2,e} \\
&\lesssim \delta_{n,t} + \delta_{\mathtt{a}, \mathcal{F}, \mathcal{G}}^\star.
\end{align*}
\end{enumerate}

The proof of (1) is similar to the second step in the proof of (2), but now we no longer have $g^\star = \Pi_{\overline{\mathcal{F}_{S_g}}}^{(e)} (m^{(e,S_g)})$. The key challenge here is to establish an upper bound on $\|g^\star - \Pi_{\overline{\mathcal{F}_{S_g}}}^{(e)} (m^{(e,S_g)})\|_{2,e}$ without imposing other population-level condition like Condition 7 in an early version of \cite{fan2024environment}. Instead, we will use the following instance-dependent bound, that
\begin{align*}
\frac{1}{|\mathcal{E}|} \sum_{e\in \mathcal{E}} \|g^\star - \Pi_{\overline{\mathcal{F}_{S_g}}}^{(e)} (m^{(e,S_g)})\|_{2,e}^2 \le C\left( (1+\gamma^\star)\bar{\mathsf{d}}_{\mathcal{G}, \mathcal{F}}(S_g) + \|g - g^\star\|_2^2\right)
\end{align*}
Such a bound is a population-level instance-dependent bound in that both the R.H.S. and L.H.S. are dependent on the function $g$.

\begin{proposition}
\label{prop:variable-selection}
Under the event defined in \cref{prop:nonasymptotic-a} and \ref{prop:nonasymptotic-pooled}, we have the event
\begin{align}
\label{eq:main-result-variable-sel}
   \mathcal{A}_+ := \left\{\forall e\in \mathcal{E} \qquad \Pi_{\overline{\mathcal{F}_{\hat{S}}}}^{(e)}(m^{(e,\hat{S})}) = g^\star  \qquad \text{for} ~~\hat{S} = S_{\hat{g}}\right\}
\end{align} occurs if the condition \eqref{eq:main-result-faster-cond} with some large universal constant $C$ holds.
\end{proposition}


\subsection{Applications of Theorem \ref{thm:oracle} and Connection to the Predecessors}

We present some examples here, sorted by the potential approximation capability of the function class $(\mathcal{G}, \mathcal{F})$.

\begin{example}[Linear $\mathcal{G}$, Linear $\mathcal{F}$]
\label{ex-linear-linear}
    The simplest case is that $\mathcal{G}$ and $\mathcal{F}$ are all linear function classes, that
    \begin{align*}
        \mathcal{G} = \mathcal{F} = \{h(x) = \beta^\top x: \beta\in \mathbb{R}^d\}:= \mathcal{H}_{\mathtt{lin}}(d).
    \end{align*}
\end{example}

The objective takes on a form that closely resembles the EILLS objective proposed in \cite{fan2024environment}. To see this, the EILLS objective is expressed as $\frac{1}{|\mathcal{E}|} \sum_{e\in \mathcal{E}} \hat{\mathbb{E}}[|Y^{(e)}-g(X^{(e)})|^2] + \frac{\gamma}{|\mathcal{E}|} \sum_{e\in \mathcal{E}} \|\hat{r}_{g}^{(e)}\|_2^2$ where $\hat{r}_{g}^{(e)} = \hat{\mathbb{E}}[\{Y^{(e)} - g(X^{(e)})\} X_{S_g}^{(e)}]$. If we take the supremum over all the $f^{(e)}\in \mathcal{F}_{S_g}$ with $e\in \mathcal{E}$, the objective in \eqref{eq:def-fair-estimator} transforms into
\begin{align*}
    \sup_{f^\mathcal{E}\in \{\mathcal{F}_{S_g}\}^{|\mathcal{E}|}} \hat{\mathsf{Q}}_\gamma(g,f^\mathcal{E}) = \frac{1}{|\mathcal{E}|} \sum_{e\in \mathcal{E}} \hat{\mathbb{E}}[|Y^{(e)}-g(X^{(e)})|^2] + \frac{\gamma}{|\mathcal{E}|} \sum_{e\in \mathcal{E}} (\hat{r}_{g}^{(e)})^\top \{\hat{\mathbb{E}}[X_S^{(e)} (X_S^{(e)})^\top]\}^{-1} (\hat{r}_{g}^{(e)}).
\end{align*} It slightly stabilizes the EILLS objective in that the regularizer has a matched moment index compared with the pooled least squares loss; see a detailed explanation and theoretical justification in \cref{sec:theory-linearx}.

\begin{example}[Linear $\mathcal{G}$, Augmented Linear $\mathcal{F}$]
\label{ex-linear-auglinear}
Consider the case where $\mathcal{F}$ is potentially larger than $\mathcal{G}$, that is, $\mathcal{G} = \mathcal{H}_{\mathtt{lin}}(d)$ and $\mathcal{F} = \{f(x) = \beta^\top x + \beta_\phi^\top \bar{\phi}(x): \beta, \beta_\phi \in \mathbb{R}^d\} := \mathcal{H}_{\mathtt{alin}}(d, \phi)$, where $\bar{\phi}(x)=(\phi(x_1), \ldots, \phi(x_d))$ applies a transformation function $\phi: \mathbb{R} \to \mathbb{R}$ to each entry of the vector $x$.
\end{example}

The proposed estimator utilizes both the heterogeneity among different environments and the strong prior knowledge that the true regression function admits linear form. It bridges the EILLS estimator in \cite{fan2024environment} and the Focused GMM estimator in \cite{fan2014endogeneity} when the instrumental variables are $[X_S, \bar{\phi}(X_S)]$ and reduces to an improved version of the latter when $|\mathcal{E}|=1$.

\begin{example}[Linear $\mathcal{G}$, Neural Network $\mathcal{F}$]
\label{ex-linear-nn}
We consider a more algorithmic version of \cref{ex-linear-auglinear} that uses neural networks to automatically learn the transformation function, that is, $\mathcal{G} = \mathcal{H}_{\mathtt{lin}}(d)$ and $\mathcal{F} = \mathcal{H}_{\mathtt{nn}}(d, L_f, N_f, B_f)$ with neural network architecture hyper-parameters of $(L_f, N_f, B_f)$. 
\end{example}

The above three estimators focus on linear $\mathcal{G}$, the simplest structural function class. We now consider a more complicated structural function class when we know the invariant association admits additive form. 
\begin{example}[Additive Neural Network $\mathcal{G}$, Neural Network $\mathcal{F}$]
\label{ex-ann-nn}
    We let $\mathcal{G} = \mathcal{H}_{\mathtt{ann}}(d, L_g, N_g, B_g) := \{g(x) = \mathrm{Tc}_{B_g}(\sum_{j=1}^d g_j(x_j)): g_j \in \mathcal{H}_{\mathtt{nn}}(1, L_g, N_g, \infty) \}$ and $\mathcal{F} = \mathcal{H}_{\mathtt{nn}}(d, L_f, N_f, B_f)$. Here $(L_g, N_g, B_g)$ and $(L_f, N_f, B_f)$ are all neural network architecture hyper-parameters.
\end{example}

Finally, we present the most algorithmic estimator, the \emph{FAIR-NN} estimator, in which both $\mathcal{G}$ and $\mathcal{F}$ are realized by fully-connected neural networks with no additional imposed structures.

\begin{example}[Neural Network $\mathcal{G}$, Neural Network $\mathcal{F}$]
\label{ex-nn-nn}
    We let $\mathcal{G} = \mathcal{H}_{\mathtt{nn}}(d, L_g, N_g, B_g)$ and $\mathcal{F} = \mathcal{H}_{\mathtt{nn}}(d, L_f, N_f, B_f)$ with neural network architecture hyper-parameters $(L_g, N_g, B_g)$ and $(L_f, N_f, B_f)$.
\end{example}

\begin{table}[!t]
    \centering
    \footnotesize
    \begin{tabular}{llllll}
      \hline
       & $\mathcal{G}$ & $\mathcal{F}$ & Category & Short Name & Result \\
      \hline
      \hline
      \cref{ex-linear-linear} & $\mathcal{H}_{\mathtt{lin}}(d)$ & $\mathcal{H}_{\mathtt{lin}}(d)$ & $\mathcal{G} \asymp \mathcal{F}$ & FAIR-Linear & \cref{thm:lglf} \\
      \cref{ex-nn-nn} & $\mathcal{H}_{\mathtt{nn}}(d, L_g, N_g, B_g)$ & $\mathcal{H}_{\mathtt{nn}}(d, L_f, N_f, B_f)$ & $\mathcal{G} \asymp \mathcal{F}$ & FAIR-NN & \cref{thm:fairnn}\\

      \cref{ex-linear-auglinear} & $\mathcal{H}_{\mathtt{lin}}(d)$ & $\mathcal{H}_{\mathtt{alin}}(d, \phi)$ & $\mathcal{G} \ll \mathcal{F}$ & FAIR-AugLinear & \cref{thm:lgbf}\\
      \cref{ex-linear-nn} & $\mathcal{H}_{\mathtt{lin}}(d)$ & $\mathcal{H}_{\mathtt{nn}}(d, L_f, N_f, B_f)$ & $\mathcal{G} \ll \mathcal{F}$ & FAIR-NNLinear & \cref{thm:lgnf}\\
      \cref{ex-ann-nn} & $\mathcal{H}_{\mathtt{ann}}(d, L_g, N_g, B_g)$ & $\mathcal{H}_{\mathtt{nn}}(d, L_f, N_f, B_f)$ & $\mathcal{G} \ll \mathcal{F}$ & FAIR-ANN & \cref{thm:fairann}\\
      \hline
    \end{tabular}
	\caption{A Glimpse of Estimators}
    \label{table:estimators}
\end{table}

Our framework requires $\mathcal{G} \subseteq \mathcal{F}$. We can divide the above estimators into two main categories that (1) $\mathcal{G}$ has roughly the same representation power as $\mathcal{F}$, denoted as $\mathcal{G} \asymp \mathcal{F}$, and (2) $\mathcal{F}$ has at least as good representation power as $\mathcal{G}$, denoted as $\mathcal{G} \ll \mathcal{F}$. For the former, our framework uses only heterogeneity among different environments to identify the invariant association. For the latter, our framework utilizes both the heterogeneity and strong prior structural assumption that the invariant association cannot be significantly better approximated by $\mathcal{F}$ than by $\mathcal{G}$ to jointly identify the invariant association. We summarize the proposed estimators above and divide them into these two categories in \cref{table:estimators}.

\subsection{FAIR-ANN: Bridging Invariance and Additional Structural Knowledge}
\label{sec:fairann}

We next consider the estimator that utilizes both heterogeneity and the strong structural assumption that the invariant association $m^\star$ admits additive form to identify $m^\star$, which can be summarized as the following assumption.

\begin{condition}[Invariance and Nondegenerate Covariate for FAIR-ANN]
\label{cond-fairann-invariance}
There exists some set $S^\star$ and $m^\star: \mathbb{R}^{|S^\star|} \to \mathbb{R}$ such that $m^{(e,S^\star)}(x) \equiv m^\star(x_{S^\star})=\sum_{j\in S^\star} m^\star_j(x_j)$ for any $e\in \mathcal{E}$. Moreover, for any $S\subseteq [d]$ with $S^\star \setminus S \neq \emptyset$, $\inf_{m\in \Theta_S} \|m - m^\star\|_{2}^2 \ge s_{\min}>0$.
\end{condition}

\begin{condition}[Boundedness in Nonparametric Regression]
    \label{cond-fairnn-bounded}
    There exists some constants $b_x$ and $b_m$ such that (1) $X \in [-b_x, b_x]^d$ $\bar{\mu}$-a.s. and (2) $\|m^{(e,S)}\|_\infty \le b_m$ for any $S\subseteq [d]$ and $e\in \mathcal{E}$.
\end{condition}

\begin{condition}
\label{cond-fairann-rsc}
There exists some constant $C_a$ such that
\begin{align*}
    \left\|\sum_{j=1}^d m_j(x_j)\right\|_{2}^2 \ge C_a^{-1} \sum_{j=1}^d \|m_j(x_j)\|_{2}^2 \qquad \forall (m_1,\ldots, m_d)\in \prod_{j=1}^d \Theta_{\{j\}} ~\text{with}~ \int m_j(x_j) \bar{\mu}_x(dx)\equiv0.
\end{align*}
\end{condition}
The above condition is referred to as the nonparametric version of the restricted strong convexity condition, which is widely used in the theoretical analysis for nonparametric high-dimension additive models \citep{van2008high, raskutti2012minimax, yuan2016minimax}. This condition is imposed to let $\prod_{j\in S} \Theta_{\{j\}}$ be a closed subspace of $\Theta_S$, where we can define
\begin{align*}
    A_{S}(h) = \argmin_{u\in \prod_{j\in S} \Theta_{\{j\}}} \|h - u\|_2,
\end{align*} which finds a unique additive function dependent on $x_S$ that fits $h$ best in $\|\cdot\|_2$ norm.

\begin{condition}[Identification for FAIR-ANN]
\label{cond-fairann-ident}
For any $S\subseteq [d]$ such that $\bar{\mu}(\{m^\star \neq A_{S\cup S^\star}(\bar{m}^{(S\cup S^\star)})\}) > 0$, either of the two holds: (1) there exists some $e, e'\in \mathcal{E}$ such that $(\mu^{(e)}\land \mu^{(e')})(\{m^{(e, S)} \neq m^{(e', S)}\}) > 0$, or (2) $\bar{\mu}(\{\bar{m}^{(S)} \neq A_S(\bar{m}^{(S)})\})>0$.
\end{condition}

With network hyper-parameter $N,L$, we realize the $\mathcal{G}$ and $\mathcal{F}$ as
\begin{align}
\label{eq:fairann-function-class}
    \mathcal{G}=\mathcal{H}_{\mathtt{ann}}(d, L, N, b_m) \qquad \text{and} \qquad \mathcal{F} = \mathcal{H}_{\mathtt{nn}}(d, L+2, 2dN, 2b_m).
\end{align} Similarly to the choice of for FAIR-NN \eqref{eq:fairnn-function-class}, the choice of $\mathcal{F}$ is to ensure $\mathcal{G} - \mathcal{G} \subseteq \mathcal{F}$. 

\begin{theorem}[Optimal Rate for FAIR-ANN Least Squares Estimator]
\label{thm:fairann}
    Assume \cref{cond:general-dgp},\ref{cond:general-response}, and \ref{cond-fairann-invariance}--\ref{cond-fairann-ident} hold. Assume further that all the conditional moments $\{m^{(e,S)}\}_{e\in \mathcal{E}, S\subseteq [d]}$ are $(\beta', C')$-smooth for some $\beta'>0$ and $C'>0$, and $\delta_{\mathtt{opt}}=o(1)$. Consider the FAIR-ANN estimator that solves \eqref{eq:approx-solution} with $\ell(y,v) = \frac{1}{2}(y-v)^2$ using $\gamma \ge 8\gamma^\star_{\mathtt{AN}}$ with
    \begin{align}
    \label{eq:fairann-gamma}
        \gamma_{\mathtt{AN}}^\star:= \sup_{S\subseteq [d]: \bar{\mu}(\{m^\star \neq A_{S\cup S^\star}(\bar{m}^{(S\cup S^\star)})\})>0} \frac{\|m^\star - A_{S\cup S^\star}(\bar{m}^{(S\cup S^\star)}))\|_2^2}{\frac{1}{|\mathcal{E}|} \sum_{e\in \mathcal{E}} \|m^{(e,S)} - A_{S}(\bar{m}^{(S)}))\|_{2,e}^2},
    \end{align} and function class \eqref{eq:fairann-function-class} with $L, N$ satisfying $LN \asymp \{n (\log n)^{8\beta^\star-3}\}^{\frac{1}{2(2\beta^\star+1)}}$ and $(\log n)/(N\land L) = o(1)$.
    Then, we have (1) $\gamma_{\mathtt{AN}}^\star\le \gamma_{\mathtt{NN}}^\star$, and (2) for $n$ large enough, the following event occurs with probability at least $1-\tilde{C}n^{-100}$
    \begin{align}
    \label{eq:rate-fairann}
            \sup_{\substack{m^\star=\sum_{j\in S^\star} m^\star_j(x_j) ~\text{with}~ m_j^\star \in \mathcal{H}_{\mathtt{HS}}(1, \beta^\star, C^\star) \\ \|m^\star\|_\infty \le b_m} }\|\hat{g} - m^\star\|_2 \le \tilde{C}\left\{\delta_{\mathtt{opt}} + \left(\frac{\log^7 n}{n}\right)^{-\frac{\beta^\star}{2\beta^\star+1}} \right\},
    \end{align} where $\tilde{C}$ is a constant that depends on $(C_1, d, \beta^\star, C^\star, \sigma_y,C_y, b_x, b_m)$ but independent of $\gamma, \delta_{\mathtt{opt}}$ and $n$.
\end{theorem}

The choice of $N, L$, and the convergence rate align with FAIR-NN with $\alpha^\star = \beta^\star$. Given the strong structural prior knowledge that the true regression function is additive, FAIR-ANN requires weaker identification condition \cref{cond-fairann-ident} and also smaller critical threshold of $\gamma$. In particular, \cref{cond-fairann-ident} requires that for any $S$ such that regressing $Y$ on $X_{S\cup S^\star}$ via additive models yields biased estimation, there should be either (1) a shift in conditional moments $m^{(e, S)}$ across different environments, or (2) one of the conditional moments $m^{(e, S)}$ is non-additive. This characteristic is called the ``\emph{double identifiable}'' property since meeting either of these conditions can consistently estimate $m^\star$. Notably, the critical threshold $\gamma^\star_{\mathtt{AN}}$ can be smaller than that of the FAIR-NN estimator. A small $\gamma$ can be adopted if either the signal of violating the additive structure or the signal of heterogeneity is strong. 

\subsection{Theoretical Analysis for Linear Prediction Function Class}
\label{sec:theory-linearx}

In this section, we apply our result in \cref{thm:oracle} to the cases where the target regression function $g^\star$ is linear. As such, we use linear function class $\mathcal{H}_{\mathtt{lin}}(d)$ as our predictor function class $\mathcal{G}$. Our theorem suggests that enhancing the potential approximation ability of the discriminator function class $\mathcal{F}$ will result in (1) a stronger condition on invariance, and (2) a weaker identification condition and a reduced choice of critical threshold $\gamma^\star$. 

\subsubsection{Linear Testing Function Class}

We first consider the case where we use linear discriminator function class $\mathcal{F}=\mathcal{H}_{\mathtt{lin}}(d)$. We introduce some notations used in linear regression and state some standard regularity conditions used in linear regression and are also imposed in \cite{fan2024environment}. 

\begin{condition}
\label{cond:lglf}
Suppose the following holds:
\begin{itemize}
    \item[(1)] The data satisfies \cref{cond:general-dgp} with $|\mathcal{E}| \le n^{C_1}$ for some constant $C_1$.
    \item[(2)] The covariance matrix $\Sigma^{(e)}=\mathbb{E}[X^{(e)} (X^{(e)})^\top] \in \mathbb{R}^{d\times d}$ in each environment satisfies $\lambda(\Sigma^{(e)}) \ge \kappa_L$ for some constant $\kappa_L > 0$.
    \item[(3)] Define the pooled covariance matrix ${\Sigma} := |\mathcal{E}|^{-1} \sum_{e\in \mathcal{E}} \Sigma^{(e)}$. There exists some positive constant $C_x, \sigma_x$ such that
    \begin{align*}
        \forall e\in \mathcal{E}, ~ \forall v\in \mathbb{R}^d ~\text{with} ~\|v\|_2=1, ~\forall t\in [0,\infty), \qquad \mathbb{P}\left( |v^\top ({\Sigma})^{-1/2} X^{(e)}| \ge t\right) \le C_x e^{-t^2 / (2\sigma_x^2)} 
    \end{align*}
    \item[(4)] \cref{cond:general-response} holds.
\end{itemize}
\end{condition}

Under \cref{cond:lglf} that the covariance matrices are all positive definite, we can define
\begin{align*}
    \beta^{(e,S)} = \argmin_{\beta \in \mathbb{R}^d: \beta_{S^c}=0} \mathbb{E}[|Y^{(e)} - \beta^\top X^{(e)}|^2]
\end{align*}
We can state the invariance and identification condition in this case.
\begin{condition}[Invariance in Linear $\mathcal{G}$ and Linear $\mathcal{F}$]
\label{cond:lglf-invariance}
There exists some $S^\star \subseteq [d]$ and $\beta^\star \in \mathbb{R}^d$ with $\beta^\star_{(S^\star)^c}=0$ and $\min_{j\in S^\star} |\beta_j^\star| = \beta_{\min} > 0$ such that
\begin{align}
\label{eq:lglf:invariance}
    \forall e\in \mathcal{E} \qquad \beta^{(e,S)} = \beta^\star.
\end{align}
\end{condition}
Let $\varepsilon^{(e)} = Y^{(e)} - (\beta^\star)^\top X^{(e)}$, the above invariance equality \eqref{eq:lglf:invariance} is equivalent to that $X_{S^\star}$ are exogenous across all the environments, that is,
\begin{align*}
\forall e\in \mathcal{E} \qquad \mathbb{E}[\varepsilon^{(e)} X_{S^\star}^{(e)}] = 0
\end{align*}

\begin{condition}[Identification for Linear $\mathcal{G}$ and Linear $\mathcal{F}$]
\label{cond:lglf:identification}
For any  $S \subseteq [d]$ with $\sum_{e\in \mathcal{E}} \mathbb{E}[X_S^{(e)} \varepsilon^{(e)}] \neq 0$, there exists $e, e'\in \mathcal{E}$ such that $\beta^{(e,S)} \neq \beta^{(e',S)}$.
\end{condition}

We are ready to state the result using truncated linear function class with bounded $L_2$ norm, that is,
\begin{align*}
    \mathcal{H}_{\mathtt{lin}}(d, B_1, B_2) = \left\{f(x)=\mathrm{Tc}_{B_2}(\beta^\top x): \beta \in \mathbb{R}^d, \|\Sigma^{1/2} \beta\|_2 \le B_1\right\}.
\end{align*}

\begin{theorem}[Linear $\mathcal{G}$ and Linear $\mathcal{F}$] 
\label{thm:lglf}
Suppose \cref{cond:lglf}--\ref{cond:lglf:identification} hold, and we choose 
\begin{align*}
\mathcal{G} = \mathcal{H}_{\mathtt{lin}}(d, C_2, C_2 \sqrt{\log n})\qquad \text{and}\qquad \mathcal{F} = \mathcal{H}_{\mathtt{lin}}(d, 2C_2, 2C_2 \sqrt{\log n})
\end{align*}with some constant $C_2 \ge 2(\sigma_x \lor 1) \max_{e\in \mathcal{E}, S\subseteq[d]} \|\Sigma^{1/2} \beta^{(e,S)}\|_2$. Then, there exists some constant $\tilde{C}$ that only depends on $(C_1, C_2, \sigma_x, C_x, \sigma_y, C_y)$ such that the FAIR least squares estimator using the above function class and hyper-parameter $\gamma$ satisfying $\gamma \ge 8\gamma_{\mathtt{LL}}^\star=8\sup_{S:\mathsf{b}_{\mathtt{LL}}(S) > 0} \mathsf{b}_{\mathtt{LL}}(S)/\bar{\mathsf{d}}_{\mathtt{LL}}(S)$, where 
\begin{align}
\label{eq:biasvar-lglf}
\begin{split}
    \mathsf{b}_{\mathtt{LL}}(S) &= \bigg\|\frac{1}{|\mathcal{E}|}\sum_{e\in \mathcal{E}} \mathbb{E}[X_{S\cup S^\star}^{(e)} \varepsilon^{(e)}]\bigg\|_{(\bar{\Sigma}_{S\cup S^\star})^{-1}}^2 \le (\kappa_L)^{-1}\bigg\|\frac{1}{|\mathcal{E}|}\sum_{e\in \mathcal{E}} \mathbb{E}[X_S^{(e)} \varepsilon^{(e)}]\bigg\|_2^2, \\
    \bar{\mathsf{d}}_{\mathtt{LL}}(S) &= \frac{1}{|\mathcal{E}|} \sum_{e\in \mathcal{E}} \|\beta^{(e,S)}_S - \beta^{(S)}_\dagger\|_{\Sigma_S^{(e)}}^2 \ge \kappa_L \frac{1}{|\mathcal{E}|} \sum_{e\in \mathcal{E}} \|\beta^{(e,S)} - \bar{\beta}^{(S)}\|_2^2
\end{split}
\end{align} with $\beta_\dagger^{(S)} = (\bar{\Sigma}_S)^{-1} \{\frac{1}{|\mathcal{E}|} \sum_{e\in \mathcal{E}} \mathbb{E}[X^{(e)}_{S} Y^{(e)}]\}$ and $\bar{\beta}^{(S)} = \frac{1}{|\mathcal{E}|}\sum_{e\in \mathcal{E}} \beta^{(e,S)}$, satisfies, with probability at least $1-\tilde{C}n^{-100}$, \begin{align}
\label{eq:lglf:l2-general}
    \forall n \ge 3 \qquad \|{\Sigma}^{1/2}(\beta_{\hat{g}} - \beta^\star)\|_2 \le \tilde{C}(1 + \gamma) \sqrt{\frac{d\log^5 (n)}{n}},
\end{align} for $\hat{g}(x) = \mathrm{Tc}_B(\beta_{\hat{g}}^\top x)$. Moreover, if $d=o((1+\gamma^2)n/(\log^6 n))$, then for large enough $n$, we further have
\begin{align}
\label{eq:lglf:l2-faster}
    \|{\Sigma}^{1/2}(\beta_{\hat{g}} - \beta^\star)\|_2 \le \tilde{C} \sqrt{\frac{d\log^5(n)}{n}}
\end{align}
\end{theorem}

\begin{remark}
We present the results using truncated function classes, and there exist poly-$\log n$ factors in the non-asymptotic $L_2$ error bounds. These are for technical convenience such that we can directly apply our result \cref{thm:oracle} which focuses on uniformly bounded function classes. Indeed, one can use a finer analysis and obtain the $\ell_2$ error bound
\begin{align*}
    \sqrt{\frac{d + \log n}{n}}
\end{align*} using unbounded linear function class. 
\end{remark}

The obtained results in \cref{thm:lglf} align with (up to $\log(n)$ factors) and offer significant enhancements over Theorem 2 \& 3 from \cite{fan2024environment}. Firstly, the ``invariance'' condition gets relaxed, we only assume that the noise $\varepsilon^{(e)}$ and the true important variables $X_{S^\star}^{(e)}$ are uncorrelated rather than conditional independent across different environments. Meanwhile, the identification condition \cref{cond:lglf:identification} exactly matches that in \cite{fan2024environment} (refer to Condition 5 therein), and the choice of critical threshold $\gamma^\star$ gets reduced as indicated by the inequality in \eqref{eq:biasvar-lglf} and given that $\kappa_L = O(1)$. Such an improvement can be attributed to the term $-\frac{1}{2} \{f^{(e)}\}^2$ in our minimax regularization that stabilizes the objective. To see this, consider $\beta$ with $\supp(\beta) = S^\star$, the population-level EILLS objective can be written as
\begin{align*}
     (\beta - \beta^\star)^\top \Sigma (\beta - \beta^\star) + \gamma \frac{1}{|\mathcal{E}|} \sum_{e\in \mathcal{E}} (\beta - \beta^\star)_{S^\star}^\top (\Sigma^{(e)}_{S^\star})^2 (\beta - \beta^\star)_{S^\star},
\end{align*} where a square of the covariance matrix appears in the regularizer. This does not match what it is in the empirical risk part and will make the objective less stable. Meanwhile, the population-level FAIR objective with sup-$f$ in this case is
\begin{align*}
    (1+\gamma) (\beta - \beta^\star)^\top \Sigma (\beta - \beta^\star),
\end{align*} which the problem of mismatched covariance matrix order disappears.

We've also refined the non-asymptotic $L_2$ error bounds. On the one hand, we can derive the error bound without further imposing stronger population-level conditions (Condition 7 required by Theorem 3 in \cite{fan2024environment}). On the other, the faster $\ell_2$ error bound for sufficiently large $n$ remains independent of the hyper-parameter $\gamma$ we choose. These refinements result from our tighter characterization of the instance-dependent error bounds compared to the ones in \cite{fan2024environment}; see the discussion on technical novelties in \cref{subsec:proof-sketch}.

\subsubsection{Augmented Linear Testing Function Class}

Here we consider the case where the discriminator function class $\mathcal{F}$ is potentially larger than the predictor function class $\mathcal{G}$. We introduce the following notations. We let $[x, y]$ be the concatenation of two vectors $x\in \mathbb{R}^{d_1}$ and $y\in \mathbb{R}^{d_2}$ as a $d_1+d_2$ dimensional vector. For each $S\subseteq [d]$, we define $\tilde{X}_S^{(e)} = [X_S^{(e)}, \bar{\phi}(X_S^{(e)})]\in \mathbb{R}^{2|S|}$, $\tilde{\Sigma}^{(e)}_S=\mathbb{E}[\tilde{X}_S^{(e)} (\tilde{X}_S^{(e)})^\top] \in \mathbb{R}^{(2|S|)\times(2|S|)}$ and let $\tilde{X}^{(e)} = \tilde{X}^{(e)}_{[d]}$ and $\tilde{\Sigma}^{(e)} = \tilde{\Sigma}^{(e)}_{[d]}$. We impose additional regularity conditions due to the incorporation of basis function $\phi$.

\begin{condition}
\label{cond:lgbf}
There exists some constant $\tilde{\kappa}_L > 0$ such that $\lambda_{\min}(\tilde{\Sigma}^{(e)}) \ge \tilde{\kappa}_L$ for any $e\in \mathcal{E}$. Moreover, define $\tilde{\Sigma}:= |\mathcal{E}|^{-1} \sum_{e\in \mathcal{E}} \tilde{\Sigma}^{(e)}$. There exists some positive constant $C_{\tilde{x}}, \sigma_{\tilde{x}}$ such that
\begin{align*}
    \forall e\in \mathcal{E}, ~ \forall v\in \mathbb{R}^{2d} ~\text{with} ~\|v\|_2=1, ~\forall t\in [0,\infty), \qquad \mathbb{P}\left( |v^\top (\tilde{\Sigma})^{-1/2} \tilde{X}^{(e)}| \ge t\right) \le C_{\tilde{x}} e^{-t^2 / (2\sigma_{\tilde{x}}^2)} 
\end{align*}
\end{condition}
Under \cref{cond:lgbf} such that the covariance matrix for $\tilde{X}$ are positive definite, we can define 
\begin{align*}
\tilde{\beta}^{(e,S)} = [\breve{\beta}, \breve{\beta}^\phi] ~~~~\text{with}~~~~ (\breve{\beta}, \breve{\beta}^\phi) = \argmin_{({\beta}, \beta^\phi) \in (\mathbb{R}^d)^2, \beta_{S^c} = \beta_{S^c}^\phi = 0} \mathbb{E}[|Y^{(e)} - \beta^\top X^{(e)} - (\beta^\phi)^\top \bar{\phi}(X^{(e)})|^2],
\end{align*} and $\tilde{\beta}^{(e,S)}_S = [\breve{\beta}_S, \breve{\beta}^\phi_S]$ be a $2|S|$-dimensional vector. The invariance and identification conditions in this case are as follows.

\begin{condition}[Invariance in Linear $\mathcal{G}$ and Augmented Linear $\mathcal{F}$]
\label{cond:lgbf-invariance}
There exists some $S^\star \subseteq [d]$ and $\beta^\star \in \mathbb{R}^d$ with $\beta^\star_{(S^\star)^c}=0$ and $\min_{j\in S^\star} |\beta_j^\star| = \beta_{\min} > 0$ such that
\begin{align}
\label{eq:lgbf:invariance}
    \forall e\in \mathcal{E} \qquad \tilde{\beta}^{(e,S)} = [\beta^\star, 0].
\end{align}
\end{condition}
Let $\varepsilon^{(e)} = Y^{(e)} - (\beta^\star)^\top X^{(e)}$ be the noise, the above invariance equality \eqref{eq:lglf:invariance} is equivalent to that both $X_{S^\star}$ and $\bar{\phi}(X_{S^\star})$ are uncorrelated with noise across all the environments, that is,
\begin{align*}
\forall e\in \mathcal{E} \qquad \mathbb{E}[\varepsilon^{(e)} X_{S^\star}^{(e)}] = \mathbb{E}[\varepsilon^{(e)} \bar{\phi}(X_{S^\star}^{(e)})] = 0
\end{align*}

\begin{condition}[Identification for Linear $\mathcal{G}$ and Augmented Linear $\mathcal{F}$]
\label{cond:lgbf:identification}
For any  $S \subseteq [d]$ with $\sum_{e\in \mathcal{E}} \mathbb{E}[X_S^{(e)} \varepsilon^{(e)}] \neq 0$, either (1) there exists some $e\in \mathcal{E}$ such that $\tilde{\beta}^{(e,S)} \neq [\beta^{(e,S)}, 0]$, or (2) there exists $e, e'\in \mathcal{E}$ such that $\beta^{(e,S)} \neq \beta^{(e',S)}$.
\end{condition}

For technical convenience, we also used truncated function class the discriminator class, defined as $\mathcal{H}_{\mathtt{alin}}(d, \phi, B) = \{\tilde{f}=\mathrm{Tc}_B(f): f \in \mathcal{H}_{\mathtt{alin}}\}$.

\begin{theorem}[Linear $\mathcal{G}$ and Augmented Linear $\mathcal{F}$] 
\label{thm:lgbf}
Suppose \cref{cond:lglf}, \ref{cond:lgbf}--\ref{cond:lgbf:identification} hold, and we choose 
\begin{align*}
\mathcal{G} = \mathcal{H}_{\mathtt{lin}}(d, C_2, C_2 \sqrt{\log n}) \qquad \text{and} \qquad \mathcal{F} = \mathcal{H}_{\mathtt{alin}}(d, \phi, 2C_2 \sqrt{\log n})
\end{align*}
with some constant $C_2 \ge 2(\sigma_{\tilde{x}} \lor 1) \max_{e\in \mathcal{E}, S\subseteq[d]} \|\tilde{\Sigma}^{1/2} \tilde{\beta}^{(e,S)}\|_{2}$. Then, there exists some constant $\tilde{C}$ that only depends on $(C_1, C_2, \sigma_{\tilde{x}}, C_{\tilde{x}}, \sigma_y, C_y)$ such that the FAIR least squares estimator using the above function classes and hyper-parameter $\gamma$ satisfying $\gamma \ge 8\gamma^\star_{\mathtt{LA}}=8\sup_{S:\mathsf{b}_{\mathtt{LL}}(S) > 0} \mathsf{b}_{\mathtt{LL}}(S)/\bar{\mathsf{d}}_{\mathtt{LA}}(S)$, where 
\begin{align}
\label{eq:biasvar-lgbf}
    \bar{\mathsf{d}}_{\mathtt{LA}}(S) = \frac{1}{|\mathcal{E}|} \sum_{e\in \mathcal{E}} \|\tilde{\beta}^{(e,S)}_S - [\beta^{(S)}_\dagger, 0]\|_{\tilde{\Sigma}_S^{(e)}}^2 \ge \bar{\mathsf{d}}_{\mathtt{LL}}(S) \qquad \text{with} ~~ \beta^{(S)}_\dagger \text{defined in \cref{thm:lglf}},
\end{align} satisfies the $L_2$ error bound \eqref{eq:lglf:l2-general} with probability at least $1-\tilde{C}n^{-100}$. Moreover, if $d=o((1+\gamma^2)n/(\log^6 n))$, for large enough $n$, the error bound \eqref{eq:lglf:l2-faster} also holds with probability at least $1-\tilde{C}n^{-100}$.
\end{theorem}

We can see that the proposed estimator utilizes both the heterogeneity among different environments and strong prior knowledge that the true regression function admits linear form to help the identification. It bridges the EILLS estimator in \cite{fan2024environment} and the Focused GMM (FGMM) estimator in \cite{fan2014endogeneity} when the instrumental variables are $[X_S, \bar{\phi}(X_S)]$ and hence has some advantages over the individual ones. We illustrate this as follows.

\begin{itemize}[noitemsep]
\item[1.] When there are multiple environments $|\mathcal{E}|> 1$, the identification condition \cref{cond:lgbf:identification} is weaker to both the EILLS and FGMM estimators. In particular, a consistent estimate $\beta^\star$ is attainable if incorporating variables $x_j$ with $\sum_{e\in \mathcal{E}} \mathbb{E}[X_j^{(e)} \varepsilon^{(e)}] \neq 0$ will result in either (1) a shift in the best linear predictor across environments or (2) the fitted residuals is strongly correlated with some nonlinear basis. We refer to this property as ``\emph{double identifiable}'' property, given satisfying either condition can lead to the consistent estimation of the true parameter. Furthermore, the critical threshold $\gamma^\star$ can be smaller than that of the EILLS estimator according to the inequality $\bar{\mathsf{d}}_{\mathtt{LA}}(S) \ge \bar{\mathsf{d}}_{\mathtt{LL}}(S)$. This implies that the estimation is sample efficient, which allows for a small $\gamma$, if either the signal of nonlinear basis or the signal of heterogeneity is strong.
\item[2.] If there is only one environment $|\mathcal{E}|=1$, it reduces to an estimator similar to the FGMM estimator. Consistent estimation remains feasible in this case but completely impossible for EILLS estimator. Moreover, the identification condition, in this case, resembles and relaxes that in \cite{fan2014endogeneity}.
\end{itemize}

At the same time, it should be noted that the above advantages over the EILLS estimator (linear $\mathcal{F}$) are at the cost of imposing stronger invariance condition \cref{cond:lgbf-invariance}, which assures that the noise should not only be uncorrelated with $X_j^{(e)}$ but also be uncorrelated with $\phi(X_j^{(e)})$ for any $j\in S^\star$ and $e\in \mathcal{E}$. 

\subsubsection{Neural Network Testing Function Class}

We impose some regularity conditions on the regression function.
\begin{condition}
\label{cond:lgnf}
    There exists some constant $(C_m, \sigma_m)$ such that $m^{(e,S)}$ is $C_m$ Lipschitz and $|m^{(e,S)}(0)|\le C_m$ for any $e\in \mathcal{E}$ and $S\subseteq [d]$ and
    \begin{align*}
        \mathbb{P}(|m^{(e,S)}(X_S^{(e)})| \ge t) \le C_m e^{-t^2/(2\sigma_m^2)} \qquad \forall t\in [0,\infty)
    \end{align*}
\end{condition}

In this case, we consider the strongest invariance condition together with the weakest identification when the predictor function class  $\mathcal{G}$ is linear.

\begin{condition}[Invariance in Linear $\mathcal{G}$ and Neural Network $\mathcal{F}$]
\label{cond:lgnf-invariance}
There exists some $S^\star \subseteq [d]$ and $\beta^\star \in \mathbb{R}^d$ with $\beta^\star_{(S^\star)^c}=0$ and $\min_{j\in S^\star} |\beta_j^\star| = \beta_{\min} > 0$ such that
\begin{align}
\label{eq:lgnf:invariance}
    \forall e\in \mathcal{E} \qquad \mathbb{E}[Y^{(e)}|X_{S^\star}^{(e)}] \equiv (\beta^\star)^\top X^{(e)}
\end{align}
\end{condition}

\begin{condition}[Identification for Linear $\mathcal{G}$ and Neural Network $\mathcal{F}$]
\label{cond:lgnf:identification}
For any  $S \subseteq [d]$ with $\sum_{e\in \mathcal{E}} \mathbb{E}[X_S^{(e)} \varepsilon^{(e)}] \neq 0$, either (1) there exists some $e\in \mathcal{E}$ such that $\mu^{(e)}(\{m^{(e,S)} \neq X^\top \beta^{(e,S)}\}) > 0$, or (2) there exists $e, e'\in \mathcal{E}$ such that $\beta^{(e,S)} \neq \beta^{(e',S)}$.
\end{condition}

\begin{theorem}[Linear $\mathcal{G}$ and Neural Network $\mathcal{F}$] 
\label{thm:lgnf}
Suppose \cref{cond:lglf}, \ref{cond:lgnf}--\ref{cond:lgnf:identification} hold, and we choose the function classes $\mathcal{G} = \mathcal{H}_{\mathtt{lin}}(d, C_2, C_2\sqrt{\log n})$ and $\mathcal{H}_{\mathtt{nn}}(d, \log^d n, \log^d n, C_2\sqrt{\log n})$ with some constant $C_2 \ge (1\lor\sigma_x\lor \sigma_m) \max_{e\in \mathcal{E}, S\subseteq[d]}\|\Sigma^{1/2}\beta^\star\|_2$. 
 Then, there exists some constant $\tilde{C}$ that only depends on $(C_1, C_2, d, \sigma_{m}, C_{m}, \sigma_y, C_y, \sigma_x, C_x)$ such that the FAIR estimator using the above function classes and hyper-parameter $\gamma$ satisfying $\gamma \ge 8\gamma^\star_{\mathtt{LN}}=8\sup_{S:\mathsf{b}_{\mathtt{LL}}(S) > 0} \mathsf{b}_{\mathtt{LL}}(S)/\bar{\mathsf{d}}_{\mathtt{LN}}(S)$, where 
 \begin{align}
 \label{eq:biasvar-lgnf}
 \bar{\mathsf{d}}_{\mathtt{LN}}(S) = \frac{1}{|\mathcal{E}|} \sum_{e\in \mathcal{E}} \|m^{(e,S)} - (\beta^{(S)}_\dagger)^\top x_S\|_{2,e}\ge \mathsf{d}_{\mathtt{LA}}(S),
 \end{align} satisfies, for large enough $n$,
\begin{align*}
    \qquad \|\beta_{\hat{g}} - \beta^\star\|_2 \le \tilde{C} (\log^{d+3} n) n^{-1/2}
\end{align*} with probability at least $1-\tilde{C}n^{-100}$.
\end{theorem}

The estimator can be viewed as an advanced version of the one using $\mathcal{F} = \mathcal{H}_{\mathtt{alin}}(d, \phi)$. It leverages neural networks to search for appropriate basis function $\phi$ with strong signals. With the proper choice of the neural network hyper-parameters, the estimator still maintains a parametric optimal rate (up to logarithmic factors). Additionally, it requires a weaker identification condition as described by \cref{cond:lgnf:identification} and reduced critical threshold $\gamma^\star$ according to the inequality $\bar{\mathsf{d}}_{\mathtt{LN}}(S) \ge \bar{\mathsf{d}}_{\mathtt{LA}}(S)$ in \cref{thm:lgnf}.

\section{Further Details of Experiments}
\label{sec:exper}

\subsection{Pseudo-code of the Gradient Descent Ascent Algorithm}
\label{sec:code}

\begin{algorithm}
\caption{FAIR Gradient Descent Ascent Training}
\begin{algorithmic}[1]
\State \textbf{SGD Hyper-parameters: } iteration steps $T$, batch size $m$, predictor/discriminator iter steps $T_g$/$T_f$.
\State \textbf{FAIR Hyper-parameters: } invariance regularization $\gamma$.
\State \textbf{Annealing Hyper-parameters: } Initial $\tau_0$ and final $\tau_T$.
\State \textbf{Models: } predictor $g(x;\theta)$, discriminators $\{f^{(e)}(x;\phi^{(e)})\}_{e\in \mathcal{E}}$, gate $w$. 
\State \textbf{Input: } data $\{\mathcal{D}^{(e)}\}_{e\in \mathcal{E}}$ with $\mathcal{D}^{(e)} = \{(x_i^{(e)}, y_i^{(e)})\}_{i=1}^n$ from $|\mathcal{E}|$ environments, loss function $\ell(\cdot, \cdot)$.
\State \textbf{Output: } Parameters of the prediction model: $w$ and $\theta$
\State Initialize $\theta, \{\phi^{(e)}\}_{e\in \mathcal{E}}$ with random weights, $w = 0$.
\For{$t \in \{1,\ldots, T\}$}
    \State Set $\tau_t = \tau_0 \times (\tau_T/\tau_0)^{t/T}$
    \For{$t_f \in \{1,\ldots, T_f\}$} \Comment{Discriminator Ascent}
        \State Sample $\{u_{j}\}_{j=1}^d=\{(u_{j,1},u_{j,2})\}_{j=1}^d$ from $\mathrm{Gumbel}(0,1)$.
        \State Calculate $a=(a_1,\ldots, a_d)$ with $a_j = V_{\tau_t}(w_j, u_{j})$, where $V(\cdot)$ is defined in \eqref{fan6}. 
        \For{$e\in \mathcal{E}$} \Comment{Update $f^{(e)}$}
        \State Sample minibatch of $m$ examples $\{(x^{(e, i)},y^{(e,i)})\}_{i=1}^m$ from $\mathcal{D}^{(e)}$.
        \State Update the discriminator by ascending its stochastic gradient:
        \[
        \nabla_{\phi^{(e)}} \frac{\gamma}{m} \sum_{i=1}^{m} \left[ \{y^{(e,i)} - g(x^{(e,i)})\}f_{\phi^{(e)}}(x^{(e,i)}) - \frac{1}{2} \{f_{\phi^{(e)}}(\phi^{(e)})\}^2\right]
        \] ~~~~~~~~~~~~~ where $g(x) = g(a(w) \odot x; \theta) ~ \text{and} ~ f_{\phi^{(e)}}(x) = f(a(w) \odot x; \phi^{(e)})$

        \EndFor
    \EndFor
    \For{$t_g \in \{1,\ldots, T_g\}$} \Comment{Predictor Descent}
        \State Sample $\{u_{j}\}_{j=1}^d=\{(u_{j,1},u_{j,2})\}_{j=1}^d$ from $\mathrm{Gumbel}(0,1)$.
        \State Calculate $a=(a_1,\ldots, a_d)$ with $a_j = V_{\tau_t}(w_j, u_{j})$, where $V(\cdot)$ is defined in \eqref{fan6}. 
        \For{$e\in \mathcal{E}$} \Comment{Enumerate Environments}
            \State Sample minibatch of $m$ examples $\{(x^{(e, i)},y^{(e,i)})\}_{i=1}^m$ from $\mathcal{D}^{(e)}$.
            \State Calculate loss as function of $\theta$ and $w$, that is
            \begin{align*}
                L^{(e)}(\theta, w) &= \frac{\gamma}{m} \sum_{i=1}^{m} \left[ \{y^{(e,i)} - g_{w, \theta}(x^{(e,i)})\}f_{w}(x^{(e,i)}) - \frac{1}{2} \{f_{w}(x^{(e, i)})\}^2\right]  \\
                &~~~~~~ + \frac{1}{m} \sum_{i=1}^m \left[\ell\left(y^{(e, i)}, g_{w, \theta}( x^{(e,i)})\right)\right]
            \end{align*} ~~~~~~~~~~~~~ where $g_{w, \theta}(x) = g(a(w) \odot x; \theta) ~ \text{and} ~ f_{w}(x) = f(a(w) \odot x; \phi^{(e)})$
        \EndFor
        \State Update the predictor weights $w, \theta$ by descending its stochastic gradient:
    \[
    \nabla_{(\theta, w)} \sum_{e\in \mathcal{E}} L^{(e)}(\theta, w)
    \]
    \EndFor
\EndFor
\end{algorithmic}
\label{algo-sgd}

\end{algorithm}

See \cref{algo-sgd}.

\subsection{Finite Performance of FAIR-NN Estimator}
\label{sec:fairnn-simulation}

\noindent \textbf{Data Generating Process.} We consider the following data generating process with $d=26$ and $|\mathcal{E}|=2$ in each trial as
\begin{align*}
    X_i^{(e)} &\gets \begin{cases}
        \varepsilon_i^{(e)} &\qquad i \le 5 \\
        f_{i,0}^{(e)}(Y^{(e)}) + \varepsilon_i^{(e)} & \qquad 6 \le i \le 9 \\
        \sum_{j\in \mathtt{pa}(i) \subseteq [8]} f_{i,j}^{(e)}(X_j^{(e)}) + \varepsilon_i^{(e)} &\qquad 10 \le i \le 26
    \end{cases} \\
    Y^{(e)} &\gets m_k^\star(X_1^{(e)},\ldots, X_5^{(e)}) + \varepsilon_0,
\end{align*} where the regression function $m^\star$ is either $m^\star_1(x)=\sum_{k=1}^5 m_{0,j}(x_j)$ with random chosen $m_{0,j}$ or a hierarchical composition model $m^\star_2(x) = x_1 x_2^3 + \log(1 + e^{\tanh(x_3)} + e^{x_4}) + \sin(x_5)$; see detailed model and omitted implementation details in \cref{sec:nonlinear-simu-append}. In the two environments, the cause-effect relationships are shared. The variable $Y$'s parent set is $\{1,2,3,4,5\}$, its children set is $\{6,7,8,9\}$, and may have potential descendants in $\{9,\ldots, 26\}$.  The above data generating process can be regarded as one observation environment $e=0$ and an interventional environment $e=1$ where the random and simultaneous interventions are applied to all the variables other than the variable $Y$, while the assignment from $Y$'s parent to $Y$ remains and furnishes the target regression function $m_k^\star(x)$ with $k\in \{1,2\}$ in pursuit. \cref{fig:visualize-nonlinear} (a) visualizes the induced graph in one trial. 

\begin{figure}[!t]
\centering
\begin{tabular}{cc}
\subfigure[]{
\includegraphics[scale=0.35]{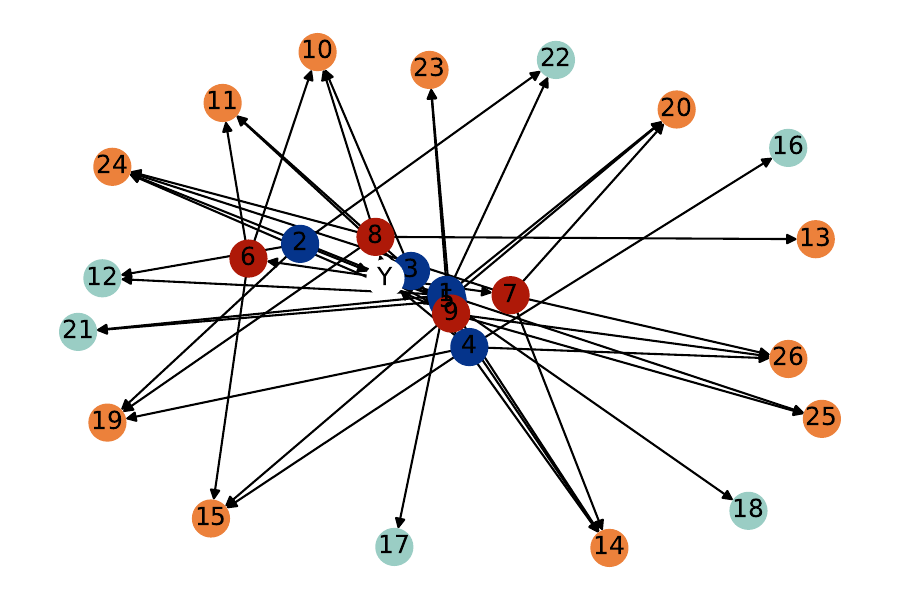}
}&
\subfigure[]{
\includegraphics[scale=0.35]{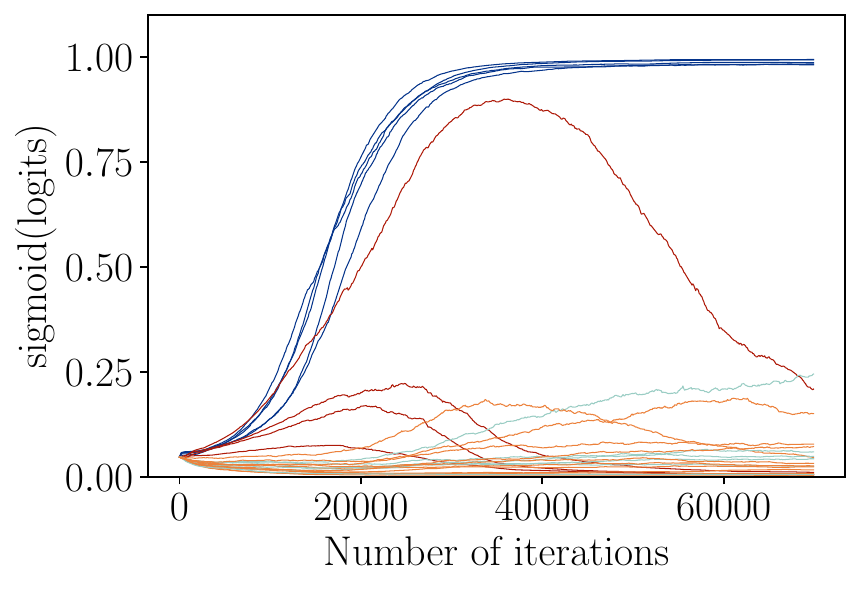}
}
\end{tabular}
\caption{The visualization of (a) the SCM and (b) the $\sig{w}$ during training in one trial for FAIR-NN estimator when $k=1$. We use different colors to represent the different relationships with $Y$: \myblue{blue} = parent, \myred{red} = child, \myorange{orange} = offspring, \mylightblue{lightblue} = other. 
}
\label{fig:visualize-nonlinear}
\end{figure}

\medskip
\noindent \textbf{Implementation.} We let $\mathcal{G}$ be the class of ReLU neural network with depth $2$ and width $128$ and $\mathcal{F}$ be the class of ReLU neural network with depth $2$ and width $196$, and run gradient descent ascent using similar experimental configurations. We use the following empirical mean squared square computed using another $2\times n_{\mathrm{test}}=2\times 30000$ i.i.d. sampled data 
\begin{align*}
    \hat{\mathtt{MSE}} = \frac{1}{2n_{\mathrm{test}}}\sum_{e\in \mathcal{E}} \sum_{i=1}^{n_{\mathrm{test}}} \{m^\star(x^{(e)}_i) - \hat{m}(x^{(e)}_i)\}^2
\end{align*} as the evaluation metric. We report the median of $\hat{\mathtt{MSE}}$ over $100$ replications for the estimators (1) -- (4) akin to that for the linear model. For (1), (2), and (4), we also use a ReLU neural network with depth $2$ and width $128$ in running least squares. \cref{fig:visualize-nonlinear} (b) also visualizes how the Gumbel gate values for different covariates $\sig{w}$ evolve during training in one trial. We can see that the training dynamics for $\sig{w}$ is much more challenging and interesting than that for the linear model depicted in \cref{fig:visualize-linear}: the weight for some $Y$'s children quickly increases at a comparable rate than the variables in $S^\star$ at the beginning, but such a trend slows down and finally completely reverses in the middle. We leave the rigorous and in-depth analysis behind such dynamics for future studies.

\begin{figure}[!t]
\centering
\begin{tabular}{cc}
\subfigure[]{
\includegraphics[scale=0.4]{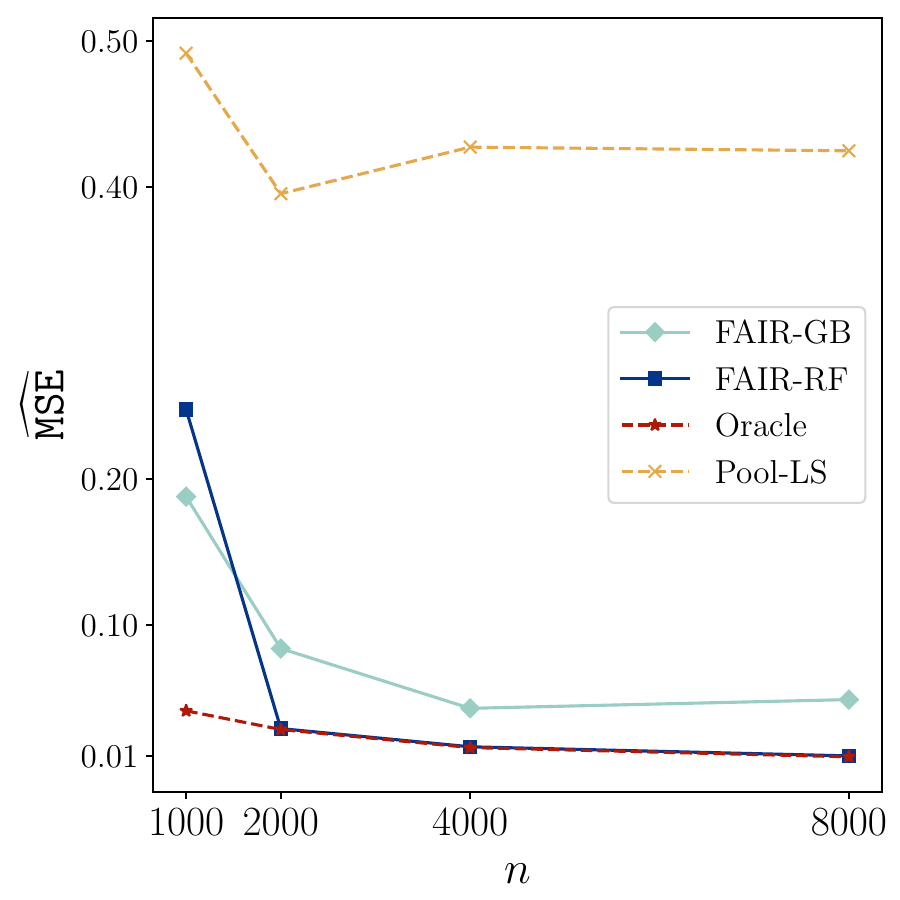}
}&
\subfigure[]{
\includegraphics[scale=0.4]{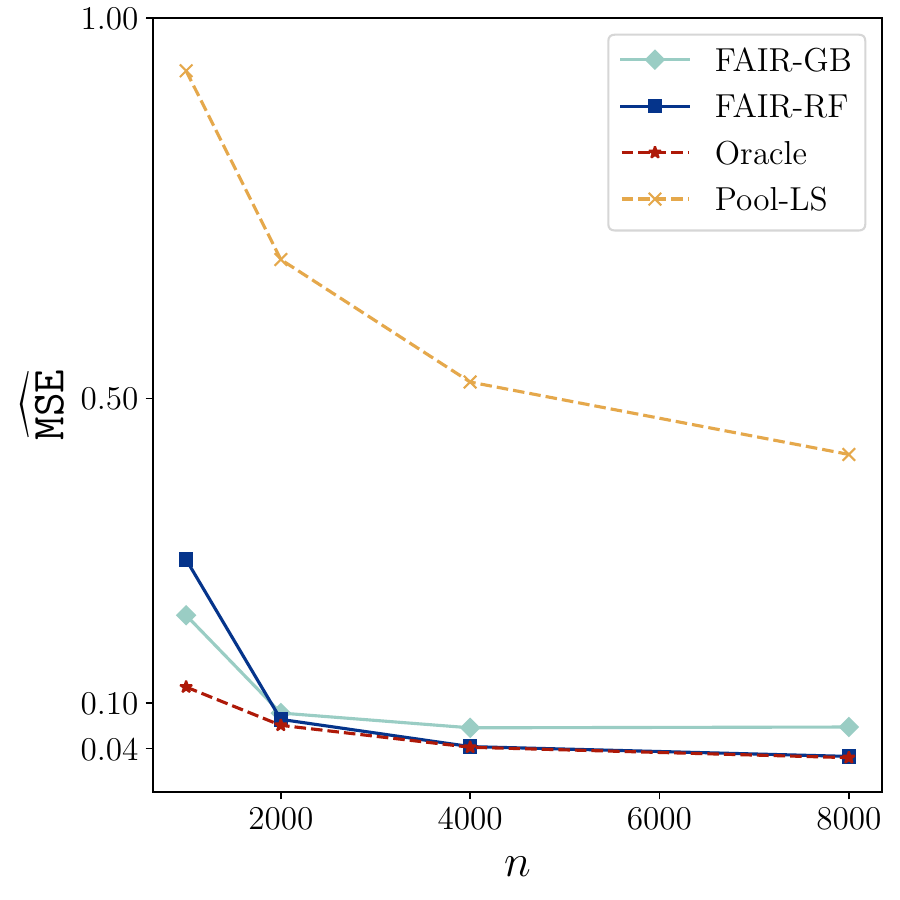}
}
\end{tabular}
\caption{The simulation results for nonlinear models with (a) $m_1^\star$ and (b) $m_2^\star$. Both figures depict how the median estimation errors (based on $50$ replications) for different estimators (marked with different shapes and colors) change when $n$ varies in $\{1000, 2000, 3000, 5000\}$ for (a) and $\{1000, 2000, 3000, 5000, 10000\}$ for (b). 
}
\label{fig:simulation-result-nonlinear}
\end{figure}

\medskip
\noindent \textbf{Results.} The results are shown in \cref{fig:simulation-result-nonlinear} and the messages are similar to those for FAIR-Linear estimators. The pooled least squares yield biased estimation, while our proposed FAIR-NN estimator can unveil the invariant association $m^\star$ from the two environments. Moreover, the refitted FAIR-NN estimator can obtain a near-oracle performance when $n$ is large.

\subsection{Application II: Prediction Based on Extracted Features}
\label{sec:app-bird}

We consider an image object classification task with a spurious background. The target is to classify water birds ($Y=1$) and land birds ($Y=0$) (see examples in \cref{fig:app2} (a)) under backgrounds of water or land based on the feature $X\in \mathbb{R}^{500}$ extracted from ResNet pre-trained on ImageNet. We train a linear classifier on top of $X$ using data from two environments $\{\mathcal{D}_k\}_{k=1}^2$. In the first environment $\mathcal{D}_1$, $r_w=95\%$ water birds appear on the water background and $r_l=90\%$ land birds stay in the land background. The spurious correlation numbers are $r_w=75\%$ and $r_l=70\%$ in $\mathcal{D}_2$. A good predictor should be based on the core features related to the bird's appearance rather than the strong spurious correlation between the background and label. The trained model is evaluated in a test environment $\mathcal{D}_{3}$ where the spurious correlation reverses: $r_w=2\%$ and $r_l=2\%$. We repeat the experiment $10$ times, where in each trial the training dataset are sampled from a larger dataset with sizes $n=|\mathcal{D}_1|=|\mathcal{D}_2|=30k$, while the testing dataset are fixed with size $|\mathcal{D}_{3}|=30k$. We compare our FAIR estimator using $\mathcal{G}=\{\sig{\beta^\top x}\}, \mathcal{F}=\{\beta^\top x\}$ and classification loss $\ell(y, v) = -\log(1-v) - y\log\{v/(1-v)\}$ ({FAIR-GB}) with invariant risk minimization ({IRM}) \citep{arjovsky2019invariant} and group distributionally robust optimization ({GroupDRO}) \citep{sagawa2019distributionally}. We also consider running Lasso on different environments for reference, including (1) using all the data $\mathcal{D}_1\cup \mathcal{D}_2$ ({Pooled Lasso}); (2) using data in $\mathcal{D}_2$ ({Lasso on D2}); (3) using another randomized controlled environment $\mathcal{D}_4$ with $r_w=r_l=50\%$ and $|\mathcal{D}_4|=n$ ({Oracle}). All the models are linear, and the performance of (3) can be seen as the upper bound of the performance using linear models; see data collection and experiential configuration details in {\cref{app:waterbird}}.
\begin{figure}
\centering
\subfigure[]{
\includegraphics[scale=0.23]{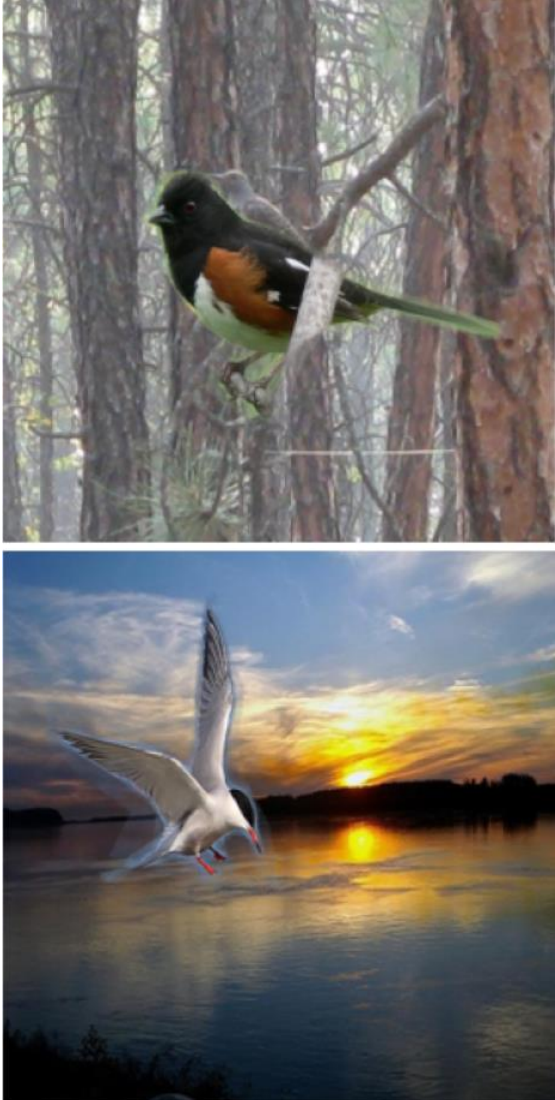}
} 
\subfigure[]{
\begin{tikzpicture}
\node at (0,0.2) {\textbf{Method}};
\node at (2.5,0.2) {\textbf{Test Accuracy}};

\node at (0,-0.5) {Oracle};
\node at (2.5,-0.5) {91.06 $\pm$ 0.24 \%};

\node at (0,-1) {Lasso on D2};
\node at (2.5,-1) {84.57 $\pm$ 0.71 \%};

\node at (0,-1.5) {Pooled Lasso};
\node at (2.5,-1.5) {79.08 $\pm$ 0.54 \%};

\node at (0,-2.5) {IRM};
\node at (2.5,-2.5) {80.32 $\pm$ 0.67 \%};

\node at (0,-3) {GroupDRO};
\node at (2.5,-3) {82.75 $\pm$ 1.10 \%};

\node at (0,-3.5) {FAIR-GB};
\node at (2.5,-3.5) {89.56 $\pm$ 0.53 \%};

\end{tikzpicture}
}
\subfigure[]{
\includegraphics[scale=0.48]{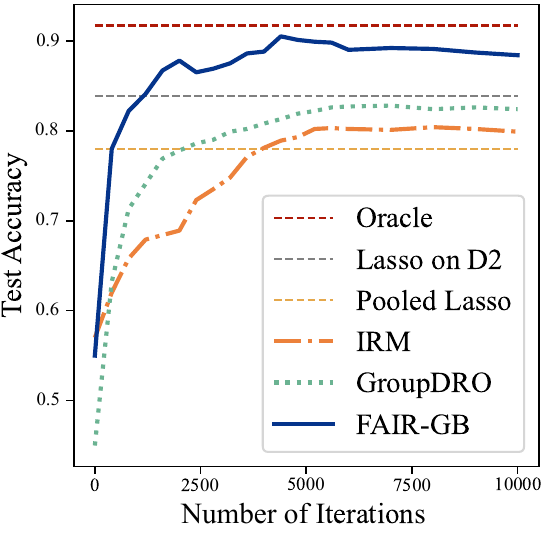}
}
\caption{Prediction Based on Extracted Features: (a) provides two sample images in the dataset: land bird on land (up) and water bird in water (bottom). (b) reports the average $\pm$ standard deviation of test accuracy over 10 trials for different estimators. (c) shows how the test accuracy changes over iterations for IRM, GroupDRO, and our FAIR estimator in one trial, where the number of iterations refers to the number of gradient descent updates.} 
\label{fig:app2}
\end{figure}
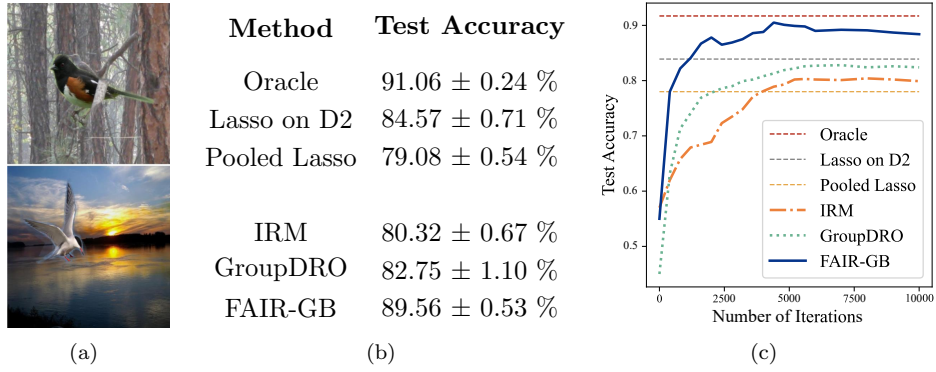

The performances are reported in \cref{fig:app2} (b). \cref{fig:app2} (c) also depicts how test accuracy changes as iterations, namely, the number of gradient descent updates, in one trial. {FAIR-GB} performs similar to {Oracle} and significantly outperforms {Lasso on D2}, while other methods ({IRM}, {GroupDRO}) falls behind { Lasso on D2}. This indicates that these methods cannot go beyond interpolating spurious associations in $\mathcal{D}_1$ and $\mathcal{D}_2$, while our method can nearly eliminate the spurious signals using the relatively small perturbations in the two environments.

\subsection{Detailed Simulation Configuration}

\subsubsection{Linear Model with d=15}
\label{sec:linear-simu-append}

\paragraph{Data Generating Process.} The data-generating process is similar to that described in \cref{sec:simu-linear}. We also let $|\mathcal{E}| = 2$, and use the same procedure to generate parent-children relationship and structural assignment except that (1) we use $d=15$ and let the variable $Z_{8}$ be $Y$; and (2) we enforce that $Y$ has at least $3$ parents and $3$ children (3) the structural assignment for variable $Y$ is
\begin{align*}
    Y^{(8)} = Z^{(8)} \gets \sum_{k\in \mathtt{pa}(8)} C_{8,k} Z_k^{(e)} + C_{8,8} \varepsilon_8,
\end{align*} that is we let the variance noise to be the same for the two environments. This is because we will include ICP in our simulation comparisons, which requires conditional distribution invariance. 

\paragraph{Implementation.} We use the same configurations in the implementation of FAIR-GB and FAIR-RF. We also use fixed $\gamma=36$ for all the FAIR family estimators including EILLS. It is worth noticing that ICP, anchor regression, and IRM introduce an additional hyper-parameter, we pick it in an oracle way for them: that is, we enumerate all the candidate hyper-parameters and select the one that minimizes the $L_2$ estimation error. We report the performance for $n\in \{100, 200, 500, 800, 1000\}$.

\paragraph{Discussion of Results.} For anchor regression and IRM, their performance and the corresponding relationships w.r.t. Pool-LS are similar to the 12 variable illustrations in \cite{fan2024environment}. The anchor regression is almost the same as Pool-LS because it is essentially the same as standard least squares when the environments are discrete: indeed, in $|\mathcal{E}|=2$, it just runs least squares with a difference intercept for the interventional environment $e=1$. The IRM is better than vanilla least squares by slightly decreasing the bias, while the performance improvement is negligible compared with the bias it has. 

For ICP, the performance is even worse than pooled least squares because it collapses to conservative solutions like $0$. Note that we apply interventions to all the variables in environment $e=1$, under which it is possible for ICP to identify $\beta^\star$ and $S^\star$ when $n=\infty$. The large estimation error it depicts is due to its inefficiency in estimation. 

We can also see that the performance of FAIR-BF and FAIR-RF are similar, demonstrating the effectiveness of our proposed gradient descent ascent algorithm with Gumbel approximation. The performance of FAIR-GF and FAIR-RF is slightly better than EILLS. This is because the FAIR estimator is essentially doing the most efficient pooled least squares when it selects the correct variable. 

\subsubsection{Nonlinear Model}

\label{sec:nonlinear-simu-append}

\paragraph{Data Generating Process. }

For the structural assignment, we let $\varepsilon_i^{(e)} = \varepsilon_i$ for $i\le 5$ and $\varepsilon_i^{(e)} = C_{i,i}^{(e)} \varepsilon_i$ where $(\varepsilon_1,\ldots, \varepsilon_{26})$ are independent $\mathrm{Uniform}([-1.5, 1.5])$ random variables to let the covariates to be uniformly bounded and $C_{i,i}^{(e)}$ are scalars that are randomly generated in each trial. $\varepsilon_0$ is standard normal distributed that is independent of $(\varepsilon_1,\ldots, \varepsilon_{26})$. 

For the assignments for the children of $Y$, we let $f_{i,0}^{(e)}(u) = C_{i,0}^{(e)} \tanh(u)$, where $C_{i,0}^{(e)}$ are scalars that is randomly sampled from $\mathrm{Uniform}([-1.5,1.5])$ for $e=0$ and $\mathrm{Uniform}([-5,5])$ for $e=1$, the noise level $C_{i,i}^{(e)}$ is a scalar generated from $\mathrm{Uniform}([1,1.5])$. For the assignments for other variables $X_i$ with $i\ge 10$, we let $f_{i,j}^{(e)}(u) = C_{i,j}^{(e)} h_{i,j}^{(e)}(u)$ where $h_{i,j}^{(e)}$ are randomly picked from the function set $\{\tanh(x), \sin(x), \cos(x)\}$, the noise level $C_{i,i}^{(e)}$ is a scalar generated from $\mathrm{Uniform}([2, 3])$. For $m_1^\star$, it is $\sum_{k=1}^5 f_{0,j}(x)$ with $f_{0,j}(x)$ randomly picked from $\{\tanh(x), \sin(x), \max(0, x), x\}$. 

\paragraph{Implementation.} For the FAIR-NN implementation using Gumbel approximation, we also run gradient descent ascent using the Adam optimizer using a learning rate of 1e-3, batch size $64$. The number of iterations is $70k$ for $m_1^\star$ and $80k$ for $m_2^\star$. In each iteration, one gradient descent update of the neural network parameters in $g$ and the Gumbel logits parameter $w$ is conducted followed by three gradient ascent updates of the neural network parameters in $f^{(0)}$ and $f^{(1)}$. We also use fixed $\gamma = 36$. The implementation details for the estimators are:
\begin{itemize}[itemsep=0pt]
    \item[(1)] Pool-LS: it simply runs least squares on the full covariate $X$ using all the data.
    \item[(2)] FAIR-GB: Our FAIR-NN estimator with Gumbel approximation, its prediction on the test dataset is evaluated by averaging the predictions over $100$ Gumbel samples. 
    \item[(3)] FAIR-RF: it first selects the variables $x_j$ in the fitted model in (2) with $\sig{w_j}>t$, i.e., $\hat{S}=\{j: \sig{w_j}>t\}$, and runs least squares again on $X_{\hat{S}}$ using all the data. Here we let $t=0.6$ for $n\le 2000$ and $t=0.9$ for $n>2000$. 
    \item[(4)] Oracle: it runs least squares on $X_{S^\star}$ using all the data.
\end{itemize}

For FAIR-GB, we report the estimated MSE for the model in the last iteration. For other estimators, we also run gradient descent using the Adam optimizer for 10k iterations. We report the estimated MSE for the model with early stopping regularization: that is, we report the estimated MSE of the model that has the smallest validation error, and the validation data is sampled independently and identically to the training data with sample size $n_{\mathrm{valid}} = \lfloor 3n/ 7\rfloor$. 

\subsection{Details of the Discovery in Real Physical System Application}
\label{appendix:app1}

\noindent \textbf{Data Collection.} We directly use the dataset `lt\_interventions\_standard\_v1' released in \cite{gamella2024causal}. 

For the training dataset, given fixed sample size $n$, the data in the first environment $e=0$ is sampled from the experimental setting `uniform\_reference'. For the second environment $e=1$, a mixture of interventions is applied. To be specific, a weak intervention on the variables $\tilde{V}_3, \tilde{V}_1, \tilde{V}_2, \tilde{I}_1, \tilde{I}_2$ with probability $(1/3, 1/6, 1/6, 1/6, 1/6)$, respectively. This is equivalent to sample data from the experimental setting `t\_vis\_3\_weak', `t\_vis\_1\_weak', `t\_vis\_2\_weak,', `t\_ir\_1\_weak', `t\_ir\_2\_weak' with weights $(1/3, 1/6, 1/6, 1/6, 1/6)$. 

For the test data used for evaluation in \cref{fig:app1} (b)--(c), we use the data from the experimental setting `t\_vis\_3\_strong', `t\_vis\_1\_strong', `t\_vis\_2\_strong', `t\_ir\_1\_strong', `t\_ir\_2\_strong'. Since there is an out-of-support issue for the intervention, i.e.,
\begin{align*}
    | \mathrm{Mean}_{\mu_{X,i}}(X) - \mathrm{Mean}_{\bar{\mu}_n}(X) | > 1.6 \cdot \mathrm{Std}_{\bar{\mu}_n}(X)
\end{align*} where $\mu_{X,i}$ is the empirical distribution of $X$ in the experimental setting where strong intervention is intervened on $X$, and $\bar{\mu}_n$ is the empirical distribution of $X$ in the training dataset. Thus, we recenter the variable $X$ in the corresponding test intervention environment such that it has the same empirical mean as that in the training dataset. 

\medskip
\noindent \textbf{Explanation on the Equivalent Graph.} We regress $\tilde{I}_3$ on $(R, G, B, \theta_1, \theta_2, \tilde{V}_1, \tilde{V}_2, \tilde{V}_3, \tilde{I}_1, \tilde{I}_2)$. There are several hidden confounders, hence there should be an arrow from $\tilde{V}_3$ to $\tilde{I}_3$ and an arrow from $\tilde{I}_3$ to $\tilde{V}_3$ if $\tilde{V}_3$ is not intervened given the existence of hidden confounders $(L_{3,1}, L_{3,2})$. Introducing the variable $\tilde{V}_3$ in predicting $\tilde{I}_3$ can increase the predictive power given it can provide additional information of $(L_{3,1}, L_{3,2})$. The (equivalent) arrow from $\tilde{V}_3$ to $\tilde{I}_3$ do disappear because of the intervention on $\tilde{V}_3$ will make the association perturbs. 

\medskip
\noindent \textbf{Experimental Setup.} For the FAIR-NN implementation using Gumbel approximation, we also run gradient descent ascent using the Adam optimizer using a learning rate of 1e-3, batch size $64$. The number of iterations is $100k$. In each iteration, one gradient descent update of the neural network parameters in $g$ and the Gumbel logits parameter $w$ is conducted followed by three gradient ascent updates of the neural network parameters in $f^{(0)}$ and $f^{(1)}$. We also use fixed $\gamma = 36$. The neural network architectures for all the estimators are the same and are the same as in the simulation of FAIR-NN. The implementation details for all the estimators are:
\begin{itemize}[itemsep=0pt]
    \item[(1)] Pooled-NN: it simply runs least squares on the full covariate $X$ using all the data.
    \item[(2)] FAIR-NN-GB: Our FAIR-NN estimator with Gumbel approximation, its prediction on the test dataset is evaluated by averaging the predictions over $100$ Gumbel samples.
    \item[(3)] FAIR-NN-RF: it first selects the variables $x_j$ in the fitted model in (2) with $\sig{w_j}>0.9$, i.e., $\hat{S}=\{j: \sig{w_j}>t\}$, and runs least squares again on $X_{\hat{S}}$ using all the data.
    \item[(4)] Oracle-NN: it runs least squares on $X_{S^\star}$ using all the data and neural networks.
    \item[(5)] Oracle-Linear: it runs least squares on $X_{S^\star}$ using all the data and linear model.
\end{itemize}
The out-of-sample $R^2$ for all the estimators is reported based on the model selection using the validation set that is sampled from the same source as training data with sample size $n'=0.6 n$. Such a model selection is adopted to prevent the model from over-fitting.

\subsection{Details of the Prediction Based on Extracted Features}\label{app:waterbird}

We generate datasets by combining the bird images in the CUB dataset \citep{wah2011caltech} and the background images in the Places dataset \citep{zhou2017places} using specific probabilities, which is similar to the waterbird setting in \cite{sagawa2019distributionally} except the spurious correlation ratio. In each environment, there are $50\%$ water birds and $50\%$ land birds. The probabilities of each environment are as follows:

\begin{itemize}
\item[(a)]
Environment-1. We place $95\%$ of all water birds against a water background, with the remaining $5\%$ against a land background. We place $90\%$ of all land birds against a land background, with the remaining $10\%$ against a water background. The dataset is denoted by $\bar{\mathcal{D}}_1$, with $50$k images.

\item[(b)]
Environment-2. We place $75\%$ of all waterbirds against a water background, with the remaining $25\%$ against a land background. We place $70\%$ of all landbirds against a land background, with the remaining $30\%$ against a water background. The dataset is denoted by $\bar{\mathcal{D}}_2$, with $50$k images.

\item[(c)]
Environment-3 (Test Environment). We only place $2\%$ of all waterbirds against a water background, with the remaining $98\%$ against a land background. We place $2\%$ of all landbirds against a land background, with the remaining $98\%$ against a water background. The dataset is denoted by $\mathcal{D}_3$, with $30$k images.

\item[(d)]
Environment-4 (Oracle Environment). We place $50\%$ of all waterbirds against a water background, with the remaining $50\%$ against a land background. We place $50\%$ of all landbirds against a land background, with the remaining $50\%$ against a water background. The dataset is denoted by $\mathcal{D}_{4}$, with $30$k images.
\end{itemize}

\medskip
\noindent \textbf{Class Identification.} We apply the CUB dataset \cite{wah2011caltech}, which contains images of birds, along with pixel-level segmentation masks for each bird. When generating the dataset, we classify each bird into waterbird if it belongs to the seabird or waterfowl categories (e.g., albatross, auklet, cormorant, frigatebird, fulmar, gull, jaeger, kittiwake, pelican, puffin, tern, gadwall, grebe, mallard, merganser, guillemot, or Pacific loon) and the land birds if it does not belong to the seabird or waterfowl categories.

\medskip
\noindent \textbf{Image Generation.} When picking bird images from the CUB dataset, we use the provided pixel-level segmentation masks to crop each bird from its original background. Then we decide which environment they should be placed in and select either a water background like ocean and lake or a land background like bamboo forest and broadleaf forest sourced from the Places dataset \cite{zhou2017places}. We randomly select $70\%$ of the images in the CUB dataset as a training set and the remaining $30\%$ as a testing set and generate our dataset for training and testing based on the split CUB dataset.

\medskip
\noindent \textbf{Feature Extraction.} Based on the dataset, we use the Pytorch torchvision implementation of the ResNet50 model \cite{he2016deep} with the pre-trained weights to extract the feature of the images, obtaining a dataset of the feature vector of 2048 dimensions. Then we apply principal components analysis (PCA) to reduce the dimensions of the feature vector to $500$ based on the whole training data $\bar{\mathcal{D}}_1$ and $\bar{\mathcal{D}}_2$. We apply the same dimensionality reduction transformation to data in other environments.

\medskip
\noindent \textbf{Experiment Setup.} We run FAIR-Linear with Gumbel approximation on the dataset. Following the standard setting, we apply the logistic loss and Adam optimizer using a learning rate of $1e-2$, weight decay of $1e-4$, and batch size $4096$ for $10000$ iterations. In each iteration, one gradient descent update of the neural network parameters in $g$ and the Gumbel logits parameter $\omega$ is conducted based on $5$ gradient ascent updates of the neural network parameters in $f$. We fix $\gamma$ as $200$. For any methods that use data in $\bar{\mathcal{D}}_1$ and $\bar{\mathcal{D}}_2$, it is trained in a random subset $\mathcal{D}_k \in \bar{\mathcal{D}}_k$ with $|\mathcal{D}_k|=30k$. The implementation details for all the estimators are:

(1) Oracle: it runs logistic regression with $\ell_1$ penalty and penalty weight $\alpha=0.001$ on the oracle environment $\mathcal{D}_3$ for $1000$ iterations.

(2) Pooled Lasso: it runs logistic regression with $\ell_1$ penalty and penalty weight $\alpha=0.001$ on $\mathcal{D}_1$ and $\mathcal{D}_2$ for $1000$ iterations.

(3) Lasso on D2: it runs logistic regression with $\ell_1$ penalty and penalty weight $\alpha=0.001$ on the $\mathcal{D}_2$ for $1000$ iterations.

(4) FAIR-GB: Our FAIR-Linear estimator with Gumbel approximation trained on $\mathcal{D}_1$ and $\mathcal{D}_2$ for $10000$ iterations.

(5) IRM: it runs Invariant Risk Minimization (IRM) trained on $\mathcal{D}_1$ and $\mathcal{D}_2$ with $\ell_2$ regularizer weight $0.001$ and penalty weight $100$ for $10000$ iterations.

(6) GroupDRO: it runs Group Distributionally Robust Optimization (Group-DRO) on $\mathcal{D}_1$ and $\mathcal{D}_2$ using ResNet50 and $\gamma=0.1$ for $10000$ iterations.



\section{Proofs for Main Results in Appendix \ref{sec:main-result} and Theorem \ref{thm:oracle-2}}
\label{sec:proof-main-result}

\subsection{Proof of the Faster L2 Error Bound and Variable Selection Result}

In this section, we prove the faster $L_2$ error bound and variable selection result in \eqref{eq:main-result-rate}. The proof proceeds conditioned on that both events in \cref{prop:nonasymptotic-a} and \ref{prop:nonasymptotic-pooled} occurs, and also the event $\mathcal{A}_+$ defined in the proof of \eqref{eq:main-result-variable-sel} occurs such that the selected set $\hat{S}$ satisfies
\begin{align}
\label{eq:proof-main-rate-eq1}
\forall e\in \mathcal{E}, \qquad \Pi_{\overline{\mathcal{F}_{\hat{S}}}}^{(e)}(m^{(e,\hat{S})}) = g^\star.
\end{align}

Note that \cref{prop:variable-selection} explicitly establishes that $\hat{S} \supseteq S^\star$ and $\bar{\mathsf{d}}_{\mathcal{G},\mathcal{F}}(\hat{S}) = 0$, which concludes the proof of the variable selection result in \eqref{eq:main-result-rate}.

Now we use \cref{thm:population} and \cref{prop:nonasymptotic-pooled}, \ref{prop:nonasymptotic-a} in a different way. On one hand, we apply \cref{thm:population} with $\delta=0.5$, $g, f^{\mathcal{E}} = \hat{g}, \hat{f}^{\mathcal{E}}$ and $\tilde{g}, \tilde{f}^{\mathcal{E}}$ to be determined as
\begin{align*}
    &\mathsf{Q}_\gamma(\hat{g}, \hat{f}^{\mathcal{E}}) - \mathsf{Q}_\gamma(\tilde{g}, \tilde{f}^{\mathcal{E}}) \\
    &~~~~\ge 0.25 \|\hat{g} - \tilde{g}\|_2^2 +  \frac{\gamma}{4} \bar{\mathsf{d}}_{\mathcal{G}, \mathcal{F}}(S) + \frac{\gamma}{2} \|\hat{g} - \Pi_{\overline{\mathcal{G}_{\hat{S}}}}(\bar{m}^{(\hat{S})})\|_2^2\\
    &~~~~~~~~~~ - \frac{\gamma}{2|\mathcal{E}|} \sum_{e\in \mathcal{E}}  \|\hat{f}^{(e)} - \{\Pi_{\overline{\mathcal{F}_{\hat{S}}}}^{(e)} (m^{(e,\hat{S})}) - \hat{g}\}\|_{2,e}^2 - \frac{4 + \gamma}{2} \|\tilde{g} - g^\star\|_2^2  \\
    &~~~~= 0.25 \|\hat{g} - \tilde{g}\|_2^2 + 0.5 \gamma \|\hat{g} - g^\star\|_2^2 \\
    &~~~~~~~~~~ - \frac{\gamma}{2|\mathcal{E}|} \sum_{e\in \mathcal{E}}  \|\hat{f}^{(e)} - \{g^\star - \hat{g}\}\|_{2,e}^2 - \frac{4 + \gamma}{2} \|\tilde{g} - g^\star\|_2^2 
\end{align*} where the last identity follows from \eqref{eq:proof-main-rate-eq1} and the fact that $\|\Pi_{\overline{\mathcal{G}_{\hat{S}}}}(\bar{m}^{(\hat{S})}) - g^\star\|_{2} = 0$ because $\mathsf{b}_{\hat{S}}=0$. Combing with the fact that 
\begin{align*}
    \|\hat{g} - g^\star\|_2^2 = \|\hat{g} - \tilde{g} + \tilde{g} - g^\star\|_2^2 \ge 0.5 \|\hat{g} - \tilde{g}\|_{2}^2 - \|\tilde{g} - g^\star\|_2^2,
\end{align*} we obtain
\begin{align}
\label{eq:proof-main-rate-lb}
\begin{split}
\mathsf{Q}_\gamma(\hat{g}, \hat{f}^{\mathcal{E}}) - \mathsf{Q}_\gamma(\tilde{g}, \tilde{f}^{\mathcal{E}}) &\ge 0.25 (1 + \gamma) \|\hat{g} - \tilde{g}\|_2^2 - 2(1 + \gamma) \|\tilde{g} - g^\star\|_2^2 \\
&~~~~~~~~~~ - \frac{\gamma}{2|\mathcal{E}|} \sum_{e\in \mathcal{E}}  \|\hat{f}^{(e)} - \{g^\star - \hat{g}\}\|_{2,e}^2.
\end{split}
\end{align}

On the other hand, we also have
\begin{align}
    \mathsf{Q}_\gamma(\hat{g}, \hat{f}^{\mathcal{E}}) - \mathsf{Q}_\gamma(\tilde{g}, \tilde{f}^{\mathcal{E}}) &= \mathsf{Q}_\gamma(\hat{g}, \hat{f}^{\mathcal{E}}) - \hat{\mathsf{Q}}_\gamma(\hat{g}, \hat{f}^{\mathcal{E}}) + \hat{\mathsf{Q}}_\gamma(\tilde{g}, \tilde{f}^{\mathcal{E}}) - \mathsf{Q}_\gamma(\tilde{g}, \tilde{f}^{\mathcal{E}}) \nonumber \\
    & ~~~~~~~~~~ + \hat{\mathsf{Q}}_\gamma(\hat{g}, \hat{f}^{\mathcal{E}}) - \hat{\mathsf{Q}}_\gamma(\tilde{g}, \tilde{f}^{\mathcal{E}}) \nonumber \\
    & = \hat{\mathsf{Q}}_\gamma(\hat{g}, \hat{f}^{\mathcal{E}}) - \hat{\mathsf{Q}}_\gamma(\tilde{g}, \tilde{f}^{\mathcal{E}}) - \Delta_{\mathsf{R}}(\hat{g}, \tilde{g}) \nonumber \\
    & ~~~~~~~~~~ + \frac{\gamma}{|\mathcal{E}|} \sum_{e\in \mathcal{E}} \Delta_{\mathsf{A}}(g, \tilde{g}, \hat{f}^{(e)}, \tilde{f}^{(e)}) \nonumber \\
    \begin{split}
    &\le \hat{\mathsf{Q}}_\gamma(\hat{g}, \hat{f}^{\mathcal{E}}) - \hat{\mathsf{Q}}_\gamma(\tilde{g}, \tilde{f}^{\mathcal{E}}) + |\Delta_{\mathsf{R}}(\hat{g}, \tilde{g})| \\
    & ~~~~~~~~~~ + \frac{\gamma}{|\mathcal{E}|} \sum_{e\in \mathcal{E}} |\Delta_{\mathsf{A}}(g, \tilde{g}, \hat{f}^{(e)}, \tilde{f}^{(e)})| \label{eq:proof-main-rate-eq2}
    \end{split}
\end{align} Now we choose $\tilde{g}$ such that 
\begin{align}
\label{eq:proof-main-rate-tilde-g-choice}
\|\tilde{g} - g^\star\|_{2}^2 \le \inf_{g\in \mathcal{G}_{S^\star}} \|g - g^\star\|_{2}^2 + \frac{1}{n} \le \delta^2_{\mathtt{a}, \mathcal{G}} + \frac{1}{n}
\end{align} and $\tilde{f}^{\mathcal{E}} = \tilde{f}^{\mathcal{E}}_{\tilde{g}}$ such that
\begin{align}
\label{eq:proof-main-rate-tilde-f-choice}
    \hat{\mathsf{Q}}_\gamma(\tilde{g}, \tilde{f}^{\mathcal{E}}) \ge \sup_{\breve{f}^{\mathcal{E}} \in \{\mathcal{F}_{S^\star}\}^{|\mathcal{E}|}} \hat{\mathsf{Q}}_\gamma(\tilde{g}, \breve{f}^{\mathcal{E}}) - (\gamma+1) \delta_{\mathtt{opt}}^2.
\end{align} Combining with the fact that our choice of $\hat{f}^{\mathcal{E}}$ satisfies $\hat{\mathsf{Q}}_\gamma(\hat{g}, \hat{f}^{\mathcal{E}}) \ge \sup_{f^{\mathcal{E}} \in \{\mathcal{F}_{\hat{S}}\}^{|\mathcal{E}|}} \hat{\mathsf{Q}}_\gamma(\hat{g}, f^{\mathcal{E}}) - (1+\gamma)\delta_{\mathtt{opt}}^2$, we find that
\begin{align}
\label{eq:proof-main-rate-opt-error}
\begin{split}
\hat{\mathsf{Q}}_\gamma(\hat{g}, \hat{f}^{\mathcal{E}}) - \hat{\mathsf{Q}}_\gamma(\tilde{g}, \tilde{f}^{\mathcal{E}}_{\tilde{g}}) &\le \inf_{g\in \mathcal{G}} \sup_{{f}^{\mathcal{E}} \in \{\mathcal{F}_{S_g}\}^{|\mathcal{E}|}} \hat{\mathsf{Q}}_\gamma({g}, {f}^{\mathcal{E}}) + (1+\gamma)\delta_{\mathtt{opt}}^2 - \hat{\mathsf{Q}}_\gamma(\tilde{g}, \tilde{f}^{\mathcal{E}}_{\tilde{g}}) \\ 
&\le \inf_{g\in \mathcal{G}} \sup_{{f}^{\mathcal{E}} \in \{\mathcal{F}_{S_g}\}^{|\mathcal{E}|}} \hat{\mathsf{Q}}_\gamma({g}, {f}^{\mathcal{E}}) - \sup_{\breve{f}^{\mathcal{E}} \in \{\mathcal{F}_{S^\star}\}^{|\mathcal{E}|}} \hat{\mathsf{Q}}_\gamma(\tilde{g}, \breve{f}^{\mathcal{E}}) + 2(1+\gamma)\delta_{\mathtt{opt}}^2 \le 2(1+\gamma)\delta_{\mathtt{opt}}^2.
\end{split}
\end{align} Here the first inequality follows from the optimization objective \eqref{eq:approx-solution}, the second inequality follows from \eqref{eq:proof-main-rate-tilde-f-choice}, and the last inequality follows from the definition of infimum. 

Putting \eqref{eq:proof-main-rate-opt-error} back into \eqref{eq:proof-main-rate-eq2} with our choice of $\tilde{g}$ and $\tilde{f}^{\mathcal{E}}$ in \eqref{eq:proof-main-rate-tilde-g-choice} and \eqref{eq:proof-main-rate-tilde-f-choice} gives
\begin{align*}
\mathsf{Q}_\gamma(\hat{g}, \hat{f}^{\mathcal{E}}) - \mathsf{Q}_\gamma(\tilde{g}, \tilde{f}_{\tilde{g}}^{\mathcal{E}}) \le 2(1+\gamma)\delta_{\mathtt{opt}}^2 + \frac{\gamma}{|\mathcal{E}|} \sum_{e\in \mathcal{E}} |\Delta_{\mathsf{A}}(g, \tilde{g}, \hat{f}^{(e)}, \tilde{f}^{(e)}_{\tilde{g}})| + |\Delta_{\mathsf{R}}(g,\tilde{g})|
\end{align*}

Applying the instance-dependent error bounds in \cref{prop:nonasymptotic-a} and \cref{prop:nonasymptotic-pooled} to the above inequality, we further have
\begin{align}
&\mathsf{Q}_\gamma(\hat{g}, \hat{f}^{\mathcal{E}}) - \mathsf{Q}_\gamma(\tilde{g}, \tilde{f}_{\tilde{g}}^{\mathcal{E}}) \nonumber\\
&~~~\le 2(1+\gamma)\delta_{\mathtt{opt}}^2 + \frac{\gamma}{|\mathcal{E}|} \sum_{e\in \mathcal{E}} |\Delta_{\mathsf{A}}(\hat{g}, \tilde{g}, \hat{f}^{(e)}, \tilde{f}^{(e)}_{\tilde{g}})| + |\Delta_{\mathsf{R}}(\hat{g},\tilde{g})| \nonumber \\
&~~~\le \gamma CU \frac{1}{|\mathcal{E}|}\sum_{e\in \mathcal{E}} \left\{\delta_{n,t} (\|\tilde{g} - g\|_{2,e} + \|\tilde{g} + \tilde{f}^{(e)}_{\tilde{g}} - \hat{g} - \hat{f}^{(e)}\|_{2,e}) + \delta_{n,t}^2\right\} \nonumber \\ 
&~~~~~~~~~~~ + CU \frac{1}{|\mathcal{E}|} \sum_{e\in \mathcal{E}} \left\{\delta_{n,t} \|g - \tilde{g}\|_{2,e} + \delta_{n,t}^2\right\} + 2(1+\gamma)\delta_{\mathtt{opt}}^2 \nonumber \\
&~~~\le \gamma CU \frac{1}{|\mathcal{E}|} \sum_{e\in \mathcal{E}}\left\{ 10 CU \delta_{n,t}^2 + \delta_{n,t}^2 + (CU)^{-1} 0.05 \left(\|\tilde{g} - \hat{g}\|_{2,e}^2 + \|\tilde{g} + \tilde{f}^{(e)}_{\tilde{g}} - \hat{g} - \hat{f}^{(e)}\|_{2,e}^2\right)\right\} \nonumber \\
&~~~~~~~~~~~ + CU \frac{1}{|\mathcal{E}|} \sum_{e\in \mathcal{E}} \left\{0.05(CU)^{-1}\|g - \tilde{g}\|_{2,e} + 10CU \delta_{n,t}^2 + \delta_{n,t}^2\right\} +  2(1+\gamma)\delta_{\mathtt{opt}}^2 \nonumber \\
\begin{split}
\label{eq:proof-main-rate-up1}
&~~~\le C (\gamma+1) U^2 \delta_{n,t}^2  + 0.05 (1+\gamma)\|\tilde{g} - \hat{g}\|_2^2 + 2(1+\gamma)\delta_{\mathtt{opt}}^2 \\
&~~~~~~~~~~~ + \gamma \frac{1}{|\mathcal{E}|} \sum_{e\in \mathcal{E}} \|\tilde{g} + \tilde{f}^{(e)}_{\tilde{g}} - \hat{g} - \hat{f}^{(e)}\|_{2,e}^2
\end{split}
\end{align} where the third inequality follows from the fact that $ab\le 0.5a^2\eta + 0.5b^2\eta^{-1}$ for any $\eta>0$, and the last inequality follows from the fact that $\frac{1}{|\mathcal{E}|}\sum_{e\in \mathcal{E}}\|\cdot\|_{2,e} = \|\cdot\|_2$. Observe the set $\hat{S}$ satisfies \eqref{eq:proof-main-rate-eq1}, thus it follows from triangle inequality that
\begin{align}
    \|\tilde{g} + \tilde{f}^{(e)}_{\tilde{g}} - \hat{g} - \hat{f}^{(e)}\|_{2,e}^2 &= \|\tilde{g} + \tilde{f}^{(e)}_{\tilde{g}} - g^\star + \Pi_{\overline{\mathcal{F}_S}}^{(e)}(m^{(e,\hat{S})}) - \hat{g} - \hat{f}^{(e)}\|_{2,e}^2 \nonumber \\
    &\le 2\|g^\star - \tilde{g} - \tilde{f}^{(e)}_{\tilde{g}}\|_{2,e}^2 + 2\|\Pi_{\overline{\mathcal{F}_S}}^{(e)}(m^{(e,\hat{S})}) - \hat{g} - \hat{f}^{(e)}\|_{2,e}^2 \nonumber
\end{align} Applying the second argument in \cref{prop:characterize-f} with $\eta = 1/4$, the instance-dependent characterization of $f^{(e)}$, to $\tilde{f}^{(e)}_{\tilde{g}}$ and $\hat{f}^{(e)}$ simultaneously, we find
\begin{align}
    &\frac{1}{|\mathcal{E}|} \sum_{e\in \mathcal{E}} \|\tilde{g} + \tilde{f}^{(e)}_{\tilde{g}} - \hat{g} - \hat{f}^{(e)}\|_{2,e}^2 \nonumber\\
    &~~~\le \frac{2}{|\mathcal{E}|} \sum_{e\in \mathcal{E}} \|g^\star - \tilde{g} - \tilde{f}^{(e)}_{\tilde{g}}\|_{2,e}^2 + \frac{2}{|\mathcal{E}|} \sum_{e\in \mathcal{E}} \|\Pi_{\overline{\mathcal{F}_S}}^{(e)}(m^{(e,\hat{S})}) - \hat{g} - \hat{f}^{(e)}\|_{2,e}^2 \nonumber \\
    &~~~\le 48\left(\delta_{\mathtt{a}, \mathcal{F}, \mathcal{G}}(S^\star) + \delta_{\mathtt{a}, \mathcal{F}, \mathcal{G}}(\hat{S})\right) + 72 C^2U^2\delta_{n,t}^2 \nonumber \\
    &~~~~~~~~~~~~ + \frac{16}{\gamma} \left(\sup_{\breve{f} \in \{\mathcal{F}_{S_{\tilde{g}}}\}^{|\mathcal{E}|}}\hat{\mathsf{Q}}_\gamma(\tilde{g}, \breve{f}^{\mathcal{E}}) - \hat{\mathsf{Q}}_\gamma(\tilde{g}, \tilde{f}_{\tilde{g}}^{\mathcal{E}})\right) \nonumber \\
    &~~~~~~~~~~~~ + \frac{16}{\gamma} \left(\sup_{\breve{f} \in \{\mathcal{F}_{S_{\hat{g}}}\}^{|\mathcal{E}|}}\hat{\mathsf{Q}}_\gamma(\hat{g}, \breve{f}^{\mathcal{E}}) - \hat{\mathsf{Q}}_\gamma(\hat{g}, \hat{f}^{\mathcal{E}})\right) \nonumber \\
    &~~~\le C \left\{ U^2 \delta_{n,t}^2 + \delta^2_{\mathtt{a}, \mathcal{F}, \mathcal{G}}(\hat{S}) + \delta^2_{\mathtt{a}, \mathcal{F}, \mathcal{G}}(S^\star) + \frac{1+\gamma}{\gamma}\delta_{\mathtt{opt}}^2 \right\}.\label{eq:proof-main-rate-sume}
\end{align} Here the second inequality follows from the \cref{prop:characterize-f}, the third inequality follows from our choice of $\tilde{f}^{\mathcal{E}}_{\tilde{g}}$ \eqref{eq:proof-main-rate-tilde-f-choice} and the definition of approximate solution \eqref{eq:approx-solution}. Substituting it back into \eqref{eq:proof-main-rate-up1}, we have
\begin{align*}
    &\mathsf{Q}_\gamma(\hat{g}, \hat{f}^{\mathcal{E}}) - \mathsf{Q}_\gamma(\tilde{g}, \tilde{f}_{\tilde{g}}^{\mathcal{E}}) \\
    &~~~~ \le 0.05 (1+\gamma) \|\tilde{g} - \hat{g}\|_2^2 + C(\gamma+1)\left(U^2\delta_{n,t}^2 + \delta^2_{\mathtt{a}, \mathcal{F}, \mathcal{G}}(\hat{S}) + \delta^2_{\mathtt{a}, \mathcal{F}, \mathcal{G}}(S^\star) + \delta_{\mathtt{opt}}^2\right).
\end{align*} Combining it with the lower bound \eqref{eq:proof-main-rate-lb} we obtained for $\mathsf{Q}_\gamma(\hat{g}, \hat{f}^{\mathcal{E}}) - \mathsf{Q}_\gamma(\tilde{g}, \tilde{f}_{\tilde{g}}^{\mathcal{E}})$ gives,
\begin{align*}
0.2 (1 + \gamma) \|\hat{g} - \tilde{g}\|_2^2  &\le C(\gamma+1)U^2\delta_{n,t}^2 + C\gamma \left(\delta^2_{\mathtt{a}, \mathcal{F}, \mathcal{G}}(\hat{S}) + \delta^2_{\mathtt{a}, \mathcal{F}, \mathcal{G}}(S^\star)\right)  + C(1+\gamma)\delta_{\mathtt{opt}}^2  \\
& ~~~~~~~~ + (1 + \gamma) \|\tilde{g} - g^\star\|_2^2 + \frac{\gamma}{2|\mathcal{E}|} \sum_{e\in \mathcal{E}}  \|\hat{f}^{(e)} - \{g^\star - \hat{g}\}\|_{2,e}^2 \\
&\le C(\gamma + 1) \left\{U^2 \delta_{n,t}^2 + \delta^2_{\mathtt{a}, \mathcal{F}, \mathcal{G}}(\hat{S}) + \delta^2_{\mathtt{a}, \mathcal{F}, \mathcal{G}}(S^\star) + \delta_{\mathtt{a}, \mathcal{G}}^2 + \delta_{\mathtt{opt}}^2\right\}
\end{align*} where the last equality follows from our choice of $\tilde{g}$ and application of \cref{prop:characterize-f} that similar to \eqref{eq:proof-main-rate-sume}. Note that
\begin{align*}
\delta^2_{\mathtt{a}, \mathcal{F}, \mathcal{G}}(\hat{S}) \lor \delta^2_{\mathtt{a}, \mathcal{F}, \mathcal{G}}(S^\star) &\le \frac{1}{|\mathcal{E}|} \sum_{e\in \mathcal{E}} \sup_{g\in \mathcal{G}} \inf_{f\in \mathcal{F}_{S_g}} \|g^\star - g - f\|^2_{2,e} 
\end{align*} because of \eqref{eq:main-result-variable-sel}, we have
\begin{align*}
    \|\hat{g} - \tilde{g}\|_2 \le C\left(U\delta_{n,t} + \delta_{\mathtt{a}, \mathcal{G}} + \sqrt{\frac{1}{|\mathcal{E}|} \sum_{e\in \mathcal{E}} \sup_{g\in \mathcal{G}} \inf_{f\in \mathcal{F}_{S_g}} \|g^\star - g - f\|^2_{2,e}} + \delta_{\mathtt{opt}}\right)
\end{align*} applying triangle equality yields,
\begin{align*}
\|\hat{g} - g^\star\|_2 &\le 2\|\hat{g} - \tilde{g}\|_2 + 2\|\tilde{g} - g^\star\|_2 \\
&\le C\left(U\delta_{n,t} + \delta_{\mathtt{a}, \mathcal{G}} + \sqrt{\frac{1}{|\mathcal{E}|} \sum_{e\in \mathcal{E}} \sup_{g\in \mathcal{G}} \inf_{f\in \mathcal{F}_{S_g}} \|g^\star - g - f\|^2_{2,e}} + \delta_{\mathtt{opt}}\right).
\end{align*} This completes the proof.

\subsection{Proof of the General Error Bound}

The proof of \eqref{eq:main-result-general-rate} is similar to the proof of \eqref{eq:main-result-rate}, the key difference is that we no longer have \eqref{eq:proof-main-rate-eq1} such that we need to establish bound on $\|\Pi_{\overline{\mathcal{F}_{\hat{S}}}}^{(e)}(m^{(e,S)}) - g^\star\|_2$.

On one hand, we apply \cref{thm:population} with $\delta=0.5$, $g, f^{\mathcal{E}} = \hat{g}, \hat{f}^{\mathcal{E}}$ and $\tilde{g}, \tilde{f}^{\mathcal{E}}$ to be determined as
\begin{align*}
    &\mathsf{Q}_\gamma(\hat{g}, \hat{f}^{\mathcal{E}}) - \mathsf{Q}_\gamma(\tilde{g}, \tilde{f}^{\mathcal{E}}) \\
    &~~~~\ge 0.25 \|\hat{g} - \tilde{g}\|_2^2 +  \frac{\gamma}{4} \bar{\mathsf{d}}_{\mathcal{G}, \mathcal{F}}(\hat{S}) + \frac{\gamma}{2} \|\hat{g} - \Pi_{\overline{\mathcal{G}_{\hat{S}}}}(\bar{m}^{(\hat{S})})\|_2^2\\
    &~~~~~~~~~~ - \frac{\gamma}{2|\mathcal{E}|} \sum_{e\in \mathcal{E}}  \|\hat{f}^{(e)} - \{\Pi_{\overline{\mathcal{F}_{\hat{S}}}}^{(e)} (m^{(e,\hat{S})}) - \hat{g}\}\|_{2,e}^2 - \frac{4 + \gamma}{2} \|\tilde{g} - g^\star\|_2^2 
\end{align*} We define
\begin{align*}
\bar{g}^{({S})} = \Pi_{\overline{\mathcal{G}_{S}}}^{(e)} (\bar{m}^{({S})}) \qquad \text{and} \qquad f^{(e,S)} = \Pi_{\overline{\mathcal{F}_{{S}}}}^{(e)} (m^{(e,{S})}).
\end{align*} Then
\begin{align}
\label{eq:proof-general-rate-lb}
\begin{split}
\mathsf{Q}_\gamma(\hat{g}, \hat{f}^{\mathcal{E}}) - \mathsf{Q}_\gamma(\tilde{g}, \tilde{g}^{\mathcal{E}}) &\ge 0.25 \|\hat{g} - \tilde{g}\|_2^2 + \frac{\gamma}{4} \bar{\mathsf{d}}_{\mathcal{G}, \mathcal{F}}(\hat{S}) - 2(1 + \gamma) \|\tilde{g} - g^\star\|_2^2 \\
&~~~~~~~~~~ - \frac{\gamma}{2|\mathcal{E}|} \sum_{e\in \mathcal{E}}  \|\hat{f}^{(e)} - \{f^{(e,\hat{S})} - \hat{g}\}\|_{2,e}^2.
\end{split}
\end{align}

Following a same choice of $\tilde{g}$ and $\tilde{f}^{\mathcal{E}}$ as \eqref{eq:proof-main-rate-tilde-g-choice}, \eqref{eq:proof-main-rate-tilde-f-choice}, and apply \cref{prop:nonasymptotic-a}, \cref{prop:nonasymptotic-pooled} in a same way, we obtain
\begin{align}
&\mathsf{Q}_\gamma(\hat{g}, \hat{f}^{\mathcal{E}}) - \mathsf{Q}_\gamma(\tilde{g}, \tilde{f}_{\tilde{g}}^{\mathcal{E}}) \nonumber \\
&~~~\le 2(1+\gamma)\delta_{\mathtt{opt}}^2 + \frac{\gamma}{|\mathcal{E}|} \sum_{e\in \mathcal{E}} |\Delta_{\mathsf{A}}(\hat{g}, \tilde{g}, \hat{f}^{(e)}, \tilde{f}^{(e)}_{\tilde{g}})| + |\Delta_{\mathsf{R}}(\hat{g},\tilde{g})| \nonumber \\
&~~~\le 2(1+\gamma)\delta_{\mathtt{opt}}^2 + CU \frac{1}{|\mathcal{E}|} \sum_{e\in \mathcal{E}} \left\{\delta_{n,t} \|\hat{g} - \tilde{g}\|_{2,e} + \delta_{n,t}^2\right\} \nonumber \\ 
&~~~~~~~~~~~ + \gamma CU \frac{1}{|\mathcal{E}|}\sum_{e\in \mathcal{E}} \left\{\delta_{n,t} (\|\tilde{g} - \hat{g}\|_{2,e} + \|\tilde{g} + \tilde{f}^{(e)}_{\tilde{g}} - \hat{g} - \hat{f}^{(e)}\|_{2,e}) + \delta_{n,t}^2\right\}  \nonumber \\
&~~~\le 2(1+\gamma)\delta_{\mathtt{opt}}^2 + CU \frac{1}{|\mathcal{E}|} \sum_{e\in \mathcal{E}} \left\{0.05(CU)^{-1}\|\hat{g} - \tilde{g}\|_{2,e} + 10CU \delta_{n,t}^2 + \delta_{n,t}^2\right\} \nonumber \\
&~~~~~~~~~~~ + \gamma CU \frac{1}{|\mathcal{E}|} \sum_{e\in \mathcal{E}}\left\{ 500 CU (\gamma+1) \delta_{n,t}^2 + \gamma \delta_{n,t}^2 \right\} \nonumber \\
&~~~~~~~~~~~ + \gamma CU \frac{1}{|\mathcal{E}|} \sum_{e\in \mathcal{E}}(CU(\gamma+1))^{-1} 0.001 \left\{\left(\|\tilde{g} - \hat{g}\|_{2,e}^2 + \|\tilde{g} + \tilde{f}^{(e)}_{\tilde{g}} - \hat{g} - \hat{f}^{(e)}\|_{2,e}^2\right)\right\} \nonumber \\
&~~~\le C (\gamma+1)^2 U^2 \delta_{n,t}^2  + 0.05 \|\tilde{g} - \hat{g}\|_2^2 + 2(1+\gamma)\delta_{\mathtt{opt}}^2 + 0.001 \|\hat{g} - \tilde{g}\|_2^2 \nonumber \\
&~~~~~~~~~~~ + 0.001\frac{\gamma}{\gamma+1}\frac{1}{|\mathcal{E}|} \sum_{e\in \mathcal{E}} \|\tilde{g} + \tilde{f}^{(e)}_{\tilde{g}} - \hat{g} - \hat{f}^{(e)}\|_{2,e}^2 \nonumber \\
\begin{split}
&~~~\overset{(a)}{\le} C (\gamma+1)^2 U^2 \delta_{n,t}^2 + 0.05  \|\tilde{g} - \hat{g}\|_2^2 + 2(1+\gamma)\delta_{\mathtt{opt}}^2 \\
&~~~~~~~~~~~ + 0.001 \|\tilde{g} - \hat{g}\|_2^2 + 0.002 \frac{\gamma}{1+\gamma} \sum_{e\in \mathcal{E}} \|g^\star - f^{(e,\hat{S})}\|_{2,e}^2 \\
&~~~~~~~~~~~ + \frac{0.004\gamma}{\gamma+1} \frac{1}{|\mathcal{E}|} \sum_{e\in \mathcal{E}} \left(\|\tilde{g} + \tilde{f}^{(e)}_{\tilde{g}} - g^\star\|_{2,e}^2 + \|f^{(e,\hat{S})} - \hat{g} - \hat{f}^{(e)}\|_{2,e}^2\right)
\end{split}
\label{eq:proof-general-rate-eq2}
\end{align}
where $(a)$ follows from the fact that
\begin{align*}
    \|\tilde{g} + \tilde{f}^{(e)}_{\tilde{g}} - \hat{g} - \hat{f}^{(e)}\|_{2,e}^2 &= \|\tilde{g} + \tilde{f}^{(e)}_{\tilde{g}} - g^\star + f^{(e,\hat{S})} - \hat{g} - \hat{f}^{(e)} + g^\star - f^{(e,\hat{S})}\|_{2,e}^2 \\
    &\le 2\|g^\star - f^{(e,\hat{S})}\|_{2,e}^2 + 4\|\tilde{g} + \tilde{f}^{(e)}_{\tilde{g}} - g^\star\|_{2,e}^2 + 4\|f^{(e,\hat{S})} - \hat{g} - \hat{f}^{(e)}\|_{2,e}^2.
\end{align*}

We claim that 
\begin{align}
\frac{1}{|\mathcal{E}|} \sum_{e\in \mathcal{E}} \|g^\star - f^{(e,\hat{S})}\|_{2,e}^2 \le (2+28\gamma^\star)\bar{\mathsf{d}}_{\mathcal{G}, \mathcal{F}}(\hat{S}) + 24 \left(\|\hat{g} - \tilde{g}\|_2^2 + \delta_{\mathtt{a}, \mathcal{G}}^2 + \frac{1}{n} \right)
\label{eq:proof-general-rate-eq3}
\end{align} and defer its proof to the end of this section.

Substituting \eqref{eq:proof-general-rate-eq3} back into \eqref{eq:proof-general-rate-eq2} and combining such an upper bound with the lower bound \eqref{eq:proof-general-rate-lb}, we obtain
\begin{align*}
    0.25 \|\hat{g} - \tilde{g}\|_2^2 + &\frac{\gamma}{4} \mathsf{d}_{\mathcal{G}, \mathcal{F}}(\hat{S}) \\
    &\le C(1+\gamma) \|\tilde{g} - g^\star\|_2^2 + C(\gamma+1)^2 U^2 \delta_{n,t}^2 + 0.05  \|\tilde{g} - \hat{g}\|_2^2 + 2(1+\gamma)\delta_{\mathtt{opt}}^2\\
    &~~~~~~~~ + \left(\frac{\gamma}{\gamma+1}+\gamma\right) \frac{1}{2|\mathcal{E}|} \sum_{e\in \mathcal{E}}  \|\hat{f}^{(e)} - \{f^{(e,\hat{S})} - \hat{g}\}\|_{2,e}^2 \\
    &~~~~~~~~ + \frac{\gamma}{\gamma+1} \frac{1}{|\mathcal{E}|} \sum_{e\in \mathcal{E}} \|\tilde{g} + \tilde{f}^{(e)}_{\tilde{g}} - g^\star\|_{2,e}^2 \\
    &~~~~~~~~ + 0.05 \gamma \mathsf{d}_{\mathcal{G}, \mathcal{F}}(\hat{S}) + 0.05 \|\hat{g} - \tilde{g}\|_2^2 + 0.05 \left(\delta_{\mathtt{a}, \mathcal{G}}^2 + \frac{1}{n} \right).
\end{align*} Plugging in our choice of $\tilde{g}$ \eqref{eq:proof-main-rate-tilde-g-choice} gives
\begin{align*}
    0.1 \|\hat{g} - \tilde{g}\|_2^2 &\le C\left\{(1+\gamma) (\delta_{\mathtt{a}, \mathcal{G}}^2 + \delta_{\mathtt{opt}}^2) + (1+\gamma)^2 U^2 \delta_{n,t}^2 \right\} \\
    &~~~~~~~~+\gamma(1+1/(\gamma+1)) \frac{1}{2|\mathcal{E}|} \sum_{e\in \mathcal{E}}  \|\hat{f}^{(e)} - \{f^{(e,\hat{S})} - \hat{g}\}\|_{2,e}^2 \\
    &~~~~~~~~+\frac{\gamma}{1+\gamma} \frac{1}{|\mathcal{E}|} \sum_{e\in \mathcal{E}} \|\tilde{g} + \tilde{f}^{(e)}_{\tilde{g}} - g^\star\|_{2,e}^2
\end{align*} because $n^{-1} \le \delta_{n,t}$ and $\bar{\mathsf{d}}_{\mathcal{G}, \mathcal{F}}(\hat{S}) \ge 0$. Applying \cref{prop:characterize-f}, we find
\begin{align*}
    \gamma(1+1/(\gamma+1))\frac{1}{2|\mathcal{E}|} \sum_{e\in \mathcal{E}}  \|\hat{f}^{(e)} - \{f^{(e,\hat{S})} - \hat{g}\}\|_{2,e}^2 \le (1+\gamma) C \left\{\delta_{\mathtt{opt}}^2 + \delta^2_{\mathtt{a}, \mathcal{F}, \mathcal{G}}(\hat{S}) + U^2\delta_{n,t}^2\right\},
\end{align*} and
\begin{align*}
\frac{\gamma}{1+\gamma} \frac{1}{|\mathcal{E}|} \sum_{e\in \mathcal{E}} \|\tilde{g} + \tilde{f}^{(e)}_{\tilde{g}} - g^\star\|_{2,e}^2\le C\left\{\delta_{\mathtt{opt}}^2 + \delta^2_{\mathtt{a}, \mathcal{F}, \mathcal{G}}(S^\star) + U^2\delta_{n,t}^2\right\}.
\end{align*} Substituting them back yields
\begin{align*}
\|\hat{g} - \tilde{g}\|_2 \le C(1+\gamma)\left\{\delta_{\mathtt{a}, \mathcal{G}} + \delta_{\mathtt{a}, \mathcal{F}, \mathcal{G}}(\hat{S}) + \delta_{\mathtt{a}, \mathcal{F}, \mathcal{G}}(S^\star) + \delta_{\mathtt{opt}} + U \delta_{n,t}^2 \right\}.
\end{align*} Applying the triangle inequality, that
\begin{align*}
\|\hat{g} - g^\star\|_2 \le \|\hat{g} - \tilde{g}\|_2 + \|\tilde{g} - g^\star\|_2
\end{align*} completes the proof.

\noindent \emph{Proof of the Claim \eqref{eq:proof-general-rate-eq3}.} It follows from triangle inequality that
\begin{align*}
    \frac{1}{|\mathcal{E}|} \sum_{e\in \mathcal{E}} \|g^\star - f^{(e,\hat{S})}\|_{2,e}^2 &= \frac{1}{|\mathcal{E}|} \sum_{e\in \mathcal{E}} \|g^\star - \bar{g}^{(\hat{S})} + \bar{g}^{(\hat{S})} - f^{(e,\hat{S})}\|_{2,e}^2 \\
    &\le \frac{1}{|\mathcal{E}|} \sum_{e\in \mathcal{E}} 2\|g^\star - \bar{g}^{(\hat{S})}\|_{2,e}^2 + 2\|\bar{g}^{(\hat{S})} - f^{(e,\hat{S})}\|_{2,e}^2 \\
    &= 2\|g^\star - \bar{g}^{(\hat{S})}\|_2^2 + 2\bar{\mathsf{d}}_{\mathcal{G}, \mathcal{F}}(\hat{S}) \\
    &=2\|g^\star - \bar{g}^{(\hat{S} \cup S^\star)} + \bar{g}^{(\hat{S} \cup S^\star)} - \bar{g}^{(\hat{S})}\|_2^2 + 2\bar{\mathsf{d}}_{\mathcal{G}, \mathcal{F}}(\hat{S}) \\
    &\le 4\|g^\star - \bar{g}^{(\hat{S} \cup S^\star)}\|_{2}^2 + 4\|\bar{g}^{(\hat{S} \cup S^\star)} - \bar{g}^{(\hat{S})}\|_{2}^2 + 2\bar{\mathsf{d}}_{\mathcal{G}, \mathcal{F}}(\hat{S})\\
    &= 4\mathsf{b}_{\mathcal{G}}(\hat{S}) + 4\|\bar{g}^{(\hat{S} \cup S^\star)} - \bar{g}^{(\hat{S})}\|_{2}^2 + 2\bar{\mathsf{d}}_{\mathcal{G}, \mathcal{F}}(\hat{S}) \\
    &\le (4\gamma^\star+2) \mathsf{d}_{\mathcal{G,F}}(\hat{S}) + 4\|\bar{g}^{(\hat{S} \cup S^\star)} - \bar{g}^{(\hat{S})}\|_{2}^2.
\end{align*}
We claim that
\begin{align*}
\|\bar{g}^{(\hat{S} \cup S^\star)} - \bar{g}^{(\hat{S})}\|_{2}^2 \le \|\hat{g} - \bar{g}^{(\hat{S} \cup S^\star)}\|_{2}^2,
\end{align*} this implies
\begin{align}
\label{eq:proof-general-rate-eq5}
\frac{1}{|\mathcal{E}|} \sum_{e\in \mathcal{E}} \|g^\star - f^{(e,\hat{S})}\|_{2,e}^2 \le (4\gamma^\star+2) \mathsf{d}_{\mathcal{G,F}}(\hat{S}) + 4\|\bar{g}^{(\hat{S} \cup S^\star)} - \hat{g}\|_{2}^2.
\end{align} This is because
\begin{align*}
    \|\bar{g}^{(\hat{S})} - \bar{g}^{(\hat{S} \cup S^\star)}\|_{2}^2 &\overset{(a)}{=} \frac{1}{|\mathcal{E}|} \sum_{e\in \mathcal{E}} \mathbb{E}\left[\{Y^{(e)} - \bar{g}^{(\hat{S})}(X^{(e)})\}^2 - \{Y^{(e)} - \bar{g}^{(\hat{S} \cup S^\star)}(X^{(e)})\}^2\right] \\
    &~~~~~~~~~~ + 2\int (\bar{m}^{(\hat{S}\cup S^\star)}(x) - \bar{g}^{(e,\hat{S}\cup S^\star)}(x)) (\bar{g}^{(e,\hat{S}\cup S^\star)}(x) - \bar{g}^{(\hat{S})}(x))\bar{\mu}_x(dx) \\
    &\overset{(b)}{=} \frac{1}{|\mathcal{E}|} \sum_{e\in \mathcal{E}} \mathbb{E}\left[\{Y^{(e)} - \bar{g}^{(\hat{S})}(X^{(e)})\}^2 - \{Y^{(e)} - \bar{g}^{(\hat{S} \cup S^\star)}(X^{(e)})\}^2\right] \\
    &\overset{(c)}{\le} \frac{1}{|\mathcal{E}|} \sum_{e\in \mathcal{E}} \mathbb{E}\left[\{Y^{(e)} - \hat{g}(X^{(e)})\}^2 - \{Y^{(e)} - \bar{g}^{(\hat{S} \cup S^\star)}(X^{(e)})\}^2\right] \\
    &\overset{(d)}{=} \frac{1}{|\mathcal{E}|} \sum_{e\in \mathcal{E}} \mathbb{E}\left[\{Y^{(e)} - \hat{g}(X^{(e)})\}^2 - \{Y^{(e)} - \bar{g}^{(\hat{S} \cup S^\star)}(X^{(e)})\}^2\right] \\
    & ~~~~~~~~~~ + 2\int (m^{(\hat{S}\cup S^\star)}(x) - \bar{g}^{(\hat{S}\cup S^\star)}(x)) (\bar{g}^{(\hat{S}\cup S^\star)}(x) - \hat{g}(x))\bar{\mu}_x(dx) \\
    &\overset{(e)}{=} \|\hat{g} - \bar{g}^{(\hat{S} \cup S^\star)}\|_2^2
\end{align*} where $(a)$ and $(e)$ follows from \cref{lemma:pooled-l2-diff}, $(b)$ and $(d)$ follows from projection theorem \cref{thm:projection} and our definition of $\bar{g}^{(S)}$, $(c)$ follows from the fact that
\begin{align*}
    &\frac{1}{|\mathcal{E}|} \sum_{e\in \mathcal{E}} \mathbb{E}\left[\{Y^{(e)} - \hat{g}(X^{(e)})\}^2 \right] - \frac{1}{|\mathcal{E}|} \sum_{e\in \mathcal{E}} \mathbb{E}\left[\{Y^{(e)} - \bar{g}^{(\hat{S})}(X^{(e)})\}^2 \right] \\
    &~~~~= 2\int (\bar{m}^{(\hat{S})}(x) - \bar{g}^{(\hat{S})}(x))(\hat{g}(x) - \bar{g}^{(\hat{S})}(x)) \bar{\mu}(dx) + \|\bar{g}^{(\hat{S})} - \hat{g}\|_2^2 = \|\bar{g}^{(\hat{S})} - \hat{g}\|_2^2 \ge 0.
\end{align*}
It then follows from triangle inequality that
\begin{align*}
\|\bar{g}^{(\hat{S} \cup S^\star)} - \bar{g}^{(\hat{S})}\|_{2}^2 &\le \|\hat{g} - \tilde{g} + \tilde{g} - g^\star + g^\star - \bar{g}^{(\hat{S} \cup S^\star)}\|_{2}^2 \\
&\overset{(a)}{\le} 3\|\hat{g} - \tilde{g}\|_2^2 + 3\|\tilde{g} - g^\star\|_2^2 + 3\|g^\star - \bar{g}^{(\hat{S}\cup S^\star)}\|_2^2 \\
&\overset{(b)}{\le} 3\|\hat{g} - \tilde{g}\|_2^2 + 3\left(\delta_{\mathtt{a}, \mathcal{G}}^2 + \frac{1}{n}\right) + 3\mathsf{b}_{\mathcal{G}}(\hat{S}) \\
&\overset{(c)}{\le} 3\left(\|\hat{g} - \tilde{g}\|_2^2 + \delta_{\mathtt{a}, \mathcal{G}}^2 + \frac{1}{n} + \gamma^\star \mathsf{d}_{\mathcal{G,F}}(\hat{S})\right).
\end{align*} Here $(a)$ follows from the fact that $(a+b+c)^3\le 3(a^2+b^2+c^2)$, $(b)$ follows from the definition of $\tilde{g}$ and $\mathsf{b}_{\mathcal{G}}(S)$, $(c)$ follows from the definition of $\gamma^\star$. Substituting it back into \eqref{eq:proof-general-rate-eq5} completes the proof of the claim \eqref{eq:proof-general-rate-eq3}.

\subsection{Proof of Theorem \ref{thm:population}}
\label{subsec:proof-population}

Let $S=S_g$, then we can decompose the difference of loss into
\begin{align*}
    \mathsf{Q}_\gamma(g, f^\mathcal{E}) - \mathsf{Q}_\gamma(\tilde{g}, \tilde{f}^\mathcal{E}) &=\frac{1}{|\mathcal{E}|} \sum_{e\in \mathcal{E}} \mathbb{E} \left[\frac{1}{2} \left\{ (Y^{(e)}- g(X^{(e)}))^{2} - (Y^{(e)} - \tilde{g}(X^{(e)}))^2\right\}\right] \\
    &~~~~~~~~~~+\gamma \frac{1}{|\mathcal{E}|} \sum_{e\in \mathcal{E}} \mathbb{E}\left[ \{Y^{(e)}-g(X^{(e)})\} f^{(e)}(X^{(e)}) - \frac{1}{2}\{f^{(e)}(X^{(e)})\}^2\right] \\
    &~~~~~~~~~~-\gamma  \frac{1}{|\mathcal{E}|} \sum_{e\in \mathcal{E}} \mathbb{E}\left[ \{Y^{(e)}-\tilde{g}(X^{(e)})\} \tilde{f}^{(e)}(X^{(e)}) - \frac{1}{2}\{\tilde{f}^{(e)}(X^{(e)})\}^2\right] \\
    &= \mathsf{T}_1(g, \tilde{g}) + \gamma \mathsf{T}_2(g, f^{\mathcal{E}}) - \gamma \mathsf{T}_3(\tilde{g}, \tilde{f}^{\mathcal{E}}).
\end{align*} We will establish lower bounds on the three terms. We first introduce some additional notations, let
\begin{align*}
    \langle f, g\rangle = \int f(x) g(x) \bar{\mu}_x(dx) \qquad \text{and} \qquad \langle f, g\rangle_e = \int f(x) g(x) \mu_x^{(e)}(dx)
\end{align*}

\noindent {\sc Step 1. Lower Bound on $\mathsf{T}_1$.} Denote $\overline{S} = S\cup S^\star$ It follows from \cref{lemma:pooled-l2-diff} that
\begin{align}
\label{eq:proof-population-t1-lb1}
\mathsf{T}_1(g, \tilde{g}) = \frac{1}{2} \|g - \tilde{g}\|_2^2 + \underbrace{\langle \tilde{g} - \bar{m}^{(\overline{S})}, g - \tilde{g}\rangle }_{\mathsf{T}_{1,1}(g, \tilde{g})}
\end{align}
It then follows from the projection theorem that, 
\begin{align*}
   \forall h \in \overline{\mathcal{G}_{\overline{S}}} \qquad  \langle \bar{m}^{(\overline{S})} - \Pi_{\overline{\mathcal{G}_{\overline{S}}}}(\bar{m}^{(\overline{S})}), h\rangle = 0.
\end{align*} Given that $\overline{\mathcal{G}_{\overline{S}}}$ is a linear subspace such that $\partial \mathcal{G} = \{g - g': g, g'\in \mathcal{G}\} \subseteq \overline{\mathcal{G}_{\overline{S}}}$, we obtain
\begin{align*}
    \langle \tilde{g} - \bar{m}^{(\overline{S})}, g - \tilde{g}\rangle &= \langle \tilde{g} - \Pi_{\overline{\mathcal{G}_{\overline{S}}}}(\bar{m}^{(\overline{S})}) + \Pi_{\overline{\mathcal{G}_{\overline{S}}}}(\bar{m}^{(\overline{S})}) - \bar{m}^{(\overline{S})}, g - \tilde{g}\rangle\\
    &= \langle \tilde{g} - \Pi_{\overline{\mathcal{G}_{\overline{S}}}}(\bar{m}^{(\overline{S})}), g - \tilde{g} \rangle.
\end{align*}
It then follows from the Cauchy-Schwartz inequality that
\begin{align*}
\left|\mathsf{T}_{1,1}(g, \tilde{g})\right| &\le \left\| \tilde{g} - \Pi_{\overline{\mathcal{G}_{\overline{S}}}}(\bar{m}^{(\overline{S})}) \right\|_2 \left\|g - \tilde{g}\right\|_2\\
&\le \frac{\delta^{-1}}{2} \left\| \tilde{g} - \Pi_{\overline{\mathcal{G}_{\overline{S}}}}(\bar{m}^{(\overline{S})}) \right\|_2^2 + \frac{\delta}{2} \|g - \tilde{g}\|_2^2 \\
&\le \delta^{-1} \|\tilde{g} - g^\star\|_2^2 + \delta^{-1} \|\Pi_{\overline{\mathcal{G}_{\overline{S}}}}(\bar{m}^{(\overline{S})}) - g^\star\|_2^2 + \frac{\delta}{2} \|g - \tilde{g}\|_2^2.
\end{align*} Here the last inequality follows from triangle inequality and Young's inequality that $\|f+g\|_2^2\le 2\|f\|_2^2+2\|g\|_2^2$. Substituting it back into \eqref{eq:proof-population-t1-lb1}, we conclude that
\begin{align}
\label{eq:proof-population-step1-final}
    \mathsf{T}_1(g, \tilde{g}) \ge &\frac{1-\delta}{2}\|g - \tilde{g}\|_{2}^2 - \delta^{-1} \|\tilde{g} - g^\star\|_2^2 - \delta^{-1} \mathsf{b}_{\mathcal{G}}(S)
\end{align}

\noindent {\sc Step 2. Calculating $\mathsf{T}_2$.}  For $\mathsf{T}_2$, it follows from the tower rule of conditional expectation and the fact that $ab-\frac{1}{2}b^2=\frac{1}{2}\{a^2-(a-b)^2\}$ that
\begin{align*}
        \mathsf{T}_2(g, f^{\mathcal{E}}) &= |\mathcal{E}|^{-1} \sum_{e\in \mathcal{E}} \mathbb{E}\left[(m^{(e,S)}(X) - g(X^{(e)})) f^{(e)}(X^{(e)}) - \frac{1}{2} \{f^{(e)}(X^{(e)})\}^2\right]\\
        &= (2|\mathcal{E}|)^{-1} \sum_{e\in \mathcal{E}} \mathbb{E}\Big[\{m^{(e,S)}(X^{(e)}) - g(X^{(e)})\}^2\Big] \\
        &~~~~~~~~~~~~~~~~~~~~~~~~~~~~~~~ - \mathbb{E}\Big[\{m^{(e,S)}(X^{(e)}) - g(X^{(e)}) - f^{(e)}(X^{(e)})\}^2\Big].
\end{align*} Hence we have
\begin{align}
\label{eq:proof-ds-step0-lb}
\mathsf{T}_2(g, f^{\mathcal{E}}) = (2|\mathcal{E}|)^{-1} \sum_{e\in \mathcal{E}} \|m^{(e,S)}-g\|_{2,e}^2 - \|\{m^{(e,S)}-g\} - f^{(e)}\|_{2,e}^2.
\end{align}
Observe that
\begin{align*}
        \sum_{e\in \mathcal{E}} \|m^{(e,S)} - g\|_{2,e}^2 &= \sum_{e\in \mathcal{E}} \|m^{(e,S)} - \Pi_{\overline{\mathcal{G}_S}} (\bar{m}^{(S)}) + \Pi_{\overline{\mathcal{G}_S}} (\bar{m}^{(S)}) - g\|_{2,e}^2  \\
        &\overset{(a)}{=} \sum_{e\in \mathcal{E}} \|m^{(e,S)} - \Pi_{\overline{\mathcal{G}}} (\bar{m}^{(S)})\|_{2,e}^2 + \sum_{e\in \mathcal{E}} \| \Pi_{\overline{\mathcal{G}_S}} (\bar{m}^{(S)}) - g\|_{2,e}^2 \\
        &~~~~~~~~~~ + \int \left(\sum_{e\in \mathcal{E}} m^{(e,S)}(x) \rho^{(e)}_S(x) - |\mathcal{E}| \Pi_{\overline{\mathcal{G}_S}}(\bar{m}^{(S)})(x) \right) \\
        &~~~~~~~~~~~~~~~~~~~~~~~~~~~~~~~~~~~~ \times (\Pi_{\overline{\mathcal{G}_S}} (\bar{m}^{(S)})(x) - g(x)) \bar{\mu}_{x}(dx) \\
        &\overset{(b)}{=} \Bigg(\sum_{e\in \mathcal{E}} \|m^{(e,S)} - \Pi_{\overline{\mathcal{G}_S}} (\bar{m}^{(S)})\|_{2,e}^2 \Bigg)+ |\mathcal{E}| \|\Pi_{\overline{\mathcal{G}_S}} (\bar{m}^{(S)}) - g\|_{2}^2 \\
        & ~~~~~~~~~~ + 2 |\mathcal{E}| \langle \bar{m}^{(S)} - \Pi_{\overline{\mathcal{G}_S}} (\bar{m}^{(S)}), \Pi_{\overline{\mathcal{G}_S}} (\bar{m}^{(S)}) - g\rangle \\
        &\overset{(c)}{=} \Bigg(\sum_{e\in \mathcal{E}} \|m^{(e,S)} - \Pi_{\overline{\mathcal{G}_S}} (\bar{m}^{(S)})\|_{2,e}^2 \Bigg)+ |\mathcal{E}| \|\Pi_{\overline{\mathcal{G}_S}} (\bar{m}^{(S)}) - g\|_{2}^2.
\end{align*} where $(a)$ follows from the fact that $\frac{1}{|\mathcal{E}|}\sum_{e\in \mathcal{E}}\rho^{(e)}_S(x) \equiv 1$, and $(b)$ follows from the definition of $\bar{m}^{(e)}$ that $\sum_{e\in \mathcal{E}} m^{(e,S)}(x_S) \rho^{(e)}_S(x_S) \equiv |\mathcal{E}| \bar{m}^{(S)}(x_S)$, $(c)$ follows from the projection theorem and the fact $\Pi_{\overline{\mathcal{G}_S}} (\bar{m}^{(S)}) - g \in \overline{\mathcal{G}_S}$ since $\overline{\mathcal{G}_S}$ is a subspace. Moreover, for any $e\in \mathcal{E}$, it follows from the projection theorem and the fact that $\overline{\mathcal{G}_S} \subseteq \overline{\mathcal{F}_S}$ that 
    \begin{align*}
    \|m^{(e,S)} - \Pi_{\overline{\mathcal{G}_S}} (\bar{m}^{(S)})\|_{2,e}^2 &= \|m^{(e,S)} - \Pi_{\overline{\mathcal{F}_S}}^{(e)} (m^{(e,S)}) + \Pi_{\overline{\mathcal{F}_S}}^{(e)} (m^{(e,S)}) - \Pi_{\overline{\mathcal{G}_S}} (\bar{m}^{(S)})\|_{2,e}^2 \\
    &=\|m^{(e,S)} - \Pi_{\overline{\mathcal{F}_S}}^{(e)} (m^{(e,S)})\|_{2,e}^2 + \|\Pi_{\overline{\mathcal{F}_S}}^{(e)} (m^{(e,S)}) - \Pi_{\overline{\mathcal{G}_S}} (\bar{m}^{(S)})\|_{2,e}^2
\end{align*} 
Combining the above two claims yields that, for any $g\in \mathcal{G}$, 
\begin{align}
\label{eq:ds-step1}
    \begin{split}
        \frac{1}{|\mathcal{E}|} \sum_{e\in \mathcal{E}} \|m^{(e,S)} - g\|_{2,e}^2 &= \|\Pi_{\overline{\mathcal{G}_S}} (\bar{m}^{(S)}) - g\|_2^2 + \frac{1}{|\mathcal{E}|} \sum_{e\in \mathcal{E}} \|\Pi_{\overline{\mathcal{F}_S}}^{(e)} (m^{(e,S)}) - \Pi_{\overline{\mathcal{G}_S}} (\bar{m}^{(S)})\|_{2,e}^2\\
        &~~~~~~~~~~ + \frac{1}{|\mathcal{E}|} \sum_{e\in \mathcal{E}}  \|m^{(e,S)} - \Pi_{\overline{\mathcal{F}_S}}^{(e)} (m^{(e,S)})\|_{2,e}^2
    \end{split}
\end{align}
At the same time, it also follows from the projection theorem and the fact $f - (\Pi_{\overline{\mathcal{F}_S}}^{(e)} (m^{(e,S)}) - g) \in \overline{\mathcal{F}_S}$ that, for any $g \in \mathcal{G}_S$ and $f \in \mathcal{F}_S$
\begin{align*}
    \|f - (m^{(e,S)} - g)\|_{2,e}^2 = \|f - \{\Pi_{\overline{\mathcal{F}_S}}^{(e)} (m^{(e,S)}) - g\}\|_{2,e}^2 + \|m^{(e,S)} - \Pi_{\overline{\mathcal{F}_S}}^{(e)} (m^{(e,S)})\|_{2,e}^2
\end{align*}
Combining the above identity with \eqref{eq:ds-step1} and substituting them back into \eqref{eq:proof-ds-step0-lb} yields
\begin{align}
\label{eq:proof-population-step2}
    \mathsf{T}_2(g, f^{\mathcal{E}}) = \frac{\bar{\mathsf{d}}_{\mathcal{G}, \mathcal{F}}(S)}{2} + \frac{1}{2}\|\Pi_{\overline{\mathcal{G}_S}} (\bar{m}^{(S)}) - g\|_2^2 - \frac{1}{2|\mathcal{E}|} \sum_{e\in \mathcal{E}}  \|f^{(e)} - \{\Pi_{\overline{\mathcal{F}_S}}^{(e)} (m^{(e,S)}) - g\}\|_{2,e}^2
\end{align}

\noindent {\sc Step 3. Calculating $\mathsf{T}_3$.} Note that $\tilde{f}^{(e)} \in \mathcal{F}_{S^\star}$ and $S_g = S^\star$. We apply the result obtained in {\sc Step 2} in this case. According to the assumption that
\begin{align*}
    \Pi^{(e)}_{\overline{\mathcal{F}_{S^\star}}}(m^{(e,S^\star)}) = \Pi_{\overline{\mathcal{G}_{S^\star}}}(\bar{m}^{(S^\star)}) = g^\star
\end{align*} in \cref{cond:general-gf}, we obtain $\bar{\mathsf{d}}_{\mathcal{G}, \mathcal{F}}(S^\star) = 0$, and 
then
\begin{align}
\label{eq:proof-population-step3}
    \mathsf{T}_3(\tilde{g}, \tilde{f}^{\mathcal{E}}) = \frac{1}{2}\|g^\star - \tilde{g}\|_2^2 - \frac{1}{2|\mathcal{E}|} \sum_{e\in \mathcal{E}}  \|\tilde{f}^{(e)} - \{g^\star - \tilde{g}\}\|_{2,e}^2.
\end{align} 

Combining \eqref{eq:proof-population-step1-final}, \eqref{eq:proof-population-step2} and \eqref{eq:proof-population-step3} in three steps together, we have
\begin{align*}
    &\mathsf{T}_1(g, \tilde{g}) + \gamma \mathsf{T}_2(g, f^{\mathcal{E}}) - \gamma \mathsf{T}_3(\tilde{g}, \tilde{f}^{\mathcal{E}}) \\
    \ge& \frac{1-\delta}{2} \|g - \tilde{g}\|_2^2 - \delta^{-1} \|\tilde{g} - g^\star\|_2^2 +  \left(-\delta^{-1} \mathsf{b}_{\mathcal{G}}(S) + \frac{4\delta^{-1}\gamma^\star}{4} \bar{\mathsf{d}}_{\mathcal{G}, \mathcal{F}}(S)\right) + \frac{\gamma}{4} \bar{\mathsf{d}}_{\mathcal{G}, \mathcal{F}}(S)\\
    &~~~~~~ + \frac{\gamma}{2} \|g - \Pi_{\overline{\mathcal{G}_S}}(\bar{m}^{(S)})\|_2^2 - \frac{\gamma}{2|\mathcal{E}|} \sum_{e\in \mathcal{E}}  \|f^{(e)} - \{\Pi_{\overline{\mathcal{F}_S}}^{(e)} (m^{(e,S)}) - g\}\|_{2,e}^2 - \frac{\gamma}{2} \|\tilde{g} - g^\star\|_2^2
\end{align*} Substituting the defintion on $\gamma^*$ completes the proof.

\subsection{Proof of Proposition \ref{prop:nonasymptotic-pooled}}
\label{subsec:proof-prop1}

We can decompose $\Delta_{\mathsf{R}}(g, \tilde{g})$ into
\begin{align*}
    \Delta_{\mathsf{R}}(g, \tilde{g}) &= \frac{1}{|\mathcal{E}|} \sum_{e\in \mathcal{E}} (\hat{\mathbb{E}} - \mathbb{E}) \left[\frac{1}{2}\{Y^{(e)} - g(X^{(e)})\}^2 - \frac{1}{2}\{Y^{(e)} -\tilde{g}(X^{(e)})\}^2 \right] \\
    &= \frac{1}{|\mathcal{E}|} \sum_{e\in \mathcal{E}} (\hat{\mathbb{E}} - \mathbb{E}) \left[ Y^{(e)} (g - \tilde{g})(X^{(e)})\right] \\
    & ~~~~~~~~~~~~~~~~ + (\hat{\mathbb{E}} - \mathbb{E})\left[ -\tilde{g}(X^{(e)}) (g - \tilde{g})(X^{(e)})\right] \\
    & ~~~~~~~~~~~~~~~~ + (\hat{\mathbb{E}} - \mathbb{E})\left[ \frac{1}{2} \{g(X^{(e)}) - \tilde{g}(X^{(e)})\}^2  \right] \\
    &= \frac{1}{|\mathcal{E}|} \sum_{e\in \mathcal{E}} \mathsf{T}_1^{(e)}(g, \tilde{g}) + \mathsf{T}_2^{(e)}(g, \tilde{g}) + \mathsf{T}_3^{(e)}(g, \tilde{g}) 
\end{align*}
We apply \cref{lemma:ep1} to derive high probability bounds on $\mathsf{T}^{(e)}_k$. 

\noindent {\sc Step 1. Bounds on $\mathsf{T}_{1}^{(e)}$.} For $\mathsf{T}_1^{(e)}$, we will use truncation argument. To be specific, let $K>0$ to be determined, we can decompose $\mathsf{T}_1^{(e)}$ into
\begin{align*}
    \mathsf{T}_1^{(e)}(g,\tilde{g}) &= (\hat{\mathbb{E}} - \mathbb{E}) \left[ 1\{|Y^{(e)}|\le K\} Y^{(e)} (g - \tilde{g})(X^{(e)})\right] \\
    & ~~~~~~~~~~ + (\hat{\mathbb{E}} - \mathbb{E}) \left[ 1\{|Y^{(e)}|>K\} Y^{(e)} (g - \tilde{g})(X^{(e)})\right] \\
    &= \mathsf{T}_{1,1}^{(e)}(g,\tilde{g},K) + \mathsf{T}_{1,2}^{(e)}(g,\tilde{g},K).
\end{align*} 
For $\mathsf{T}_{1,1}^{(e)}(g,\tilde{g},K)$, applying \cref{lemma:ep1} with $\mathcal{H} = \mathcal{G}$, $v(g,\tilde{g}, z)=1\{|Y^{(e)}|\le K\} Y^{(e)}$ that is uniformly bounded by $K$, and $\phi(x)=x$, we find that, for any $e\in \mathcal{E}$ and $u>0$, the following event
\begin{align}
\label{eq:prop1-sn1}
\begin{split}
    &\mathcal{C}_{1,1}^{(e)}(u) = \left\{\forall g,\tilde{g}\in \mathcal{G}, ~|\mathsf{T}_{1,1}^{(e)}(g,\tilde{g},K)| \le CBK \left(s_{n,1}\|\tilde{g} - g\|_{2,e} + s_{n,1}^2\right)\right\} \\
    &~~~~~~~~~~~~~~~~~~~~~~~~~~~~~ \text{with}~ s_{n,1} = \delta_n + \sqrt{\frac{u + 1 + \log(nB))}{n}}
\end{split}
\end{align} occurs with probability at least $1-e^{-u}$ for some universal constant $C$. Applying union bound over all the $e\in \mathcal{E}$, we obtain that, 
\begin{align}
    \mathbb{P}\left[\mathcal{C}_{1,1}(t)\right] &= \mathbb{P}\left[\forall e\in \mathcal{E},~\forall g,\tilde{g}\in \mathcal{G}, ~|\mathsf{T}_{1,1}^{(e)}(g,\tilde{g},K)| \le C BK (s_{n,2}\|\tilde{g}- g\|_{2,e} + s_{n,2}^2)\right] \nonumber \\
    &\ge 1 - \sum_{e\in \mathcal{E}} \mathbb{P}\left[\left(\mathcal{C}_{1,1}^{(e)}(u)\right)^c\right]\ge 1 - |\mathcal{E}| e^{-u} \ge 1-e^{-t} \label{eq:prop1-c11}
\end{align} where $s_{n,2} = \delta_n + \sqrt{\frac{t + 1 + \log(nB|\mathcal{E}|))}{n}}$.

For $\mathsf{T}_{1,2}^{(e)}(g,\tilde{g},K)$, it follows from Markov inequality that, for any given $e\in \mathcal{E}$ and $x>0$,
\begin{align*}
    &\mathbb{P}\left[\sup_{g,\tilde{g}\in \mathcal{G}}|\mathsf{T}_{1,2}^{(e)}(g,\tilde{g},K)| > x\right] \\
    &~~~\le x^{-1} \mathbb{E} \left[\sup_{g,\tilde{g}\in \mathcal{G}}(\hat{\mathbb{E}} - \mathbb{E}) \left[1\{|Y^{(e)}|>K\} Y^{(e)}  (g - \tilde{g})(X^{(e)})\right]\right]\\
    &~~~\le x^{-1} \mathbb{E} \left[\sup_{g,\tilde{g}\in \mathcal{G}}(\hat{\mathbb{E}} + \mathbb{E}) \left[\left|1\{|Y^{(e)}|>K\} Y^{(e)}  (g - \tilde{g})(X^{(e)})\right|\right]\right] \\
    &~~~\le x^{-1} 4B \mathbb{E}[|Y^{(e)}| 1\{|Y^{(e)}| > K\}]
\end{align*} It then follows from the sub-Gaussian response condition \cref{cond:general-response} that
\begin{align*}
    \mathbb{E}[|Y^{(e)}| 1\{|Y^{(e)}| > K\}] &= \int |y| 1\{|y|\ge K\} \mu^{(e)}_y(dy) \\
    &= \int 1\{|y|\ge K\} \left(\int_0^\infty 1\{t\le |y|\} dt\right) v \mu^{(e)}_y(dy)\\
    &= \int_0^\infty \int 1\{|y| \ge K\lor t\}\mu_y^{(e)}(dy) dt \\
    &= \int_0^\infty \mathbb{P}\left(|Y^{(e)}| \ge t\lor K\right) dt \\
    &\le K C_{y} e^{-K^2/(2\sigma_y^2)} + \frac{C_y \sigma_y^2}{K} e^{-K^2/(2\sigma_y^2)}.
\end{align*} Hence, we can conclude that, for any fixed $e\in \mathcal{E}$, $x>0$ and $K>0$, 
\begin{align}
\label{eq:proof-prop1-truncation-argument}
\mathbb{P}\left[\sup_{g,\tilde{g}\in \mathcal{G}} |\mathsf{T}_{1,2}^{(e)}(g,\tilde{g},K)| > x\right] \le 4BC_y e^{-K^2/(2\sigma_y^2)} x^{-1} (K+\sigma_y^2/K)
\end{align} Applying union bound with $x=44Bn^{-1}$, $K=\sigma_y\sqrt{102\log(n|\mathcal{E}|)}$, we have
\begin{align}
\label{eq:prop1-c12}
\begin{split}
    &\mathbb{P}[\mathcal{C}_{1,2}] = \mathbb{P}\left[\forall e\in \mathcal{E}, \sup_{g,\tilde{g}\in \mathcal{G}}\mathsf{T}_{1,2}^{(e)}(g,\tilde{g},K) \le \frac{44B}{n}\right] \\
    &~~~\ge 1- |\mathcal{E}| \times C_y \frac{1}{11 (n\cdot B|\mathcal{E}|)^{102}} n (\sigma_y \sqrt{102\log(n|\mathcal{E}|)} + 1)\\
    &~~~\ge 1- C_y(\sigma_y+1) n^{-100}, 
\end{split}
\end{align}

\noindent {\sc Step 2. Bounds on $\mathsf{T}_{2}^{(e)}$.} Applying \cref{lemma:ep1} with $\mathcal{H}=\mathcal{G}$, $v(g, \tilde{g}, z) = -\tilde{g}$ uniformly bounded by $L_1= B$, and $\phi(x)=x$, we have that, for any $e$ and $u>0$, the following event
\begin{align*}
    \mathcal{C}_2^{(e)}(u) = \left\{ |\mathsf{T}_2^{(e)}(g, \tilde{g})| \le C B^2  (s_{n,1} \|\tilde{g} - g\|_{2,e} + s_{n,1}^2)\right\}
\end{align*} occurs with probability $1-e^{-u}$ for some universal constant $C$, where $s_{n,1}$ is the same quantity defined in \eqref{eq:prop1-sn1}. Applying union bound over all the $e\in \mathcal{E}$ gives that
\begin{align}
\label{eq:prop1-c2}
    \mathbb{P}\left[\mathcal{C}_{2}(t)\right]= \mathbb{P}\left[\forall e\in \mathcal{E}, ~\forall g,\tilde{g}\in \mathcal{G},~|\mathsf{T}_{2}^{(e)}(g,\tilde{g})| \le C B^2  (s_{n,2}\|\tilde{g}- g\|_{2,e} + s_{n,2}^2)\right] \ge 1-e^{-t}.
\end{align} where $s_{n,2}=s_{n,1}$ with $u=t + \log(|\mathcal{E}|)$.

\noindent {\sc Step 3. Bounds on $\mathsf{T}_{3}^{(e)}$.} Following a similar procedure as what we do for $\mathsf{T}_2^{(e)}$, we apply \cref{lemma:ep1} with $\mathcal{H}=\mathcal{G}$, $v(g, \tilde{g}, z) = 0.5$ uniformly bounded by $L_1=1/2$, and $\phi(x)=x^2$ that is $2B$-Lipschitz due to the boundedness of $(g,\tilde{g})$, followed by using union bound over all the $e\in \mathcal{E}$. Therefore,
\begin{align}
\label{eq:prop1-c3}
    \mathbb{P}\left[\mathcal{C}_{3}(t)\right]= \mathbb{P}\left[\forall e\in \mathcal{E}, ~\forall g,\tilde{g}\in \mathcal{G}, ~|\mathsf{T}_{3}^{(e)}(g,\tilde{g})| \le C B^2 (s_{n,2}\|\tilde{g}- g\|_{2,e} + s_{n,2}^2)\right] \ge 1-e^{-t}.
\end{align}

\noindent {\sc Step 4. Putting the pieces together.} Recall our choice of $K=\sigma\sqrt{102\log(n|\mathcal{E}|)}$. Combining \eqref{eq:prop1-c11}, \eqref{eq:prop1-c12}, \eqref{eq:prop1-c2}, and \eqref{eq:prop1-c3} together, we can conclude that, under $\mathcal{C}_{1,1}(t) \cap \mathcal{C}(1,2) \cap \mathcal{C}_2(t) \cap \mathcal{C}_{3}(t)$ that occurs with probability at least $1-3e^{-t}-11C_y(\sigma_y+1) n^{-100}$, the following holds
\begin{align*}
    \forall g, \tilde{g}, ~~ |\Delta_{\mathsf{R}}(g, \tilde{g})| &\le \frac{1}{|\mathcal{E}|} \sum_{e\in \mathcal{E}} |\mathsf{T}_{1,1}^{(e)}(g,\tilde{g},K)| + |\mathsf{T}_{1,2}^{(e)}(g,\tilde{g},K)| + |\mathsf{T}_{2}^{(e)}(g,\tilde{g})| + |\mathsf{T}_{3}^{(e)}(g,\tilde{g})| \\
    &\overset{(a)}{\le} \frac{1}{|\mathcal{E}|} \sum_{e\in \mathcal{E}} \Bigg\{CBK \left(s_{n,2}\|g - \tilde{g}\|_{2,e} + s_{n,2}^2\right) \\
    &~~~~~~~~~~~~~~~~~~~+ \frac{44B}{n} + 2CB^2  \left(s_{n,2} \|g - \tilde{g}\|_{2,e} + s_{n,2}^2\right) \Bigg\}\\
    &\overset{(b)}{\le} C  B(B+\sqrt{\log(n|\mathcal{E}|)}) \left\{ s_{n,2}^2 + s_{n,2} \frac{1}{|\mathcal{E}|} \sum_{e\in \mathcal{E}} \|g - \tilde{g}\|_{2,e}\right\} 
\end{align*} where (a) follows from the instant-dependent error bounds in \eqref{eq:prop1-c11}, \eqref{eq:prop1-c12}, \eqref{eq:prop1-c2} and \eqref{eq:prop1-c3}, and (b) follows from our choice of $K$ and the fact $1/n\le \delta_n \le s_{n,2}$. This completes the proof. \qed

\subsection{Proof of Proposition \ref{prop:nonasymptotic-a}}
\label{subsec:proof-prop2}

It follows from the fact $ab-\frac{1}{2}b^2 = \frac{1}{2}a^2 - \frac{1}{2}(a-b)^2$ that,
\begin{align*}
    &\Delta_{\mathsf{A}}^{(e)}(g, \tilde{g}, f^{(e)}, \tilde{f}^{(e)}) \\
    &~~~= (\mathbb{E} - \hat{\mathbb{E}}) \Bigg[\{Y^{(e)} - g(X^{(e)})\}f^{(e)}(X^{(e)})-\frac{1}{2}\{f^{(e)}(X^{(e)})\}^2\\
    &~~~~~~~~~~~~~~~~~~~~~~~~~~~~~ -\left\{ \{Y^{(e)} - \tilde{g}(X^{(e)})\}\tilde{f}^{(e)}(X^{(e)}) - \frac{1}{2}\{\tilde{f}^{(e)}(X^{(e)})\}^2 \right\} \Bigg]\\
    &~~~= \frac{1}{2}(\mathbb{E} - \hat{\mathbb{E}}) \left[\{Y^{(e)} - g(X^{(e)})\}^2 - \{Y^{(e)} - \tilde{g}^{(e)}\}^2\right] \\
    &~~~~~~~~~ + \frac{1}{2}(\mathbb{E} - \hat{\mathbb{E}}) \left[ \{Y^{(e)} - g^{(e)}(X^{(e)}) - {f}^{(e)}(X^{(e)})\}^2 - \{Y^{(e)} - \tilde{g}^{(e)}(X^{(e)}) - \tilde{f}^{(e)}(X^{(e)})\}^2\right] \\
    &~~~= \frac{1}{2}(\mathbb{E} - \hat{\mathbb{E}}) \left[\{\tilde{g}(X^{(e)}) - g(X^{(e)})\}^2\right] \\
    &~~~~~~~~~ + (\mathbb{E} - \hat{\mathbb{E}}) \left[Y^{(e)} \{\tilde{g}(X^{(e)}) - g(X^{(e)})\} \right] +  (\mathbb{E} - \hat{\mathbb{E}}) \left[- \tilde{g}(X^{(e)}) \{\tilde{g}(X^{(e)}) - g(X^{(e)})\} \right] \\
    &~~~~~~~~~ + \frac{1}{2}(\mathbb{E} - \hat{\mathbb{E}}) \left[\{\tilde{g}(X^{(e)}) + \tilde{f}^{(e)}(X^{(e)}) - g(X^{(e)}) - f^{(e)}(X^{(e)})\}^2\right] \\
    &~~~~~~~~~ + (\mathbb{E} - \hat{\mathbb{E}}) \left[Y^{(e)}\{\tilde{g}(X^{(e)}) + \tilde{f}^{(e)}(X^{(e)}) - g(X^{(e)}) - f^{(e)}(X^{(e)})\}\right] \\
    &~~~~~~~~~ + (\mathbb{E} - \hat{\mathbb{E}}) \left[-\{\tilde{g}(X^{(e)}) + \tilde{f}^{(e)}(X^{(e)})\}\{\tilde{g}(X^{(e)}) + \tilde{f}^{(e)}(X^{(e)}) - g(X^{(e)}) + f^{(e)}(X^{(e)})\}\right] \\
    &~~~= \sum_{k=1}^3 \mathsf{T}_{k}^{(e)}(\tilde{g}, g) + \sum_{k=4}^6 \mathsf{T}_{k}^{(e)}(\tilde{g}, g, \tilde{f}^{(e)}, f^{(e)}).
\end{align*}

We define
\begin{align*}
s_n = \delta_n + \sqrt{\frac{t + \log(nB|\mathcal{E}|) + 1}{n}}.
\end{align*} 

For $\mathsf{T}_{1}^{(e)}$ and $\mathsf{T}_{4}^{(e)}$, applying \cref{lemma:ep1} with (1) $v=1$ uniformly bounded by 1, (2) $\phi(\cdot)$ that is $4B$-Lipschitz by the boundedness of $\mathcal{G} \cup \mathcal{F}$, and (3) $\mathcal{H}=\mathcal{G}$ and $\mathcal{H}=\mathcal{G}+\mathcal{F}$ respectively, followed by using union bound across $e\in \mathcal{E}$, we obtain, the following two events
\begin{align*}
    \mathcal{C}_{1}(t) &= \left\{\forall e\in \mathcal{E}, ~\forall g, \tilde{g} \in \mathcal{G}, |\mathsf{T}_{1}^{(e)}(g,\tilde{g})| \le CB^2 (s_n\|g - \tilde{g}\|_{2,e} + s_n^2)\right\} \\
    \mathcal{C}_{4}(t) &= \Bigg\{\forall e\in \mathcal{E}, ~\forall (g, \tilde{g}, f^{(e)}, \tilde{f}^{(e)}) \in \mathcal{M}^{(e)}(\mathcal{G},\mathcal{F}), \\
    &~~~~~~~~~~~~~~~~ |\mathsf{T}_{4}^{(e)}(g,\tilde{g}, f^{(e)}, \tilde{f}^{(e)})| \le CB^2 (s_n\|\tilde{g} + \tilde{f}^{(e)} - g - f^{(e)}\|_{2,e} + s_n^2)\Bigg\}
\end{align*} both occur with probability at least $1-e^{-t}$ for some universal constant $C>0$.

For $\mathsf{T}_{2}^{(e)}$, we will use truncation argument. Let $K=\sigma\sqrt{102 \log(n|\mathcal{E}|)}$, we define
\begin{align*}
    \mathsf{T}_{2,1}^{(e)}(g,\tilde{g};K) &= (\mathbb{E} - \hat{\mathbb{E}}) \left[Y^{(e)} 1\{|Y^{(e)}| \le K\}\{\tilde{g}(X^{(e)}) - g(X^{(e)})\} \right] \\
    \mathsf{T}_{2,2}^{(e)}(g,\tilde{g};K) &= (\mathbb{E} - \hat{\mathbb{E}}) \left[Y^{(e)} 1\{|Y^{(e)}| > K\}\{\tilde{g}(X^{(e)}) - g(X^{(e)})\} \right]
\end{align*}
For the first part, applying \cref{lemma:ep1} with (1) $v$ uniformly bounded by $K$, (2) $\phi(t)=t$ that is $1$-Lipschitz and $\mathcal{H} = \mathcal{G}$ followed by applying union bound over all the $e\in \mathcal{E}$, we obtain
\begin{align*}
    \forall e\in \mathcal{E}, ~\forall g, \tilde{g} \in \mathcal{G}, ~
    |\mathsf{T}_{2,1}^{(e)}(g,\tilde{g};K)| \le CKB(s_n \|\tilde{g} - g\|_{2,e} + s_n^2)
\end{align*} occurs with probability at least $1-e^{-t}$ for some universal constant $C>0$. Moreover, following a similar argument to \eqref{eq:prop1-c12}, we find
\begin{align*}
    \forall e\in \mathcal{E}, ~\forall g, \tilde{g} \in \mathcal{G}, ~ |\mathsf{T}_{2,2}^{(e)}(g,\tilde{g};K)| \le \frac{44B}{n}
\end{align*} with probability at least $1-C_y(\sigma_y+1) n^{-100}$. Combining the two pieces together with triangle inequality $|\mathsf{T}_{2}^{(e)}| \le |\mathsf{T}_{2,1}^{(e)}(g,\tilde{g};K)| + |\mathsf{T}_{2,2}^{(e)}(g,\tilde{g};K)|$ yields, the event
\begin{align*}
    \mathcal{C}_{2}(t) = \{\forall e\in \mathcal{E}, ~\forall g, \tilde{g} \in \mathcal{G}, ~ |\mathsf{T}_{2,1}^{(e)}(g,\tilde{g};K)| \le C(K+1)B(s_n \|\tilde{g} - g\|_{2,e} + s_n^2)\}
\end{align*} occurs with probability at least $1-e^{-t}-C_y(\sigma_y+1) n^{-100}$. It follows almost the same procedure except that we apply \cref{lemma:ep1} with $\mathcal{H} = \mathcal{G} + \mathcal{F}$ that, with probability at least $1-e^{-t} - C_y(\sigma_y+1) n^{-100}$, the following event 
\begin{align*}
    \mathcal{C}_{5}(t) = \Big\{\forall e\in \mathcal{E}, ~ &\forall (g, \tilde{g}, f^{(e)}, \tilde{f}^{(e)}) \in \mathcal{M}^{(e)}(\mathcal{G},\mathcal{F}), \\&~~~~~~~~~~|\mathsf{T}_{5}^{(e)}(g,\tilde{g}, f^{(e)}, \tilde{f}^{(e)})| \le CBK (s_n\|\tilde{g} + \tilde{f}^{(e)} - g - f^{(e)}\|_{2,e} + s_n^2)\Big\}
\end{align*} occurs.

For $\mathsf{T}_{3}^{(e)}$ and $\mathsf{T}_{5}^{(e)}$, applying \cref{lemma:ep1} with (1) $v(h,h',z)=h'(x)$ that is uniformly bounded by $2B$, (2) $\phi(x)=x$ that is $1$-Lipschitz, and (3) $\mathcal{H} = \mathcal{G}$ and $\mathcal{H} = \mathcal{G} + \mathcal{F}$ respectively, followed by using union bound over all the $e\in \mathcal{E}$, we claim, the following two events,
\begin{align*}
    \mathcal{C}_{3}(t) &= \left\{\forall e\in \mathcal{E}, ~\forall g, \tilde{g} \in \mathcal{G}, |\mathsf{T}_{3}^{(e)}(g,\tilde{g})| \le CB^2 (s_n\|g - \tilde{g}\|_{2,e} + s_n^2)\right\} \\
    \mathcal{C}_{6}(t) &= \Big\{\forall e\in \mathcal{E}, ~\forall (g, \tilde{g}, f^{(e)}, \tilde{f}^{(e)}) \in \mathcal{M}^{(e)}(\mathcal{G},\mathcal{F}), \\
    &~~~~~~~~~~~~~~~~~~~~~~~~~|\mathsf{T}_{6}^{(e)}(g,\tilde{g}, f^{(e)}, \tilde{f}^{(e)})| \le CB^2 (s_n\|\tilde{g} + \tilde{f}^{(e)} - g - f^{(e)}\|_{2,e} + s_n^2)\Big\}
\end{align*} both occur with probability at least $1-e^{-t}$ for some universal constant $C>0$.

Combining all the pieces, we can conclude that, under the event $\cap_{k=1}^6\mathcal{C}_{k}(t)$, which occurs with probability at least $1-3e^{-t}-C_y(\sigma_y+1) n^{-100}$, the following error bound holds, that $\forall e\in \mathcal{E}$ and $\forall (g, \tilde{g}, f^{(e)}, \tilde{f}^{(e)}) \in \mathcal{M}(\mathcal{G}, \mathcal{F})$, 
\begin{align*}
    |\Delta_{\mathsf{A}}^{(e)}(g, \tilde{g}, f^{(e)}, \tilde{f}^{(e)})| &\le \sum_{k=1}^3 |\mathsf{T}_{k}^{(e)}(\tilde{g}, g)| + \sum_{k=4}^6 |\mathsf{T}_{k}^{(e)}(\tilde{g}, g, \tilde{f}^{(e)}, f^{(e)})| \\
    &\le CB(B+\sigma_y\sqrt{\log(n|\mathcal{E}|)}) \left(s_n \|\tilde{g} - g\|_{2,e} + s_n^2\right) \\
    &~~~~~~~~ + CB(B+\sigma_y\sqrt{\log(n|\mathcal{E}|)}) \left(s_n \|\tilde{g} + \tilde{f}^{(e)} - g - f^{(e)}\|_{2,e} + s_n^2\right)
\end{align*}
This completes the proof.

\qed

\subsection{Proof of Proposition \ref{prop:characterize-f}}

\label{subsec:proof-prop3}

We use the decomposition that, for any $g\in \mathcal{G}$, $f^{(e)} \in \mathcal{F}_{S_g}, \breve{f} \in \mathcal{F}_{S_g}$,
\begin{align*}
    \mathsf{A}^{(e)}(g, \breve{f}) - \mathsf{A}^{(e)}(g, f^{(e)}) =&~ \mathsf{A}^{(e)}(g, \breve{f}) - \hat{\mathsf{A}}^{(e)}(g, f^{(e)}) + \hat{\mathsf{A}}^{(e)}(g, \breve{f}) - \mathsf{A}^{(e)}(g, f^{(e)}) \\
    & ~~~~~~ + \hat{\mathsf{A}}^{(e)}(g, \breve{f}) - \hat{\mathsf{A}}^{(e)}(g, f^{(e)}) \\
    =&~ \Delta_{\mathsf{A}}^{(e)}(g, g, \breve{f}, f^{(e)}) + \hat{\mathsf{A}}^{(e)}(g, \breve{f}) - \hat{\mathsf{A}}^{(e)}(g, f^{(e)})
\end{align*} Applying the error bound in \cref{prop:nonasymptotic-a}, we have
\begin{align*}
    \mathsf{A}^{(e)}(g, \breve{f}) - \mathsf{A}^{(e)}(g, f^{(e)}) &\le |\Delta_{\mathsf{A}}^{(e)}(g, g, \breve{f}, f^{(e)})| + \hat{\mathsf{A}}^{(e)}(g, \breve{f}) - \hat{\mathsf{A}}^{(e)}(g, f^{(e)}) \\
    &\le \frac{\eta}{2} \|f^{(e)} - \breve{f}\|_{2,e}^2 + (\eta^{-1}/2+1)C^2U^2 s_n^2 + \hat{\mathsf{A}}^{(e)}(g, \breve{f}) - \hat{\mathsf{A}}^{(e)}(g, f^{(e)})
\end{align*} holds for any $g \in \mathcal{G}$ and $f^{(e)}, \breve{f} \in \mathcal{F}_{S_g}$, where the last inequality follows from the fact that $ab \le \frac{\eta}{2} a^2 + \frac{1}{2\eta} b^2$
On the other hand, according to the definition of $\mathsf{A}^{(e)}$, we also have
\begin{align*}
&\mathsf{A}^{(e)}(g, \breve{f}) - \mathsf{A}^{(e)}(g, f^{(e)}) \\
&~~~= \mathbb{E} \left[\{Y^{(e)} - g(X^{(e)})\} \{\breve{f}(X^{(e)}) - f^{(e)}(X^{(e)})\}\right] - \frac{1}{2} \|\breve{f}\|_{2,e}^2 + \frac{1}{2} \|f^{(e)}\|_{2,e}^2 \\
&~~~= \mathbb{E} \left[\{\Pi_{\overline{\mathcal{F}},e}(m^{(e,S)}) - g(X^{(e)})\} \{\breve{f}(X^{(e)}) - f^{(e)}(X^{(e)})\}\right] - \frac{1}{2} \|\breve{f}\|_{2,e}^2 + \frac{1}{2} \|f^{(e)}\|_{2,e}^2 \\
&~~~= \frac{1}{2} \|f^{(e)} - \breve{f}\|_{2,e}^2 - \int (f^{(e)} - \tilde{f})(\Pi_{\overline{\mathcal{F}_S}}^{(e)}(m^{(e,S)}) - g(X^{(e)}) - \breve{f})\mu^{(e)}(dx) \\
&~~~\ge \frac{1}{2} (1-\eta) \|f^{(e)} - \breve{f}\|_{2,e}^2 - \frac{\eta^{-1}}{2} \|\Pi_{\overline{\mathcal{F}_S}}^{(e)}(m^{(e,S)}) - g - \breve{f}\|_{2,e}^2
\end{align*} Combining the upper bound and lower bound on $\mathsf{A}^{(e)}(g, \breve{f}) - \mathsf{A}^{(e)}(g, f^{(e)})$ together, we find
\begin{align*}
    \frac{1}{2} (1-2\eta) \|f^{(e)} - \breve{f}\|_{2,e}^2 &\le (\eta^{-1}/2+1)C^2U^2 s_n^2 + \frac{\eta^{-1}}{2} \|\Pi_{\overline{\mathcal{F}_S}}^{(e)}(m^{(e,S)}) - g - \breve{f}\|_{2,e}^2 \\
    &~~~~~~~~ + (\hat{\mathsf{A}}^{(e)}(g, \breve{f}) - \hat{\mathsf{A}}^{(e)}(g, f^{(e)})).
\end{align*} Note that here $g$, $f^{(e)}, \breve{f} \in \mathcal{F}_{S_g}$ are both arbitrary, we let $\breve{f}$ to be that
\begin{align*}
\|\Pi_{\overline{\mathcal{F}_S}}^{(e)}(m^{(e,S)}) - g(X^{(e)}) - \breve{f}\|_{2,e} \le \delta_{\mathtt{a},\mathcal{G}, \mathcal{F}}(e,S_g) + \epsilon
\end{align*} for some $\epsilon>0$. Then, one has
\begin{align*}
\frac{1}{2} (1-2\eta) \|f^{(e)} - \breve{f}\|_{2,e}^2 &\le (\eta^{-1}/2+1)C^2 s_n^2 + \frac{\eta^{-1}}{2} (\delta_{\mathtt{a},\mathcal{G}, \mathcal{F}}(e,S_g) + \epsilon) \\
    &~~~~~~~~ + \sup_{f\in \mathcal{F}_{S_g}}(\hat{\mathsf{A}}^{(e)}(g, f) - \hat{\mathsf{A}}^{(e)}(g, f^{(e)}))
\end{align*} and thus
\begin{align*}
    \|f^{(e)} - \breve{f}\|_{2,e}^2 &\le \frac{\eta^{-1} + 2}{1-2\eta} C^2U^2 \delta_{n,t}^2 + \frac{\eta^{-1}}{1-2\eta} (\delta_{\mathtt{a},\mathcal{G}, \mathcal{F}}(e,S_g) + \epsilon) \\
    &~~~~~~~ + \frac{2}{1-2\eta} \sup_{f\in \mathcal{F}_{S_g}}(\hat{\mathsf{A}}^{(e)}(g, f) - \hat{\mathsf{A}}^{(e)}(g, f^{(e)})).
\end{align*} Then
\begin{align*}
    &\|\Pi_{\overline{\mathcal{F}_S}}^{(e)}(m^{(e,S)}) - g(X^{(e)}) - f^{(e)}\|_{2,e}^2 \\ 
    &~~~\le 2\|f^{(e)} - \breve{f}\|_{2,e}^2 + 
    2 \|\Pi_{\overline{\mathcal{F}_S}}^{(e)}(m^{(e,S)}) - g(X^{(e)}) - f^{(e)}\|_{2,e}^2 \\
    &~~~\le \frac{2\eta^{-1} + 4}{1-2\eta} C^2U^2 s_n^2 + \frac{2\eta^{-1}+2-4\eta}{1-2\eta}(\delta_{\mathtt{a},\mathcal{G}, \mathcal{F}}(e,S_g) + \epsilon) \\
    &~~~~~~~~~~~~~ + \frac{4}{1-2\eta} \sup_{f\in \mathcal{F}_{S_g}}(\hat{\mathsf{A}}^{(e)}(g, f) - \hat{\mathsf{A}}^{(e)}(g, f^{(e)})).
\end{align*} Letting $\epsilon \to 0$ completes the proof of the first claim. Averaging over all the $e\in \mathcal{E}$ then completes the proof of the second claim by noting that
\begin{align*}
    \frac{1}{|\mathcal{E}|} \sum_{e\in \mathcal{E}} \sup_{f\in \mathcal{F}_{S_g}} \hat{\mathsf{A}}^{(e)}(g, f) &= \sup_{\breve{f}^{\mathcal{E}}\in \{\mathcal{F}_{S_g}\}^{|\mathcal{E}|}} \frac{1}{|\mathcal{E}|} \sum_{e\in \mathcal{E}} \hat{\mathsf{A}}^{(e)}(g, \breve{f}^{(e)}) \\
    &= \sup_{\breve{f}^{\mathcal{E}}\in \{\mathcal{F}_{S_g}\}^{|\mathcal{E}|}} \gamma^{-1} \left(\hat{\mathsf{Q}}_\gamma(g, \breve{f}^{\mathcal{E}}) - \hat{\mathsf{R}}(g)\right),
\end{align*} and
\begin{align*}
    \frac{1}{|\mathcal{E}|} \sum_{e\in \mathcal{E}} \hat{\mathsf{A}}^{(e)}(g, f^{(e)}) = \gamma^{-1}\left(\hat{\mathsf{Q}}_\gamma(g, f^{\mathcal{E}}) - \hat{\mathsf{R}}(g)\right).
\end{align*}

\subsection{Proof of Proposition \ref{prop:variable-selection}}
\label{subsec:proof-prop-variable-selection}

Define the set of ``bad'' variable index set 
\begin{align*}
\overline{S} = \mathcal{S} \cup \left\{S: \mathsf{b}_{\mathcal{G}}(S) = 0, S\supseteq S^\star, \exists e\in \mathcal{E}, \|\Pi_{\overline{\mathcal{F}_S}}^{(e)}(m^{(e,S)}) - g^\star \|_{2,e} > 0\right\}
\end{align*} with
\begin{align*}
    \mathcal{S} &= \left\{S \subseteq [d]: \mathsf{b}_{\mathcal{G}}(S) > 0 ~\text{or}~ S^\star \setminus S \neq \emptyset \right\}.
\end{align*}

Define the event $\mathcal{A}_+$ as
\begin{align*}
&\forall g ~\text{with}~S_g \in \overline{\mathcal{S}}, ~~~~\exists f^{\mathcal{E}} \in \{\mathcal{F}_{S_g}\}^{|\mathcal{E}|} ~\text{and}~\tilde{g} ~\text{with}~ S_{\tilde{g}} = S^\star ~\text{s.t.}\\
&~~~~~~~~~~~~~~~~~~~~ \hat{\mathsf{Q}}_{\gamma}(g, f^{\mathcal{E}}) - \sup_{\tilde{f}^{\mathcal{E}} \in \{\mathcal{F}_{S^\star}\}^{|\mathcal{E}|}}\hat{\mathsf{Q}}_{\gamma}(\tilde{g}, \tilde{f}^{\mathcal{E}}) > 2(1+\gamma)\delta_{\mathtt{opt}}^2.
\end{align*}

Denote $\hat{S}=S_{\hat{g}}$. We claim that $\mathcal{A}_+ \subseteq \{\hat{S} \notin \overline{\mathcal{S}}\}$. This is because the solution of the minimax optimization objective $\hat{g}, \hat{f}^{\mathcal{E}}$ satisfies
\begin{align*}
    \hat{\mathsf{Q}}_{\gamma}(\hat{g}, \hat{f}^{\mathcal{E}}) \le \inf_{g\in \mathcal{G}} \sup_{f^{\mathcal{E}}\in\{\mathcal{F}_{S_g}\}^{|\mathcal{E}|}} \hat{\mathsf{Q}}_{\gamma}({g}, f^{\mathcal{E}}) + (1+\gamma)\delta_{\mathtt{opt}}^2 \le \sup_{\tilde{f}^{\mathcal{E}}\in\{\mathcal{F}_{S^\star}\}^{|\mathcal{E}|}} \hat{\mathsf{Q}}_{\gamma}(\tilde{g}, \tilde{f}^{\mathcal{E}}) + (1+\gamma)\delta_{\mathtt{opt}}^2
\end{align*} and
\begin{align*}
    \hat{\mathsf{Q}}_{\gamma}(\hat{g}, \hat{f}^{\mathcal{E}}) \ge \sup_{\breve{f}^{\mathcal{E}} \in \{\mathcal{F}_{S_g}\}^{|\mathcal{E}|}} \hat{\mathsf{Q}}_{\gamma}(\hat{g}, \breve{f}^{\mathcal{E}}) - \gamma\delta_{\mathtt{opt}}^2 \ge \hat{\mathsf{Q}}_{\gamma}(\hat{g}, f^{\mathcal{E}}) - (\gamma+1)\delta_{\mathtt{opt}}^2
\end{align*} simultaneously. Combining the above two inequalities indicates that $\hat{g}$ in the minimax optimization solution satisfies
\begin{align*}
\forall \tilde{g} \in \mathcal{G}, ~\forall f^{\mathcal{E}} \in \{\mathcal{F}_{\hat{S}}\}^{|\mathcal{E}|}, ~~ \hat{\mathsf{Q}}_{\gamma}(\hat{g}, f^{\mathcal{E}}) &\le \sup_{\tilde{f}^{\mathcal{E}}\in\{\mathcal{F}_{S^\star}\}^{|\mathcal{E}|}} \hat{\mathsf{Q}}_{\gamma}(\tilde{g}, \tilde{f}^{\mathcal{E}}) + 2(1+\gamma)\delta_{\mathtt{opt}}^2
\end{align*} Therefore, we can conclude that if $\mathcal{A}_+$ occurs, then $\hat{S} \notin \overline{\mathcal{S}}$.

The rest of the proof is to show that if condition \eqref{eq:main-result-faster-cond} holds with some large enough universal constant $C>0$, then $\mathcal{A}_+$ occurs. Such a condition is equivalent to
\begin{align}
\label{eq:thm1-cond-n}
    (1+\gamma) \left(\delta_{\mathtt{opt}}^2 + \sup_{S\subseteq[d]} \delta_{\mathtt{a}, \mathcal{F}, \mathcal{G}}(S)^2 +  \delta_{\mathtt{a}, \mathcal{G}}^2 + UB\delta_{n,t} \right) \le \left( {s_{\min}} \land \gamma \inf_{S: \bar{\mathsf{d}}_{\mathcal{G}, \mathcal{F}}(S) > 0} \bar{\mathsf{d}}_{\mathcal{G}, \mathcal{F}}(S)\right)/\breve{C}
\end{align} for some large $\breve{C}>0$. Without loss of generality, we assume $\delta_{n,t}<1$ because we can choose large enough $C$ in condition \eqref{eq:main-result-faster-cond}.

To this end, we use the decomposition that, for any $g\in \mathcal{G}$, $f^{\mathcal{E}} \in \{\mathcal{F}_{S_g}\}^{|\mathcal{E}|}$, $\tilde{g} \in \mathcal{G}$ with $S_{\tilde{g}} = S^\star$ and $\tilde{f}^{\mathcal{E}} \in \{\mathcal{F}_{S^\star}\}^{|\mathcal{E}|}$, the following holds, that
\begin{align*}
    &\hat{\mathsf{Q}}_\gamma(g, f^{\mathcal{E}}) - \hat{\mathsf{Q}}_\gamma(\tilde{g}, \tilde{f}^{\mathcal{E}}) \\
    &~~~=  \hat{\mathsf{Q}}_\gamma(g, f^{\mathcal{E}}) - {\mathsf{Q}}_\gamma(g, f^{\mathcal{E}}) + {\mathsf{Q}}_\gamma(\tilde{g}, \tilde{f}^{\mathcal{E}}) - \hat{\mathsf{Q}}_\gamma(\tilde{g}, \tilde{f}^{\mathcal{E}}) \\
    & ~~~~~~~~~~~ + {\mathsf{Q}}_\gamma(g, f^{\mathcal{E}}) -  {\mathsf{Q}}_\gamma(\tilde{g}, \tilde{f}^{\mathcal{E}}) \\
    &~~~= \Delta_{\mathsf{R}}(g, \tilde{g}) - \frac{1}{|\mathcal{E}|} \sum_{e\in \mathcal{E}} \Delta_{\mathsf{A}}(g, \tilde{g}, f^{(e)}, \tilde{f}^{(e)}) + {\mathsf{Q}}_\gamma(g, f^{\mathcal{E}}) -  {\mathsf{Q}}_\gamma(\tilde{g}, \tilde{f}^{\mathcal{E}})\\
    &~~~\ge -|\Delta_{\mathsf{R}}(g, \tilde{g})| - \frac{\gamma}{|\mathcal{E}|} \sum_{e\in \mathcal{E}} |\Delta_{\mathsf{A}}(g, \tilde{g}, f^{(e)}, \tilde{f}^{(e)})| + {\mathsf{Q}}_\gamma(g, f^{\mathcal{E}}) -  {\mathsf{Q}}_\gamma(\tilde{g}, \tilde{f}^{\mathcal{E}}) \\
    &~~~\overset{(a)}{\ge} - CU \frac{1}{|\mathcal{E}|} \sum_{e\in \mathcal{E}} \left\{\delta_{n,t} \|\tilde{g} - g\|_{2,e} + \delta_{n,t}^2\right\}\\
    & ~~~~~~~~~~~ - CU \frac{\gamma}{|\mathcal{E}|} \sum_{e\in \mathcal{E}} \left\{\delta_{n,t} \left(\|\tilde{g} - g\|_{2,e} + \|\tilde{g} + \tilde{f}^{(e)} - g - f^{(e)}\|_{2,e}\right) + \delta_{n,t}^2 \right\}\\
    & ~~~~~~~~~~~ + 0.25 \|g - \tilde{g}\|_2^2 +  \frac{\gamma}{4} \bar{\mathsf{d}}_{\mathcal{G}, \mathcal{F}}(S) + \frac{\gamma}{2} \|g - \Pi_{\overline{\mathcal{G}_S}}(\bar{m}^{(S)})\|_2^2\\
    &~~~~~~~~~~~ - \frac{\gamma}{2|\mathcal{E}|} \sum_{e\in \mathcal{E}}  \|f^{(e)} - \{\Pi_{\overline{\mathcal{F}_S}}^{(e)} (m^{(e,S)}) - g\}\|_{2,e}^2 - (2 + 0.5\gamma) \|\tilde{g} - g^\star\|_2^2.
\end{align*} Here in $(a)$ we apply \cref{thm:population} with $\delta=1/2$, \cref{prop:nonasymptotic-pooled} and \cref{prop:nonasymptotic-a}.

It follows from the uniform boundedness of $\mathcal{G} \cup \mathcal{F}$ that
\begin{align*}
    &\frac{1}{|\mathcal{E}|} \sum_{e\in \mathcal{E}} \left\{CU\delta_{n,t} \left(\|\tilde{g} - g\|_{2,e} + \|\tilde{g} + \tilde{f}^{(e)} - g - f^{(e)}\|_{2,e}\right) + CU \delta_{n,t}^2 \right\} \\
    \le& ~ \frac{1}{|\mathcal{E}|} \sum_{e\in \mathcal{E}} \left\{CU \delta_{n,t}^2 + C UB\delta_{n,t}\right\} \le CUB \delta_{n,t}
\end{align*} provided that $\delta_{n,t}<1$ and $B>1$, and
\begin{align*}
CU \frac{1}{|\mathcal{E}|} \sum_{e\in \mathcal{E}} \left\{\delta_{n,t} \|\tilde{g} - g\|_{2,e} + \delta_{n,t}^2\right\} \le C UB \{\delta_{n,t} + \delta_{n,t}^2\} \le C UB \delta_{n,t}
\end{align*}
Substituting the above inequality back, we have
\begin{align*}
    &\hat{\mathsf{Q}}_\gamma(g, f^{\mathcal{E}}) - \hat{\mathsf{Q}}_\gamma(\tilde{g}, \tilde{f}^{\mathcal{E}}) \\
    &~~~\ge 0.25 \|g - \tilde{g}\|_{2}^2 + \frac{\bar{\mathsf{d}}_{\mathcal{G}, \mathcal{F}}(S)}{4} - C(1 + \gamma)UB \delta_{n,t} \\
    &~~~~~~~~~ - \left\{\frac{\gamma}{2|\mathcal{E}|} \sum_{e\in \mathcal{E}}  \|f^{(e)} - \{\Pi_{\overline{\mathcal{F}},e} (m^{(e,S)}) - g\}\|_{2,e}^2 +  (2 + 0.5\gamma) \|\tilde{g} - g^\star\|_2^2\right\}
\end{align*} for arbitrary $\tilde{g} \in \mathcal{G}$ with $S_{\tilde{g}} = S^\star$ and $\tilde{f}^{\mathcal{E}} \in \{\mathcal{F}_{S^\star}\}^{|\mathcal{E}|}$. Now we choose all the $f^{(e)}$ satisfying
\begin{align*}
    \|f^{(e)} - \{\Pi_{\overline{\mathcal{F}_S}}^{(e)} (m^{(e,S)}) - g\}\|_{2,e}^2 \le \{\delta_{\mathtt{a}, \mathcal{F}, \mathcal{G}}(e, S)\}^2 + \epsilon,
\end{align*} and choose $\tilde{g}$ satisfying
\begin{align*}
    \|\tilde{g} - g^\star\|_2^2 \le \delta_{\mathtt{a}, \mathcal{G}}^2 + \epsilon
\end{align*} with arbitrary small $\epsilon>0$, then
\begin{align}
\label{eq:thm1-proof-eq1}
\begin{split}
&\hat{\mathsf{Q}}_\gamma(g, f^{\mathcal{E}}) - \hat{\mathsf{Q}}_\gamma(\tilde{g}, \tilde{f}^{\mathcal{E}}) - 2(1+\gamma)\delta_{\mathtt{opt}}^2 \\
&~~~\ge 0.05 \|g - \tilde{g}\|_{2}^2 + \frac{\gamma}{4} \bar{\mathsf{d}}_{\mathcal{G}, \mathcal{F}}(S) - 2(1+\gamma)\delta_{\mathtt{opt}}^2\\
&~~~~~~~~~~~ -\left\{C (1 + \gamma)UB \delta_{n,t} + (0.5\gamma+2) \delta_{\mathtt{a}, \mathcal{G}}^2  + \epsilon(0.5\gamma+2)\right\}\\
&~~~~~~~~~~~ - \gamma \delta_{\mathtt{a}, \mathcal{F}, \mathcal{G}}(S)
\end{split}
\end{align} for any $g\in \mathcal{G}$, $\tilde{f}^{\mathcal{E}} \in \mathcal{F}_{S^\star}$, $f^{\mathcal{E}}$ that depends on $g$, and some $\tilde{g}$ with $S_{\tilde{g}} = S^\star$.

We conclude the proof by showing that for any $g \in \mathcal{G}$ with $S=S_g \in \overline{\mathcal{S}}$, the R.H.S. of \eqref{eq:thm1-proof-eq1} is greater than $0$ under \eqref{eq:thm1-cond-n}, thus $\mathcal{A}_+$ holds. The proof is divided into three cases.

\noindent \emph{Case 1. $\mathsf{b}_{\mathcal{G}}(S) > 0$.} In this case, we have $\bar{\mathsf{d}}_{\mathcal{G}, \mathcal{F}}(S) > 0$ by (2) in \cref{cond:general-gf}. Then, we have 
\begin{align*}
\hat{\mathsf{Q}}_\gamma(g, f^{\mathcal{E}}) - \hat{\mathsf{Q}}_\gamma(\tilde{g}, \tilde{f}^{\mathcal{E}}) - 2(1+\gamma)\delta_{\mathtt{opt}}^2 \ge 0.25  \|g - \tilde{g}\|_{2}^2 + \frac{\gamma}{8} \bar{\mathsf{d}}_{\mathcal{G}, \mathcal{F}}(S)  \ge 0 + \frac{\gamma}{8} \bar{\mathsf{d}}_{\mathcal{G}, \mathcal{F}}(S) > 0
\end{align*}
by letting $\epsilon \to 0$ under the condition that
\begin{align*}
    2(1+\gamma)\delta_{\mathtt{opt}}^2 + \gamma \delta_{\mathtt{a}, \mathcal{F}, \mathcal{G}}(S) + (0.5\gamma+2) (\delta_{\mathtt{a}, \mathcal{G}}^2 + CUB\delta_{n,t})  \le  \gamma \frac{\bar{\mathsf{d}}_{\mathcal{G}, \mathcal{F}}(S)}{8}.
\end{align*} which is implied by \eqref{eq:thm1-cond-n} with large enough $\breve{C}$.

\noindent \emph{Case 2. $S^\star \setminus S\neq \emptyset$. } In this case, applying Young's inequality $2ab \ge -0.5a^2 - 2b^2$ gives
\begin{align*}
    \|g - \tilde{g}\|_2^2 \ge \|g - g^\star + g^\star - \tilde{g}\|_2^2 &\ge 0.5 \|g - g^\star\|_2^2 - \|g^\star - \tilde{g}\|_2^2 \\
    &\overset{(a)}{\ge} 0.5 s_{\min} - \delta_{\mathtt{a}, \mathcal{G}} - \epsilon.
\end{align*} Here $(a)$ follows from (2) in \cref{cond:general-gf}. Similarly, we have $\hat{\mathsf{Q}}_\gamma(g, f^{\mathcal{E}}) - \hat{\mathsf{Q}}_\gamma(\tilde{g}, \tilde{f}^{\mathcal{E}}) - 2(1+\gamma)\delta_{\mathtt{opt}}^2 \ge 0.25\cdot 0.5 s_{\min} > 0$ via letting $\epsilon \to 0$ provided
\begin{align*}
2(1+\gamma)\delta_{\mathtt{opt}}^2 + \gamma \delta_{\mathtt{a}, \mathcal{F}, \mathcal{G}}(S) + (0.5\gamma+2) (\delta_{\mathtt{a}, \mathcal{G}}^2 + CUB\delta_{n,t}) \le 0.5  s_{\min}.
\end{align*} And the above inequality holds under \eqref{eq:thm1-cond-n} with large enough $\breve{C}$.

\noindent \emph{Case 3. $S\in \overline{\mathcal{S}} \setminus {\mathcal{S}}$. } In this case, it follows from the definition of $\overline{\mathcal{S}}$ that 
\begin{align}
\label{eq:proof-thm1-trick1}
    \sum_{e\in \mathcal{E}} \|\Pi_{\overline{\mathcal{F}_S}}^{(e)} (m^{(e,S)}) - g^\star\|_{2,e}^2 > 0.
\end{align} At the same time, given that $\mathsf{b}_{\mathcal{G}}(S) = 0$ and $S\supseteq S^*$ because $S\notin \mathcal{S}$, we obtain
\begin{align*}
    0 = \mathsf{b}_{\mathcal{G}}(S) = \|\Pi_{\overline{\mathcal{G}_S}} (\bar{m}^{(S)}) - g^\star\|_{2}^2 = \frac{1}{|\mathcal{E}|}\sum_{e\in \mathcal{E}} \|\Pi_{\overline{\mathcal{G}_S}} (\bar{m}^{(S)}) - g^\star\|_{2,e}^2
\end{align*} indicating that $\Pi_{\overline{\mathcal{G}_S}} (\bar{m}^{(S)}) - g^\star = 0$ $\mu^{(e)}$-a.s. for any $e\in \mathcal{E}$. Therefore, we have
\begin{align*}
    \bar{\mathsf{d}}_{\mathcal{G}, \mathcal{F}}(S) = \frac{1}{|\mathcal{E}|} \sum_{e\in \mathcal{E}} \|\Pi_{\overline{\mathcal{G}_S}} (\bar{m}^{(S)}) - \Pi_{\overline{\mathcal{F}_S}}^{(e)}(m^{(e,S)})\|_{2,e}^2 \overset{(a)}{=} \frac{1}{|\mathcal{E}|} \sum_{e\in \mathcal{E}} \|g^\star - \Pi_{\overline{\mathcal{F}_S}}^{(e)}(m^{(e,S)})\|_{2,e}^2 \overset{(b)}{>} 0
\end{align*} Here (a) follows from the fact that $\Pi_{\overline{\mathcal{G}_S}} (\bar{m}^{(S)}) - g^\star = 0$ $\mu^{(e)}$-a.s., and (b) follows from \eqref{eq:proof-thm1-trick1}. Similar to case 1, we have $\hat{\mathsf{Q}}_\gamma(g, f^{\mathcal{E}}) - \hat{\mathsf{Q}}_\gamma(\tilde{g}, \tilde{f}^{\mathcal{E}}) - 2(1+\gamma)\delta_{\mathtt{opt}}^2 \ge 0.05 \|g - \tilde{g}\|_{2}^2 + \frac{\gamma}{8} \bar{\mathsf{d}}_{\mathcal{G}, \mathcal{F}}(S) \ge \frac{\gamma}{8} \bar{\mathsf{d}}_{\mathcal{G}, \mathcal{F}}(S) > 0$ by letting $\epsilon \to 0$ under the condition that
\begin{align*}
    2(1+\gamma)\delta_{\mathtt{opt}}^2 + \gamma \delta_{\mathtt{a}, \mathcal{F}, \mathcal{G}}(S) + (0.5\gamma + 2) (\delta_{\mathtt{a}, \mathcal{G}}^2 + CUB\delta_{n,t})  \le \gamma \frac{\bar{\mathsf{d}}_{\mathcal{G}, \mathcal{F}}(S)}{8}.
\end{align*} And the above inequality holds under \eqref{eq:thm1-cond-n} with large enough $\breve{C}$.

Combining the above three cases completes the proof.

\subsection{Proof of Proposition \ref{prop:bias}}

Define $\langle f, g\rangle = \int f(x) g(x) \bar{\mu}_x(dx)$
It follows from \cref{lemma:pooled-l2-diff} that for any $g, \tilde{g} \in \mathcal{G}_S$, 
\begin{align*}
    \mathsf{R}(g) - \mathsf{R}(\tilde{g}) = \frac{1}{2}\|g - \tilde{g}\|_2^2 - \langle g - \tilde{g}, \bar{m}^{(S)} - \tilde{g}\rangle
\end{align*} Observe that 
\begin{align*}
\langle g - \tilde{g}, \bar{m}^{(S)} - \tilde{g}\rangle = \langle g - \tilde{g}, \Pi_{\overline{\mathcal{G}_S}}(\bar{m}^{(S)}) - \tilde{g}\rangle 
\end{align*} by projection theorem \cref{thm:projection}. Applying Cauchy-Schwarz inequality, we obtain
\begin{align*}
    \mathsf{R}(g) - \mathsf{R}(\tilde{g}) \ge \frac{1}{4} \|g - \tilde{g}\|_2^2 - \|\tilde{g} - \Pi_{\overline{\mathcal{G}_S}}(\bar{m}^{(S)})\|_2^2.
\end{align*}
On the other hand, for $\hat{g}_{\mathtt{R}}$ and any $\tilde{g} \in G_S$, we have the following decomposition:
\begin{align*}
    \mathsf{R}(\hat{g}_{\mathtt{R}}) - \mathsf{R}(\tilde{g}) &= \mathsf{R}(\hat{g}_{\mathtt{R}}) - \hat{\mathsf{R}}(\hat{g}_{\mathtt{R}}) + \hat{\mathsf{R}}(\hat{g}_{\mathtt{R}}) - \hat{\mathsf{R}}(\tilde{g}) + \hat{\mathsf{R}}(\tilde{g}) - \mathsf{R}(\tilde{g}) \\
    &\overset{(a)}{\le} - \Delta_{\mathsf{R}}(\hat{g}_{\mathtt{R}}, \tilde{g}) \\
    &\overset{(b)}{\le} C U \left\{ \delta_{n,t}^2 + \delta_{n,t} \frac{1}{|\mathcal{E}|} \sum_{e\in \mathcal{E}} \|\hat{g}_{\mathtt{R}} - \tilde{g}\|_{2,e}\right\} \le \{CU + 4(CU)^2\} \delta_{n, t}^2 + \frac{1}{8} \|\hat{g}_{\mathtt{R}} - \tilde{g}\|_2^2
\end{align*} where $(a)$ follows from the fact that $\hat{g}_{\mathtt{R}}$ is the empirical risk minimizer, $(b)$ follows from \cref{prop:nonasymptotic-pooled}. Putting the upper bound and lower bound on $\mathsf{R}(\hat{g}_{\mathtt{R}}) - \mathsf{R}(\tilde{g})$ together, we obtain, with probability at least $1-\{C_y(\sigma_y + 1)+1\}n^{-100}$, 
\begin{align*}
    \|\hat{g}_{\mathtt{R}} - \tilde{g}\|_{2} \le C_1 \left(U\delta_{n, 100\log n} + \|\tilde{g} - \Pi_{\overline{\mathcal{G}_S}}(\bar{m}^{(S)})\|_2\right).
\end{align*} We let $\delta_a = \inf_{g\in \mathcal{G}_S} \|g - \Pi_{\overline{\mathcal{G}_{S}}}(\bar{m}^{(S)})\|_2$. By definition, there exists some $\tilde{g} \in \mathcal{G}_S$ such that $\|\tilde{g} - \Pi_{\overline{\mathcal{G}_{S}}}(\bar{m}^{(S)})\|_2 \le \delta_a + \frac{1}{n}$. It then follows from triangle inequality and our choice of $\tilde{g}$ above that
\begin{align*}
    \|\hat{g}_{\mathtt{R}} - \Pi_{\overline{\mathcal{G}_{S}}}(\bar{m}^{(S)})\|_2 \le \|\Pi_{\overline{\mathcal{G}_{S}}}(\bar{m}^{(S)}) - \tilde{g}\|_2 + \|\tilde{g} - \hat{g}_{\mathtt{R}}\|_2
    &\le 200 C_1 (U\delta_{n, \log n} + \delta_a + n^{-1}).
\end{align*}
Meanwhile, it follows from Cauchy-Schwarz inequality that
\begin{align*}
    \|\hat{g}_{\mathtt{R}} - g^\star\|_2^2 &= \|\hat{g}_{\mathtt{R}} - \Pi_{\overline{\mathcal{G}_{S}}}(\bar{m}^{(S)}) + \Pi_{\overline{\mathcal{G}_{S}}}(\bar{m}^{(S)}) - g^\star\|_2^2 \\
    &= \mathsf{b}_{\mathcal{G}}(S) + \|\hat{g}_{\mathtt{R}} - \Pi_{\overline{\mathcal{G}_{S}}}(\bar{m}^{(S)})\|_2^2 + 2\langle \hat{g}_{\mathtt{R}} - \Pi_{\overline{\mathcal{G}_{S}}}(\bar{m}^{(S)}), \Pi_{\overline{\mathcal{G}_{S}}}(\bar{m}^{(S)}) - g^\star\rangle\\
    &\ge (1-\eta) \mathsf{b}_{\mathcal{G}}(S) - (\eta^{-1}-1) \|\hat{g}_{\mathtt{R}} - \Pi_{\overline{\mathcal{G}_{S}}}(\bar{m}^{(S)})\|_2^2
\end{align*} for any $\eta \in (0,1)$. Since $\delta_{a} + U\delta_{n, \log n}=o(1)$, setting $\eta = 0.005$ gives
\begin{align*}
\|\hat{g}_{\mathtt{R}} - g^\star\|_2^2 \ge 0.995 \mathsf{b}_{\mathcal{G}}(S) - o(1).
\end{align*} dividing both sides by positive $\mathsf{b}_{\mathcal{G}}(S)$ completes the proof of the lower bound provided $n$ is large enough. The upper bound follows similarly.

\subsection{Technical Lemmas}

\begin{lemma}[Instance-dependent error bound on empirical process]
\label{lemma:ep1}
Suppose the function class $\mathcal{H}$ satisfies $\sup_{h\in \mathcal{H}} \|h\|_\infty \le b$, and for any $\delta \ge \delta_n \ge 1/n$, the local population Rademacher complexity satisfies
\begin{align*}
    R_{n, \nu}(\delta;\partial \mathcal{H}) \le b\delta_n \delta
\end{align*} and the function $\Phi(h, h', z): \mathcal{H} \times \mathcal{H} \times \mathcal{Z}$ satisfies that, $\nu$-a.s.,
\begin{align*}
    \Phi(h, h', Z) = v(h, h', Z) \phi(h - h') ~~~~\text{with}~~~~ |v(h, h', z)|\le L_1, \phi\text{ is } L_2 \text{-Lipschitz and } \phi(0)=0.
\end{align*}
Then let $\delta_* = \delta_n + \sqrt{\frac{t+1+\log(nb)}{n}}$
\begin{align*}
    \mathbb{P}\Bigg[\forall h,h'\in \mathcal{H}, ~~ &\Big| \frac{1}{n} \sum_{i=1}^n \Phi(h, h', Z_i) - \mathbb{E}[\Phi(h, h', Z_i)] \Big| \\
    &~~~~~~~~~~\le C(bL_1L_2)\{\delta_* \|h-h'\|_{L_2(\nu)} + \delta_*^2\} \Bigg] \ge 1-e^{-t}.
\end{align*} for some universal constant $C>0$.
\end{lemma}
\begin{proof}[Proof of \cref{lemma:ep1}]
    Throughout the proof, we let $\|\cdot\|_2 = \|\cdot\|_{L_2(\nu)}$. Define 
    \begin{align*}
        Z_n(\delta) := \sup_{h, h'\in \mathcal{H}, \|h-h'\|_2 \le \delta} \Big| \frac{1}{n} \sum_{i=1}^n \Phi(h, h', Z_i) - \mathbb{E}[\Phi(h, h', Z_i)] \Big|
    \end{align*}
    \noindent {\sc Step 1. Bound on $\mathbb{E}[Z_n(\delta)]$.} It follows from the symmetrization argument that
    \begin{align*}
        \mathbb{E}[Z_n(\delta)] \le 2 \mathbb{E}\left[\sup_{h, h'\in \mathcal{H}, \|h-h'\|_2 \le \delta} \Big| \frac{1}{n} \sum_{i=1}^n \varepsilon_i \Phi(h, h', Z_i) \Big|\right]
    \end{align*} for i.i.d. Rademacher random variables $\varepsilon_1,\ldots, \varepsilon_n$ that is also independent of $Z_1, \ldots, Z_n$. We claim that
    \begin{align}
    \label{eq:proof-claim-ep1-1}
        \mathbb{E}\left[\sup_{h, h'\in \mathcal{H}, \|h-h'\|_2 \le \delta} \Big| \frac{1}{n} \sum_{i=1}^n \varepsilon_i \Phi(h, h', Z_i) \Big|\right] \le 2L_1 L_2 R_{n, \nu}(\delta;\partial \mathcal{H})
    \end{align} using a similar argument w.r.t. Talagrand contraction inequality. To this end, let $\epsilon>0$, $ \mathcal{H}_{\partial, \delta} = \{(h, h')\in \mathcal{H}, \|h-h'\|_2\le \delta\}$, and
    \begin{align*}
        \mathsf{T}_m(h,h') = \frac{1}{n} \sum_{i=1}^{m} \varepsilon_i \Phi(h, h', Z_i).
    \end{align*} For fixed $\varepsilon_1,\ldots, \varepsilon_{n-1}$ and $Z_1,\ldots, Z_n$, let $h_{+}, h'_{+}$ be such that
    \begin{align*}
        \mathsf{T}_{n-1}(h_+, h_+') + \frac{1}{n} \Phi(h_+, h_+', Z_i) \ge \sup_{(h, h')\in \mathcal{H}_{\partial, \delta}} \mathsf{T}_{n-1}(h,h') + \frac{1}{n} \Phi(h, h', Z_i) - \epsilon,
    \end{align*} and $h_{-}$ and $h'_{-}$ be such that
    \begin{align*}
        \mathsf{T}_{n-1}(h_-, h_-') - \frac{1}{n} \Phi(h_-, h_-', Z_i) \ge \sup_{(h, h')\in \mathcal{H}_{\partial, \delta}} \mathsf{T}_{n-1}(h, h') - \frac{1}{n} \Phi(h, h', Z_i) - \epsilon.
    \end{align*} Then given the definition of Rademacher random variable that $\varepsilon_n \sim \mathrm{Unif}\{-1,+1\}$, 
    \begin{align*}
        &\mathbb{E}\left[\sup_{(h,h')\in\mathcal{H}_{\partial, \delta}}  \mathsf{T}_n(h,h') \Big|Z_{1}^n, \varepsilon_{1}^{n-1}\right]\\ 
        = &\frac{1}{2} \mathbb{E}\left[\sup_{(h,h')\in \mathcal{H}_{\partial, \delta}} \mathsf{T}_{n-1}(h,h') + \frac{1}{n} \Phi(h, h', Z_n) \Bigg|Z_1^n,\varepsilon_1^{n-1}\right]\\
        & ~~~~~~ + \frac{1}{2} \mathbb{E}\left[\sup_{(h,h')\in \mathcal{H}_{\partial, \delta}} \left|\mathsf{T}_{n-1}(h,h')- \frac{1}{n} \Phi(h, h', Z_n)\right| \Big|Z_1^n,\varepsilon_1^{n-1}\right] \\
        \le & \frac{1}{2} \mathbb{E} \left[\mathsf{T}_{n-1}(h_+, h_+') + \frac{1}{n} \Phi(h_+, h_+', Z_n)| \Big| Z_1^n,\varepsilon_1^{n-1}\right] \\
        & ~~~~~~ + \frac{1}{2} \mathbb{E}\left[\mathsf{T}_{n-1}(h_-,h_-')- \frac{1}{n} \Phi(h_-, h_-', Z_n) \Bigg|Z_1^n,\varepsilon_1^{n-1}\right] + \epsilon \\
        &= \frac{1}{2} \left\{\mathsf{T}_{n-1}(h_+, h_+') + \mathsf{T}_{n-1}(h_-, h_-') \right\} + \frac{1}{2}\left\{\frac{1}{n} \Phi(h_+, h_+', Z_n) - \frac{1}{n} \Phi(h_-, h_-', Z_n)\right\} + \epsilon
    \end{align*}
    Applying the condition on $\Phi$ gives
    \begin{align*}
        |\Phi(h_+, h_+', Z_n) - \Phi(h_-, h_-', Z_n)| \le \text{sign}(h_+-h_+'-\{h_--h_-'\}) (L_1 L_2) \left\{(h_+-h_+')-(h_--h_-')\right\},
    \end{align*} now that $u = \text{sign}(h_+-h_+'-\{h_--h_-'\}) \in \{-1,+1\}$ and only depends on $h_+, h_+', h_-, h_-'$, then
    \begin{align*}
        &\frac{1}{2} \left\{\mathsf{T}_{n-1}(h_+, h_+') + \mathsf{T}_{n-1}(h_-, h_-') \right\} + \frac{1}{2}\left\{\frac{1}{n} \Phi(h_+, h_+', Z_n) - \frac{1}{n} \Phi(h_-, h_-', Z_n)\right\} \\
        \le & \frac{1}{2} \left\{\mathsf{T}_{n-1}(h_+, h_+') + \mathsf{T}_{n-1}(h_-, h_-') \right\} + \frac{u}{2n} (L_1L_2) (h_+-h_+') - \frac{u}{2n} (L_1L_2) (h_--h_-')\\
        \overset{(a)}{\le} & \sup_{h,h'\in \mathcal{H}_{\partial, \delta}} \frac{1}{2} \left\{\mathsf{T}_{n-1}(h, h') + \frac{(L_1L_2)}{n} (h-h')\right\} + \sup_{h,h'\in \mathcal{H}_{\partial, \delta}} \frac{1}{2} \left\{\mathsf{T}_{n-1}(h, h') - \frac{(L_1L_2)}{n} (h-h')\right\} \\
        \le & \mathbb{E}\left[\sup_{(h,h')\in\mathcal{H}_{\partial, \delta}}  \mathsf{T}_{n-1}(h,h') + \frac{1}{n} (L_1L_2) \varepsilon_n(h-h')(Z_n) \Big|Z_{1}^n, \varepsilon_{1}^{n-1}\right]
    \end{align*} where $(a)$ follows from a discussion on whether $u=+1$ or $u=-1$: if $u=+1$, then 
    \begin{align*}
    &\frac{1}{2} \left\{\mathsf{T}_{n-1}(h_+, h_+') + \mathsf{T}_{n-1}(h_-, h_-') \right\} + \frac{u}{2n} (L_1L_2) (h_+-h_+') - \frac{u}{2n} (L_1L_2) (h_--h_-') \\
    =& \frac{1}{2} \left\{\mathsf{T}_{n-1}(h_+, h_+') + \frac{1}{2n} (L_1L_2) (h_+-h_+') \right\} + \left\{\mathsf{T}_{n-1}(h_-, h_-') - \frac{1}{2n} (L_1L_2) (h_--h_-') \right\} \\
    \le& \mathrm{R.H.S.~of~} (a).
    \end{align*} Meanwhile, if $u=-1$, then
    \begin{align*}
    &\frac{1}{2} \left\{\mathsf{T}_{n-1}(h_+, h_+') + \mathsf{T}_{n-1}(h_-, h_-') \right\} + \frac{u}{2n} (L_1L_2) (h_+-h_+') - \frac{u}{2n} (L_1L_2) (h_--h_-') \\
    =& \frac{1}{2} \left\{\mathsf{T}_{n-1}(h_+, h_+') - \frac{1}{2n} (L_1L_2) (h_+-h_+') \right\} + \left\{\mathsf{T}_{n-1}(h_-, h_-') + \frac{1}{2n} (L_1L_2) (h_--h_-') \right\} \\
    \le& \mathrm{R.H.S.~of~} (a).
    \end{align*}
    Applying the tower rule of conditional expectation to both sides of the inequality we obtained so far, that 
    \begin{align*}
        &\mathbb{E}\left[\sup_{(h,h')\in\mathcal{H}_{\partial, \delta}}  \mathsf{T}_n(h,h') \Big|Z_{1}^n, \varepsilon_{1}^{n-1}\right] \\
        &~~~~~~~~ \le \mathbb{E}\left[\sup_{(h,h')\in\mathcal{H}_{\partial, \delta}}  \mathsf{T}_{n-1}(h,h') + \frac{1}{n} \varepsilon_n (L_1L_2) (h-h')(Z_n) \Big|Z_{1}^n, \varepsilon_{1}^{n-1}\right] + \epsilon
    \end{align*} and letting $\epsilon\to 0$, we find
    \begin{align*}
        \mathbb{E}\left[\sup_{(h,h')\in\mathcal{H}_{\partial, \delta}}  \mathsf{T}_n(h,h')\right] \le \mathbb{E}\left[\sup_{(h,h')\in\mathcal{H}_{\partial, \delta}}  \mathsf{T}_{n-1}(h,h') + \frac{L_1L_2}{n} \varepsilon_n (h-h')(Z_n)\right].
    \end{align*} Iteratively applying this yields
    \begin{align*}  
    \mathbb{E}\left[\sup_{(h,h')\in\mathcal{H}_{\partial, \delta}}  \mathsf{T}_n(h,h')\right]\le L_1L_2 \mathbb{E}\left[\sup_{(h,h')\in\mathcal{H}_{\partial, \delta}}  \frac{1}{n} \sum_{i=1}^n \varepsilon_i(h-h')(Z_i)\right] \le (L_1L_2) R_{n,\nu}(\delta; \partial \mathcal{H})
    \end{align*} Following a similar argument, we can also obtain $\mathbb{E}[\sup_{(h,h')\in\mathcal{H}_{\partial, \delta}}  -\mathsf{T}_n(h,h')] \le (L_1L_2) R_{n,\nu}(\delta; \partial \mathcal{H})$. Combining with the fact that $\sup_{a}|v(a)| \le \sup_{a} v(a) + \sup_{a} -v(a)$ completes the proof of the claim \eqref{eq:proof-claim-ep1-1}. Then, we can upper bound $\mathbb{E}[Z_n(\delta)]$ as
    \begin{align*}
    \mathbb{E}[Z_n(\delta)] \le 4(L_1L_2) R_{n, \nu}(\delta;\partial \mathcal{H}).
    \end{align*}

    \noindent {\sc Step 2. Establish a high probability bound on $Z_n(\delta)$.} Given that for any $\delta \ge \delta_n$, 
    \begin{align*}
        \mathbb{E}[Z_n(\delta)] \le 4(L_1L_2) R_{n, \nu}(\delta;\partial \mathcal{H}) \le 4 \delta_n \delta b L_1L_2
    \end{align*} and
    It follows from the facts that $h$ is upper bounded by $b$ and $|v(h,h',Z)| \le L_1$ and $\phi(x) \le L_2 |x|$ that
    \begin{align*}
        \sigma^2 &:= \sup_{h,h'\in \mathcal{H}_{\partial, \delta}} \mathbb{E} (\Phi - \mathbb{E}\Phi)^2 \le \sup_{h,h'\in \mathcal{H}_{\partial, \delta}} \mathbb{E} (\Phi)^2 \le \sup_{h,h'\in \mathcal{H}_{\partial, \delta}} \mathbb{E} [L_1^2L_2^2(h-h')^2] \le (L_1L_2)^2 \delta^2\\
        \mathsf{U} &:= \sup_{h,h'\in \mathcal{H}_{\partial, \delta}} \|\Phi - \mathbb{E}[\Phi]\|_\infty \le 2L_1L_2 b \\
        \mathsf{E} &:= \mathbb{E}[Z_n(\delta)] \le 4\delta_n \delta b L_1 L_2.
    \end{align*} 
    Combining the above facts and applying Talagrand inequality for the supremum of empirical process, we have 
    \begin{align*}
    \forall x>0, \qquad \mathbb{P}\left[Z_n(\delta) \ge \mathsf{E} + \sqrt{2(\sigma^2 + \mathsf{U}\mathsf{E}) x} + \frac{\mathsf{U}x}{3}\right] \le e^{-nx}.
    \end{align*} Therefore, we can conclude that, for all $\delta > \delta_n$ and $u>0$,
\begin{align*}
    \mathbb{P}\left[Z_n(\delta) \ge (L_1L_2) \delta \left(12 b\delta_n + \sqrt{\frac{u}{n}}\right) + b(L_1L_2) \frac{u}{n}\right] \ge e^{-u}.
\end{align*}

    \noindent {\sc Step 3. Apply peeling device.} 
    Define
    \begin{align*}
        \mathcal{A}_k = \{(h, h')\in \partial \mathcal{H}, a_{k-1} \delta_n \le \|h - h'\|_2 \le a_k \delta_n\}
    \end{align*} with $a_k = 2^{k-1}$ for $k \ge 1$ and $a_0 = 0$. Then, let $\mathfrak{C}=24bL_1L_2$, and $u=\log(2\log(4b/\delta_n))+t$
    \begin{align*}
        &\mathbb{P}\left[\exists h,h'\in \mathcal{H}, ~ \left| (\hat{\mathbb{E}}-\mathbb{E})[\Phi] \right| > \mathfrak{C} \left(
        (\delta_n + \sqrt{u/n}) (\|h-h'\|_{2} + \delta_n) + \frac{u}{n}\right) \right] \\
        \le& \sum_{k=1}^{\lceil \log_2(2b/\delta_n) \rceil} \mathbb{P}\left[\exists (h, h')\in \mathcal{A}_k, ~\left| (\hat{\mathbb{E}}-\mathbb{E})[\Phi] \right| > \mathfrak{C}\left((\delta_n + \sqrt{u/n})(\|h - h'\|_2 + \delta_n) + \frac{t}{n}\right)\right] \\
        \le& \sum_{k=1}^{\lceil \log_2(2b/\delta_n) \rceil} \mathbb{P}\left[Z_n(a_k \delta_n) > \mathfrak{C}\left((\delta_n + \sqrt{u/n})(a_{k-1} + 1) \delta_n + \frac{u}{n}\right)\right] \\
        \le& \sum_{k=1}^{\lceil \log_2(2b/\delta_n) \rceil} \mathbb{P}\left[Z_n(a_k \delta_n) > \frac{\mathfrak{C}}{2} a_k \delta_n(\delta_n + \sqrt{u/n}) + \mathfrak{C}\frac{u}{n}\right] \\
        \le& 2\log(4b/\delta_n)e^{-u} \le e^{-t}
    \end{align*} where the first inequality follows from union bound and the fact that $\cup_{k=1}^{\lceil \log_2(2b/\delta_n) \rceil}\mathcal{A}_k = \partial \mathcal{H}$, the second inequality follows from the definition of $\mathcal{A}_k$, the third inequality follows from the fact that $a_{k-1} + 1 \ge \frac{1}{2} a_k$. This completes the proof by the fact that $\log(2\log(4b/\delta_n)) \le C(\log(nb) + 1)$ for some universal constant $C>0$.
    
\end{proof}

\begin{theorem}[Projection Theorem in Hilbert Space]
\label{thm:projection}
Suppose $\mathcal{M}$ is a closed subspace of the Hilbert space $\mathcal{H}$ equipped with the inner product $\langle \cdot, \cdot \rangle$. Then for any $h\in \mathcal{H}$, there exists a unique $a \in \mathcal{M}$ such that
\begin{itemize}
    \item[1.] $\|h-a\|_2^2 = \inf_{y\in \mathcal{M}} \|h - y\|_2^2$, where $\|x\|_2^2 = \langle x, x\rangle$.
    \item[2.] $\langle h - a, y\rangle = 0$ for any $y\in \mathcal{M}$.
\end{itemize}
\end{theorem}

\begin{lemma}
\label{lemma:pooled-l2-diff}
For any $g, \tilde{g} \in \Theta$, we have
\begin{align*}
    \frac{1}{|\mathcal{E}|} \sum_{e\in \mathcal{E}} \mathbb{E}\left[ \{Y^{(e)} - g(X^{(e)})\}^2 - \{Y^{(e)} - \tilde{g}(X^{(e)})\}^2 \right] = \|g - \tilde{g}\|_2^2 - 2\langle g - \tilde{g}, \bar{m}^{(S_g\cup S_{\tilde{g}})} - \tilde{g}\rangle_{\bar{\mu}_x},
\end{align*} where $\langle h, g \rangle_{\nu} = \int h(x) g(x) \nu(dx)$.
\end{lemma}

\begin{proof}[Proof of \cref{lemma:pooled-l2-diff}]
It follows from the fact that $(a - b)^2 - (a - c)^2 = -2(a - c)(b - c)+(b - c)^2$ with $a = y$, $b = g$ and $c=\tilde{g}$ that
\begin{align*}
    \left\{(y - g(x))^2 - (y - \tilde{g}(x))^2 \right\} = - 2(y - \tilde{g}(x))(g(x) - \tilde{g}(x)) + ({g}(x) - \tilde{g}(x))^2.
\end{align*} This implies that 
\begin{align*}
     \mathsf{T}(g,\tilde{g}):=&~\frac{1}{|\mathcal{E}|} \sum_{e\in \mathcal{E}} \mathbb{E} \left[(Y^{(e)}- g(X^{(e)}))^{2} - (Y^{(e)} - \tilde{g}(X^{(e)}))^2\right] \\
    = &~\frac{1}{|\mathcal{E}|} \sum_{e\in \mathcal{E}} \Bigg\{2\mathbb{E}\left[ \{Y^{(e)} - \tilde{g}(X^{(e)})\} \{g(X^{(e)}) - \tilde{g}(X^{(e)})\}\right] + \|\tilde{g} - g\|_{2,e}^2 \Bigg\} 
\end{align*} 
It then follows from the fact $\|\cdot\|_2 = \frac{1}{|\mathcal{E}|}\sum_{e\in \mathcal{E}} \|\cdot\|_{2,e}$ that
\begin{align}
\mathsf{T}(g, \tilde{g}) = \|g - \tilde{g}\|_2^2 + \underbrace{\frac{1}{|\mathcal{E}|}\sum_{e\in \mathcal{E}} \mathbb{E}\left[ \{Y^{(e)} - \tilde{g}(X^{(e)})\} \{g(X^{(e)}) - \tilde{g}(X^{(e)})\}\right]}_{\mathsf{T}_{1}(g, \tilde{g})}
\end{align}
Denote $\overline{S} = S\cup \tilde{S}$, where $S=S_g$ and $\tilde{S}=S_{\tilde{g}}$, it follows from the tower rule of conditional expectation that
\begin{align*}
    &\mathsf{T}_{1}(g, \tilde{g}) \\
    =&\frac{1}{|\mathcal{E}|} \sum_{e\in \mathcal{E}} \mathbb{E}\Big[\{Y^{(e)} - \tilde{g}(X^{(e)})\} \{g(X^{(e)}) - \tilde{g}(X^{(e)})\} \big|X_{\overline{S}}\big]\Big] \\
    =&\frac{1}{|\mathcal{E}|} \sum_{e\in \mathcal{E}} \int \left\{\tilde{g}_{\tilde{S}}(x_{\tilde{S}}) - m^{(e,\overline{S})}(x_{\overline{S}})\right\} \left\{g_S(x_S) - \tilde{g}_{\tilde{S}}(x_{\tilde{S}})\right\} \mu^{(e)}_{x,\overline{S}}(dx_{\overline{S}}) \\
    =&\frac{1}{|\mathcal{E}|} \sum_{e\in \mathcal{E}} \int \left\{\tilde{g}_{\tilde{S}}(x_{\tilde{S}}) - m^{(e,\overline{S})}(x_{\overline{S}})\right\} \rho^{(e)}_{\overline{S}}(x_{\overline{S}}) \left\{g_S(x_S) - \tilde{g}_{\tilde{S}}(x_{\tilde{S}})\right\} \bar{\mu}_{x,\overline{S}}(dx_{\overline{S}}) \\
    =&\int \left\{\tilde{g}_{\tilde{S}}(x_{\tilde{S}}) \frac{1}{|\mathcal{E}|}\sum_{e\in \mathcal{E}} \rho_{\overline{S}}^{(e)}(x_{\overline{S}}) - \frac{1}{|\mathcal{E}|}\sum_{e\in \mathcal{E}} \rho_{\overline{S}}^{(e)}(x_{\overline{S}}) m^{(e,\overline{S})}(x_{\overline{S}})\right\} \left\{g_S(x_S) - \tilde{g}_{\tilde{S}}(x_{\tilde{S}})\right\} \bar{\mu}_{x,\overline{S}}(dx_{\overline{S}}) \\
    =&\int \left\{ \tilde{g}_{\tilde{S}}(x_{\tilde{S}}) - \bar{m}^{(e,\overline{S})}(x_{\overline{S}}) \right\} \left\{g_S(x_S) - \tilde{g}_{\tilde{S}}(x_{\tilde{S}})\right\} \bar{\mu}_{x,\overline{S}}(dx_{\overline{S}}) \\
    =& \langle \tilde{g} - \bar{m}^{(\overline{S})}, g - \tilde{g}\rangle_{\bar{\mu}_x}.
\end{align*} This completes the proof.
\end{proof}

\subsection{Proof of Theorem \ref{thm:oracle-2}}

The proof is similar to the proof of \cref{thm:oracle}, we only highlight the differences. We will use the following two propositions to (1) establish approximate strong convexity w.r.t. $m^\star$, and (2) establish instance-dependent error bound for pooled risks, respectively, which is similar to \cref{thm:population} and \cref{prop:nonasymptotic-pooled}, the proofs can be found in \cref{sec:proof-prop-general-loss}.

\begin{proposition}
\label{prop:population-general-loss}
Assume \cref{cond:general-loss}--\ref{cond:general-gf-2} hold. Let $\delta \in (0,1)$ be arbitrary. Then the following holds, for any $\gamma \ge 4\delta^{-1} \zeta^2 \gamma^\star$, 
    \begin{align*}
        \mathsf{Q}_\gamma(g, f^{\mathcal{E}}) - \mathsf{Q}_\gamma(\tilde{g}, \tilde{f}^{\mathcal{E}}) \ge& \frac{1-\delta}{2} \zeta^{-1} \|g - \tilde{g}\|_2^2 +  \frac{\gamma}{4} \bar{\mathsf{d}}(S) + \frac{\gamma}{2} \|g - \bar{m}^{(S)}\|_2^2\\
        &~~~~~~ - \frac{\gamma}{2|\mathcal{E}|} \sum_{e\in \mathcal{E}}  \|f^{(e)} - \{m^{(e,S)} - g\}\|_{2,e}^2 - (\delta^{-1}\zeta^2 + \gamma/2) \|\tilde{g} - g^\star\|_2^2  
    \end{align*} for any $g\in \mathcal{G}$, $\tilde{g} \in \mathcal{G}_{S^\star}$ and $S_{\tilde{g}} = S^\star$, $f^{\mathcal{E}} \in \{\Theta_{S_g}\}^{|\mathcal{E}|}$, and $\tilde{f}^{\mathcal{E}} \in \{\Theta_{S^\star}\}^{|\mathcal{E}|}$.
\end{proposition}

We define 
\begin{align*}
    \mathsf{R}(g) = \mathbb{E}[\ell(Y^{(e)}, g(X^{(e)}))] \qquad \text{and} \qquad \hat{\mathsf{R}}(g) = \frac{1}{n} \sum_{i=1}^n \ell(Y^{(e)}_i, g(X^{(e)}_i)).
\end{align*} The following proposition is used to establish a high probability instance-dependent error bound for the following quantity
\begin{align*}
    \Delta_\mathsf{R}(g, \tilde{g}) := \{\hat{\mathsf{R}}(g) - \hat{\mathsf{R}}(\tilde{g})\} - \{{\mathsf{R}}(g) - \mathsf{R}(\tilde{g})\}.
\end{align*}

\begin{proposition}[Instance-dependent error bounds for general pooled risk]
\label{prop:nonasymptotic-pooled-general-loss}
Suppose \cref{cond:general-dgp}, \cref{cond:general-response}, \ref{cond:general-function-class}, and \cref{cond:general-loss} hold. There exists some universal constant $C$ such that for any $\eta>0$ and $t>0$, the following event
\begin{align*}
    \forall g, \tilde{g} \in \mathcal{G}, ~~ |\Delta_\mathsf{R}(g, \tilde{g})| \le C U \zeta \left\{ \delta_{n,t}^2 + \delta_{n,t} \frac{1}{|\mathcal{E}|} \sum_{e\in \mathcal{E}} \|g - \tilde{g}\|_{2,e}\right\}
\end{align*} occurs with probability at least $1-3e^{-t}-C_y(\sigma_y+1) n^{-100}$.
\end{proposition}

We are ready to prove the claims in \cref{thm:oracle-2}.

\begin{proof}[Proof of the Rate~\eqref{eq:main-result-general-rate-2}]

The proof is similar to that for \eqref{eq:main-result-general-rate}, we only highlight the differences.
Let $\hat{S} = S_{\hat{g}}$. We first apply \cref{prop:population-general-loss} with $\delta=0.5$ and substitute $g=\hat{g}$, $f^{\mathcal{E}}=\hat{f}^{\mathcal{E}}$, $\tilde{g}$ be that $\|\tilde{g} - m^\star\|_2 \le \delta_{\mathtt{a}, \mathcal{G}} + n^{-1}$ and $\tilde{f}^{\mathcal{E}}$ be that $\mathsf{Q}_\gamma(\tilde{g}, \tilde{f}^{\mathcal{E}}) \ge \sup_{f^{\mathcal{E}} \in \{\mathcal{F}_{S^\star}\}^{|\mathcal{E}|}} \mathsf{Q}_\gamma(\tilde{g}, f^{\mathcal{E}}) -\gamma \delta_{\mathtt{opt}}^2$. Then,
\begin{align*}
    \mathsf{Q}_\gamma(\hat{g}, \hat{f}^{\mathcal{E}}) - \mathsf{Q}_\gamma(\tilde{g}, \tilde{f}^{\mathcal{E}}) &\ge 0.25 \zeta^{-1} \|\hat{g} - \tilde{g}\|_2^2 + \frac{\gamma}{4} \bar{\mathsf{d}}(S) + \frac{\gamma}{2} \|\hat{g} - \bar{m}^{(\hat{S})}\|_2^2 \\
    &~~~~~~~~ - \frac{\gamma}{2|\mathcal{E}|} \sum_{e\in \mathcal{E}} \|\hat{f}^{(e)} - \{m^{(e,S)} - \hat{g}\}\|_{2,e}^2 - \frac{\gamma+4\zeta^2}{2} (\delta_{\mathtt{a}, \mathcal{G}}^2 + n^{-2}|).
\end{align*}
Apply \cref{prop:nonasymptotic-pooled-general-loss} and \ref{prop:nonasymptotic-a} in a similar way to \eqref{eq:proof-general-rate-eq2}, we obtain
\begin{align*}
\mathsf{Q}_\gamma(\hat{g}, \hat{f}^{\mathcal{E}}) - \mathsf{Q}_\gamma(\tilde{g}, \tilde{f}^{\mathcal{E}}) &\le 2(1+\gamma) \delta_{\mathtt{opt}}^2 + CU\zeta \frac{1}{|\mathcal{E}|} \sum_{e\in \mathcal{E}} \left\{\delta_{n, t} \|\hat{g} - \tilde{g}\|_{2,e} + \delta_{n, t}^2 \right\} \\
&~~~~~~~~ + C \gamma U \frac{1}{|\mathcal{E}|} \sum_{e\in \mathcal{E}} \left\{\delta_{n, t} \left(\|\tilde{g} - \hat{g}\|_{2,e}+ \|\tilde{g} + \tilde{f}^{(e)} - \hat{g} - \hat{f}^{(e)}\|_{2,e}\right) + \delta_{n, t}^2\right\}
\end{align*}
Combining the derivation in \eqref{eq:proof-general-rate-eq2} with the fact that
\begin{align*}
CU\zeta \frac{1}{|\mathcal{E}|} \sum_{e\in \mathcal{E}} \left\{\delta_{n, t} \|\hat{g} - \tilde{g}\|_{2,e} + \delta_{n, t}^2 \right\} \le 10C^2U^2\zeta^3  \delta_{n,t}^2 + CU\zeta \delta_{n, t}^2 + 0.05\zeta^{-1}\|\hat{g} - \tilde{g}\|_2^2,
\end{align*} we have
\begin{align*}
    \mathsf{Q}_\gamma(\hat{g}, \hat{f}^{\mathcal{E}}) - \mathsf{Q}_\gamma(\tilde{g}, \tilde{f}^{\mathcal{E}}) &\le 2(1+\gamma) \delta_{\mathtt{opt}}^2 + C(\gamma + \zeta)^2 \zeta U^2 \delta_{n, t}^2 + 0.05 \zeta^{-1}\|\hat{g} - \tilde{g}\|_2^2 \\
    &~~~~~~~~+ 0.001 \zeta^{-1} \|\hat{g} - \tilde{g}\|_2^2 + 0.002 \frac{\gamma}{\gamma+1} \sum_{e\in \mathcal{E}} \| m^\star - m^{(e,\hat{S})} \|_{2,e}^2 \\
    &~~~~~~~~+ \frac{0.004\gamma}{\gamma+1} \frac{1}{|\mathcal{E}|}\sum_{e\in \mathcal{E}} \left\{\|\tilde{g} + \tilde{f}^{(e)} - m^\star\|_{2,e}^2 + \|m^{(e,\hat{S})} - \hat{g} - \hat{f}^{(e)}\|_{2,e}^2 \right\} 
\end{align*} 
Similar to \eqref{eq:proof-general-rate-eq3}, we can also claim that
\begin{align}
\label{eq:proof-general-rate-general-loss-eq1}
    \frac{\gamma}{\gamma+1} \sum_{e\in \mathcal{E}} \| m^\star - m^{(e,\hat{S})} \|_{2,e}^2 \le (2+28\gamma^\star) \bar{\mathsf{d}}(\hat{S}) + 24\left( \|\hat{g} - \tilde{g}\|_{2}^2 + \delta_{\mathtt{a}, \mathcal{G}}^2 + \frac{1}{n}\right).
\end{align} 
Combining the upper bound and lower bound on $\mathsf{Q}_\gamma(\hat{g}, \hat{f}^{\mathcal{E}}) - \mathsf{Q}_\gamma(\tilde{g}, \tilde{f}^{\mathcal{E}})$ together and plugging in \eqref{eq:proof-general-rate-general-loss-eq1}, we obtain
\begin{align*}
    &0.15 \zeta^{-1} \|\hat{g} - \tilde{g}\|_2^2 \\
    &~~~\le C \Big\{ (\zeta^2 + \gamma) (\delta_{\mathtt{a}, \mathcal{G}}^2 + \delta_{\mathtt{opt}}^2) + (\zeta+\gamma)^2\zeta U^2\delta_{n,t}^2 \Big\}\\
    &~~~~~~~~~~~ + \gamma (1+1/(\gamma+1)) \frac{1}{2|\mathcal{E}|} \sum_{e\in \mathcal{E}} \left(\|\hat{f}^{(e)} - (m^{(e,\hat{S})}-\hat{g})\|_{2,e}^2 + \|\tilde{f}^{(e)} - (m^\star-\tilde{g})\|_{2,e}^2 \right)\\
    &~~~\overset{(a)}{\le} C \Big\{ (\zeta^2 + \gamma) (\delta_{\mathtt{a}, \mathcal{G}}^2 + \delta_{\mathtt{opt}}^2) + (\zeta+\gamma)^2\zeta U^2\delta_{n,t}^2 \Big\} \\
    &~~~~~~~~~~~ + C (1+\gamma)(\delta_{\mathtt{opt}}^2 + \delta^2_{\mathtt{a}, \mathcal{F}, \mathcal{G}}(\hat{S}) + \delta^2_{\mathtt{a}, \mathcal{F}, \mathcal{G}}(S^\star) + U^2\delta_{n, t}^2),
\end{align*} where $(a)$ follows from \cref{prop:characterize-f}. This further implies that
\begin{align*}
    \|\hat{g} - \tilde{g}\|_2^2 \le C \zeta^2 (\zeta + \gamma)^2 \left\{U^2\delta_{n, t}^2 + \delta^2_{\mathtt{a}, \mathcal{F}, \mathcal{G}}(\hat{S}) + \delta^2_{\mathtt{a}, \mathcal{F}, \mathcal{G}}(S^\star) + \delta_{\mathtt{a}, \mathcal{G}}^2 + \delta_{\mathtt{opt}} \right\}.
\end{align*} Applying the triangle inequality $\|\hat{g} - m^\star\|_2 \le \|\hat{g} - \tilde{g}\|_2 + \|\tilde{g} - m^\star\|_2$ concludes the proof.
\end{proof}

\begin{proof}[Proof of Rate~\eqref{eq:main-result-rate-2}]
The proof is also similar to that for \eqref{eq:main-result-rate}, we only highlight the differences. Let $\hat{S}=S_{\hat{g}}$. We first show that the following event
\begin{align*}
\mathcal{A}_+ = \left\{\forall e\in \mathcal{E},~ m^{(e,\hat{S})}=m^\star\right\}
\end{align*} occurs under the event defined in \cref{prop:nonasymptotic-pooled-general-loss} and \cref{prop:nonasymptotic-a} if Condition \eqref{eq:main-result-faster-cond-2} holds with some large universal constant $C$. Following a similar idea to the proof of \cref{prop:variable-selection}, it suffices to show that the following event
\begin{align*}
    &\forall g ~\text{with}~S_g \in \{S\subseteq [d]: \exists e\in \mathcal{E}, \|m^{(e,S)}-m^\star\|_{2,e}>0\}, ~~~~\exists f^{\mathcal{E}} \in \{\mathcal{F}_{S_g}\}^{|\mathcal{E}|} ~\text{and}~\tilde{g} ~\text{with}~ S_{\tilde{g}} = S^\star ~\text{s.t.}\\
&~~~~~~~~~~~~~~~~~~~~ \hat{\mathsf{Q}}_{\gamma}(g, f^{\mathcal{E}}) - \sup_{\tilde{f}^{\mathcal{E}} \in \{\mathcal{F}_{S^\star}\}^{|\mathcal{E}|}}\hat{\mathsf{Q}}_{\gamma}(\tilde{g}, \tilde{f}^{\mathcal{E}}) > 2(1+\gamma)\delta_{\mathtt{opt}}^2.
\end{align*} Now we apply \cref{prop:nonasymptotic-pooled-general-loss}, \cref{prop:nonasymptotic-a} and \cref{prop:population-general-loss} (with $\delta=0.5$) and choose $\tilde{g}$ be such that $\|\tilde{g} - m^\star\|_2 \le \delta_{\mathtt{a},\mathcal{G}} + n^{-1}$, we can obtain a similar lower bound to \eqref{eq:thm1-proof-eq1} that for any $g$, we can choose a corresponding $f^{\mathcal{E}}=f^{\mathcal{E}}_g$ such that 
\begin{align}
\label{eq:variable-sel-lb1}
\begin{split}
&\hat{\mathsf{Q}}_\gamma(g, f^{\mathcal{E}}_g) - \hat{\mathsf{Q}}_\gamma(\tilde{g}, \tilde{f}^{\mathcal{E}}) - 2(1+\gamma)\delta_{\mathtt{opt}}^2 \\
&~~~\ge 0.25\zeta^{-1} \|g - \tilde{g}\|_{2}^2 + \frac{\gamma}{4} \bar{\mathsf{d}}(S) - 2(1+\gamma)\delta_{\mathtt{opt}}^2\\
&~~~~~~~~~~~ -\left\{C (\zeta + \gamma)UB \delta_{n,t} + (0.5\gamma+2) \delta_{\mathtt{a}, \mathcal{G}}^2 + 0.5\gamma \delta_{\mathtt{a}, \mathcal{F}, \mathcal{G}}(S)\right\}
\end{split}
\end{align} provided $\delta_{n,t}<1$ and $B\ge 1$. Therefore, it follows similarly to the three case discussions in the proof of \cref{prop:variable-selection} that, if
\begin{align*}
    (\zeta + \gamma) \left(\delta_{\mathtt{opt}}^2 + \sup_{S\subseteq[d]} \delta_{\mathtt{a}, \mathcal{F}, \mathcal{G}}(S)^2 + \delta_{\mathtt{a}, \mathcal{G}}^2 + UB\delta_{n, t}\right) \le \left(\zeta^{-1} s_{\min} \land \gamma \inf_{S:\bar{\mathsf{d}}(S)>0} \bar{\mathsf{d}}(S) \right) / \breve{C}
\end{align*} for some large enough universal constant $\breve{C}$, then 
\begin{align*}
\hat{\mathsf{Q}}_\gamma(g, f^{\mathcal{E}}_g) - \hat{\mathsf{Q}}_\gamma(\tilde{g}, \tilde{f}^{\mathcal{E}}) - 2(1+\gamma)\delta_{\mathtt{opt}}^2 > 0
\end{align*} for any $g$ with $S_g \in \{S\subseteq [d]: \exists e\in \mathcal{E}, \|m^{(e,S)}-m^\star\|_{2,e}>0\}$. This concludes the proof of the claim that the event $\mathcal{A}^+$ occurs under the event defined in \cref{prop:nonasymptotic-pooled-general-loss} and \cref{prop:nonasymptotic-a} if Condition \eqref{eq:main-result-faster-cond-2} holds with some large universal constant $C$.

The rest of the proof proceeds condition on the above variable selection property, and is similar to the proof of the error bound \eqref{eq:main-result-rate}. On one hand, we apply \cref{prop:population-general-loss} with $\delta=0.5$ and substitute $g=\hat{g}$, $f^{\mathcal{E}}=\hat{f}^{\mathcal{E}}$, $\tilde{g}$ be that $\|\tilde{g} - m^\star\|_2 \le \delta_{\mathtt{a}, \mathcal{G}} + n^{-1}$ and $\tilde{f}^{\mathcal{E}}$ be that $\mathsf{Q}_\gamma(\tilde{g}, \tilde{f}^{\mathcal{E}}) \ge \sup_{f^{\mathcal{E}} \in \{\mathcal{F}_{S^\star}\}^{|\mathcal{E}|}} \mathsf{Q}_\gamma(\tilde{g}, f^{\mathcal{E}}) -\gamma \delta_{\mathtt{opt}}^2$. Then,
\begin{align*}
    \mathsf{Q}_\gamma(\hat{g}, \hat{f}^{\mathcal{E}}) - \mathsf{Q}_\gamma(\tilde{g}, \tilde{f}^{\mathcal{E}}) &\ge 0.25\zeta^{-1} \|\hat{g} - \tilde{g}\|_2^2 + 0.5\gamma \|\hat{g} - m^\star\|_2^2 \\
    &~~~~~~~~ - \frac{\gamma+4\zeta^2}{2} \|\tilde{g} - m^\star\|_2^2 - \frac{\gamma}{2|\mathcal{E}|} \sum_{e\in \mathcal{E}} \|\hat{f}^{(e)} - (m^\star - \hat{g})\|_{2,e}^2 \\
    &\ge 0.25 (\zeta^{-1} + \gamma) \|\hat{g} - \tilde{g}\|_2^2 - 2(\zeta^2+\gamma) (\delta_{\mathtt{a},\mathcal{G}}^2 + n^{-2}) \\
    &~~~~~~~~ - \frac{\gamma}{2|\mathcal{E}|} \sum_{e\in \mathcal{E}} \|\hat{f}^{(e)} - (m^\star - \hat{g})\|_{2,e}^2
\end{align*}
On the other hand, applying \cref{prop:nonasymptotic-pooled-general-loss} and \ref{prop:nonasymptotic-a} in a similar way to \eqref{eq:proof-general-rate-eq2}, we obtain
\begin{align*}
&\mathsf{Q}_\gamma(\hat{g}, \hat{f}^{\mathcal{E}}) - \mathsf{Q}_\gamma(\tilde{g}, \tilde{f}^{\mathcal{E}}) \\
&~~~\le 2(1+\gamma) \delta_{\mathtt{opt}}^2 + CU\zeta \frac{1}{|\mathcal{E}|} \sum_{e\in \mathcal{E}} \left\{\delta_{n, t} \|\hat{g} - \tilde{g}\|_{2,e} + \delta_{n, t}^2 \right\} \\
&~~~~~~~~~~~ + C \gamma U \frac{1}{|\mathcal{E}|} \sum_{e\in \mathcal{E}} \left\{\delta_{n, t} \left(\|\tilde{g} - \hat{g}\|_{2,e}+ \|\tilde{g} + \tilde{f}^{(e)} - \hat{g} - \hat{f}^{(e)}\|_{2,e}\right) + \delta_{n, t}^2\right\} \\
&~~~\le 2(1+\gamma) \delta_{\mathtt{opt}}^2 + CU^2\zeta^3\delta_{n,t}^2 + 0.05\zeta^{-1}\|\hat{g} - \tilde{g}\|_2^2 + \gamma \left(20CU^2 \delta_{n, t}^2 + 0.05 \|\hat{g} - \tilde{g}\|_2^2 \right)\\
&~~~~~~~~~~~ + \gamma \frac{1}{|\mathcal{E}|} \sum_{e\in \mathcal{E}} \left(\|\tilde{g} + \tilde{f}^{(e)} - m^\star\|_{2,e}^2 + \|\hat{g} + \hat{f}^{(e)} - m^\star\|_{2,e}^2 \right)
\end{align*} Combing the upper bound and lower bound together, we find
\begin{align*}
0.20(\zeta^{-1} + \gamma) \|\hat{g} - \tilde{g}\|_2^2 &\le CU^2\delta_{n,t}^2 (\zeta^3 + \gamma) + 2(\zeta^2 + \gamma) \delta_{\mathtt{a}, \mathcal{G}}^2 + (1+\gamma) \delta_{\mathtt{opt}}^2\\
&~~~~~~~~ + 2\gamma \frac{1}{|\mathcal{E}|} \sum_{e\in \mathcal{E}} \left(\|\tilde{g} + \tilde{f}^{(e)} - m^\star\|_{2,e}^2 + \|\hat{g} + \hat{f}^{(e)} - m^\star\|_{2,e}^2 \right)\\
&\overset{(a)}{\le} CU^2\delta_{n,t}^2 (\zeta^3 + \gamma) + 2(\zeta^2 + \gamma) \delta_{\mathtt{a}, \mathcal{G}}^2 + (1+\gamma) \delta_{\mathtt{opt}}^2 \\
&~~~~~~~~ + (\gamma +1)C(U^2\delta_{n, t}^2 + \delta^2_{\mathtt{a}, \mathcal{F}, \mathcal{G}}(\hat{S}) + \delta^2_{\mathtt{a}, \mathcal{F}, \mathcal{G}}(S^\star))
\end{align*} where $(a)$ follows from \cref{prop:characterize-f}. Observe that $(x+a)/(x+b) \le a/b$ for any $x\ge 0$ and $a\le b$.  Therefore,
\begin{align*}
    \|\hat{g} - \tilde{g}\|_{2}^2 \le C \zeta^4 \left\{\delta_{\mathtt{opt}}^2 + U^2 \delta_{n,t}^2 + \delta_{\mathtt{a}, \mathcal{G}}^2 + \delta^2_{\mathtt{a}, \mathcal{F}, \mathcal{G}}(\hat{S}) + \delta^2_{\mathtt{a}, \mathcal{F}, \mathcal{G}}(S^\star)\right\}
\end{align*} provided $\zeta \ge 1$. This concludes the proof via observing that $\delta_{\mathtt{a}, \mathcal{F}, \mathcal{G}}(\hat{S}) \lor \delta_{\mathtt{a}, \mathcal{F}, \mathcal{G}}(S^\star) \le \delta_{\mathtt{a}, \mathcal{F}, \mathcal{G}}^\star$ and applying triangle inequality $\|\hat{g} - m^\star\|_{2}\le \|\hat{g} - \tilde{g}\|_{2} + \|\tilde{g} - m^\star\|_2$.

\end{proof}

\subsection{Proof of Proposition \ref{prop:nonasymptotic-pooled-general-loss}}

It follows from \cref{cond:general-loss} and the Taylor expansion that $\ell(y, v) - \ell(y, \tilde{v}) = \frac{\partial \ell}{\partial v}(y, \tilde{v}) (v - \tilde{v}) +  \frac{\partial^2 \ell}{\partial v}(y, \bar{v})\frac{(v - \tilde{v})^2}{2}$ with some $\bar{v} \in [v \land \tilde{v}, v \lor \tilde{v}]$. We can thus let $\frac{\partial \ell}{\partial v}(y, \tilde{v}) = \psi(\tilde{v}) (y - \tilde{v})$, and let $\frac{\partial^2 \ell}{\partial v}(y, \bar{v}) = q(y, v, \tilde{v})$. Then, we can decompose $\Delta_{\mathsf{R}}(g, \tilde{g})$ into
\begin{align*}
    \Delta_{\mathsf{R}}(g, \tilde{g}) &= \frac{1}{|\mathcal{E}|} \sum_{e\in \mathcal{E}} (\hat{\mathbb{E}} - \mathbb{E}) \left[\ell(Y^{(e)}, g(X^{(e)})) - \ell(Y^{(e)}, \tilde{g}(X^{(e)})) \right] \\
    &= \frac{1}{|\mathcal{E}|} \sum_{e\in \mathcal{E}} (\hat{\mathbb{E}} - \mathbb{E}) \left[ Y^{(e)} \psi(\tilde{g}(X^{(e)})) (g - \tilde{g})(X^{(e)})\right] \\
    & ~~~~~~~~~~~~~~~~ + (\hat{\mathbb{E}} - \mathbb{E})\left[ -\tilde{g}(X^{(e)}) \psi(\tilde{g}(X^{(e)})) (g - \tilde{g})(X^{(e)})\right] \\
    & ~~~~~~~~~~~~~~~~ + (\hat{\mathbb{E}} - \mathbb{E})\left[ \frac{q(Y^{(e)}, g(X^{(e)}), \tilde{g}(X^{(e)}))}{2} \{g(X^{(e)}) - \tilde{g}(X^{(e)})\}^2  \right] \\
    &= \frac{1}{|\mathcal{E}|} \sum_{e\in \mathcal{E}} \mathsf{T}_1^{(e)}(g, \tilde{g}) + \mathsf{T}_2^{(e)}(g, \tilde{g}) + \mathsf{T}_3^{(e)}(g, \tilde{g}) 
\end{align*}
We apply \cref{lemma:ep1} to derive high probability bounds on $\mathsf{T}^{(e)}_k$. 

\noindent {\sc Step 1. Bounds on $\mathsf{T}_{1}^{(e)}$.} For $\mathsf{T}_1^{(e)}$, we will use truncation argument. To be specific, let $K>0$ to be determined, we can decompose $\mathsf{T}_1^{(e)}$ into
\begin{align*}
    \mathsf{T}_1^{(e)}(g,\tilde{g}) &= (\hat{\mathbb{E}} - \mathbb{E}) \left[ 1\{|Y^{(e)}|\le K\} Y^{(e)} \psi(\tilde{g}(X^{(e)})) (g - \tilde{g})(X^{(e)})\right] \\
    & ~~~~~~~~~~ + (\hat{\mathbb{E}} - \mathbb{E}) \left[ 1\{|Y^{(e)}|>K\} Y^{(e)} \psi(\tilde{g}(X^{(e)})) (g - \tilde{g})(X^{(e)})\right] \\
    &= \mathsf{T}_{1,1}^{(e)}(g,\tilde{g},K) + \mathsf{T}_{1,2}^{(e)}(g,\tilde{g},K).
\end{align*} 
For $\mathsf{T}_{1,1}^{(e)}(g,\tilde{g},K)$, applying \cref{lemma:ep1} with $\mathcal{H} = \mathcal{G}$, $v(g,\tilde{g}, z)=Y^{(e)} \psi(\tilde{g}(X^{(e)}))$ that is uniformly bounded by $K\zeta$ by \cref{cond:general-loss} (2), and $\phi(x)=x$, we find that, for any $e\in \mathcal{E}$ and $u>0$, the following event
\begin{align}
\label{eq:zeta-prop1-sn1}
\begin{split}
    \mathcal{C}_{1,1}^{(e)}(u) &= \left\{\forall g,\tilde{g}\in \mathcal{G}, ~|\mathsf{T}_{1,1}^{(e)}(g,\tilde{g},K)| \le CBK\zeta \left(s_{n,1}\|\tilde{g} - g\|_{2,e} + s_{n,1}^2\right)\right\} \\
    &\qquad\qquad\qquad \qquad ~~ \text{with}~ s_{n,1} = \delta_n + \sqrt{\frac{u + 1 + \log(nB))}{n}}
\end{split}
\end{align} occurs with probability at least $1-e^{-u}$ for some universal constant $C$. Applying union bound over all the $e\in \mathcal{E}$, we obtain that, 
\begin{align}
    \mathbb{P}\left[\mathcal{C}_{1,1}(t)\right] &= \mathbb{P}\left[\forall e\in \mathcal{E},~\forall g,\tilde{g}\in \mathcal{G}, ~|\mathsf{T}_{1,1}^{(e)}(g,\tilde{g},K)| \le C BK\zeta (s_{n,2}\|\tilde{g}- g\|_{2,e} + s_{n,2}^2)\right] \nonumber \\
    &\ge 1 - \sum_{e\in \mathcal{E}} \mathbb{P}\left[\left(\mathcal{C}_{1,1}^{(e)}(u)\right)^c\right]\ge 1 - |\mathcal{E}| e^{-u} \ge 1-e^{-t} \label{eq:zeta-prop1-c11}
\end{align} where $s_{n,2} = \delta_n + \sqrt{\frac{t + 1 + \log(nB|\mathcal{E}|))}{n}}$.

For $\mathsf{T}_{1,2}^{(e)}(g,\tilde{g},K)$, it follows from Markov inequality that, for any given $e\in \mathcal{E}$ and $x>0$,
\begin{align*}
    &\mathbb{P}\left[\sup_{g,\tilde{g}\in \mathcal{G}}|\mathsf{T}_{1,2}^{(e)}(g,\tilde{g},K)| > x\right] \\
    &~~~\le x^{-1} \mathbb{E} \left[\sup_{g,\tilde{g}\in \mathcal{G}}(\hat{\mathbb{E}} - \mathbb{E}) \left[1\{|Y^{(e)}|>K\} Y^{(e)} \psi(\tilde{g}(X^{(e)})) (g - \tilde{g})(X^{(e)})\right]\right]\\
    &~~~\le x^{-1} \mathbb{E} \left[\sup_{g,\tilde{g}\in \mathcal{G}}(\hat{\mathbb{E}} + \mathbb{E}) \left[\left|1\{|Y^{(e)}|>K\} Y^{(e)} \psi(\tilde{g}(X^{(e)})) (g - \tilde{g})(X^{(e)})\right|\right]\right] \\
    &~~~\le x^{-1} 4B\zeta \mathbb{E}[|Y^{(e)}| 1\{|Y^{(e)}| > K\}]
\end{align*} It then follows from the sub-Gaussian response condition \cref{cond:general-response} that
\begin{align*}
    \mathbb{E}[|Y^{(e)}| 1\{|Y^{(e)}| > K\}] &= \int |y| 1\{|y|\ge K\} \mu^{(e)}_y(dy) \\
    &= \int 1\{|y|\ge K\} \left(\int_0^\infty 1\{t\le |y|\} dt\right) v \mu^{(e)}_y(dy)\\
    &= \int_0^\infty \int 1\{|y| \ge K\lor t\}\mu_y^{(e)}(dy) dt = \int_0^\infty \mathbb{P}\left(|Y^{(e)}| \ge t\lor K\right) dt \\
    &\le K C_{y} e^{-K^2/(2\sigma_y^2)} + \frac{C_y \sigma_y^2}{K} e^{-K^2/(2\sigma_y^2)}.
\end{align*} Hence, we can conclude that, for any fixed $e\in \mathcal{E}$, $x>0$ and $K>0$, 
\begin{align}
\label{eq:zeta-proof-prop1-truncation-argument}
\mathbb{P}\left[\sup_{g,\tilde{g}\in \mathcal{G}} |\mathsf{T}_{1,2}^{(e)}(g,\tilde{g},K)| > x\right] \le 4B\zeta C_y e^{-K^2/(2\sigma_y^2)} x^{-1} (K+\sigma_y^2/K)
\end{align} Applying union bound with $x=n^{-1} 44 B\zeta$, $K=\sigma_y\sqrt{102\log(n|\mathcal{E}|)}$, we have
\begin{align}
\label{eq:zeta-prop1-c12}
\begin{split}
    &\mathbb{P}[\mathcal{C}_{1,2}] = \mathbb{P}\left[\forall e\in \mathcal{E}, \sup_{g,\tilde{g}\in \mathcal{G}}\mathsf{T}_{1,2}^{(e)}(g,\tilde{g},K) 
    \le \frac{44\zeta B}{n}\right] \\
    &~~~\ge 1- |\mathcal{E}| \times C_y \frac{1}{(n\cdot |\mathcal{E}|)^{102}} n (\sigma_y \sqrt{102\log(n|\mathcal{E}|)} + 1)\\
    &~~~\ge 1- C_y (\sigma_y+1) n^{-100}, 
\end{split}
\end{align}

\noindent {\sc Step 2. Bounds on $\mathsf{T}_{2}^{(e)}$.} Applying \cref{lemma:ep1} with $\mathcal{H}=\mathcal{G}$, $v(g, \tilde{g}, z) = -\tilde{g} \psi(\tilde{g})$ uniformly bounded by $L_1= B\zeta$, and $\phi(x)=x$, we have that, for any $e$ and $u>0$, the following event
\begin{align*}
    \mathcal{C}_2^{(e)}(u) = \left\{ |\mathsf{T}_2^{(e)}(g, \tilde{g})| \le C B^2 \zeta (s_{n,1} \|\tilde{g} - g\|_{2,e} + s_{n,1}^2)\right\}
\end{align*} occurs with probability $1-e^{-u}$ for some universal constant $C$, where $s_{n,1}$ is the same quantity defined in \eqref{eq:zeta-prop1-sn1}. Applying union bound over all the $e\in \mathcal{E}$ gives that
\begin{align}
\label{eq:zeta-prop1-c2}
    \mathbb{P}\left[\mathcal{C}_{2}(t)\right]= \mathbb{P}\left[\forall e\in \mathcal{E}, ~\forall g,\tilde{g}\in \mathcal{G},~|\mathsf{T}_{2}^{(e)}(g,\tilde{g})| \le C B^2 \zeta (s_{n,2}\|\tilde{g}- g\|_{2,e} + s_{n,2}^2)\right] \ge 1-e^{-t}.
\end{align} where $s_{n,2}=s_{n,1}$ with $u=t + \log(|\mathcal{E}|)$.

\noindent {\sc Step 3. Bounds on $\mathsf{T}_{3}^{(e)}$.} Following a similar procedure as what we do for $\mathsf{T}_2^{(e)}$, we apply \cref{lemma:ep1} with $\mathcal{H}=\mathcal{G}$, $v(g, \tilde{g}, z) = 0.5 q(g, \tilde{g}, y)$ uniformly bounded by $L_1=\zeta/2$, and $\phi(x)=x^2$ that is $2B$-Lipschitz due to the boundedness of $(g,\tilde{g})$, followed by using union bound over all the $e\in \mathcal{E}$. Therefore,
\begin{align}
\label{eq:zeta-prop1-c3}
    \mathbb{P}\left[\mathcal{C}_{3}(t)\right]= \mathbb{P}\left[\forall e\in \mathcal{E}, ~\forall g,\tilde{g}\in \mathcal{G}, ~|\mathsf{T}_{3}^{(e)}(g,\tilde{g})| \le C B^2\zeta (s_{n,2}\|\tilde{g}- g\|_{2,e} + s_{n,2}^2)\right] \ge 1-e^{-t}.
\end{align}

\noindent {\sc Step 4. Putting the pieces together.} Recall our choice of $K=\sigma_y \sqrt{102\log(n|\mathcal{E}|)}$. Combining \eqref{eq:zeta-prop1-c11}, \eqref{eq:zeta-prop1-c12}, \eqref{eq:zeta-prop1-c2}, and \eqref{eq:zeta-prop1-c3} together, we can conclude that, under $\mathcal{C}_{1,1}(t) \cap \mathcal{C}_{1,2} \cap \mathcal{C}_2(t) \cap \mathcal{C}_{3}(t)$ that occurs with probability at least $1-3e^{-t}-C_y(\sigma_y+1) n^{-100}$, the following holds
\begin{align*}
    \forall g, \tilde{g}, ~~ |\Delta_{\mathsf{R}}(g, \tilde{g})| &\le \frac{1}{|\mathcal{E}|} \sum_{e\in \mathcal{E}} |\mathsf{T}_{1,1}^{(e)}(g,\tilde{g},K)| + |\mathsf{T}_{1,2}^{(e)}(g,\tilde{g},K)| + |\mathsf{T}_{2}^{(e)}(g,\tilde{g})| + |\mathsf{T}_{3}^{(e)}(g,\tilde{g})| \\
    &\overset{(a)}{\le} \frac{1}{|\mathcal{E}|} \sum_{e\in \mathcal{E}} CBK\zeta \left(s_{n,2}\|g - \tilde{g}\|_{2,e} + s_{n,2}^2\right) \\
    &~~~~~~~~~~ + \frac{44B\zeta}{n} + 2CB^2 \zeta  \left(s_{n,2} \|g - \tilde{g}\|_{2,e} + s_{n,2}^2\right) \\
    &\overset{(b)}{\le} C  B(B+\sqrt{\log(n|\mathcal{E}|)})\zeta \left\{ s_{n,2}^2 + s_{n,2} \frac{1}{|\mathcal{E}|} \sum_{e\in \mathcal{E}} \|g - \tilde{g}\|_{2,e}\right\} 
\end{align*} where (a) follows from the instant-dependent error bounds in \eqref{eq:zeta-prop1-c11}, \eqref{eq:zeta-prop1-c12}, \eqref{eq:zeta-prop1-c2} and \eqref{eq:zeta-prop1-c3}, and (b) follows from our choice of $K$ and the fact $1/n\le \delta_n \le s_{n,2}$. This completes the proof. \qed

\subsection{Proof of Proposition \ref{prop:population-general-loss}}
\label{sec:proof-prop-general-loss}

Note that the regularizer loss coincides with that \cref{thm:population} with $\ell(y,v) = \frac{1}{2}(y-v)^2$, then it follows similar to {\sc Step 2} and {\sc Step 3} in the proof of \cref{thm:population} that
\begin{align*}
    \mathsf{Q}_\gamma(g, f^{\mathcal{E}}) - \mathsf{Q}_\gamma(\tilde{g}, \tilde{f}^{\mathcal{E}}) &\ge \frac{1}{|\mathcal{E}|} \sum_{e\in \mathcal{E}} \mathbb{E}[\ell(Y^{(e)}, g(X^{(e)})) - \ell(Y^{(e)}, \tilde{g}(X^{(e)}))]\\
    &~~~~~~~~ + \gamma \left\{\frac{\bar{\mathsf{d}}(S)}{2} + \frac{1}{2}\| \bar{m}^{(S)} - g\|_2^2 - \frac{1}{2|\mathcal{E}|} \sum_{e\in \mathcal{E}}  \|f^{(e)} - m^{(e,S)} - g\}\|_{2,e}^2\right\} \\
    &~~~~~~~~ - \gamma \left\{ \frac{\bar{\mathsf{d}}(S^\star)}{2} + \frac{1}{2} \|m^\star - \tilde{g}\|_2^2\right\}
\end{align*} Applying Taylor expansion to $\ell$, it follows from \cref{cond:general-loss} (1) that
\begin{align*}
\ell(y, v) - \ell(y, \tilde{v}) = \frac{\partial \ell}{\partial v}(\tilde{v}) (v - \tilde{v}) + \frac{\partial^2 \ell}{\partial v^2}(\bar{v}) \frac{(v - \tilde{v})^2}{2} = (y - \tilde{v}) g(\tilde{v}) (v - \tilde{v}) + \frac{\partial^2 \ell}{\partial v^2}(\bar{v}) \frac{(v - \tilde{v})^2}{2}
\end{align*} for some $\bar{v} \in [v\land \tilde{v}, v\lor \tilde{v}]$, substituting $y=Y^{(e)}$, $v=g(X^{(e)})$ and $\tilde{v} = \tilde{g}(X^{(e)})$ in and taking expectations, we obtain
\begin{align*}
    & \frac{1}{|\mathcal{E}|} \sum_{e\in \mathcal{E}} \mathbb{E}[\ell(Y^{(e)}, g(X^{(e)})) - \ell(Y^{(e)}, \tilde{g}(X^{(e)}))] \\
    &~~= \frac{1}{|\mathcal{E}|} \sum_{e\in \mathcal{E}} \mathbb{E} \left[\{Y^{(e)} - \tilde{g}(X^{(e)}) \}\psi(\tilde{g}(X^{(e)})) \{g(X^{(e)}) - \tilde{g}(X^{(e)})\} + \frac{\partial^2 \ell}{\partial v^2}(\bar{V}) \frac{\{g(X^{(e)}) - \tilde{g}(X^{(e)})\}^2}{2} \right] \\
    &~~\overset{(a)}{\ge} \frac{1}{|\mathcal{E}|} \sum_{e\in \mathcal{E}} \mathbb{E} \left[\mathbb{E}\left[\{Y^{(e)} - \tilde{g}(X^{(e)}) \}\psi(\tilde{g}(X^{(e)})) \{g(X^{(e)}) - \tilde{g}(X^{(e)})\}\big|X_{S_g\cup S^\star}\right]\right] + \frac{\zeta^{-1}}{2} \|g - \tilde{g}\|_{2,e}^2 \\
    &~~\overset{(b)}{=} \frac{1}{|\mathcal{E}|} \sum_{e\in \mathcal{E}} \int \{m^{(e,S_g\cup S^\star)}(x) - \tilde{g}(x)\}\psi(\tilde{g}(x))\{g(x) - \tilde{g}(x)\} \mu^{(e)}(dx) + \frac{\zeta^{-1}}{2} \|g - \tilde{g}\|_{2,e}^2 \\
    &~~\overset{(c)}{=} \int \{\bar{m}^{(S_g\cup S^\star)}(x) - \tilde{g}(x)\} \psi(\tilde{g}(x)) \{g(x) - \tilde{g}(x)\} \bar{\mu}_x(dx) + \frac{\zeta^{-1}}{2} \|g - \tilde{g}\|_2^2
\end{align*} where $(a)$ follows from the lower bound on second derivative of $\ell$ on the domain \cref{cond:general-loss} (2) and tower rule of conditional expectation, $(b)$ follows from the definition of conditional expectation, $(c)$ follows from definition of $\bar{m}^{(S)}$ and the fact that $\frac{1}{|\mathcal{E}|} \sum_{e\in \mathcal{E}} \|\cdot\|_{2,e}^2 = \|\cdot\|_2^2$. It then follows from Cauchy-Schwarz inequality and the fact that $\psi(v)$ is bounded by $\zeta$ that
\begin{align*}
   &\left| \int \{\bar{m}^{(S_g\cup S^\star)}(x) - \tilde{g}(x)\} \psi(\tilde{g}(x)) \{g(x) - \tilde{g}(x)\} \bar{\mu}_x(dx) \right| \\
   &~~\le \|g - \tilde{g}\|_{2} \|\psi(\tilde{g})\{\bar{m}^{(S_g\cup S^\star)}- \tilde{g}\}\|_2 \\
   &~~\le \frac{\delta \zeta^{-1}}{2} \|g - \tilde{g}\|_{2}^2 + \frac{\zeta^2 \delta^{-1}}{2} \|\bar{m}^{(S_g\cup S^\star)} - \tilde{g}\|_2^2 \\
   &~~\le \frac{\delta \zeta^{-1}}{2} \|g - \tilde{g}\|_{2} + \zeta^2 \delta^{-1} \left\{ \mathsf{b}(S) +\|m^\star-\tilde{g}\|_2^2\right\}
\end{align*} Putting all the pieces together, we conclude that whenever $\gamma \ge 4\zeta^2\delta^{-1} \gamma^\star$,
\begin{align*}
    &\mathsf{Q}_\gamma(g, f^{\mathcal{E}}) - \mathsf{Q}_\gamma(\tilde{g}, \tilde{f}^{\mathcal{E}}) \\
    &~~~\ge \frac{1-\delta}{2} \zeta^{-1} \|g - \tilde{g}\|_2^2 + \frac{\gamma}{4} \bar{\mathsf{d}}(S) + \frac{\gamma}{2} \|\bar{m} - g\|_2^2\\
    &~~~~~~~~~~~ + \left(\frac{ 4\zeta^2\delta^{-1} \gamma^\star}{4} \bar{\mathsf{d}}(S) - \zeta^2 \delta^{-1} \mathsf{b}(S) \right) \\
    &~~~~~~~~~~~ - \gamma \frac{1}{2|\mathcal{E}|} \sum_{e\in \mathcal{E}}  \|f^{(e)} - m^{(e,S)} - g\}\|_{2,e}^2 - \left(\zeta^2 \delta^{-1} + \frac{\gamma}{2} \right) \|m^\star - \tilde{g}\|_2^2.
\end{align*} This completes the proof provided the definition of $\gamma^\star$. \qed


\section{Proofs for Example Estimators}
\label{sec:proof-application}

\subsection{Preliminaries}

We will use the following facts relating to the approximation error and stochastic error for neural networks. We first introduce some notations. Define
\begin{align*}
    \|h\|_{\infty, X} = \sup_{x \in X} |h|.
\end{align*}
Let $\mathcal{H}$ be a function class defined on $\mathcal{Z}$, we denote $\mathcal{N}(\epsilon, \mathcal{H}, d(\cdot, \cdot))$ to be the $\epsilon$-covering number of function class $\mathcal{H}$ with respect to the metric $d$, we let
\begin{align*}
    &\mathcal{N}_p(\epsilon, \mathcal{H}, z_1^n) = \mathcal{N}\left(\epsilon, \mathcal{H}, d\right) \\
    &\qquad \text{with} \qquad d(f, g) = \begin{cases}
        \left(\frac{1}{n} \sum_{i=1}^n |f(z_i) - g(z_i)|^p\right)^{1/p} \qquad & 1\le p<\infty \\
        \max_{1\le i\le n} |f(z_i) - g(z_i)| \qquad & p = \infty
    \end{cases}
\end{align*} for any $p \in [1,\infty]$, and define the uniform covering number $\mathcal{N}_\infty(\varepsilon, \mathcal{H}, n)$ as
\begin{align*}
    \mathcal{N}_\infty(\epsilon, \mathcal{H}, n) = \sup_{z_1,\ldots, z_n} \mathcal{N}_{\infty}(\epsilon, \mathcal{H}, z_1^n)
\end{align*}

We will use the following facts regarding the uniform covering number.
\begin{lemma}
\label{lemma:uniform-cover-number-plus}
    The following holds
    \begin{align*}
    \log \mathcal{N}_\infty(\epsilon, \mathcal{H}_1 + \mathcal{H}_2, n) \le \log \mathcal{N}_\infty(\epsilon/2, \mathcal{H}_1, n) + \log \mathcal{N}_\infty(\epsilon/2, \mathcal{H}_2, n)
    \end{align*}
\end{lemma}
\begin{proof}[Proof of \cref{lemma:uniform-cover-number-plus}]
    For fixed $z_1,\ldots, z_n$, we let $N_1 = \mathcal{N}_\infty(\epsilon/2, \mathcal{H}_1, n)$ and $N_2 = \mathcal{N}_\infty(\epsilon/2, \mathcal{H}_2, n)$, it follows from the definition of uniform covering number that, there exists $\{f_1,\ldots, f_{N_1} \} \subseteq \mathcal{H}_1$ and $\{g_1,\ldots, g_{N_2}\} \subseteq \mathcal{H}_2$ such that, for any $f \in \mathcal{H}_1$ and $g\in \mathcal{H}_2$, 
    \begin{align*}
        \sup_{i\in [n]} \inf_{k\in [N_1]} |f(z_i) - f_k(z_i)| \le \epsilon / 2 \qquad \text{and} \qquad \sup_{i\in [n]} \inf_{k\in [N_2]} |g(z_i) - g_k(z_i)| \le \epsilon / 2
    \end{align*} We construct another covering set $V = \{v(z) = f_{k_1}(z) + g_{k_2}(z): k_1\in [N_1], k_2\in [N_2]\}$. It is obvious that $|V| \le N_1N_2$. Moreover, for any $h\in \mathcal{H}_1 + \mathcal{H}_2$ with $h=f+g$, 
    \begin{align*}
        \sup_{i\in [n]} \inf_{v \in V} |h(z_i) - v(z_i)| &= \sup_{i\in [n]} \inf_{v \in V} |f(z_i) + g(z_i) - v(z_i)| \\
        &= \sup_{i\in [n]} \inf_{k_1\in [N_1], k_2 \in [N_2]} |f(z_i) + g(z_i) - f_{k_1}(z_i) - f_{k_2}(z_i)|\\
        &\le \sup_{i\in [n]} \inf_{k_1\in [N_1]} |f(z_i) - f_{k_1}(z_i)| + \sup_{i\in [n]} \inf_{k_2\in [N_2]} |f(z_i) - f_{k_2}(z_i)| \le \epsilon.
    \end{align*}
    Therefore, one has
    \begin{align*}
    \mathcal{N}_{\infty}(\epsilon, \mathcal{H}_1 + \mathcal{H}_2, z_1^n) \le |V| \le N_1N_2
    \end{align*} Taking supremum over all the $z_1,\ldots, z_n$ completes the proof.
\end{proof}

\begin{lemma}[Calculating Local Rademacher Complexity with Uniform Covering Number]
\label{lemma:local-rademacher-uniform-covering-number}
    Let $Z_1,\ldots, Z_n \overset{i.i.d.}{\sim} \nu$ be random variables on $\mathcal{Z}$, and $\mathcal{H}$ be a function class satisfying $\sup_{h\in \mathcal{H}}\|h\|_\infty \le b$ \begin{align}
\label{eq:entropy-h}
    \log \mathcal{N}_\infty(\epsilon, \mathcal{H}, n) \le A_1 \log (A_2/\epsilon) \qquad \qquad \forall \epsilon \in (0, b]
\end{align} where $(A_1, A_2)$ are dependent on $\mathcal{H}$ and $n$ but independent of $\epsilon$. Then there exists some universal constant $C$ such that, for any $n\ge 3$
\begin{align*}
    R_{n,\nu}(\delta; \mathcal{H}) \le b \delta_n \delta\qquad \forall \delta \in [\delta_n, b] 
\end{align*} with $\delta_n =C\sqrt{n^{-1} (A_1\log (A_2n)+\log(bn))}$.
\end{lemma}
\begin{proof}
Without loss of generality, we prove the claim when $b=1$.
Let $Z_1,\ldots, Z_n \overset{i.i.d.}{\sim} \nu$, and $\varepsilon_1,\ldots, \varepsilon_n$ be i.i.d. Rademacher random variables that is also independent of $(Z_1,\ldots, Z_n)$. In the proof, we denote $\|h\|_{2,\nu} = \{\int h^2\nu(dx)\}^{1/2}$ and $\|h\|_{n,\nu}=\{\frac{1}{n} \sum_{i=1}^n h^2(Z_i)\}^{1/2}$.

\noindent {\sc Step 1. Application of Chaining.} Let  $V_u = \frac{A_1 \log (A_2 n) + u+1}{n}$ for arbitrary $u>0$. Define the event
\begin{align}
    \mathcal{A}_t(\delta) = \left\{\sup_{h\in \mathcal{H}, \|h\|_{n,\nu} \le \delta} \left|\frac{1}{n} \sum_{i=1}^n \varepsilon_i h(Z_i) \right|\le 10 \delta \sqrt{\frac{A_1 \log (A_2 n) + u+1}{n}}\right\}
\end{align} we claim that $\mathbb{P}[\mathcal{A}_t(\delta)] \ge 1-4e^{-u}$. For any fixed $(Z_1,\ldots, Z_n)$, we define $\mathcal{H}_n(\delta) = \{h\in \mathcal{H}: \|h\|_{n, v} \le \delta\}$, $X_h = \frac{1}{n} \sum_{i=1}^n h(Z_i) \varepsilon_i$ for any $h \in \mathcal{H}$, and let define the distance $d(h, h')=n^{-1/2}\|h - h'\|_{n, \nu}$. It is easy to see that $\sup_{h, h'\in \mathcal{H}_n(\delta)} d(h, h') \le 2n^{-1/2}\delta$, and
\begin{align*}
    \mathbb{P}\left[|X_h - X_{h'}| \ge u|Z_1^n\right] \le e^{-{u^2}/(2d(h, h'))}
\end{align*} using Hoeffding bound and the fact that $\varepsilon_1,\ldots,\varepsilon_n$ are i.i.d. random variable with $\mathbb{E} [e^{\lambda \varepsilon_1}] \le e^{\frac{1}{2} \lambda^2}$ for any $\lambda > 0$. Then it follows from the chaining argument (Lemma 11 in \cite{fan2024factor}) that 
\begin{align*}
    \sup_{h, h'\in \mathcal{H}_n(\delta)} |X_h - X_{h'}| \le \alpha(\epsilon, u) + \int_{\epsilon/4}^{2n^{-1/2}\delta} \sqrt{\log \mathcal{N}(\omega, \mathcal{H}_n(\delta), d)} d\omega + (\sqrt{u} + 1) \frac{2\delta}{\sqrt{n}}
\end{align*} with probability at least $1-2e^{-u}$, where $\alpha(\epsilon, u)$ satisfies $\mathbb{P}\left[\sup_{h,h'\in \mathcal{H}_n(\delta), d(h,h') \le \epsilon} |X_h-X_{h'}| \le \alpha(\epsilon, u)\right] \ge 1-e^{-u}$. We choose $\epsilon=n^{-1.5}$. On one hand, it follows from Cauchy-Schwarz inequality that
\begin{align*}
    \sup_{h,h'\in \mathcal{H}_n(\delta), d(h,h') \le \epsilon} \le \sqrt{\frac{1}{n} \sum_{i=1}^n \varepsilon_i^2} \times  \|h-h'\|_{n,\nu} \le \epsilon \sqrt{n} \le n^{-1} = \alpha(\epsilon, u)
\end{align*} with probability 1 since $\frac{1}{n} \sum_{i=1}^n \varepsilon_i^2=1$. On the other hand, observe that any $\epsilon$-net of $\mathcal{H}_n(\delta)$ with respect to $\|\cdot\|_{\infty, \{Z_1^n\}}$ is also a $\sqrt{n}\epsilon$-net of $\mathcal{H}_n(\delta)$ with respect to the metric $d$ we defined. Therefore, applying \eqref{eq:entropy-h} yields
\begin{align*}
    &\int_{n^{-3/2}/4}^{2n^{-1/2}\delta}  \sqrt{\log \mathcal{N}(\omega, \mathcal{H}_n(\delta), d)} d\omega \\
    &~~~\le \int_{n^{-3/2}/4}^{2n^{-1/2}\delta}  \sqrt{\log \mathcal{N}_\infty(\omega/n^{1/2} , \mathcal{H}, n)} d\omega \\
    &~~~\le \int_{n^{-3/2}/4}^{2n^{-1/2}\delta} \sqrt{A_1 \log (A_2\sqrt{n}/\omega)} d\omega
    \le 4n^{-1/2} \delta \sqrt{A_1 \log (A_2 n)}.
\end{align*} Putting all the pieces together, we have
\begin{align*}
    \sup_{h, h'\in \mathcal{H}_n(\delta)} (X_h - X_{h'}) \le n^{-1} + 8\left(\sqrt{\frac{A_1\log(A_2 n) + u + 1}{n}} \right) 
\end{align*} Combine with the fact that for fixed $\breve{h}$, $\mathbb{P}[|X_{\breve{h}}| \ge \sqrt{u/n} \delta|Z_1^n] \le 2e^{-u}$, we find
\begin{align*} 
\sup_{h\in \mathcal{H}_n(\delta)} |X_h| \le \sup_{h\in \mathcal{H}_n(\delta)} |X_h - X_h'| + |X_{\breve{h}}| \le 10\left(\sqrt{\frac{A_1\log(A_2 n) + u + 1}{n}} \right) 
\end{align*} with probability at least $1-4e^{-u}$. Therefore, we conclude $\mathbb{P}[\mathcal{A}_t(\delta)] = \mathbb{E}[\mathbb{P}(\mathcal{A}_t(\delta)|Z_1^n)] \le 1-4e^{-u}$. 

\noindent {\sc Step 2. Application of peeling.} Let $\delta_\star=\sqrt{n^{-1} (5+t+A_1\log(A_2 n) + \log(bn))}$. In this step, we show that the following event
\begin{align*}
    \mathcal{B}_t = \left\{ \left|\frac{1}{n} \sum_{i=1}^n h(Z_i) \varepsilon_i \right| \le 20 \left(\|h\|_{n, \nu} \delta_\star + \delta_\star^2\right)\right\}
\end{align*} occurs with probability at least $1-e^{-t}$. To this end, we apply a peeling device. Define $\eta_\ell = 2^{\ell-1}$ for $\ell \ge 1$ and $\eta_0 = 0$. Let $\mathcal{H}_\ell = \{h\in \mathcal{H}: \eta_{\ell-1} \delta_* \le \|h\|_{n,\nu} \le \eta_\ell \delta_*\}$, then
\begin{align*}
    \mathbb{P}\left[\mathcal{B}_t^c\right] &= \mathbb{P}\left[\exists h\in \mathcal{H}, \left|\frac{1}{n} \sum_{i=1}^n h(Z_i) \varepsilon_i \right| \ge 20 \left(\|h\|_{n, \nu} \delta_\star + \delta_\star^2\right)\right] \\
    &\overset{(a)}{\le} \sum_{\ell=1}^{\lceil\log_2(b/\delta_*)\rceil} \mathbb{P}\left[\exists h\in \mathcal{H}_\ell, \left|\frac{1}{n} \sum_{i=1}^n h(Z_i) \varepsilon_i \right| \ge 20 \left(\|h\|_{n, \nu} \delta_\star + \delta_\star^2\right)\right] \\
    &\overset{(b)}{\le} \sum_{\ell=1}^{\lceil\log_2(b/\delta_*)\rceil} \mathbb{P}\left[\sup_{h\in \mathcal{H}, \|h\|_{n, \nu} \le \eta_l \delta_* }\left|\frac{1}{n} \sum_{i=1}^n h(Z_i) \varepsilon_i \right| \ge 20 \left((\eta_{\ell-1} \delta_\star)\delta_\star + \delta_\star^2\right)\right] \\
    &\overset{(c)}{\le} \sum_{\ell=1}^{\lceil\log_2(b/\delta_*)\rceil} \mathbb{P}\left[\sup_{h\in \mathcal{H}, \|h\|_{n, \nu} \le 
    \eta_l \delta_* }\left|\frac{1}{n} \sum_{i=1}^n h(Z_i) \varepsilon_i \right| \ge 10 (\eta_{\ell} \delta_\star^2) \right] \\
    &\overset{(d)}{\le} 4\log(bn) 4e^{-t-4-2\log(bn)} \le e^{-t}.
\end{align*} where $(a)$ follows from union bound and the uniform boundedness of $\mathcal{H}$, $(b)$ follows from the definition of $\mathcal{H}_\ell$, $(c)$ follows from the fact that $\eta_{\ell-1}+1 \ge \frac{1}{2}\eta_{\ell}$, $(d)$ follows from the result we obtained in {\sc Step 1}. 

\noindent {\sc Step 3. } Define the event
\begin{align*}
    \mathcal{C}_t = \left\{\left|\|h\|_{n, \nu}^2 - \|h\|_{2, \nu}^2\right| \le \frac{1}{2} \|h\|_{2, \nu}^2 + C_1b^2\delta_*\right\}
\end{align*} for some large enough constant $C_1>0$, we apply Theorem 19.3 in \cite{gyorfi2002distribution} to show that $\mathbb{P}[\mathcal{C}_t] \ge 1-e^{-t}$. Without loss of generality, we prove for the case where $b=1$, the results for general $b$ can be obtained via re-scaling argument. 

Let $\bar{\mathcal{H}} = \{h^2: h\in \mathcal{H}\}$. We have $\|\bar{h}\|_\infty \le 1$ for any $\bar{h} \in \bar{\mathcal{H}}$ and $\mathbb{E}[\bar{h}^2] \le \mathbb{E}[\bar{h}]$.  Following their notations, we choose $\epsilon=1/2$ and $\alpha=C_2\frac{A_1 \log (A_2n)+t+1}{n}$ for some large enough universal constant $C_2>0$ to be determined. Note any $\epsilon$-net of $\mathcal{H}$ w.r.t. $\|\cdot\|_{\infty, \{z_1^n\}}$ norm is also a $(b\epsilon)$-net of $\bar{\mathcal{H}}$ with respect to $\|\cdot\|_{n, \nu}$ norm, then for any $\omega \in (0,ne)$, it follows from \eqref{eq:entropy-h} that
\begin{align*}
    \log \mathcal{N}_2(\omega, \{h\in \mathcal{H}, \|h\|_{n, v}^2\le 16\delta\}, z_1^n) \le \log \mathcal{N}_\infty(\omega/b,\mathcal{H}, z_1^n) \le A_1\log (A_2/\omega).
\end{align*} We let $C_2$ be large enough such that the condition (Eq (19.11) therein)
\begin{align*}
    \int_{\delta/(128)}^{\sqrt{\delta}} \sqrt{\log \mathcal{N}_2(\omega, \{h\in \mathcal{H}, \|h\|_{n, v}^2\le 16\delta\}, z_1^n)} \le \sqrt{\delta A_1\log (A_2 128n)} \le \frac{\sqrt{n} \delta}{748\sqrt{2}}
\end{align*} is satisfied for any $\delta \ge \alpha/8$. Applying Theorem 19.3 in \cite{gyorfi2002distribution}, we obtain
\begin{align*}
    \mathbb{P}\left[\sup_{h\in \mathcal{H}} \frac{|\|h\|_{n, \nu}^2 - \|h\|_{2, \nu}^2|}{\|h\|_{2, \nu}^2+\alpha} \ge 1/2\right] \le C_3 e^{-n\alpha/C_3}
\end{align*} for some large enough constant $C_3>0$. We let $C_1$ be also large enough such that $C_2 \ge C_3 \log C_3$. Then the above inequality implies that 
\begin{align*}
    \forall h\in \mathcal{H} \qquad |\|h\|_{n, \nu}^2 - \|h\|_{2, \nu}^2| \le \frac{1}{2} \left(\|h\|_{2, \nu}^2+\alpha\right)
\end{align*} with probability at least $1-e^{-n\alpha} \ge 1-e^{-t}$. This completes the proof.

\noindent {\sc Step 4. Conclude the proof.} By the definition of $R_{n,\nu}(\delta;\mathcal{H})$, for fixed $\delta \in [\delta_*, b]$, 
\begin{align*}
R_{n,\nu}(\delta;\mathcal{H}) &= \mathbb{E} \left[\sup_{h\in \mathcal{H}, \|h\|_2 \le \delta} \left|\frac{1}{n} \sum_{i=1}^n h(Z_i) \varepsilon_i\right|\right] \\
&= \mathbb{E} \left[\sup_{h\in \mathcal{H}, \|h\|_{2,\nu} \le \delta} \left|\frac{1}{n} \sum_{i=1}^n h(Z_i) \varepsilon_i\right| 1\{\mathcal{B}_{\log n}, \mathcal{C}_{\log n} ~\text{all occur}\}\right] \\
&~~~~~~~~~~~~ + \mathbb{E} \left[\sup_{h\in \mathcal{H}, \|h\|_2 \le \delta} \left|\frac{1}{n} \sum_{i=1}^n h(Z_i) \varepsilon_i\right| 1\{\text{one of } \mathcal{B}_{\log n}, \mathcal{C}_{\log n} ~\text{does not occur}\}\right] \\
&\le \mathbb{E} \left[\sup_{h\in \mathcal{H}, \|h\|_{2,\nu} \le \delta} 20\left(\|h\|_{n,\nu} \delta_* + \delta_*^2\right) 1\{\mathcal{C}_{\log n} \text{ occurs}\}\right] + \frac{2b}{n}\\
&\le \mathbb{E} \left[\sup_{h\in \mathcal{H}, \|h\|_{2,\nu} \le \delta} 20\left\{\sqrt{2\|h\|_{2,\nu}^2 + 2C_1b^2\delta_*}\delta_* + \delta_*^2\right\}\right] + \frac{2b}{n}\\
&\le C_4 \left(\delta \delta_* + b\delta_*^2\right).
\end{align*} for some universal constant $C_4>0$. Therefore it is clear that $R_{n,\nu}(\delta;\mathcal{H}) \le b\delta \delta_n$ for any $\delta \ge \delta_n = \sqrt{2C_4} \delta_*$.

\end{proof}

The following lemma is used to characterize the non-asymptotic approximation error for the ReLU neural network. 

\begin{lemma}[Upper bound on neural network approximation error for $\mathcal{H}_{\mathtt{HS}}$ and $\mathcal{H}_{\mathtt{HCM}}$]
\label{lemma:nn-approx-error}
There exists some universal constants $C$ depending only on $(d, \beta_s)$ such that for arbitrary $g\in \mathcal{H}_{\mathtt{HS}}(d, \beta_s, C_s)$ and $N, L\in \mathbb{N}^+ \setminus \{1\}$, there exists a deep ReLU network $g^\dagger \in \mathcal{H}_{\mathtt{nn}}(d, C\lceil L\log L\rceil, C\lceil N\log N\rceil, \infty, \infty)$ satisfying
\begin{align*}
    \|g^\dagger-g\|_{\infty,[0,1]^d} \le C C_s(NL)^{-2\beta_s/d}.
\end{align*}
Let $b$ be some fixed positive constant, and $\mathcal{P}$ be such that $\sup_{(\beta, t)\in \mathcal{P}} (\beta \lor t) <\infty$. There also exists some universal constants $C$ depending on $(d, l, \mathcal{P}, C_h, \sup_{(\beta, t)\in \mathcal{P}} (\beta \lor t), b)$ such that for arbitrary $g\in \mathcal{H}_{\mathtt{HCM}}(d, l,\mathcal{P}, C_h)$ and  $N, L\in \mathbb{N}^+ \setminus \{1\}$, there exists a deep ReLU network $g^\dagger \in \mathcal{H}_{\mathtt{nn}}(d, C\lceil L\log L\rceil, C\lceil N\log N\rceil, \infty, \infty)$ satisfying
\begin{align*}
    \|g^\dagger-g\|_{\infty,[-b,b]^d} \le C (NL)^{-\inf_{(\beta, t) \in \mathcal{P}}(2\beta/t)}.
\end{align*}
\end{lemma}
\begin{proof}[Proof of \cref{lemma:nn-approx-error}] The proof for $\mathcal{H}_{\mathtt{HS}}$ can be found in Theorem 1 in \cite{lu2021deep}. The proof for $\mathcal{H}_{\mathtt{HCM}}$ can be found in Proposition 3.4 in \cite{fan2024noise}.
\end{proof}

We also need the following simple approximation result.

\begin{lemma}
\label{lemma:approx-minus}
    Let $L \in \mathbb{N}$, $N \ge 4$, $b\in \mathbb{R}^+$ and $V\ge b \lor 1$ be arbitrary.
    For any $g_1\in \mathcal{H}_{\mathtt{nn}}(d, L, N, b, V)$ and $g_2\in \mathcal{H}_{\mathtt{nn}}(d, L, N, b, V)$, there exists some $g^\dagger \in \mathcal{H}_{\mathtt{nn}}(d, L+2, 2N, 2b, V)$ such that
    \begin{align*}
    \forall x\in \mathbb{R}^d\qquad
        g^\dagger(x) = g_1(x) - g_2(x)
    \end{align*}
\end{lemma}
\begin{proof}[Proof of \cref{lemma:approx-minus}]
Suppose $g_1 = \mathrm{Tc}_b(\tilde{g}_1)$ and $g_2 = \mathrm{Tc}_b(\tilde{g}_2)$ for $\tilde{g}_1, \tilde{g}_2 \in \mathcal{H}_{\mathtt{nn}}(d, L, N, \infty, V)$. We construct 
\begin{align*}
    g^\dagger = \mathrm{Tc}_{2b} \left(\max(\min(\tilde{g}_1(x), b), -b) - \max(\min(\tilde{g}_2(x), b), -b)\right).
\end{align*} It is easy to verify that $g^\dagger = g_1(x) - g_2(x)$.

Meanwhile, it follows from Lemma 12 in \cite{fan2024factor} that $\min, \max$ can be implemented using ReLU neural networks with depth $1$, width $4$, and weights upper bounded by $1$, then our constructed $g^\dagger$ can be implemented via ReLU network with depth $L+2$ and width $2N\lor 8$. This completes the proof.
\end{proof}

\subsection{Proof of Theorem \ref{thm:fairnn}}
\label{sec:proof-fairnn}

Throughout the proof, we note that the quantities defined in \cref{thm:oracle} satisfy
\begin{align*}
    \max\left\{\sigma_y, C_y,\frac{\log |\mathcal{E}|}{\log n}, B\right\} \le \mathrm{poly}(C_0) \lesssim 1
\end{align*}
To apply \cref{thm:oracle}, we need to first calculate $\mathsf{b}_{\mathcal{G}}(S)$ and $\bar{\mathsf{d}}_{\mathcal{G}, \mathcal{F}}(S)$ and verify the general condition \cref{cond:general-gf}, and then calculate the approximation error and stochastic error to establish concrete $L_2$ error bound.

\noindent {\sc Step 1. Verify \cref{cond:general-gf}.} We choose $\overline{\mathcal{G}_S}=\overline{\mathcal{F}_S} = \Theta_S$, thus the first part ``invariance'' and the last part ``nondegenerate covariate'' are satisfied with $g^\star = m^\star$ provided \cref{cond:regularity-fairnn} (b)--(c) holds. Given that the projections are all the identity map, we have
\begin{align}
\label{eq:bias-var-nn}
    \mathsf{b}_{\mathcal{G}}(S) = \|m^\star - \bar{m}^{(S\cup S^\star)}\|_{2}^2 \qquad \text{and} \qquad \bar{\mathsf{d}}_{\mathcal{G}, \mathcal{F}}(S) = \frac{1}{|\mathcal{E}|} \sum_{e\in \mathcal{E}} \|\bar{m}^{(S)} - m^{(e,S)}\|_{2,e}^2.
\end{align} Moreover, 
\begin{align*}
\forall S\subseteq [d] ~with~ \mathsf{b}_{\mathcal{G}}(S) > 0 \qquad \overset{(a)}{\Longrightarrow} \qquad & \bar{\mu}(\{\bar{m}^{(e,S\cup S^\star)} \neq m^\star\}) > 0\\
\overset{(b)}{\Longrightarrow} \qquad & \exists e, e'\in \mathcal{E} ~s.t.~ (\mu^{(e)}\land \mu^{(e')})(\{m^{(e,S)} \neq m^{(e',S)}\}) > 0  \\
\overset{(c)}{\Longrightarrow} \qquad & \bar{\mathsf{d}}_{\mathcal{G}, \mathcal{F}}(S) > 0
\end{align*}
where $(a)$ follows from the calculate of $\mathsf{b}_{\mathcal{G}}(S)$ and the fact that $\|f\|_{L_2(\nu)}>0$ implies $\nu(\{f\neq 0\})>0$, $(b)$ follows from \cref{cond-fairnn-ident}. The derivation of $(c)$ follows from the fact that
\begin{align*}
    |\mathcal{E}| \bar{\mathsf{d}}_{\mathcal{G}, \mathcal{F}}(S) = \sum_{e\in \mathcal{E}} \|m^{(e,S)} - \bar{m}^{(S)} \|_{2,e}^2 &\ge \|m^{(e,S)} - \bar{m}^{(S)} \|_{2,e}^2 + \|m^{(e',S)} - \bar{m}^{(S)} \|_{2,e'}^2 \\
    &\overset{(a)}{\ge} \int \frac{\{m^{(e,S)} - m^{(e',S)}\}^2}{2} (\mu^{(e)} \land \mu^{(e')})(dx) \overset{(b)}{>} 0.
\end{align*} where $(a)$ follows from the fact that $ (x-a)^2 + (x-b)^2 \ge \frac{1}{2}(b-a)^2$, $(b)$ follows from the fact that $\nu(\{f\neq 0\})>0$ implies $\|f\|_{L_2(\nu)}>0$. Therefore, ``heterogeneity'' in \cref{cond:general-gf} is verified.

\noindent {\sc Step 2. Verifying \cref{cond:general-function-class} and specifying $\delta_n$.} Applying further Theorem~7 of \cite{BHLM2019} yields the bound ${\rm Pdim}(\mathcal{H}_{\mathtt{nn}})  \lesssim W L \log(W)$, where $W$ is the number of parameters of the network $\mathcal{H}_{\mathtt{nn}}$, this indicates that
\begin{align*}
    {\rm Pdim}(\mathcal{G}) + {\rm Pdim}(\mathcal{F}) \lesssim (LN^2 + dN) L \log(LN^2 + dN) \lesssim C_1 L^2N^2 (1+\log n).
\end{align*}

It then follows from Theorem 12.2 of \cite{AB1999} that, for any $\epsilon \in (0,2B]$
\begin{align*}
    \log \mathcal{N}_\infty(\varepsilon, \mathcal{G}, n) \lor \log \mathcal{N}_\infty(\varepsilon, \mathcal{F}, n) &\le \left(\mathrm{Pdim}(\mathcal{H}) \lor \mathrm{Pdim}(\mathcal{F}) \right)\log\left(\frac{eBn}{\epsilon}\right) \\
    &\lesssim C_1 (NL)^2 (1+\log n) \log\left(eBn/\epsilon\right)
\end{align*}

Then it follows from \cref{lemma:uniform-cover-number-plus} and the fact that $n\ge 3$ that that
\begin{align*}
\log \mathcal{N}_\infty(\varepsilon, \partial \mathcal{G}, n) \le 2 \log \mathcal{N}_\infty(\varepsilon/2, \mathcal{G}, n) \lesssim C_1 (NL)^2 (\log n) \log\left(eBn/\epsilon\right)
\end{align*} and
\begin{align*}
\log \mathcal{N}_\infty(\varepsilon, \partial (\mathcal{G}+\mathcal{F}), n) \le 2 \log \mathcal{N}_\infty(\varepsilon/2, (\mathcal{G}+\mathcal{F}), n) \lesssim C_1 (NL)^2 (\log n) \log\left(eBn/\epsilon\right)
\end{align*} Applying \cref{lemma:local-rademacher-uniform-covering-number}, we find that \cref{cond:general-function-class} holds with $\delta_n = \tilde{C}_1 n^{-1/2} (NL)\log n$ for some large constant $\tilde{C}_1>0$ that may depend on $(d, C_0)$.

\noindent {\sc Step 3. Calculating Approximation Error.} 
We define
\begin{align*}
&\delta_{\mathtt{a}, \mathtt{NN}}^\star := \max_{e\in \mathcal{E}} \inf_{g\in [\mathcal{H}_{\mathtt{nn}}(d, L, N, b_m)]_{S^\star}} \|m^\star - g\|_{2,e} \qquad \text{and} \qquad \\
&~~~~~\delta_{\mathtt{a}, \mathtt{NN}}^\dagger := \max_{e\in \mathcal{E}, S\subseteq[d]}\inf_{g\in [\mathcal{H}_{\mathtt{nn}}(d, L, N, b_m)]_S} \|m^{(e,S)} - g\|_{2,e}
\end{align*}
By the boundedness condition $B\ge C_0$ and \cref{cond:regularity-fairnn} (d), it is easy to verify that $\delta_{\mathtt{a}, \mathcal{G}} \le \delta_{\mathtt{a}, \mathtt{NN}}^\star$. At the same time, for any $e\in \mathcal{E}$ and $S\subseteq [d]$ and arbitrary $\epsilon$, there exist some $u^{(e,S)}\in \mathcal{H}_{\mathtt{nn}}(d, L, N, b_m)$ such that
\begin{align*}
    \|u^{(e,S)} - m^{(e,s)}\|_{2,e} \le \delta_{\mathtt{a}, \mathtt{NN}}^\dagger + \epsilon.
\end{align*} It then follows \cref{lemma:approx-minus} that,  for any $e\in \mathcal{E}$ and $S\subseteq [d]$,
\begin{align*}
    \sup_{g\in \mathcal{G}: S_g = S} \inf_{f\in \mathcal{F}_S} \|m^{(e,S)} - g - f\|_{2,e} &\le \sup_{g\in \mathcal{G}: S_g = S} \|m^{(e,S)} - g - (u^{(e,S)} - g)\|_{2,e} \le  \delta_{\mathtt{a}, \mathtt{NN}}^\dagger + \epsilon,
\end{align*} provided $\mathcal{G} = \mathcal{H}_{\mathtt{nn}}(d, L, N, b_m)$ and $\mathcal{F} = \mathcal{H}_{\mathtt{nn}}(d, L+2, 2N, 2b_m)$ and $N\ge 4$. This yields \begin{align*}
    \delta_{\mathtt{a}, \mathcal{F}, \mathcal{G}}(S) \le \delta_{\mathtt{a}, \mathtt{NN}}^\dagger
\end{align*} for any $S\subseteq [d]$ via letting $\epsilon \to 0$. Similarly, we also have
\begin{align*}
    \delta_{\mathtt{a}, \mathcal{F}, \mathcal{G}}^\star \le \delta_{\mathtt{a}, \mathtt{NN}}^\star.
\end{align*}

\noindent {\sc Step 4. Apply \cref{thm:oracle}.} Now that \cref{cond:general-gf} is validated and \cref{cond:general-function-class} is also validated with $\delta_n = \tilde{C}_1 n^{-1/2} (NL)(1+\log n)$. Observe $U\le \tilde{C}_2 \log n$ for some constant $\tilde{C}_2$ dependent on $C_0$. Applying \cref{thm:oracle} (1), we obtain that if $\gamma \ge 8\gamma^\star$ where $\gamma^\star = \sup_{\mathsf{b}_{\mathcal{G}}(S)>0} \mathsf{b}_{\mathcal{G}}(S) / \bar{\mathsf{d}}_{\mathcal{G}, \mathcal{F}}(S)$ with $\mathsf{b}_{\mathcal{G}}(S)$ and $\bar{\mathsf{d}}_{\mathcal{G}, \mathcal{F}}(S)$ defined in \eqref{eq:bias-var-nn}, then the following holds with probability at least $1-\tilde{C}_3 n^{-100}$, 
\begin{align*}
    \|\hat{g} - m^\star\|_{2} &\le \tilde{C}_3(1+\gamma)\left(U\delta_{n,\log^{100} n} + \delta_{\mathtt{a}, \mathcal{G}} + \delta_{\mathtt{a}, \mathcal{F}, \mathcal{G}}(S_{\hat{g}}) + \delta_{\mathtt{a}, \mathcal{F}, \mathcal{G}}(S^\star) + \delta_{\mathtt{opt}} \right)\\
    &\le \tilde{C}_3 (1+\gamma) \left(\frac{NL\log^{3/2} n}{\sqrt{n}} + \delta_{\mathtt{a}, \mathtt{NN}}^\dagger \right) = \tilde{C}_3(1+\gamma)\delta_{\mathtt{NN}, 1},
\end{align*} since $\delta_{\mathtt{a}, \mathtt{NN}}^\star \le \delta_{\mathtt{a}, \mathtt{NN}}^\dagger$ and $\delta_{\mathtt{opt}}=0$, where $\tilde{C}_3$ is some constant dependent on $(d, C_0)$. This completes the proof of the generic convergence rate holds for all the $n\ge 3$. Moreover, via setting large enough $\tilde{C}\ge \tilde{C}_4 > 1+C \tilde{C}_3$ in the statement such that $\delta^2 \le \delta$ for $\delta = \delta_{\mathtt{a}, \mathtt{NN}}^\dagger$, the inequality on $\delta_{\mathtt{NN},1}$ further yields that
\begin{align*}
    &(1+\gamma) \left\{\delta_{\mathtt{opt}}^2 + (\delta_{\mathtt{a}, \mathtt{NN}}^\dagger)^2 + (\delta_{\mathtt{a}, \mathtt{NN}}^\star)^2 + \frac{NL\log^{3/2} n}{\sqrt{n}} \right\} \\
    &~~~ \le  \tilde{C}_3(1+\gamma)\delta_{\mathtt{NN}, 1} \\
    &~~~\le \left\{ s_{\min}\land \left(\gamma \inf_{S: \bar{\mathsf{d}}_{\mathcal{G}, \mathcal{F}}(S) > 0} \bar{\mathsf{d}}_{\mathcal{G}, \mathcal{F}}(S)\right) \right\} / C
\end{align*} where $C$ is the universal constant in \cref{thm:oracle}, which validates the $\eqref{eq:main-result-faster-cond}$, apply \cref{thm:oracle} (2), we have, with probability at least $1-\tilde{C}_5 n^{-100}$, 
\begin{align*}
    \|\hat{g} - m^\star\|_{2} &\le C\left(U\delta_{n,\log(n^{100})} + \delta_{\mathtt{a}, \mathcal{G}} +  \delta_{\mathtt{a}, \mathcal{F}, \mathcal{G}}^\star\right)\\
    &\le \tilde{C}_5 \left(\frac{NL\log^{3/2} n}{\sqrt{n}} + \delta_{\mathtt{a}, \mathtt{NN}}^\star  \right).
\end{align*} where $\tilde{C}_5$ is some constant dependent on $(d, C_0)$. Setting $\tilde{C} = \max_{1\le i\le 5} \tilde{C}_i$ completes the proof.



\subsection{Proof of Theorem \ref{coro:fair-nn-fast}}

It follows from the second claim in \cref{lemma:nn-approx-error} and the assumption that $m^\star \in \mathcal{H}_{\mathtt{HCM}}(d, l, \mathcal{O}^\star, C_h)$ on bounded support $|X| \le C_0$ that
\begin{align*}
    \delta_{\mathtt{a}, \mathtt{NN}}^\star \le \max_{e\in \mathcal{E}} \inf_{g\in [\mathcal{H}_{\mathtt{nn}}(d, L, N, b_m)]_{S^\star}} \|m^\star - g\|_{2,e} \le \tilde{C}_2 \left(\frac{NL}{\log^2 n}\right)^{-2\alpha^\star}
\end{align*} provided $N \land L \ge \tilde{C}_2\log n$, where $\tilde{C}_2$ is a constant that dependent on $(C_0, d, l, \sup_{(\beta, t)\in \mathcal{O}^\star} (\beta \lor t), C_h)$. Then under our choice of $NL$, we have
\begin{align*}
    \delta_{\mathtt{NN}, 1} &\le (1+\gamma)\tilde{C}_3 \left(\sqrt{\frac{(NL)^2 \log^3 n}{n}} + \left(\frac{NL}{\log^2 n}\right)^{-2\alpha_0} \right) \\
    &\le 2\tilde{C}_3(1+\gamma) \left(\frac{\log^7 n}{n}\right)^{\frac{\alpha^\star \land \alpha_0}{2\alpha^\star + 1}} = o(1) 
\end{align*} for some constant $\tilde{C}_3$ dependent on $(C_0, d, l, \sup_{(\beta, t)\in \mathcal{O}} (\beta \lor t), C_h)$. This yields the bound applicable to any $n \ge 3$. Moreover, for large enough $n \ge n^\star$, the condition on $\delta_{\mathtt{NN}, 1}$ is satisfied. Applying the faster rate in \cref{thm:fairnn} gives, with probability at least $1-\tilde{C} n^{-100}$, the following holds, that
\begin{align*}
    \|\hat{g} - m^\star\|_2 &\le \tilde{C}_4 \left( \sqrt{\frac{(NL)^2 \log^3 n}{n}} + \left(\frac{NL}{\log^2 n}\right)^{-2\alpha^\star} \right) \\
    &\le \tilde{C}_4  \left(\frac{\log^7 n}{n}\right)^{-\frac{\alpha^\star}{2\alpha^\star+1}}.
\end{align*} Here $\tilde{C}_4$ is the constant dependent on $\tilde{C}_2$ and $(d, C_0)$, This completes the proof. \qed

\subsection{Proof of Theorem \ref{thm:fairann}}

We will use the following technical lemma.
\begin{lemma}
\label{lemma:rsc-to-closed}
Suppose \cref{cond-fairann-rsc} holds. For any $S \subseteq [d]$, $\prod_{j\in S} \Theta_{\{j\}}$ is a closed subspace of $\Theta_S$.
\end{lemma}

\begin{proof}[Proof of \cref{thm:fairann}]
\noindent {\sc Step 1. Verify \cref{cond:general-gf}.} We let
\begin{align*}
    \overline{\mathcal{G}_S} = \prod_{j\in S} \Theta_{\{j\}} \qquad \text{and} \qquad \overline{\mathcal{F}_S} = \Theta_S.
\end{align*} It follows from \cref{cond-fairann-rsc} and \cref{lemma:rsc-to-closed} that $\overline{\mathcal{G}_S}$ is a closed subspace of $\Theta_S$ for any $S\subseteq [d]$. Hence ``invariance'' with $g^\star = m^\star$ and ``irrepresentablity'' in \cref{cond:general-gf} is validated by \cref{cond-fairann-invariance}. Moreover, 
\begin{align}
\label{eq:fairann-bias-var}
    \mathsf{b}_{\mathcal{G}}(S) = \|A_{S\cup S^\star}(\bar{m}^{(S\cup S^\star)}) - m^\star\|_2^2 \qquad \text{and} \qquad\bar{\mathsf{d}}_{\mathcal{G}, \mathcal{F}}(S) = \frac{1}{|\mathcal{E}|} \sum_{e\in \mathcal{E}} \|m^{(e,S)} - A_S(\bar{m}^{(S)})\|_{2,e}^2.
\end{align} Denote $\langle f, g\rangle = \int f(x)g(x) \bar{\mu}_x(dx)$, it follows from the fact that $m^\star$ is additive that
\begin{align}
\label{eq:fairann-b-upperbound}
\mathsf{b}_{\mathcal{G}}(S) &= \|A_{S\cup S^\star}(\bar{m}^{(S\cup S^\star)} - m^\star)\|_2^2 \le \|\bar{m}^{(S\cup S^\star)} - m^\star\|_2^2.
\end{align} 

We argue that
\begin{align}
    \bar{\mathsf{d}}_{\mathcal{G}, \mathcal{F}}(S) &= \frac{1}{|\mathcal{E}|} \sum_{e\in \mathcal{E}} \|m^{(e,S)} - \bar{m}^{(S)} + \bar{m}^{(S)} - A_S(\bar{m}^{(S)})\|_{2,e}^2 \nonumber \\
    &\overset{(a)}{=} \frac{1}{|\mathcal{E}|} \sum_{e\in \mathcal{E}} \|m^{(e,S)} - \bar{m}^{(S)}\|_{2,e}^2 + \|\bar{m}^{(S)} - A_S(\bar{m}^{(S)})\|_{2,e}^2 \nonumber \\
    &= \left(\frac{1}{|\mathcal{E}|} \sum_{e\in \mathcal{E}} \|m^{(e,S)} - \bar{m}^{(S)}\|_{2,e}^2\right) + \|\bar{m}^{(S)} - A_S(\bar{m}^{(S)})\|_2^2 \label{eq:fairann-d-lb1} \\
    &\ge \left(\frac{1}{|\mathcal{E}|} \sum_{e\in \mathcal{E}} \|m^{(e,S)} - \bar{m}^{(S)}\|_{2,e}^2\right). \label{eq:fairann-d-lb2}
\end{align} where $(a)$ follows from the fact that
\begin{align*}
    &\frac{1}{|\mathcal{E}|} \int \sum_{e\in \mathcal{E}} (m^{(e,S)} - \bar{m}^{(S)})(\bar{m}^{(S)} - A_S(\bar{m}^{(S)})) \mu^{(e)}(dx) \\
    =&\frac{1}{|\mathcal{E}|} \int \sum_{e\in \mathcal{E}} \{m^{(e,S)}_S(x_S) - \bar{m}^{(S)}_S(x_S)\} \rho^{(e)}_S(x_S) \{\bar{m}^{(S)}_S(x_S) - [A_S(\bar{m}^{(S)})]_S(x_S)\} \bar{\mu}(dx) \\
    =& \left\langle  \frac{1}{|\mathcal{E}|} \sum_{e\in \mathcal{E}} m^{(e,S)}\rho^{(e)}_S - \bar{m}^{(S)}, \bar{m}^{(S)} - A_S(\bar{m}^{(S)})\right\rangle = \langle 0, \bar{m}^{(S)} - A_S(\bar{m}^{(S)})\rangle = 0.
\end{align*}

Therefore, we can validate ``heterogeneity'' in \cref{cond:general-gf} as
\begin{align*}
    \forall S\subseteq [d]~\text{with}~ \mathsf{b}_{\mathcal{G}}(S) > 0 \qquad &\overset{(a)}{\Longrightarrow} \qquad \bar{\mu}(\{m^\star \neq A_{S\cup S^\star}(\bar{m}^{(S\cup S^\star)})\}) > 0 \\
    &\overset{(b)}{\Longrightarrow} \qquad 
    \exists e, e'\in \mathcal{E}, (\mu^{(e)}\land \mu^{(e')})(\{m^{(e,S)}\neq m^{(e',S)}\}) > 0 \\
    &\qquad \qquad ~\text{OR}~ \exists e\in \mathcal{E}, \bar{\mu}(\{\bar{m}^{(S)} \neq A_S(\bar{m}^{(S)})\})>0\\
    &\overset{(c)}{\Longrightarrow} \qquad \bar{\mathsf{d}}_{\mathcal{G}, \mathcal{F}}(S) > 0,
\end{align*} where $(a)$ follows from the calculate of $\mathsf{b}_{\mathcal{G}}(S)$ and the fact that $\|f\|_{L_2(\nu)}>0$ implies $\nu(\{f\neq 0\})>0$, $(b)$ follows from \cref{cond-fairnn-ident}. The derivation of $(c)$ can be divided into two cases:

\noindent \emph{Case 1. $\exists e, e'\in \mathcal{E}, (\mu^{(e)}\land \mu^{(e')})(\{m^{(e,S)}\neq m^{(e',S)}\}) > 0$. } In this case, it follows from the lower bound \eqref{eq:fairann-d-lb2} and the discussion in {\sc Step 1.} in \cref{sec:proof-fairnn} that $\bar{\mathsf{d}}_{\mathcal{G}, \mathcal{F}}(S)>0$.

\noindent \emph{Case 2. $m^{(e,S)}$ are $\bar{\mu}_x$-a.s. the same but $\bar{\mu}(\{\bar{m}^{(S)} \neq A_S(\bar{m}^{(S)})\})$.} In this case, it follows from \eqref{eq:fairann-d-lb1} that $\bar{\mathsf{d}}_{\mathcal{G}, \mathcal{F}}(S) = \|\bar{m}^{(S)} - A_S(\bar{m}^{(S)})\|_{2}^2 > 0$.

At the same time, combining the definition of $\gamma^\star$ with the upper bound of $\mathsf{b}_{\mathcal{G}}(S)$ \eqref{eq:fairann-b-upperbound} and the lower bound of $\bar{\mathsf{d}}_{\mathcal{G}, \mathcal{F}}(S)$ \eqref{eq:fairann-d-lb2}, we find that $\gamma^\star_{\mathtt{AN}} \le \gamma^\star_{\mathtt{NN}}$.

\noindent {\sc Step 2. Verify \cref{cond:general-function-class} and Calculate Approximation Error.} It follows similar to {\sc Step 2} in \cref{sec:proof-fairnn} that \cref{cond:general-function-class} holds with $\delta_n = \tilde{C}_1 n^{-1/2} (NL)(1+\log n)$ for some constant $\tilde{C}_1$ that depends on $(C_1, d, b_m)$.

Observe that $\mathcal{H}_{\mathtt{ann}}(d, L, N, b_m) \subseteq \mathcal{H}_{\mathtt{nn}}(d, L, dN, b_m)$. Then it follows similar to {\sc Step 3} in \cref{sec:proof-fairnn} that
\begin{align*}
    \mathsf{\delta}_{\mathtt{a}, \mathcal{G}} &\le \tilde{C}_2(NL/\log^2 n)^{-2\beta^\star} \\
    \sup_{S\subseteq [d]}\mathsf{\delta}_{\mathtt{a}, \mathcal{F}, \mathcal{G}}(S) &\le \tilde{C}_2(NL/\log^2 n)^{-2\beta'/d} \\
    \mathsf{\delta}_{\mathtt{a}, \mathcal{F}, \mathcal{G}}^\star &\le \tilde{C}_2(NL/\log^2 n)^{-2\beta^\star}.
\end{align*} via applying \cref{lemma:nn-approx-error}.

\noindent {\sc Step 3. Apply \cref{thm:oracle}.} Now that \cref{cond:general-gf} is validated and \cref{cond:general-function-class} is also validated with $\delta_n = \tilde{C}_1 n^{-1/2} (NL)(1+\log n)$. Applying \cref{thm:oracle} (2) in a similar manner to the proof of \cref{thm:fairnn} and \cref{coro:fair-nn-fast} completes the proof.

\end{proof}

\begin{proof}[Proof of \cref{lemma:rsc-to-closed}]
It is easy to verify that $M=\prod_{j\in S} \Theta_{\{j\}}$ is a subspace of $\Theta_S$. It remains shows that such a subspace is closed. To this end, let $\{f^{(k)}\}_{k=1}^\infty$ be an arbitrary Cauchy sequence in $M$, that for any $\epsilon>0$, there exists some $K>0$ such that
\begin{align*}
    \|f^{(k)} - f^{(k')}\|_2 \le \epsilon \qquad \forall k, k'\ge K.
\end{align*}
Since $f^{(k)}$ is an element in $M$, without loss of generality, we can write 
\begin{align*}
    f^{(k)} = \sum_{j\in S} f_j^{(k)} + b^{(k)}
\end{align*} where $f_j^{(k)} \in \Theta_{j}$ satisfying $\int f_j^{(k)}(x) \bar{\mu}_x(dx)=0$, and $b^{(k)}$ is a scalar. Then it follows from \cref{cond-fairann-rsc} that
\begin{align*}
    \epsilon^2 \ge \|f^{(k)} - f^{(k')}\|_2^2 &= \left\|\sum_{j\in S} f_j^{(k)} - \sum_{j\in S} f_j^{(k')} + b^{(k)}-b^{(k')}\right\|_2^2 \\
    &\ge \left\|\sum_{j\in S} (f_j^{(k)} - f_j^{(k')})\right\|_2 + (b^{(k)} - b^{(k')})^2 \\
    &\ge C_a^{-1} \sum_{j\in S} \left\|f_j^{(k)} - f_j^{(k')}\right\|_2^2 \ge C_a^{-1} \|f_j^{(k)} - f_j^{(k')}\|_2^2.
\end{align*}
and $\epsilon^2 \ge \|f^{(k)} - f^{(k')}\|_2^2 \ge (b^{(k)} - b^{(k')})^2$.
Then we have, for any $j\in S$ and any $\epsilon>0$, there exists some $K$ such that
\begin{align*}
    \|f_j^{(k)} - f_j^{(k')}\|_2 \le \sqrt{C_a} \epsilon,
\end{align*} and $|b^{(k)}-b^{(k')}| \le \epsilon$. This indicates that for any $j\in [S]$, $\{f^{(k)}_j\}_{k=1}^\infty$ is a Cauchy sequence in $\Theta_{\{j\}}$, and $\{b^{(k)}\}_{k=1}^\infty$ is Cauchy sequence in $\mathbb{R}$. Because $\Theta_{\{j\}}$ and $\mathbb{R}$ are all closed, then there exists some $f_j^\star \in \Theta_{\{j\}}$ and $b^\star \in \mathbb{R}$ such that
\begin{align*}
    \lim_{k\to \infty}\|f_j^{(k)} - f_j^\star\|_2 = 0 \qquad \text{and} \qquad \lim_{k\to \infty} |b^{(k)} - b^\star| = 0
\end{align*} Let $f^\star = \sum_{j\in S} f_j^\star + b^\star$, it is easy to verify that $f^\star \in M$. Moreover, 
\begin{align*}
    \limsup_{k\to \infty}\|f^{(k)} - f^\star\|_2^2 &\le \limsup_{k\to \infty} \left\|\sum_{j\in S} (f_j^{(k)} - f_j^\star)\right\|_2^2 + \limsup_{k\to \infty} (b^{(k)} - b^\star)^2 \\
    &\le \sum_{j\in S}|S| \limsup_{k\to \infty} \left\|f_j^{(k)} - f_j^\star\right\|_2^2 + \limsup_{k\to \infty} (b^{(k)} - b^\star)^2 = 0,
\end{align*} which implies that all the Cauchy sequence in $M$ converges to an element in $M$, this verifies that $M$ is closed according to the definition of closed space in Hilbert space.

\end{proof}

\subsection{Proof of Theorem \ref{thm:lglf}}
\label{sec:proof-lglf}

Our proof will be divided into three steps. In the first step, we calculate the projections and verify the population-level invarianceheterogeneity condition. In the second step, we calculate the critical radius $\delta_n$. We conclude the proof via applying \cref{thm:oracle}.

\noindent {\sc Step 1. Validate \cref{cond:general-gf}.} We choose 
\begin{align}
\overline{\mathcal{G}_S} = \overline{\mathcal{F}_S} = \{f(x)=\beta^\top x: \beta_{S^c} = 0\} =: \mathcal{H}_S.
\end{align} We first show that $\mathcal{H}_S$ is a closed subspace of $\Theta_S$. To this end, we only need to show that for any Cauchy sequence $\{h_1, h_2\ldots, \} \subseteq \mathcal{H}_S$, there exists some $\tilde{h} \in \mathcal{H}_S$ such that $\lim_{k\to\infty}\|h_k - \tilde{h}\|_{2} = 0$. By the definition of $\mathcal{H}_S$ and $\|\cdot\|_{2}$, let $\beta(h) \in \mathbb{R}^d$ be such that $h(x) = (\beta(h))^\top x$, it is easy to see that $\|h - h'\|_{2} = \|\Sigma^{1/2}\{\beta(h) - \beta(h')\}\|_{2}$. Because $\Sigma$ is positive definition by \cref{cond:lglf} (2), the sequence $\beta(h_k)$ is also a Cauchy sequence in a $|S|$ dimensional subspace of $\mathbb{R}^d$. It follows from the completeness of $\mathbb{R}^{|S|}$ that there exists some $\tilde{\beta} \in \mathbb{R}^d$ such that $\tilde{\beta}_{S^c}=0$ and $\lim_{k\to \infty} \|\beta - \tilde{\beta}\|_2 \to 0$. Letting $\tilde{h} = \tilde{\beta}^\top x$ completes the proof.

Now we have already shown that $\mathcal{H}_S$ is a closed subspace of $\Theta_S$ for any $S\subseteq [d]$. Our proof in {\sc Step 1} will be divided into the following substeps. 

\begin{itemize}
\item In Step 1.1, we calculate the projections and verify \cref{cond:general-gf} ``invariance'' and ``nondegenerate covariate'' by \cref{cond:lglf-invariance}.
\item In Step 1.2, we calculate the $\mathsf{b}_{\mathcal{G}}(S)$ and $\bar{\mathsf{d}}_{\mathcal{G}, \mathcal{F}}(S)$ in this scenario, which are also denoted as $\mathsf{b}_{\mathtt{LL}}(S)$ and $\bar{\mathsf{d}}_{\mathtt{LL}}(S)$, respectively, and derive the inequalities in \eqref{eq:biasvar-lglf}.
\item In Step 1.3, we validate \cref{cond:general-gf} ``heterogeneity'' by \cref{cond:lglf:identification}.
\end{itemize}

\noindent {\sc Step 1.1 Calculate the Projections.} We claim that
\begin{align}
\label{eq:proof-corollary1-claim1}
\Pi_{\mathcal{H}_S}^{(e)} (m^{(e,S)}) = x^\top \beta^{(e,S)} \qquad \text{and} \qquad \Pi_{\mathcal{H}_S}({\bar{m}^{(S)}}) = x_S^\top \underbrace{\left\{\{{\Sigma}_S \}^{-1} \frac{1}{|\mathcal{E}|}\sum_{e\in \mathcal{E}} \mathbb{E}[X^{(e)}_S Y^{(e)}] \right\}}_{{\beta}^{(S)}_\dagger}
\end{align} For the first part, let $\langle f, g\rangle_e = \int f(x) g(x) \mu^{(e)}_x(dx)$, it follows from projection theorem that $\Pi_{\mathcal{H}_S} (m^{(e,S)})$ satisfies 
\begin{align*}
    \langle m^{(e,S)} - \Pi_{\mathcal{H}_S}^{(e)} (\bar{m}^{(S)}), x_j\rangle_{e} = 0 \qquad \forall j\in S.
\end{align*} We let $\Pi_{\mathcal{H}_S}^{(e)} (\bar{m}^{(S)}) = (\beta_{S})^\top x_S$ with some $\beta_S \in \mathbb{R}^{|S|}$, the above equation implies that
\begin{align*}
    \mathbb{E}[ (\mathbb{E}[Y^{(e)}|X_S^{(e)}] - X_S^\top \beta_S) X^{(e)}_S] = 0,
\end{align*} which is equivalent to $\beta_S = \{\Sigma_S^{(e)} \}^{-1} \mathbb{E}[Y^{(e)}X^{(e)}] = \beta^{(e,S)}_S$ as $\kappa_L > 0$. 

Similarly, for the second part, let $\Pi_{\mathcal{H}_S}({\bar{m}^{(S)}}) = \bar{\beta}_S^\top x_S$ with some $\bar{\beta}_S \in \mathbb{R}^{|S|}$, it also follows from the projection theorem \cref{thm:projection} that
\begin{align*}
    \forall j \in [S] \qquad 0&=\langle \bar{m}^{(S)} - \bar{\beta}_S^\top x_S, x_j\rangle \\
    &= \int \left\{\frac{1}{|\mathcal{E}|} \sum_{e\in \mathcal{E}} \rho^{(e)}_S(x_S) m^{(e,S)} - \bar{\beta}_S^\top x_j\right\} x_S \bar{\mu}_x(dx_S) \\
    &= \frac{1}{|\mathcal{E}|}\sum_{e\in \mathcal{E}} \mathbb{E} \Big[ \big\{\mathbb{E}[Y^{(e)}|X_S^{(e)}] - \bar{\beta}_S X_S^{(e)} \big\} X^{(e)}_j\Big] \\
    &= \frac{1}{|\mathcal{E}|} \sum_{e\in \mathcal{E}} \mathbb{E} [Y^{(e)} X_j^{(e)}] - \bar{\beta}_S^\top \left\{\frac{1}{|\mathcal{E}|} \sum_{e\in \mathcal{E}} \mathbb{E}[X_S^{(e)} X_j^{(e)}]\right\}.
\end{align*} Hence we obtain
\begin{align*}
    \frac{1}{|\mathcal{E}|} \sum_{e\in \mathcal{E}} \mathbb{E} [Y^{(e)} X_S^{(e)}] - {\Sigma}_S \bar{\beta}_S = 0.
\end{align*} Combining the fact that ${\Sigma}_S$ is positive-definite completes the proof the claim of \eqref{eq:proof-corollary1-claim1}.

It then follows from \cref{cond:lglf-invariance} that
\begin{align*}
\Pi_{\overline{\mathcal{F}_{S^\star}}}^{(e)} (m^{(e,S^\star)}) = x^\top \beta^{(e,S^\star)} = x^\top \beta^\star = x_{S^\star}^\top {\beta_{\dagger}^{(S^\star)}} = \Pi_{\overline{\mathcal{G}_{S^\star}}} (\bar{m}^{(S^\star)})
\end{align*} which validates the ``invariance'' condition in \cref{cond:general-gf} with $g^\star(x) = (\beta^\star)^\top x$. Meanwhile, for any $h(x) = \beta^\top x$ with $\beta \in \mathbb{R}^d$, if there exists some $j\in S^\star$ such that $\beta_j = 0$, then
\begin{align*}
    \|h - g^\star\|_2^2 = \|\Sigma^{1/2}(\beta - \beta^\star)\|_2^2 \ge \|\beta - \beta^\star\|_2^2 \ge \|\beta_j^\star\|_2^2 \ge \beta_{\min} > 0,
\end{align*} this verifies the ``nondegenerate covariate'' condition in \cref{cond:general-gf}.

\noindent {\sc Step 1.2. Calculate $(\mathsf{b}_{\mathcal{G}}(S), \bar{\mathsf{d}}_{\mathcal{G}, \mathcal{F}}(S))$ and Establish Upper/Lower Bounds.} It follows from the data generating process and \eqref{eq:proof-corollary1-claim1} that
\begin{align*}
    \mathsf{b}_{\mathtt{LL}}(S) := \mathsf{b}_{\mathcal{G}}(S) &= \int \left({\beta}_\dagger^{(S\cup S^\star)} - \beta^\star \right) x_{S\cup S^\star} x_{S\cup S^\star}^\top \left({\beta}_\dagger^{(S\cup S^\star)} - \beta^\star \right) \bar{\mu}(dx_{S\cup S^\star}) \\
    &= \left\|{\beta}_\dagger^{(S\cup S^\star)} - \beta^\star \right\|_{({\Sigma}_{S\cup S^\star})}^2 \\
    &= \left\|\{{\Sigma}_{S\cup S^\star}\}^{-1} \frac{1}{|\mathcal{E}|} \sum_{e\in \mathcal{E}} \left(\Sigma_{S\cup S^\star}^{(e)} \beta_{S\cup S^\star}^\star + \mathbb{E}[X_{S\cup S^\star}^{(e)} \varepsilon^{(e)}] \right) - \beta^\star \right\|_{({\Sigma}_{S\cup S^\star})}^2 \\
    &= \left\|\{{\Sigma}_{S\cup S^\star}\}^{-1} \frac{1}{|\mathcal{E}|} \sum_{e\in \mathcal{E}} \mathbb{E}[X_{S\cup S^\star}^{(e)} \varepsilon^{(e)}] \right\|_{({\Sigma}_{S\cup S^\star})}^2 \\
    &= \left\|\frac{1}{|\mathcal{E}|} \sum_{e\in \mathcal{E}} \mathbb{E}[X_{S\cup S^\star}^{(e)} \varepsilon^{(e)}] \right\|_{({\Sigma}_{S\cup S^\star})^{-1}}^2 \overset{(a)}{\le} \lambda_{\max}(({\Sigma}_{S\cup S^\star})^{-1}) \left\|\frac{1}{|\mathcal{E}|} \sum_{e\in \mathcal{E}} \mathbb{E}[X_S^{(e)} \varepsilon^{(e)}] \right\|_2^2 \\
    &\le \frac{1}{\lambda_{\min}({\Sigma}_{S\cup S^\star})} \left\|\frac{1}{|\mathcal{E}|} \sum_{e\in \mathcal{E}} \mathbb{E}[X_S^{(e)} \varepsilon^{(e)}] \right\|_2^2
    \le \kappa_L^{-1} \left\|\frac{1}{|\mathcal{E}|} \sum_{e\in \mathcal{E}} \mathbb{E}[X_S^{(e)} \varepsilon^{(e)}] \right\|_2^2.
\end{align*} where (a) follows from the fact that $\mathbb{E}[X_{S^\star}^{(e)} \varepsilon^{(e)}] \equiv 0$ by the invariance assumption $\mathbb{E}[\varepsilon^{(e)}|X_{S^\star}^{(e)}] \equiv 0$.

At the same time, substituting our calculations of projections into the general representation of $\bar{\mathsf{d}}_{\mathcal{G}, \mathcal{F}}(S)$, we have
\begin{align*}
   \bar{\mathsf{d}}_{\mathtt{LL}}(S) = \bar{\mathsf{d}}_{\mathcal{G}, \mathcal{F}}(S) &= \frac{1}{|\mathcal{E}|} \sum_{e\in \mathcal{E}} \int \left\{(\beta^{(e,S)}_S - {\beta}^{(S)}_\dagger)^\top x_S\right\}^2 \mu^{(e)}(dx_S) \\
    &= \frac{1}{|\mathcal{E}|} \sum_{e\in \mathcal{E}} \|\beta^{(e,S)}_S - {\beta}^{(S)}_\dagger\|_{\Sigma_S^{(e)}}^2\\
    &\ge \kappa_L \frac{1}{|\mathcal{E}|} \sum_{e\in \mathcal{E}} \|\beta^{(e,S)}_S - {\beta}^{(S)}_\dagger\|_2^2 \\
    &\ge \kappa_L \inf_{\beta: \beta_{S^c} = 0} \frac{1}{|\mathcal{E}|} \sum_{e\in \mathcal{E}} \|\beta^{(e,S)} - \beta\|_2^2 \ge \kappa_L  \frac{1}{|\mathcal{E}|} \sum_{e\in \mathcal{E}} \|\beta^{(e,S)} - \bar{\beta}^{(S)}\|_2^2
\end{align*} where the last inequality follows from the fact that $\bar{\beta}^{(S)} = \frac{1}{|\mathcal{E}|} \sum_{e\in \mathcal{E}} \beta^{(e,S)}$ minimizes the previous constrained minimization problem.

\noindent {\sc Step 1.3. Validate ``Heterogeneity'' Condition.} According to the above calculation, we have
\begin{align*}
\forall S\subseteq [d] ~with~ \mathsf{b}_{\mathcal{G}}(S) > 0 \qquad \overset{(a)}{\Longrightarrow} \qquad & \sum_{e\in \mathcal{E}} \mathbb{E}[X_S^{(e)} \varepsilon^{(e)}] \neq 0\\
\overset{(b)}{\Longrightarrow} \qquad & \exists e,e'\in \mathcal{E} ~s.t.~{\beta}^{(e,S)} \neq \beta^{(e',S)} \\
\overset{(c)}{\Longrightarrow} \qquad & \bar{\mathsf{d}}_{\mathcal{G}, \mathcal{F}}(S) > 0.
\end{align*} where $(a)$ follows from the calculate of $\mathsf{b}_S$, $(b)$ follows from \cref{cond:lglf:identification}, $(c)$ follows from the above lower bound that
\begin{align*}
    \mathsf{d}_S \ge \kappa \frac{1}{|\mathcal{E}|} \left(\|\beta^{(e,S)} - \bar{\beta}^{(S)}\|_2^2 + \|\beta^{(e',S)} - \bar{\beta}^{(S)}\|_2^2\right) > 0.
\end{align*} This completes the validation of \cref{cond:general-gf}.

\noindent {\sc Step 2. Validate \cref{cond:general-function-class}.} Now we choose $\mathcal{G} = \mathcal{H}_{\mathtt{lin}}(d, C_2, B)$ and $\mathcal{F} = \mathcal{H}_{\mathtt{lin}}(d, 2C_2, 2B)$ with
\begin{align*}
    B = 2C_2\sqrt{n} \ge 2\sigma_x \max_{e\in \mathcal{E}, S\subseteq[d]} \|\Sigma^{1/2} \beta^{(e,S)}\|_2 \sqrt{\log n}
\end{align*} it is easy to see that $\mathcal{G}$ and $\mathcal{F}$ can be represented as ReLU network with input dimension $d$, depth $0$ and width $0$, then it follows similarly to {\sc Step 2.} in the proof of \cref{thm:fairnn} that
\begin{align*}
    \mathrm{Pdim}(\mathcal{G}) \lor \mathrm{Pdim}(\mathcal{F}) \le d,
\end{align*} and \cref{cond:general-function-class} holds with $\delta_n = \tilde{C}_1  (\log n)^{3/2} \sqrt{d/n}$ for some constant $\tilde{C}_1$ dependent on $C_2$.

\noindent {\sc Step 3. Apply \cref{thm:oracle}.} Now that \cref{cond:general-gf} is validated and \cref{cond:general-function-class} is also validated with $\delta_n \le \tilde{C}_1 (\log n)^{3/2} \sqrt{d/n}$. Observe that $U \le 2B(\sigma_y\sqrt{C_1}\sqrt{\log n} + 2B)$. Setting $t=100 \log n$, the stochastic error $\delta_{n, t}$ satisfies
\begin{align*}
    U \delta_{n, t} \le \tilde{C}_2 \sqrt{d/n} \log^{5/2} n.
\end{align*} for some constant $\tilde{C}_2$ that depends on $(C_1, C_2, \sigma_x, \sigma_y)$ by \cref{cond:lglf} (1).
At the same time, observe that
\begin{align*}
    \delta_{\mathtt{a}, \mathcal{G}} 
    &\le \frac{1}{|\mathcal{E}|} \sum_{e\in \mathcal{E}} \mathbb{E}[|(\beta^\star)^\top X^{(e)} - \mathrm{Tc}_{B}((\beta^\star)^\top X^{(e)})|^2]\\
    &\le \frac{4}{|\mathcal{E}|} \sum_{e\in \mathcal{E}} \mathbb{E} \left[|(\beta^\star)^\top X^{(e)}|^2 1\{|(\beta^\star)^\top X^{(e)} \ge B\}\right]
\end{align*}
Observe that for any random variable $X$ and scalar $s$
\begin{align}
\label{eq:sub-gaussian-integral}
\mathbb{E}[X^2 1\{|X| \ge s\}] = \int x^2 1\{|x| \ge s\} \nu_x(dx) = \int \int 1\{t\le x^2, |x|\ge s\} dt \nu_x(dx) 
= \int_0^\infty \mathbb{P}(|X|\ge s \lor \sqrt{t}) dt
\end{align}
Let $v = (\Sigma^{1/2} \beta^\star)$, then it follows from the sub-Gaussian $\Sigma^{-1/2} X^{(e)}$ (\cref{cond:lglf} (2)) that, for any $e\in \mathcal{E}$
\begin{align*}
    \mathbb{E}[|v^\top \Sigma^{-1/2} X^{(e)}| 1\{|v^\top \Sigma^{-1/2} X^{(e)}| > B\}] 
    &= \int_0^\infty \mathbb{P}\left(|v^\top \Sigma^{-1/2} X^{(e)}| \ge \sqrt{t} \lor B\right) dt \\
    &\le B^2 C_{x} e^{-B^2/(2\|v\|_2^2\sigma_x^2)} + {2C_x \|v\|_2 \sigma_x^2} e^{-B^2/(2\|v\|_2^2 \sigma_x^2)} \lesssim \frac{\log n}{n}.
\end{align*} where last inequality follows from our choice of $B$. Therefore, 
\begin{align}
\label{eq:approx-a-g-linear}
    \delta_{\mathtt{a}, \mathcal{G}} \le \sqrt{\tilde{C}_3 \frac{\log n}{n}}
\end{align} for some constant $\tilde{C}_3$ dependent on $(C_1, C_2, \sigma_x, C_x)$. 
Observe that
\begin{align*}
    |x - \mathrm{Tc}_B(y) - \mathrm{Tc}_{2B}(x - y)|^2 \le 16\left\{1\{|x|\ge B\}|x|^2 + 1\{|y| \ge B\} B^2\right\}.
\end{align*}
This implies that, for any $e\in \mathcal{E}$ and $S\subseteq [d]$, 
\begin{align*}
    &\sup_{g\in \mathcal{G}_S: S_g=S} \inf_{f\in \mathcal{F}_S} \mathbb{E}[|(\beta^{(e,S)})^\top X^{(e)} - \mathrm{Tc}_{B}(\beta_g^\top X^{(e)}) - \mathrm{Tc}_{2B}(\beta_f^\top X^{(e)})|^2]\\ 
    &\le \sup_{g\in \mathcal{G}_S: S_g=S} \mathbb{E}[|(\beta^{(e,S)})^\top X^{(e)} - \mathrm{Tc}_{B}(\beta_g^\top X^{(e)}) - \mathrm{Tc}_{2B}((\beta^{(e,S)} - \beta_g)^\top X^{(e)})|^2] \\
    &\le 16 \left\{ \mathbb{E}[|(\beta^{(e,S)})^\top X^{(e)}|^2 1\{|\beta^{(e,S)})^\top X^{(e)}| \ge B\}] + \sup_{\beta: \|\Sigma^{1/2} \beta\|_{2} \le C_2} \mathbb{P} \left(|\beta^\top X^{(e)}| \ge B\right)\right\}\\
    &\le 32 \tilde{C}_3 \frac{\log n}{n}
\end{align*} provided our choice of $B$, this further implies that
\begin{align}
\label{eq:approx-a-f-linear}
    \sup_{S\subseteq [d]} \delta_{\mathtt{a}, \mathcal{G}, \mathcal{F}}(S) \le \sqrt{32\tilde{C}_3 \frac{\log n}{n}}.
\end{align}

Now we apply \cref{thm:oracle}, then for any $\gamma \ge 8\gamma^\star_{\mathtt{LL}}$ with $\gamma^\star_{\mathtt{LL}} = \sup_{S\subseteq[d]: \sum_{e\in \mathcal{E}} \mathbb{E}[X_S^{(e)}\varepsilon^{(e)}] \neq 0} \mathsf{b}_{\mathtt{LL}}(S) / \bar{\mathsf{d}}_{\mathtt{LL}}(S)$, the minimizer satisfies, with probability at least $1-C_y (\sigma_y+1) n^{-100} - n^{-100}$, 
\begin{align*}
    \|\hat{g} - g^\star\|_{2} &\le C(1+\gamma) \left(U\delta_{n, t}  + \sup_{S\subseteq [d]} \delta_{\mathtt{a}, \mathcal{G}, \mathcal{F}}(S) + \delta_{\mathtt{a}, \mathcal{G}} \right) \\
    &\le (1+\gamma) C\left(\tilde{C}_2 \sqrt{d/n} \log^{5/2} n + 8\sqrt{\tilde{C}_3} \sqrt{\frac{\log n}{n}}\right) \\
    &\le \tilde{C}_4 (1 + \gamma) \left(\log^{5/2} n \sqrt{\frac{d}{n}} \right)
\end{align*} where $\tilde{C}_4$ is a constant dependent on $(C_1, C_2, \sigma_x, C_x, \sigma_y, C_y)$.

Moreover, given that
\begin{align*}
    BU\delta_{n, t}  + \sup_{S\subseteq [d]} \delta_{\mathtt{a}, \mathcal{G}, \mathcal{F}}(S) + \delta_{\mathtt{a}, \mathcal{G}} = O \left((1+\gamma) (\log n)^3 \sqrt{\frac{d}{n}} \right) = o(1)
\end{align*} provided $d=o((1+\gamma^2)n/(\log^6 n))$, then for $n$ large enough, with same probability, 
\begin{align*}
    \|\hat{g} - g^\star\|_{2} \le C \left(U\delta_{n, t}  + \sup_{S\subseteq [d]} \delta_{\mathtt{a}, \mathcal{G}, \mathcal{F}}(S) + \delta_{\mathtt{a}, \mathcal{G}} \right) \le \tilde{C}_4 (\log n)^{5/2} \sqrt{\frac{d}{n}}.
\end{align*} 

Given that we derive high probability error bound on $\|\hat{g} - g^\star\|_2$, it follows from triangle equality that
\begin{align*}
    \|\Sigma^{1/2}(\beta_{\hat{g}} - \beta^\star)\|_2 = \|\beta_{\hat{g}}^\top x - g^\star\|_2 &\le \|\beta_{\hat{g}}^\top x - \hat{g}\|_2 + \|\hat{g} - g^\star\|_2 \le \sqrt{\tilde{C}_3 \frac{\log n}{n}} + \|\hat{g} - g^\star\|_2,
\end{align*} which completes the proof. Here the last inequality follows from the fact that
\begin{align*}
    \|\beta_{\hat{g}}^\top x - \hat{g}\|_2 \le \left\{\frac{4}{|\mathcal{E}|} \sum_{e\in \mathcal{E}} \mathbb{E}[|(\Sigma^{1/2} \beta_{\hat{g}})^\top(\Sigma^{-1/2} X^{(e)})|^2 1\{|(\Sigma^{1/2} \beta_{\hat{g}})^\top(\Sigma^{-1/2} X^{(e)})| \ge B\}]\right\}^{1/2}.
\end{align*}

\subsection{Proof of Theorem \ref{thm:lgbf}}
\label{sec:proof-lgbf}

The proof is similar to that of \cref{thm:lglf}. So we only highlight the difference. Let $\overline{\mathcal{G}_S} = \mathcal{H}_{\mathtt{lin}}(d) \cap \Theta_S$ and $\overline{\mathcal{F}_S} = \mathcal{H}_{\mathtt{alin}}(d, \phi) \cap \Theta_S$. First, it follows a similar idea to the proof in \cref{sec:proof-lglf} that both $\overline{\mathcal{G}_S}$ and $\overline{\mathcal{F}_S}$ are closed subspace of $\Theta_S$ because $\Sigma$ and $\tilde{\Sigma}$ are positive definite matrix by \cref{cond:lglf} (2) and \cref{cond:lgbf}.

\noindent {\sc Step 1. Verify \cref{cond:general-gf}.} We first claim that
\begin{align}
\label{eq:lgbf:projection}
    \Pi_{\overline{\mathcal{G}_S}}({\bar{m}^{(S)}}) = (\beta_\dagger^{(S)})^\top x_S \qquad \text{and} \qquad \Pi_{\overline{\mathcal{F}_S}} (m^{(e,S)}) = (\tilde{\beta}^{(e,S)})^\top \tilde{x}.
\end{align} The calculation of $\Pi_{\overline{\mathcal{G}_S}}({\bar{m}^{(S)}})$ is the same as that in \eqref{eq:proof-corollary1-claim1}. For $\Pi_{\overline{\mathcal{F}_S}} (m^{(e,S)})$, we let $\Pi_{\overline{\mathcal{F}_S}} (m^{(e,S)}) = \tilde{\beta}_S^\top \tilde{x}_S$, applying the projection theorem gives that, for any $j\in S$
\begin{align*}
    \langle m^{(e,S)} - \tilde{\beta}_S^\top \tilde{x}_S, x_j \rangle_e = 0 \qquad and \qquad \langle m^{(e,S)} - \tilde{\beta}_S^\top \tilde{x}_S, h(x_j) \rangle_e = 0,
\end{align*} where $\langle f, g\rangle = \int f(x) g(x) \mu^{(e)}_x(dx)$. Similarly, we have $\mathbb{E}[Y^{(e)} \tilde{X}_S^{(e)}] - \tilde{\Sigma}_S^{(e)} \tilde{\beta}_S^\top = 0$, combining with the fact that $\tilde{\Sigma}_S^{(e)}$ is positive definite thus is invertible, we obtain
\begin{align*}
    \Pi_{\overline{\mathcal{F}_S}} (m^{(e,S)}) = \big\{(\tilde{\Sigma}_S^{(e)})^{-1} \mathbb{E}[Y^{(e)} \tilde{X}_S^{(e)}]\big\}^\top \tilde{x}_S, 
\end{align*} which completes the proof of \eqref{eq:lgbf:projection}. Since all the projections are calculated, it then follows from \cref{cond:lgbf-invariance} that 
\begin{align*}
\Pi_{\overline{\mathcal{F}_{S^\star}}}^{(e)} (m^{(e,S^\star)}) = x^\top \beta^{(e,S^\star)} + \bar{\phi}(x)^\top 0  = x^\top \beta^\star = x_{S^\star}^\top {\beta_{\dagger}^{(S^\star)}} = \Pi_{\overline{\mathcal{G}_{S^\star}}} (\bar{m}^{(S^\star)})
\end{align*} hence the ``invariance'' condition in \cref{cond:general-gf} holds with $g^\star = x^\top \beta^\star$. Moreover, since $\Sigma$ is positive definite, the ``nondegenerate covariate'' condition \cref{cond:general-gf} holds with $s_{\min} = \beta_{\min}^2$.

Moreover, it follows from some simple calculation that
\begin{align*}
    \mathsf{b}_{\mathtt{LA}}(S) := \mathsf{b}_{\mathcal{G}}(S) = \mathsf{b}_{\mathtt{LL}}(S) \qquad \text{and} \qquad \bar{\mathsf{d}}_{\mathtt{LA}}(S) := \bar{\mathsf{d}}_{\mathcal{G}, \mathcal{F}}(S) = \frac{1}{|\mathcal{E}|} \sum_{e\in \mathcal{E}} \|\beta^{(e,S)}_S - [{\beta}^{(S)}_\dagger, 0]\|_{\tilde{\Sigma}_S^{(e)}}^2.
\end{align*} The calculation of $\mathsf{b}_{\mathtt{LA}}(S)$ is the same as that in \cref{sec:proof-lglf}, and 
\begin{align*}
    \bar{\mathsf{d}}_{\mathtt{LA}}(S) = \frac{1}{|\mathcal{E}|} \sum_{e\in \mathcal{E}} \int \left\{(\tilde{\beta}^{(e,S)}_S - [\beta^{(S)}_\dagger, 0])^\top \tilde{x}_S\right\}^2 \mu^{(e)}(dx_S) = \frac{1}{|\mathcal{E}|} \sum_{e\in \mathcal{E}} \|\beta^{(e,S)}_S - [{\beta}^{(S)}_\dagger, 0]\|_{\tilde{\Sigma}_S^{(e)}}^2 
\end{align*} For the inequality in \eqref{eq:biasvar-lgbf}, it follows from the fact that $\overline{\mathcal{G}_S} \subseteq \overline{\mathcal{F}_S}$ are both subspace of $\Theta_S$ that, 
\begin{align*}
    \|\beta^{(e,S)}_S - [{\beta}^{(S)}_\dagger, 0]\|_{\tilde{\Sigma}_S^{(e)}}^2 &= \|\Pi_{\overline{\mathcal{F}_S}}^{(e)} (m^{(e,S)}) - \Pi_{\overline{\mathcal{G}_S}} (\bar{m}^{(S)})\|_{2,e}^2 \\
    &= \|\Pi_{\overline{\mathcal{F}_S}}^{(e)} \{m^{(e,S)} - \Pi_{\overline{\mathcal{G}_S}} (\bar{m}^{(S)})\}\|_{2,e}^2 \\
    &\ge \|\Pi_{\overline{\mathcal{G}_S}}^{(e)} \{m^{(e,S)} - \Pi_{\overline{\mathcal{G}_S}} (\bar{m}^{(S)})\}\|_{2,e}^2 = \|\beta^{(e,S)}_S - {\beta}^{(S)}_\dagger\|_{\Sigma_S^{(e)}}^2,
\end{align*} summing over all the $e\in \mathcal{E}$ concludes the proof of inequality in \eqref{eq:biasvar-lgbf}.

Finally, according to the above calculation, we have
\begin{align*}
\forall S\subseteq [d] ~with~ \mathsf{b}_{\mathcal{G}}(S) > 0 \qquad \overset{(a)}{\Longrightarrow} \qquad & \sum_{e\in \mathcal{E}} \mathbb{E}[X_S^{(e)} \varepsilon^{(e)}] \neq 0\\
\overset{(b)}{\Longrightarrow} \qquad & \exists e\in \mathcal{E} ~s.t.~\tilde{\beta}^{(e,S)} \neq [\beta^{(e,S)},0] 
 ~~OR~~ \exists e,e'\in \mathcal{E} ~s.t.~{\beta}^{(e,S)} \neq \beta^{(e',S)} \\
\overset{(c)}{\Longrightarrow} \qquad & \bar{\mathsf{d}}_{\mathcal{G}, \mathcal{F}}(S) > 0.
\end{align*} This completes the validation of condition ``heterogeneity'' in \cref{cond:general-gf}. Here $(a)$ follows from the calculate of $\mathsf{b}_{\mathcal{G}}(S)$, $(b)$ follows from \cref{cond:lgbf:identification}. The derivation of $(c)$ can be divided into two cases:

\noindent \emph{Case 1. $\exists e\in \mathcal{E}$, $\tilde{\beta}^{(e,S)} \neq [\beta^{(e,S)},0]$.} Let $e_1$ be the environment such that $\tilde{\beta}^{(e_1,S)} \neq [\beta^{(e_1,S)},0]$, we have
\begin{align*}
    \bar{\mathsf{d}}_{\mathcal{G}, \mathcal{F}}(S) = \frac{1}{|\mathcal{E}|} \sum_{e\in \mathcal{E}} \|\beta^{(e,S)}_S - [{\beta}^{(S)}_\dagger, 0]\|_{\tilde{\Sigma}_S^{(e)}}^2 \ge \frac{1}{|\mathcal{E}|} \min_{e\in \mathcal{E}} \lambda_{\min}(\tilde{\Sigma}^{(e)}) \|\beta^{(e,S)}_S - [{\beta}^{(S)}_\dagger, 0]\|_2 > 0.
\end{align*}
\noindent \emph{Case 2. $\forall e\in \mathcal{E}$, $\tilde{\beta}^{(e,S)} = [\beta^{(e,S)},0]$.} In this case, we have
\begin{align*}
    \bar{\mathsf{d}}_{\mathcal{G}, \mathcal{F}}(S) = \frac{1}{|\mathcal{E}|} \sum_{e\in \mathcal{E}} \|\beta^{(e,S)}_S - {\beta}^{(S)}_\dagger\|_{\Sigma_S^{(e)}}^2.
\end{align*} Combining the condition that $\exists e,e'\in \mathcal{E} ~s.t.~{\beta}^{(e,S)} \neq \beta^{(e',S)}$ with the discussion {\sc Step 1.3} in \cref{sec:proof-lglf} concludes that $\bar{\mathsf{d}}_{\mathcal{G}, \mathcal{F}}(S)>0$ in this case.

\noindent {\sc Step 2. Verify \cref{cond:general-function-class}.} Now we choose $\mathcal{G} = \mathcal{H}_{\mathtt{lin}}(d, C_2, B)$ and $\mathcal{F} = \mathcal{H}_{\mathtt{alin}}(d, \phi, 2B)$ with
\begin{align*}
    B = C_2\sqrt{n} \ge 2\sigma_{\tilde{x}} \max_{e\in \mathcal{E}, S\subseteq[d]} \|\tilde{\Sigma}^{1/2} \tilde{\beta}^{(e,S)}\|_2 \sqrt{\log n}
\end{align*} it is easy to see that $\mathcal{G}$ and $\mathcal{F}$ can be represented as ReLU networks with input dimension at most $2d$, depth $0$ and width $0$, then it follows similarly to {\sc Step 2} in the proof of \cref{thm:fairnn} that
\begin{align*}
    \mathrm{Pdim}(\mathcal{G}) \lor \mathrm{Pdim}(\mathcal{F}) \le 2d,
\end{align*} and \cref{cond:general-function-class} holds with $\delta_n = \tilde{C}_1  (\log n)^{3/2} \sqrt{d/n}$ for some constant $\tilde{C}_1$ dependent on $C_2$.

\noindent {\sc Step 3. Apply \cref{thm:oracle}.} Now that we have validated \cref{cond:general-gf} with corresponding $(\mathsf{b}_{\mathcal{G}}(S), \mathsf{b}_{\mathcal{G}, \mathcal{F}}(S))$ and $g^\star = x^\top \beta^\star$, and validated \cref{cond:general-function-class} with $\delta_{n} = \tilde{C}_1  (\log n)^{3/2} \sqrt{d/n}$. Setting $t=100\log n$, we have
\begin{align*}
    \delta_{n, t} \le \tilde{C}_2 \log^{5/2} n \sqrt{d/n},
\end{align*} where $\tilde{C}_2$ is a constant dependent on $(C_1, C_2, \sigma_{\tilde{x}}, \sigma_y, C_y)$
It follows similar to \eqref{eq:approx-a-f-linear} and \eqref{eq:approx-a-g-linear} in {\sc Step 3} of \cref{sec:proof-lglf} that
\begin{align*}
    \delta_{\mathtt{a}, \mathcal{G}} + \sup_{S\subseteq [d]} \delta_{\mathtt{a}, \mathcal{F}, \mathcal{G}}(S) \le \sqrt{\tilde{C}_3 \frac{\log n}{n}}
\end{align*} for some constant $\tilde{C}_3$ that dependent on $(C_2, \sigma_y, C_y, \sigma_{\tilde{x}}, C_{\tilde{x}})$. Following a similar spirit of the remaining discussion in \cref{sec:proof-lglf} {\sc Step 3} concludes the proof of the error bounds. 

\subsection{Proof of Theorem \ref{thm:lgnf}}

\label{sec:proof-lgnf}

The proof is similar to that of \cref{thm:lglf} and \cref{thm:lgbf}. In this case, we let $\overline{\mathcal{G}_S} = \mathcal{H}_{\mathtt{lin}}(d) \cap \Theta_S$ and $\overline{\mathcal{F}_{S}} = \Theta_S$, it follows from the discussion in \cref{sec:proof-lglf} that $\overline{\mathcal{G}_S}$ is closed subspace of $\Theta_S$.

\noindent {\sc Step 1. Verify \cref{cond:general-gf}.} It is easy to see that $\Pi_{\overline{\mathcal{G}_S}}({\bar{m}^{(S)}})$ is the same as that in \cref{sec:proof-lglf}, and $\Pi_{\overline{\mathcal{F}_S}}^{(e)} (m^{(e,S)}) = m^{(e,S)}$. Then it follows from \cref{cond:lgnf-invariance} that \cref{cond:general-gf} ``invariance'' and ``nondegenerate covariate'' holds with $S^\star$ and $g^\star(x) = (\beta^\star)^\top x$ and $s_{\min} = \beta_{\min}^2$.

Moreover, $\mathsf{b}_{\mathtt{LN}}(S):=\mathsf{b}_{\mathcal{G}}(S)=\mathsf{b}_{\mathtt{LL}}(S)$, which is the same as the {\sc Step 2} in \cref{sec:proof-lglf}. At the same time, substituting our calculations of projections into the general representation of $\bar{\mathsf{d}}_{\mathcal{G}, \mathcal{F}}(S)$, we have
\begin{align*}
    \bar{\mathsf{d}}_{\mathtt{LN}}(S)=\bar{\mathsf{d}}_{\mathcal{G}, \mathcal{F}}(S) &= \frac{1}{|\mathcal{E}|} \sum_{e\in \mathcal{E}} \|m^{(e,S)} - \{(\beta^{(S)}_\dagger)^\top x_S \}\|_{2,e}^2
\end{align*}
For the inequality in \eqref{eq:biasvar-lgnf}, it follows from the fact that $\mathcal{S}:=\mathcal{H}_{\mathtt{alin}}(d,\phi) \cap \Theta_S \subseteq \overline{\mathcal{F}_S}$ is subspace and $\overline{\mathcal{G}_S}$ is also a subspace of $\mathcal{S}$ that, 
\begin{align*}
    \|m^{(e,S)} - \{(\beta^{(S)}_\dagger)^\top x_S \}\|_{2,e}^2 &= \|\Pi_{\overline{\mathcal{F}_S}}^{(e)} \{m^{(e,S)} - \Pi_{\overline{\mathcal{G}_S}} (\bar{m}^{(S)})\}\|_{2,e}^2 \\
    &\ge \|\Pi_{\mathcal{S}}^{(e)} \{m^{(e,S)} - \Pi_{\overline{\mathcal{G}_S}} (\bar{m}^{(S)})\}\|_{2,e}^2 = \|\tilde{\beta}^{(e,S)}_S - [{\beta}^{(S)}_\dagger, 0]\|_{\tilde{\Sigma}_S^{(e)}}^2,
\end{align*} summing over all the $e\in \mathcal{E}$ concludes the proof of inequality in \eqref{eq:biasvar-lgnf}.

Finally, we find
\begin{align*}
\forall S\subseteq [d] ~with~ \mathsf{b}_{\mathcal{G}}(S) > 0 \qquad \overset{(a)}{\Longrightarrow} \qquad & \sum_{e\in \mathcal{E}} \mathbb{E}[X_S^{(e)} \varepsilon^{(e)}] \neq 0\\
\overset{(b)}{\Longrightarrow} \qquad & \exists e\in \mathcal{E} ~s.t.~ \mu^{(e)}(\{m^{(e,S)}(X_S) \neq X_S^\top \beta^{(e,S)}\}) > 0  \\
 & ~~OR~~ \exists e,e'\in \mathcal{E} ~s.t.~{\beta}^{(e,S)} \neq \beta^{(e',S)} \\
\overset{(c)}{\Longrightarrow} \qquad & \bar{\mathsf{d}}_{\mathcal{G}, \mathcal{F}}(S) > 0.
\end{align*} where $(a)$ follows from the calculate of $\mathsf{b}_{\mathcal{G}}(S)$, $(b)$ follows from \cref{cond:lgnf:identification}. The derivation of $(c)$ can be divided into two cases:
\noindent \emph{Case 1. $\exists e\in \mathcal{E} ~s.t.~ \mu^{(e)}(\{m^{(e,S)}(X_S) \neq X_S^\top \beta^{(e,S)}\})>0$.} Let $e_1$ be such an environment, we have
\begin{align*}
    \mathsf{d}_S &= \frac{1}{|\mathcal{E}|} \sum_{e\in \mathcal{E}} \|m^{(e,S)} - (\beta^{(S)}_\dagger)^\top x_S \|_{2,e}^2 \ge \frac{1}{|\mathcal{E}|} \|m^{(e_1,S)} - (\beta^{(S)}_\dagger)^\top x_S\|_{2,e_1}^2 \\
    & \ge \frac{1}{|\mathcal{E}|}\|m^{(e_1,S)} - \Pi_{\overline{\mathcal{G}}}^{(e_1)} (m^{(e,S)}) + \Pi_{\overline{\mathcal{G}}}^{(e_1)} (m^{(e,S)}) - (\beta^{(S)}_\dagger)^\top x_S \|_{2,e_1}^2. \\
    &\overset{(a)}{=} \frac{\|m^{(e_1,S)} - \Pi_{\overline{\mathcal{G}}}^{(e_1)} (m^{(e,S)})\|_{2,e_1}^2 + \|\Pi_{\overline{\mathcal{G}}}^{(e_1)} (m^{(e,S)}) - (\beta^{(S)}_\dagger)^\top x_S \|_{2,e_1}^2}{|\mathcal{E}|} \overset{(b)}{>} 0.
\end{align*} Here $(a)$ follows from the projection theorem, $(b)$ follows from the fact that $\|f\|_{L_2(\nu)}>0$ if $\nu(\{f \neq 0\}) > 0$. 
\noindent \emph{Case 2. $\forall e\in \mathcal{E}, \mu^{(e)}(\{m^{(e,S)}(X_S) \neq X_S^\top \beta^{(e,S)}\})=0$.} In this case, we have
\begin{align*}
    \mathsf{d}_S = \frac{1}{|\mathcal{E}|} \sum_{e\in \mathcal{E}} \|\beta^{(e,S)}_S - {\beta}^{(S)}_\dagger\|_{\Sigma_S^{(e)}}^2.
\end{align*} Combining the condition that $\exists e,e'\in \mathcal{E} ~s.t.~{\beta}^{(e,S)} \neq \beta^{(e',S)}$ with the discussion in \cref{sec:proof-lglf} concludes that $\bar{\mathsf{d}}_{\mathcal{G},\mathcal{F}}(S)>0$ in this case. This completes the validation of \cref{cond:general-gf} ``heterogeneity''.

\noindent {\sc Step 2. Verify \cref{cond:general-function-class}. } Recall that we choose
\begin{align*}
    \mathcal{G} = \mathcal{H}_{\mathtt{lin}}(d, C_2, B) \qquad  \text{and} \qquad \mathcal{F} = \mathcal{H}_{\mathtt{nn}}(d, \log n, \log^{d} n, B)
\end{align*} with
\begin{align*}
    B = C_2 \sqrt{\log n} \ge 4(\sigma_x \lor \sigma_m \lor 1) \|\Sigma^{1/2} \beta^\star\|_2.
\end{align*} Following similarly to {\sc Step 2} in \cref{sec:proof-fairnn}, we have \cref{cond:general-function-class} holds with 
\begin{align*}
    \delta_{n} = \tilde{C}_1 \frac{\log^{d+2} n}{n}
\end{align*} with some constant $\tilde{C}_1$ dependent on $(d, C_2)$.

\noindent {\sc Step 3. Calculate the approximation error.} It follows similar to \eqref{eq:approx-a-g-linear} that
\begin{align*}
    \delta_{\mathtt{a}, \mathcal{G}} \le \tilde{C}_2 \sqrt{\frac{\log n}{n}}
\end{align*} for some constant $\tilde{C}_2$ dependent on $(C_1, C_2, \sigma_x, C_x)$. Now we turn to establish upper bounds on $\delta_{\mathtt{a}, \mathcal{F}, \mathcal{G}}^\star$ and $\delta_{\mathtt{a}, \mathcal{F}, \mathcal{G}}(S)$. Observe that $f(x) = \beta^\top x$ can be implemented via ReLU neural network with $L \ge 0$ and $N \ge 2$ because the identity map can be implemented via the neural network with depth $1$ and width $2$ (Lemma 12 (1) of \cite{fan2024factor}), then, $\forall e\in \mathcal{E}$,
\begin{align*}
    \sup_{g\in \mathcal{G}} \inf_{f\in \mathcal{F}_S} \|(\beta^\star)^\top x - g - f\|_{2,e}^2 &\le \sup_{\beta\in \mathcal{G}} \inf_{f\in \mathcal{F}} \|(\beta^\star)^\top x - \mathrm{Tc}_B(\beta^\top x) - f\|_{2,e}^2 \\
    &\le \sup_{\beta\in \mathcal{G}}  \|(\beta^\star)^\top x - \mathrm{Tc}_B(\beta^\top x) - \mathrm{Tc}_{2B}((\beta^\star - \beta)^\top x )\|_{2,e}^2 \\
    &\le \tilde{C}_2 \sqrt{\frac{\log n}{n}}
\end{align*} provided $N \ge 2$ and $L\ge 0$, where the last inequality follows similarly to \eqref{eq:approx-a-f-linear}.

Meanwhile, for any $C_m$ Lipschitz $m^{(e,S)}$, let $K = \sqrt{\log n}$. We can see that $g^\dagger(x)=m^{(e,S)}(K(2x-1))$ is $2KC_m$ Lipschitz, applying \cref{lemma:nn-approx-error} claim (1), there exists a ReLU network $g$ with depth $L-2$ and width $N-2$ such that
\begin{align*}
    \|g - g^\dagger\|_{\infty, [0,1]^d} \le 2KC_m\tilde{C}_3(NL/\log (\log n))^{-2/d} \le 2KC_m\tilde{C}_3 N^{-2/d}
\end{align*} provided $N, L\ge \tilde{C}_2(\log(\log n)+1)$ for some constant $\tilde{C}_2$ dependent on $d$, substituting our choice of $(N, L)$ gives $\|g - g^\dagger\|_{\infty, [0,1]^d} \le (\log n)^{-1/2}$ when $n$ is large enough. Let $\mathsf{S}(x):\mathbb{R}^d \to \mathbb{R}^d$ be such that $\mathsf{S}(x) = ((x_1 + K)/2K, \ldots, (x_d + K)/2K)$. Then we have
\begin{align*}
    &\left\|\mathrm{Tc}_{B}\left\{g(\mathsf{S}(X^{(e)}))\right\} - m^{(e,S)}(x)\right\|_{2,e}^2 \\
    &~~~~\le \mathbb{E} \left[\left|g(\mathsf{S}(X^{(e)})) - m^{(e,S)}(X^{(e)})\right|^2 1\left\{\|X^{(e)}\|_2 \le K, |g(\mathsf{S}(X^{(e)}))| > B\right\}\right] \\
    &~~~~~~~~~~~~ + \mathbb{E} [|B - m^{(e,S)}(X^{(e)})|^2 1\{\|X^{(e)}\|_2 \le K, |g(\mathsf{S}(X^{(e)}))| \ge B\}] \\
    &~~~~~~~~~~~~ +  \mathbb{E} [(2B^2+2|m^{(e,S)}(X^{(e)})|^2) 1\{\|X^{(e)}\|_2 > K\}] = \mathsf{T}_1 + \mathsf{T}_2 + \mathsf{T}_3
\end{align*} 
Observe that $\|x\|_{2} \le K$ implying $\|x\|_\infty \le K$. It follows from our construction of $g$ that
\begin{align}
\label{eq:approx-e1}
|g(\mathsf{S}(x)) - g^\dagger(\mathsf{S}(x))| \le \|g - g^\dagger\|_{\infty, [0,1]^d} \le \log^{-0.5} n \qquad \forall ~x:~ \|x\|_2 \le K
\end{align}
then it follows from the construction of $g^\dagger$ that
\begin{align*}
\mathsf{T}_1 &\le \mathbb{E}[|g(\mathsf{S}(X^{(e)})) - g^\dagger(\mathsf{S}(X^{(e)}))|^2 1\{\|X^{(e)}\|_2 \le K\}] \le \log^{-1} n.
\end{align*} At the same time, when $\|X^{(e)}\|_2 \le K$, it follows from triangle inequality and \eqref{eq:approx-e1} that $|m^{(e,S)}(x)| \ge B/2$ provided $n$ is large enough and $|g(\mathsf{S}(x))| \ge B$. It then follows from \cref{cond:lgnf} and \eqref{eq:sub-gaussian-integral} that
\begin{align*}
    \mathsf{T}_2 &\lesssim \mathbb{E} \left[|m^{(e,S)}(X^{(e)})|^2 1\{|m^{(e,S)}(X^{(e)})| \ge B/2\}\right] \\
    &\le \int_{0}^\infty \mathbb{P}(|m^{(e,S)}| \ge B/2\lor \sqrt{t}) dt \le \tilde{C}_4 \frac{\log n}{n}.
\end{align*} for some constant $\tilde{C}_4$ dependent on $C_m$ and $\sigma_m$. It also follows from Lipschitz condition in \cref{cond:lgnf} that
\begin{align*}
    \mathsf{T}_3 &\le \mathbb{E}\left[\left(2B^2 + 4C_m^2 + \|X^{(e)}\|_2^2\right) 1\{\|X^{(e)}\|_4^2 > K^4\}\right]\le \frac{1}{K^4} \tilde{C}_5 (B^2 + 1).
\end{align*} for some constant $\tilde{C}_5$ dependent on $(C_m, \sigma_m)$.

Putting all the pieces together, we conclude that under our choice of $(N,L)$, if $n$ is large enough, then for any $e\in \mathcal{E}$ and $S\subseteq[d]$, there exists some ReLU network $f$ with depth $L-2$ and width $N-2$ such that
\begin{align*}
\left\|\mathrm{Tc}_{B}(f(x)) - m^{(e,S)}(x)\right\|_{2,e}^2 \le \tilde{C}_6 (\log n)^{-1}.
\end{align*} for some constant $\tilde{C}_6$ dependent on $(d, C_m, \sigma_m,C_x, \sigma_x, C_2)$. Applying  \cref{lemma:approx-minus}, the function $\mathrm{Tc}_{B}(f(x)) - \mathsf{Trunc}_{B}(\beta^\top x)$ can be realized by a neural network with depth $L$ and width $N$. Using the fact
\begin{align*}
    |u - \mathrm{Tc}_B(v) - \mathrm{Tc}_{2B}(\tilde{u} - \mathrm{Tc}_{B}(v))|^2 \le 16\left\{1\{|u|\ge B\}|u|^2 + 1\{|v| \ge B\} B^2\right\} + |u - \tilde{u}|^2,
\end{align*}
we obtain
\begin{align*}
&\sup_{g\in \mathcal{G}_S: S_g = S} \inf_{f\in \mathcal{F}_S} \|m^{(e,S)} - g - f\|_{2,e}^2 \\
&~~~\le \inf_{\beta\in \mathbb{R}^d} \|m^{(e,S)} - \beta^\top x - \mathrm{Tc}_{2B}(\mathrm{Tc}_{B}(f(x)) - \mathrm{Tc}_{B}(\beta^\top x))\|_{2,e}^2 \\
&~~~\le \frac{\tilde{C}_6}{\log n} + 16 \mathbb{E}[|m^{(e,S)}(X^{(e)})|^2 1\{|m^{(e,S)}(X^{(e)})| \ge B\}] \\
&~~~~~~~~~~~ + 16 B^2 \sup_{\beta: \|{\Sigma}^{1/2} \beta\|_2 \le C_2} \mathbb{E} \left[|\beta^\top X^{(e)}|^2 1\{|\beta^\top X^{(e)}|\ge B\}\right] \\
&~~~\le\tilde{C}_7 \frac{1}{\log n},
\end{align*} for some constant $\tilde{C}_7$ dependent on $(d, C_m, \sigma_m, C_x, \sigma_x, C_2)$, which gives $\max_{S\subseteq [d]} \delta_{\mathtt{a}, \mathcal{F}, \mathcal{G}}(S) \le \sqrt{\tilde{C}_7(\log n)^{-1}}$.

\noindent {\sc Step 4. Apply \cref{thm:oracle}.} Now we apply \cref{thm:oracle} (2) with $t=100\log n$, under which $\delta_{n, t}$ can be upper bounded by
\begin{align*}
U \delta_{n, t} \le \tilde{C}_8n^{-1/2}\log^{d+3} n.
\end{align*} for some constant $\tilde{C}_8$ dependent on $(C_1, C_2, \sigma_y, C_y, \sigma_x, \sigma_m)$. Given that \cref{cond:general-gf} and \cref{cond:general-function-class} are all validated, it remains to verify \eqref{eq:main-result-faster-cond}, substituting the quantities we calculated in {\sc Step 2} and {\sc Step 3}, we have
\begin{align*}
    BU \delta_{n, t} + \max_{S\subseteq [d]} \delta^2_{\mathtt{a}, \mathcal{F}, \mathcal{G}}(S) + \delta^2_{\mathtt{a}, \mathcal{G}} = O((\log n)^{-1}) = o(1)
\end{align*} thus \eqref{eq:main-result-faster-cond} is validated with large enough $n$. Applying \cref{thm:oracle} (2) gives
\begin{align*}
    \|\hat{g} - (\beta^\star)^\top x\|_2 \le \delta_{\mathtt{a}, \mathcal{G}} + \delta_{\mathtt{a}, \mathcal{F}, \mathcal{G}}^\star + U \delta_{n, t} \lesssim \frac{\log^{d+3} n}{\sqrt{n}},
\end{align*} which completes the proof.


\section{Proofs for Population-level Results}
\label{sec:proof-population}
\subsection{Proof of Proposition \ref{prop:scm-invariant}}

The existence of all the conditional moments follows from the Radon-Nikodym theorem and the fact that for any $e\in \mathcal{E}$, $\mathbb{E}[|Y^{(e)}|] \le \mathbb{E}[|Y^{(e)}|^2] + 1 < \infty$. Moreover, it follows from the structural assignments \eqref{eq:scm-model} and the fact that $G$ is acyclic that, for any $x \in \mathbb{R}^d$ and $e\in \mathcal{E}$, 
\begin{align*}
    m^{(e,S^\star)}(x) = \mathbb{E}_{U_{d+1}\sim \nu_{d+1}}\left[f_{d+1}(x_{S^\star}, U_{d+1})\right]
\end{align*} and it is clear that the R.H.S. is independent of $e$.

\subsection{Proof of Proposition \ref{prop:transfer-learning}}

We first claim that
\begin{align}
\label{eq:proof-oos-sigma2}
    \mathsf{R}_{\mathtt{oos}}(m^\star; \nu_x) = \sigma^2.
\end{align}
On one hand, it follows from the definition of $\mathcal{U}_{S^\star, m^\star, \sigma^2}$ that for any $\mu \in \mathcal{U}_{S^\star, m^\star, \sigma^2}$ with $\mu_x \sim \nu_x$, the upper bound holds
\begin{align*}
    \mathbb{E}_{(X,Y)\sim \mu}\left[ |Y - m^\star(X)|^2 \right] = \mathbb{E}_{(X,Y)\sim \mu}\left[ |Y - \mathbb{E}[Y|X_{S^\star}]|^2 \right] = \mathrm{Var}_\mu[Y|X_{S^\star}] \le \sigma^2.
\end{align*} Taking supremum over all the $\mu$ on both sides yields $\mathsf{R}_{\mathtt{oos}}(m^\star; \nu_x) \le \sigma^2$. On the other hand, we let $\tilde{\nu}$ be the joint distribution of $(X,Y)$ with $\tilde{\nu}(d(x,y))=\nu_x(dx) \tilde{\mu}(dy|x)$ with $\tilde{\mu}(dy|x)\sim \mathcal{N}(m^\star(x_{S^\star}), \sigma^2)$, where $\mathcal{N}(\theta, \tau)$ is Gaussian distribution with mean $\theta$ and variance $\tau$. It is easy to verify that $\tilde{\nu}\in \mathcal{U}_{S^\star, m^\star, \sigma^2}$, which implies that $\mathsf{R}_{\mathtt{oos}}(m^\star; \nu_x) \ge \mathbb{E}_{(X,Y)\sim \tilde{\nu}}\left[ |Y - m^\star(X)|^2 \right] = \sigma^2$.

At the same time, for any $m\in \Theta^{(t)}$ and $\mu \in \mathcal{U}_{S^\star, m^\star, \sigma^2}$ with $\mu_x \sim \nu_x$, 
\begin{align*}
    \mathbb{E}_{(X,Y)\sim {\mu}}\left[ |Y - m(X)|^2 \right] &= \mathbb{E}_{(X,Y)\sim {\mu}}\left[ |Y - m^\star(X)|^2 \right] + \int |m(x) - m^\star(x_{S^\star})|^2 \nu_x(dx) \\
    &~~~~~~~~ - 2\mathbb{E}_{(X,Y)\sim {\mu}}\left[ \{Y - m^\star(X)\}\{m(X) - m^\star(X_{S^\star})\} \right].
\end{align*} 

We first characterize the corresponding excess risk when $\|m - \tilde{m}\|_{L_2(\nu_x)} \neq 0$. 
On one hand, we let $\breve{\nu}(dxdy) = \nu_x(dx) \breve{\nu}(dy|x)$, where $\breve{\nu}$ satisfies
\begin{align*}
    Y = m^\star(x_{S^\star}) - \frac{\sigma}{\|m - \tilde{m}\|_{L_2(\nu_x)}} \cdot (m(x) - \tilde{m}(x_{S^\star})) \
\end{align*} It is easy to verify that $\mathbb{E}_{\breve{\nu}}[Y|X_{S^\star}=x_{S^\star}] = m^\star(x_{S^\star})$, and $\mathbb{E}_{\mu}[\textrm{Var}_{\mu}(Y|X_{S^\star})] \le \sigma^2$. Therefore, $\breve{\nu}  \in \mathcal{U}_{S^\star, m^\star, \sigma^2}$, which further yields that
\begin{align}
    \mathsf{R}_{\mathtt{oos}}(m; \nu_x) &\ge \mathbb{E}_{(X,Y)\sim \breve{\nu}}\left[ |Y - m(X)|^2 \right] \nonumber \\
    &= \mathbb{E}_{(X,Y)\sim \breve{\nu}}\left[ |Y - m^\star(X)|^2 \right] + \int |m(x) - m^\star(x_{S^\star})|^2 \nu_x(dx) \nonumber \\
    &~~~~~~~~ + 2\frac{\sigma}{\|m - \tilde{m}\|_{L_2(\nu_x)}}\mathbb{E}_{X\sim {\nu}_x}\left[ \{m(X) - \tilde{m}(X_{S^\star})\}\{m(X) - m^\star(X_{S^\star})\} \right] \nonumber \\
    &\overset{(a)}{=} \sigma^2 + \|m - m^\star\|^2_{L_2(\nu_x)} \nonumber \\
    &~~~~~~~~ + 2\frac{\sigma}{\|m - \tilde{m}\|_{L_2(\nu_x)}}\mathbb{E}_{X\sim {\nu}_x}\left[ \{m(X) - \tilde{m}(X_{S^\star})\}\{m(X) - \tilde{m}(X_{S^\star})\} \right] \nonumber \\
    &\overset{(b)}{=} \mathsf{R}_{\mathtt{oos}}(m^\star; \nu_x) + \|m - m^\star\|^2_{L_2(\nu_x)} + 2\sigma \|m - \tilde{m}\|_{L_2(\nu_x)} \label{eq:proof-oos-lb}
\end{align} Here $(a)$ follows from the tower rule of conditional expectation and our construction of $\tilde{m}$ that
\begin{align*}
    &\mathbb{E}_{X\sim {\nu}_x}\left[ \{m(X) - \tilde{m}(X_{S^\star})\}\{\tilde{m}(X_{S^\star}) - m^\star(X_{S^\star})\} \right] \\
    &= \mathbb{E}_{X\sim {\nu}_x}\left[ \mathbb{E}\left[\{m(X) - \tilde{m}(X_{S^\star})\}\{\tilde{m}(X_{S^\star}) - m^\star(X_{S^\star})\} | X_{S^\star}\right]\right] \\
    &= \mathbb{E}_{X\sim {\nu}_x}\left[ \mathbb{E}\left[m(X) - \tilde{m}(X_{S^\star})|X_{S^\star}\right]\{\tilde{m}(X_{S^\star}) - m^\star(X_{S^\star})\}\right] = 0,
\end{align*} and $(b)$ follows from the claim \eqref{eq:proof-oos-sigma2}.

On the other hand, it also follows from the tower rule of conditional expectation and the fact that $\mathbb{E}_{\mu}[Y|X_{S^\star}] = m^\star(X_S^\star)$ for any $\mu \in \mathcal{U}_{S^\star, m^\star, \sigma^2}$ that
\begin{align*}
\mathbb{E}_{(X,Y)\sim {\mu}}\left[ \{Y - m^\star(X)\}\{\tilde{m}(X_{S^\star}) - m^\star(X_{S^\star})\} \right] = 0.
\end{align*} Further applying Cauchy Schwarz inequality and the fact that $\mathbb{E}_{(X,Y)\sim {\mu}}\left[ |Y - m^\star(X)|^2 \right] \le \sigma^2$, we find, the following holds for any $\mu \in \mathcal{U}_{S^\star, m^\star, \sigma^2}$ with $\mu_x \sim \nu_x$:
\begin{align*}
\mathbb{E}_{(X,Y)\sim {\mu}}\left[ |Y - m(X)|^2 \right] 
&= \mathbb{E}_{(X,Y)\sim {\mu}}\left[ |Y - m^\star(X)|^2 \right] + \int |m(x) - m^\star(x_{S^\star})|^2 \nu_x(dx) \\
&~~~~~~~~ - 2\mathbb{E}_{(X,Y)\sim {\mu}}\left[ \{Y - m^\star(X)\}\{m(X) - \tilde{m}(X_{S^\star})\} \right]\\
&\le \sigma^2 + \|m - m^\star\|_{L_2(\nu_x)}^2 + 2\sigma\|m - \tilde{m}\|_{L_2(\nu_x)}
\end{align*} Taking supremum on both sides over $\mu$ and substituting \eqref{eq:proof-oos-sigma2} in, we conclude that
\begin{align*}
\mathsf{R}_{\mathtt{oos}}(m; \nu_x) \le \mathsf{R}_{\mathtt{oos}}(m^\star; \nu_x) + \|m - m^\star\|_{L_2(\nu_x)}^2 + 2\sigma\|m - \tilde{m}\|_{L_2(\nu_x)}.
\end{align*} Combining it with the lower bound \eqref{eq:proof-oos-lb} completes the proof.

\subsection{Preliminaries for the Proofs in Section \ref{sec:ident-fairnn}}

We will use the terminology ``path'', ``blocked'', ``collider'', ``chain'', ``fork'', ``child'', ``parent'', ``descendant'', ``ancestor'' under the SCM framework. See a formal definition in \cite{glymour2016causal}.

We will repeatedly use the fact that
\begin{align}
\label{eq:scm-ident-fact1}
\begin{split}
&\mathbb{E}[Y|X_S, E] = \mathbb{E}[Y|X_S] \\
&\qquad \Longleftrightarrow\qquad (\mu^{(e)}\land \mu^{(e')})(\{m^{(e,S)}(X)\neq m^{(e',S)}(X)\})=0, ~~\forall e,e'\in \mathcal{E}
\end{split}
\end{align} and
\begin{align}
\label{eq:scm-ident-fact2}
\mathbb{E}[Y|X_S] = \bar{m}^{(S)},
\end{align} where the distribution of the random vector $(X,Y,E)=(X_1,\ldots, X_d, Y, E)=(Z_1,\ldots, Z_d, Z_{d+1}, Z_{d+2})$ is determined by $\tilde{M}$. The above facts \eqref{eq:scm-ident-fact1} and \eqref{eq:scm-ident-fact2} follow from the data generating process of $\tilde{M}$.

It is easy to verify that $\tilde{G}$ is also a DAG when $G$ is a DAG, this is because one can do topological sorting for $\tilde{G}$ by first choosing $E$ and then doing topological sorting for $G$.

\subsection{Proof of Theorem \ref{prop:ident-transfer-learning}}

We first show that 
\begin{align}
\label{eq:invariance1}
    E \indep Y | X_{S_\star}
\end{align} 
Because $\tilde{G}$ is a DAG, it suffices to show that $E \indep_{\tilde{G}} Y | X_{S_\star}$, i.e., all the paths connecting $Y$ and $E$ are blocked by $X_{S_\star}$ by the definition of $d$-separation. Let $p$ be a path connecting $Y$ and $E$. We divide it into two cases. 

\noindent \emph{Case 1. the path contains $Y$'s parent, i.e., $E \to \cdots X_j \to Y$.} In this case, the path is blocked by some fork or chain ``$\gets X_j \to Y$'' or ``$\to X_j \to Y$'' because $j\in \mathtt{pa}(d+1) \subseteq S_\star$. 

\noindent \emph{Case 2. the path contains $Y$'s child, i.e., $E \to \cdots X_j \gets Y$. } We consider the two sub-cases: (1) $j\in S_\star \setminus \mathtt{pa}(d+1)$ and $j$ is $Y$'s child or (2) $j$ is $Y$'s child but not in $S_\star$. 

For the first case (1), the path in the right can either be (1a) $X_k \gets X_j\gets Y$, under which the path is blocked by a chain centered by $j\in S^\star$; or (1b) $X_k \to X_j \gets Y$. 

For the case (1b), if $X_k$ is also $Y$'s child, then $X_k \in A(I)$ by the definition of $A(I)$, and verifying the path $p$ is blocked can be reduced to verifying whether the path $Y\to X_k \cdots \gets E$ that removes $X_j$ in $p$ is blocked (and there is no cyclic proof here because the length of the path will be deducted by 1). This can be done by a proof-by-induction argument. On the other hand, if $X_k$ is not $Y$'s child, then we have $k\in \mathtt{pa}(j) \setminus \{d+1\} \setminus \mathtt{ch}(d+1) \subseteq \mathtt{pa}(j) \setminus \{d+1\} \setminus A(I) \subseteq S_\star$. Therefore, the path $p$ is blocked by a fork ($X_j \gets X_k \to$) or chain ($X_j \gets X_k \gets$) centered at $k\in S_\star$.

For the case (2), the path can be written as $Y\to \underbrace{\cdots}_{(i)} \to X_k \gets \cdots \gets E$ where there is no $\gets$ in (i). We claim that $X_k$ and all its descendants are not in $S_\star$ in this case thus the path is blocked by a collider. We prove that claim using proof by contradiction. If there exists $X_a$ such that it is either $X_k$ or $X_k$'s descendant and $a\in S_\star$, then $X_a$ must be either in (2a) $A(I)$, or (2b) a parent of some $X_{\tilde{a}}$ in $A(I)$. 

For (2a), it is contrary to the definition of $A(I)$ because $X_a$ in this case has an ancestor $X_j$, who is $Y$'s child but not in $A(I)$ because $X_j$ also has an ancestor $X_b$ such that $X_b$ is $Y$'s child and $b\in I$. Therefore, $X_a$ has an ancestor $X_b$ such that $X_b$ is $Y$'s child and $b\in I$, which is contrary to the definition of $A(I)$. 

For (2b), it has a child $X_{\tilde{a}}$ such that $\tilde{a} \in A(I)$, this is also contrary to the definition of $A(I)$, because $X_{\tilde{a}}$ now also has an ancestor $X_j$, the same argument in (2a) applies. 

Putting these detailed discussions together, we can conclude \eqref{eq:invariance1} holds, which further implies that
\begin{align*}
    m_\star := \bar{m}^{(S_\star)}(x) \equiv m^{(e,S_\star)}. 
\end{align*}

We need the following lemma to characterize the topology of the set $S$ when $\mathsf{b}_{\mathtt{NN}}(S) > 0$.

\begin{lemma}
\label{lemma:11} 
If $\mathsf{b}_{\mathtt{NN}}(S) > 0$, then $S  \cap \cup_{j\in \mathtt{ch}(d+1) \setminus A(I)} \mathtt{dt}(j) \neq \emptyset$, where $\mathtt{dt}(k)$ refers to the set of descendants of variable $Z_k$ with $k\in [d+1]$, and $\mathtt{ch}(k)$ refers to the set of children of variable $Z_k$ with $k\in [d+1]$.
\end{lemma}

\begin{proof}[Proof of \cref{lemma:11}]
We use proof by contradiction argument. To be specific, we show that if $S$ doesn't contain any variables in $\cup_{j\in \mathtt{ch}(d+1) \setminus A(I)} \mathtt{dt}(j)$, then
\begin{align}
\label{eq:nobias1}
    Y \indep_{\tilde{G}} X_{S\setminus S_\star}|X_{S_\star} \qquad \Longrightarrow \qquad \mathbb{E}[Y|X_{S_\star \cup S}] = \mathbb{E}[Y|X_{S_\star}]
\end{align} which further implies that $\mathsf{b}_{\mathtt{NN}}(S) = 0$. Hence it remains to establish $Y \indep_{\tilde{G}} X_{S\setminus S_\star}|X_{S_\star}$, i.e., all the paths $p$ connecting $Y$ and $X_{S\setminus S_\star}$ is blocked by $X_{S_\star}$. 

Let $Y$ be any path connecting $Y$ and $X_j$ with $j\in S \setminus S_\star$. We can divide $p$ into three cases: (1) the path contains one of $Y$'s parent, or (2) the path contains one of $Y$'s children in $A(I)$, or (3) the path contains one of $Y$'s children not in $A(I)$. For case (1), the path is blocked by some fork or chain in the path containing $X_k$ with $k\in \mathtt{pa}(d+1) \subseteq S_\star$. 

Turning to (2), the path will be either (2a) $Y\to X_k \to X_\ell \cdots X_j$ or (2b) $Y\to X_k \gets X_\ell \cdots X_j$. For (2a), such a path is blocked by the chain $\to X_k\to $ with $k\in S_\star$ by the assumption that $X_k \in A(I) \subseteq S_\star$. For (2b), such a path is blocked by fork $\gets X_\ell \to$ or chain $\gets X_\ell \gets$ with $X_\ell \in \mathtt{pa}(k) \subseteq S_\star$. 

For last case (3), the path can be written as $Y\to X_k \to \underbrace{\cdots}_{(i)} \to X_\ell \gets \cdots \gets X_j$ where there is no $\gets$ in $(i)$ and $k \in \mathtt{ch}(d+1) \setminus A(I)$, otherwise $j$ will be in $\cup_{j\in \mathtt{ch}(d+1) \setminus A(I)} \mathtt{dt}(j)$. We claim that $X_\ell$ and all its descendants are not in $S_\star$. We prove such a claim using a proof-by-contradiction argument. To be specific, if there exists some $X_a$ who is the descendant of $X_\ell$, which is also the descendant of $X_k$ such that $X_a \in S_\star$, then either $X_a$ and one of $X_a$'s child is lying in $A(I)$. Let be either $a$ or one of $X_a$'s children such that $\tilde{a} \in A(I)$. It follows from the definition of $A(I)$ that all the ancestors of $\tilde{a}$ who is also $Y$'s child cannot be in $I$. At the same time, because $X_k$ is $Y$'s child and $X_k \notin A(I)$, then $k$ has one ancestor $b$ who is also $Y$'s child and is not in $I$. Note that $b$ is also $\tilde{a}$'s ancestor. This is contrary to the fact that $\tilde{a} \in A(I)$. 

\end{proof}

Now we are ready to validate \cref{cond-fairnn-ident} with $S^\star = S_\star$. If $\mathsf{b}_{\mathtt{NN}}(S) > 0$, then by \cref{lemma:11}, there exists some $j\in S$ such that $j$ is the descendant of $Y$'s child $o$ such that $o \notin A(I)$. It follows from the definition of $A(I)$ that $o$ has an ancestor $k$ such that $k\in \mathtt{ch}(d+1)$ and $k\in I$. Then $j$ is the descendant of $Y$'s child $k$ with $k\in I$. Therefore, we have the path $E\to X_k \gets Y$ is not blocked by $S$ because $j\in S$ and $j$ is either $k$ or $k$'s descendant, which further implies that
\begin{align*}
    E \notindep_{\tilde{G}} Y | X_S
\end{align*}
It then follows from conditions (b) and (a) in \cref{cond:expected-faithfulness} that
\begin{align*}
    \exists e, e'\in \mathcal{E}, ~~~(\mu^{(e')} \land \mu^{(e)}) \left( \{ m^{(e,S)} \neq m^{(e',S)} \}\right) > 0.
\end{align*} This verifies \cref{cond-fairnn-ident}.

\subsection{Proof of Proposition \ref{prop:rtl}}

\begin{proof}[Proof of \cref{prop:rtl-full}]
The proof is almost identical to the proof of \cref{prop:ident-transfer-learning}. We only highlight the differences.

It follows similarly to the proof of the claim \eqref{eq:invariance1} that 
\begin{align*}
    Y \indep_{\bar{M}} E | X_{S_\star}
\end{align*} holds in the new graph, this implies that $\mathbb{E}[Y^{(t)}|X_{S_\star}^{(t)}] = \mathbb{E}[Y^{(0)}|X_{S_\star}^{(0)}]$.

It follows similar to the proof of \cref{lemma:11} that the following claim holds: if $\mathbb{E}[Y^{(t)}|X^{(t)}_{S\cup S_\star}] \neq \mathbb{E}[Y^{(0)}|X^{(0)}_{S_\star}]$, then $S$ contains $\cup_{j\in \mathtt{ch}(d+1) \setminus A(I)} \mathtt{dt}(j)$. 

Let $j\in S$ be such that $j$ is the descendant of $Y$'s child $o$ such that $o \notin A(I)$. It follows from the definition of $A(I)$ that $o$ has an ancestor $k$ such that $k\in \mathtt{ch}(d+1)$ and $k\in I$. Then $j$ is the descendant of $Y$'s child $k$ with $k\in I$. Therefore, we have the path $E\to X_k \gets Y$ is not blocked by $S$ because $j\in S$ and $j$ is either $k$ or $k$'s descendant, which further implies that
\begin{align*}
    E \notindep_{\bar{G}} Y | X_S
\end{align*} It then follows from \cref{cond:expected-faithfulness-2} that
\begin{align*}
    m^{(0,S)} \neq m^{(t, S)}.
\end{align*} This completes the proof.

\end{proof}

\subsection{Proof of Proposition \ref{prop:ident-causal-discovery}}


Our first lemma characterizes the topology of the set $S$ when $\mathsf{b}_{\mathtt{NN}}(S) > 0$.
\begin{lemma}
\label{lemma:1} 
If $\mathsf{b}_{\mathtt{NN}}(S) > 0$, then $S\text{ contains }Y\text{'s descendants}$. 
\end{lemma}


\begin{proof}[Proof of \cref{lemma:1}]
To prove \cref{lemma:1}, we use the proof by contradiction argument. To be specific, we show that if $S$ does not contain any descendant of $Y$, then $\mathsf{b}_{\mathtt{NN}}(S)=0$. Let $p$ be any path connecting $Y$ and $X_j$ with $j\in S \setminus S^\star$. We can divide $p$ into two cases: (1) the path contains one of $Y$'s parents, or (2) the path contains one of $Y$'s children. For the first case, this path is blocked by some fork or chain in the path containing $X_k$ with $k\in S^\star$. For the second case, the path will be $Y\to \underbrace{\cdots}_{(i)} \to X_k \gets \cdots \gets X_j$, where there is no $\to$ in $(i)$. Because $X_j$ is not the descendant of $Y$, then the path contains a collider and is also blocked by $S^\star$. Since all the paths are blocked by $S^\star$, we can conclude that
\begin{align*}
    Y \indep_{\tilde{G}} X_{S\setminus S^\star}|X_{S^\star} \qquad \Longrightarrow \qquad \mathbb{E}[Y|X_{S^\star \cup S}] = \mathbb{E}[Y|X_{S^\star}].
\end{align*} when $S$ does not contain descendants of $Y$. This further implies that $\mathsf{b}_{\mathtt{NN}}(S)=0$ by definition.
\end{proof}

Now we are ready to prove the ``if'' direction. Let $S\subseteq [d]$ be some set such that $\mathsf{b}_{\mathtt{NN}}(S) > 0$. It then follows from \cref{lemma:1} that $S$ contains at least one descendant of $Y$. We argue that in this case, one has
\begin{align}
\label{eq:scm-ident-if}
    Y \notindep_{\tilde{G}} ~E | X_S.
\end{align} 
To this end, we only need to show that there exists a path connecting $Y$ and $E$ and is not blocked by $X_S$. Let $j \in S$ such that $X_j$ is the descendant of $Y$.  If $j\in I^\star$, then it is obvious that the path $Y\to X_j \gets E$ is not blocked by $X_S$ because $j\in S$. Otherwise, there must exist some node $k$ in $I^\star$ such that $X_k$ is the ancestor of $X_j$ by the definition of $I^\star$. Therefore, we can see that the path $Y\to X_k \gets E$ is not blocked by $X_S$ since $j\in S$. This completes the proof of the claim. It then follows from the faithfulness of $\tilde{G}$ (\cref{cond:expected-faithfulness} (2)) that
\begin{align*}
    Y \notindep E | X_S.
\end{align*} This validates \cref{cond-fairnn-ident} because of \cref{cond:expected-faithfulness} (a).

We then prove the ``only if'' direction using proof by contradiction. Specifically, we will show that if there exists some $j \in I^\star$ such $j\notin I$, then \cref{cond-fairnn-ident} will violate. Define $\overline{S} = S^\star \cup \{j\} \cup (\mathtt{pa}(j) \setminus \{d+1\})$. We argue that if $j \in I^\star$ and $j\notin I$, then
\begin{align}
    Y \indep_{\tilde{G}} E|X_{\overline{S}}.
\label{eq:scm-ident-onlyif}
\end{align} This further yields that $\mathbb{E}[Y|X_{\overline{S}}, E] = \mathbb{E}[Y|X_{\overline{S}}]$ and hence $m^{(e,\overline{S})} \equiv \bar{m}^{(\overline{S})}$, i.e., $\mathsf{d}_{\mathtt{NN}}(\overline{S}) = 0$, by \eqref{eq:scm-ident-fact1}. At the same time, because $j$ is the child of $Y$, we also have $\mathsf{b}_{\mathtt{NN}}(\overline{S})>0$ by our assumption that $m^\star \neq \bar{m}^{(\overline{S})}$. Thus \cref{cond-fairnn-ident} is violated. This completes the proof of the ``only if'' direction.

\noindent {\it Proof of \eqref{eq:scm-ident-onlyif}.} To establish \eqref{eq:scm-ident-onlyif}, we need to show that every path $p$ connecting $Y$ and $E$ is blocked by $X_{\overline{S}}$. Let $p=(E,X_\ell, \ldots, X_k, Y)$ be a path connecting $Y$ and $E$. Then our discussions will be divided into two cases: (1) $X_k$ is the parent of $Y$, or (2) $Y$ is the parent of $X_k$. For the first case, we have $k\in S^\star$ by the definition of $S^\star$ and the tuple containing the last three elements in the path is a fork or chain because the cause-effect relationship $X_k \to Y$. This means this path is blocked by the node $k\in S^\star \subseteq \overline{S}$. For the second case, we also consider two sub-cases: (2a) $k\neq j$ and (2b) $k= j$. For (2a), let the path be $Y\to X_k \to \underbrace{\cdots}_{(i)} \to X_\ell \gets \cdots \gets E$ where there is no ``$\gets$'' in $(i)$, it is easy to see that there is a collider centered around $X_\ell$ in the path. Note that $X_\ell$ and all its descendants cannot be $X_j$ by $j\in I^\star$, that is, $X_j$ does not have an ancestor that is also a child of $Y$. Combining this with the fact that $\tilde{G}$ is DAG and $S^\star$ is the set of parents of $Y$, we find that $X_\ell$ and all its descendants do not belong to $\overline{S}$. Therefore, the path is blocked in this case. For (2b), it is obvious that the length of the path is greater than $3$ because $j \notin I$. We also divide it into two cases: $Y\to X_k \to \underbrace{\cdots}_{(i)} \to X_\ell \gets \cdots \gets E$ where there is no ``$\gets$'' in $(i)$, or $Y\to X_j \gets X_\ell \cdots \gets E$. For the former, the path is blocked by a collider centered around $X_\ell$ since $X_\ell$ and all its descendants do not belong to $\overline{S}$. For the latter, the path is blocked by a fork/chain centered around $X_\ell$ provided $\ell \in \overline{S}$. Therefore, we can conclude that the path is also blocked in the case of (2b). This completes the proof. \qed

\subsection{Proof of Theorem \ref{thm:causal-identification-hc}}

We first show that 
\begin{align}
\label{eq:invariance1-hc}
    E \indep Y | X_{S_\star}
\end{align} 
Because $\tilde{G}$ is a DAG, it suffices to show that $E \indep_{\tilde{G}} Y | X_{S_\star}$, i.e., all the paths connecting $Y$ and $E$ are blocked by $X_{S_\star}$ by the definition of $d$-separation. Let $p$ be a path connecting $Y$ and $E$. We divide it into two cases. 

\noindent \emph{Case 1. the path contains $Y$'s parent (but is not $H$), i.e., $E \to \cdots X_j \to Y$.} In this case, the path is blocked by some fork or chain ``$\gets X_j \to Y$'' or ``$\to X_j \to Y$'' because $j\in \mathtt{pa}(d+1) \subseteq S_\star$. 
\medskip

\noindent \emph{Case 2.} When the parent is $H$ and it is $E\to \cdots X_j\to H \to Y$. In this case, the path is blocked by some fork or chain ``$\gets X_j \to Y$'' or ``$\to X_j \to Y$'' because $j\in \mathtt{pa}(d+2) \subseteq S_\star$. 
\medskip

\noindent \emph{Case 3.} When the parent is $H$ and it is $E\to \cdots X_j\gets H \to Y$. We consider the two sub-cases: (1) $j$ is $H$'s child and $j\in S_\star$ or (2) $j$ is $H$'s child but not in $S_\star$. 

For the first case (1). Given $j$ is the child of $H$, $j$ must be in $A(I)$: $j$ cannot be in $\mathtt{pa}(d+2)$ by the acyclic property of the graph, it also cannot be in $\mathtt{pa}(d+1) \setminus \{d+2\}$ provided $d+2\notin \cup_{j\in \mathtt{pa}(d + 1) \setminus \{d+2\}} \mathtt{at}(j)$. Given $j$ is the child of $H$, $j$ cannot be in $\bigcup_{j\in A(I)} \left(\mathtt{pa}(j) \setminus \{d+1, d+2\}\right) \setminus A(I)$ because the definition of $A(I)$ implies that if $j$ is a child of $d+2$ and $j$ is not in $A(I)$, then its children are also not in $A(I)$.

It should be noted that the path cannot be $E\to X_j \gets H \to Y$ by the the fact that $j\in A(I)$. Therefore, the path in the right can either be (1a) $X_k \gets X_j \gets H \to Y$, under which the path is blocked by a chain $X_k\gets X_j \gets H$ centered by $j\in S^\star$; or (1b) $X_k \to X_j \gets H \to Y$. 

For (1b), if $k$ is not in $A(I)$, then we have 
\begin{align*}
    k\in \mathtt{pa}(j) \setminus \{d+2\} \subseteq S_\star
\end{align*} given $j \in S_\star$ and $j$ is a child of $H$. Therefore, the path $p$ is blocked by a fork ($X_j \gets X_k \to$) or chain ($X_j \gets X_k \gets$) centered at $k\in S_\star$.

Also under (1b), if $k \in A(I)$, which indicates that $k$ is either a child of $Y$, or a child of $H$. In this case, verifying the path $p$ is blocked can be reduced to verifying whether the path $Y\to X_k \cdots \gets E$ or $Y \gets H\to X_k \cdots \gets E$ that removes $X_j$ in $p$ is blocked (and there is no cyclic proof here because the length of the path will be deducted by 1). This can be done by a proof-by-induction argument and combining arguments in a similar case in Case 4. 

For the case (2), the path can be written as $Y\gets H \to \underbrace{\cdots}_{(i)} \to X_k \gets \cdots \gets E$ where there is no $\gets$ in (i). We prove that claim using proof by contradiction. If there exists $X_a$ such that it is either $X_k$ or $X_k$'s descendant and $a\in S_\star$, then $X_a$ must be either in (2a) $A(I)$, or (2b) a parent of some $X_{\tilde{a}}$ in $A(I)$ given the graph is acyclic and $d+2$ is not the ancestor of $\mathtt{pa}(d+1)\setminus \{d+2\}$. 

For (2a), it is contrary to the definition of $A(I)$ because $X_a$ in this case has an ancestor $X_j$, who is $H$'s child but not in $A(I)$ because $X_j$ also has an ancestor $X_b$ such that $X_b$ is either $Y$'s child or $H$'s child and $b\in I$. Therefore, $X_a$ has an ancestor $X_b$ such that $X_b$ is $Y$'s or $H$'s child and $b\in I$, which is contrary to the definition of $A(I)$. 

For (2b), it has a child $X_{\tilde{a}}$ such that $\tilde{a} \in A(I)$, this is also contrary to the definition of $A(I)$, because $X_{\tilde{a}}$ now also has an ancestor $X_j$, the same argument in (2a) applies. 

\medskip

\noindent \emph{Case 4. the path contains $Y$'s child, i.e., $E \to \cdots X_j \gets Y$. } The discussion in this case is very similar to that in Case 3. We consider the two sub-cases: (1) $j$ is $Y$'s child and $j\in S_\star$ or (2) $j$ is $Y$'s child but not in $S_\star$. 

For the first case (1), we first can verify that $j\in A(I)$ (similar to (1) in Case 3), then the path in the right can either be (1a) $X_k \gets X_j\gets Y$, under which the path is blocked by a chain centered by $j\in S^\star$; or (1b) $X_k \to X_j \gets Y$. 

For the case (1b), if $k$ is not in $A(I)$, then we have $k\in \mathtt{pa}(j) \setminus \{d+2\} \subseteq S_\star$. Therefore, the path $p$ is blocked by a fork ($X_j \gets X_k \to$) or chain ($X_j \gets X_k \gets$) centered at $k\in S_\star$. On the other hand, if $k\in A(I)$, verifying the path $p$ is blocked can be reduced to verifying whether the path $Y\to X_k \cdots \gets E$ or $Y \gets H\to X_k \cdots \gets E$ that removes $X_j$ in $p$ is blocked (and there is no cyclic proof here because the length of the path will be deducted by 1). This can be done by a proof-by-induction argument and combining arguments in a similar case in Case 3. 

For the case (2), the path can be written as $Y\to \underbrace{\cdots}_{(i)} \to X_k \gets \cdots \gets E$ where there is no $\gets$ in (i). We claim that $X_k$ and all its descendants are not in $S_\star$ in this case thus the path is blocked by a collider, and the proof is identical to the (2) in Case 3. 

Putting these detailed discussions together, we can conclude \eqref{eq:invariance1-hc} holds, which further implies that
\begin{align*}
    m_\star := \bar{m}^{(S_\star)}(x) \equiv m^{(e,S_\star)}. 
\end{align*}

We need the following lemma to characterize the topology of the set $S$ when $\mathsf{b}_{\mathtt{NN}}(S) > 0$.

\begin{lemma}
\label{lemma:11-hc} 
If $\mathsf{b}_{\mathtt{NN}}(S) > 0$, then $S  \cap \cup_{j\in \{\mathtt{ch}(d+1) \cup \mathtt{ch}(d+2)\}\setminus A(I)} \mathtt{dt}(j) \neq \emptyset$, where $\mathtt{dt}(k)$ refers to the set of descendants of variable $Z_k$ with $k\in [d+1]$, and $\mathtt{ch}(k)$ refers to the set of children of variable $Z_k$ with $k\in [d+2]$.
\end{lemma}

\begin{proof}[Proof of \cref{lemma:11-hc}]
We use proof by contradiction argument. To be specific, we show that if $S$ doesn't contain any variables in $\cup_{j\in \{\mathtt{ch}(d+1) \cup \mathtt{ch}(d+2)\} \setminus A(I)} \mathtt{dt}(j)$, then
\begin{align}
    Y \indep_{\tilde{G}} X_{S\setminus S_\star}|X_{S_\star} \qquad \Longrightarrow \qquad \mathbb{E}[Y|X_{S_\star \cup S}] = \mathbb{E}[Y|X_{S_\star}]
\end{align} which further implies that $\mathsf{b}_{\mathtt{NN}}(S) = 0$. Hence it remains to establish $Y \indep_{\tilde{G}} X_{S\setminus S_\star}|X_{S_\star}$, i.e., all the paths $p$ connecting $Y$ and $X_{S\setminus S_\star}$ is blocked by $X_{S_\star}$. 

Let $Y$ be any path connecting $Y$ and $X_j$ with $j\in S \setminus S_\star$. We can divide $p$ into three cases: (1) the path contains one of $Y$'s parent (but not in $H$), or one of $H$'s parent; (2) the path contains one of $Y$'s children in $A(I)$, or $H$ and one of $H$'s children in $A(I)$; (3) the path contains one of $Y$'s children not in $A(I)$, or $H$ and one of $H$'s children not in $A(I)$. For case (1), the path is blocked by some fork or chain in the path containing $X_k$ with $k\in \mathtt{pa}(d+1) \setminus \{d+2\} \cup \mathtt{pa}(d+2) \subseteq S_\star$. 

Turning to (2), if the path starts with $Y\to X_k$, then the path will be either (2a) $Y\to X_k \to X_\ell \cdots X_j$ or (2b) $Y\to X_k \gets X_\ell \cdots X_j$. For (2a), such a path is blocked by the chain $\to X_k\to $ with $k\in S_\star$ by the assumption that $X_k \in A(I) \subseteq S_\star$. For (2b), such a path is blocked by fork $\gets X_\ell \to$ or chain $\gets X_\ell \gets$ with $X_\ell \in \mathtt{pa}(k) \subseteq S_\star$. The discussion under which the path starts with $Y\gets H\to X_k$ follows similarly. 

For the last case (3), if the path starts with $Y\to X_k$, then the path can be written as $Y\to X_k \to \underbrace{\cdots}_{(i)} \to X_\ell \gets \cdots \gets X_j$ where there is no $\gets$ in $(i)$ and $k \in \mathtt{ch}(d+1) \setminus A(I)$, otherwise $j$ will be in $\cup_{j\in \{\mathtt{ch}(d+1) \cup \mathtt{ch}(d+2)\} \setminus A(I)} \mathtt{dt}(j)$. We claim that $X_\ell$ and all its descendants are not in $S_\star$. We prove such a claim using a proof-by-contradiction argument. To be specific, if there exists some $X_a$ who is the descendant of $X_\ell$, which is also the descendant of $X_k$ such that $X_a \in S_\star$, then either $X_a$ and one of $X_a$'s child is lying in $A(I)$. Let be either $a$ or one of $X_a$'s children such that $\tilde{a} \in A(I)$. It follows from the definition of $A(I)$ that all the ancestors of $\tilde{a}$ who is also either $Y$'s or $H$'s child cannot be in $I$. At the same time, because $X_k$ is $Y$'s child and $X_k \notin A(I)$, then $k$ has one ancestor $b$ who is also $Y$'s child and is not in $I$. Note that $b$ is also $\tilde{a}$'s ancestor. This is contrary to the fact that $\tilde{a} \in A(I)$. The discussion under which the path starts with $Y\gets H\to X_k$ follows similarly.

\end{proof}

Now we are ready to validate \cref{cond-fairnn-ident} with $S^\star = S_\star$. If $\mathsf{b}_{\mathtt{NN}}(S) > 0$, then by \cref{lemma:11-hc}, there exists some $j\in S$ such that $j$ is the descendant of either $Y$'s or $H$'s child $o$ such that $o \notin A(I)$. It follows from the definition of $A(I)$ that $o$ has an ancestor $k$ such that $k\in \mathtt{ch}(d+1) \cup \mathtt{ch}(d+2)$ and $k\in I$. Then $j$ is the descendant of $Y$'s (or $H$'s) child $k$ with $k\in I$. Therefore, we have the path $E\to X_k \gets Y$ is not blocked by $S$ because $j\in S$ and $j$ is either $k$ or $k$'s descendant, which further implies that
\begin{align*}
    E \notindep_{\tilde{G}} Y | X_S
\end{align*}
It then follows from conditions (b) and (a) in \cref{cond:expected-faithfulness-3} that
\begin{align*}
    \exists e, e'\in \mathcal{E}, ~~~(\mu^{(e')} \land \mu^{(e)}) \left( \{ m^{(e,S)} \neq m^{(e',S)} \}\right) > 0.
\end{align*} This verifies \cref{cond-fairnn-ident}.